\documentclass[11pt,reqno]{amsart}

\usepackage{color}
\usepackage{amsmath, amsthm, amssymb}
\usepackage{amsfonts}
\usepackage{mathrsfs}

\usepackage{hyperref}

\usepackage[ansinew]{inputenc}
\usepackage[dvips]{epsfig}
\usepackage{graphicx}
\usepackage[english]{babel}

\theoremstyle{plain}
\newtheorem{theorem}{Theorem}[section]
\newtheorem{corollary}[theorem]{Corollary}
\newtheorem{lemma}[theorem]{Lemma}
\newtheorem{proposition}[theorem]{Proposition}

\theoremstyle{definition}
\newtheorem{definition}[theorem]{Definition}

\theoremstyle{remark}
\newtheorem{remark}[theorem]{Remark}

\numberwithin{equation}{section}

\newcommand{\average}{{\mathchoice {\kern1ex\vcenter{\hrule height.4pt
width 6pt depth0pt} \kern-9.7pt} {\kern1ex\vcenter{\hrule
height.4pt width 4.3pt depth0pt} \kern-7pt} {} {} }}
\newcommand{\ave}{\average\int}
\def\R{\mathbb{R}}

\newcommand{\ep}{\varepsilon}
\newcommand{\HH}{\mathcal H}
\newcommand{\N}{\mathbb N}

\newcommand{\anz}{\mathscr P}

\newcommand{\Xalpha}{\boldsymbol{X}_\alpha}

\newcommand{\place}[2]{%
  \overset{\substack{#1\\\smile}}{#2}%
}

\begin{document}

 \title[Generic regularity of free boundaries for the obstacle problem]{Generic regularity of free boundaries  \\for the obstacle problem}

\author{Alessio Figalli}
\address{ETH Z\"urich, Department of Mathematics, Raemistrasse 101, 8092 Z\"urich, Switzerland}
\email{alessio.figalli@math.ethz.ch}

\author{Xavier Ros-Oton}
\address{Universit\"at Z\"urich,
Institut f\"ur Mathematik, 
Winterthurerstrasse 190, 8057 Z\"urich, Switzerland \& 
ICREA, Pg.\ Llu\'is Companys 23, 08010 Barcelona, Spain \& 
Universitat de Barcelona, Departament de Matem\`atiques i Inform\`atica, Gran Via de les Corts Catalanes 585, 08007 Barcelona, Spain.}
\email{xavier.ros-oton@math.uzh.ch}

\author{Joaquim Serra}
\address{ETH Z\"urich, Department of Mathematics, Raemistrasse 101, 8092 Z\"urich, Switzerland}
\email{joaquim.serra@math.ethz.ch}

\thanks{AF and JS have received funding from the European Research Council (ERC) under the Grant Agreement No 721675.
XR was supported by the European Research Council (ERC) under the Grant Agreement No 801867.
JS was supported by Swiss NSF Ambizione Grant PZ00P2 180042.
XR and JS were supported by MINECO grant MTM2017-84214-C2-1-P (Spain).\\
The authors are grateful to the anonymous referee for a careful reading and the useful comments.}


\begin{abstract}
The goal of this paper is to establish generic regularity of free boundaries for the  obstacle problem in $\R^n$.
By classical results of Caffarelli, the free boundary is $C^\infty$ outside a set of singular points.
Explicit examples show that the singular set could be in general $(n-1)$-dimensional ---that is, as large as the regular set.
Our main result establishes that, generically, the singular set has zero $\mathcal H^{n-4}$ measure (in particular, it has codimension 3 inside the free boundary).
Thus, for $n\leq4$, the free boundary is generically a $C^\infty$ manifold.
This solves a conjecture of Schaeffer (dating back to 1974) on the generic regularity of free boundaries in dimensions $n\leq4$.
\end{abstract}

\maketitle

\tableofcontents

\section{Introduction}
\label{sect:Intro}

Several fundamental problems in science (physics, biology, finance, geometry, etc.) can be described by PDEs that exhibit a-priori unknown  interfaces or boundaries.
They are called \textit{free boundary problems}, and have been a major line of research in the PDE community in the last 60 years; see for instance \cite{LS,LS69,Kin73,BK74,KN77,C-obst,CR77,Sak91,C-obst2,W99,CKS00,M03,SU,ACS08,GP09,ALS,AlessioJoaquim}.

The obstacle problem
\begin{equation} \label{obstacle}
\begin{split}
\Delta u &= \chi_{\{u>0\}}\quad \textrm{in}\quad \Omega\subset\R^n \\
u & \geq 0,
\end{split}
\end{equation}
is the most classical and among the most important elliptic free boundary problems, and it arises in a variety of situations; see e.g. \cite{DL,obst-appl4,obst-appl1,PSU12,Serfaty}.

From the mathematical point of view, the most challenging question in this context is to understand the \emph{regularity of free boundaries}.
The modern development of the regularity theory for free boundaries started in the late 1970's with the seminal paper of Caffarelli  \cite{C-obst}, and since then it has been a very active area of research.

The main result in \cite{C-obst} establishes that, for any solution of \eqref{obstacle}, the free boundary $\partial\{u>0\}$ is $C^\infty$ outside a closed set of singular points.
Singular points arise for example when the free boundary creates {cusps}, and they may appear in any dimension $n\geq2$.
By \cite{CR76,C-obst2,M03}, these points are locally contained in a $C^1$ manifold of dimension $n-1$.
More recently, finer estimates at singular points were established  in \cite{CSV17,AlessioJoaquim}.

\vspace{3mm}

\subsection{Generic regularity for the obstacle problem}

\ \\ A major question in the understanding of singularities in PDE theory is the development of methods to prove \emph{generic regularity} results.
In the context of the obstacle problem \eqref{obstacle}, the key question is to understand the generic regularity of free boundaries.
Explicit examples \cite{Sch76} show that singular points in the obstacle problem can form a set of dimension $n-1$ (thus, as large as the whole free boundary).
Still, singular points are expected to be rare \cite{Sch1}:

\vspace{3mm}

\noindent \textbf{Conjecture} (Schaeffer, 1974): \emph{Generically, free boundaries in the obstacle problem have {no} singular~points.}

\vspace{3mm}

\noindent
The conjecture is only known to hold in the plane $\R^2$ \cite{M03}, and up to now nothing was known in the physical space $\R^3$ or in higher dimensions.

Notice that, in the obstacle problem, the question of generic regularity is particularly relevant, since in such context the singular set can be as large as the regular set ---while in other problems the singular set has lower Hausdorff dimension \cite{Giusti}.
Also, from the point of view of applications (see \cite{Baiocchi,DL,obst-appl1,Serfaty}), it is particularly important to understand the problem in the physical space $\R^3$.\\

A main goal of this paper is to prove Schaeffer's conjecture in $\R^3$ and $\R^4$.
To this aim, 
we consider any monotone family of solutions $\{u^t\}_{t\in(-1,1)}$ of \eqref{obstacle} in $B_1$ satisfying the following ``uniform monotonicity'' condition:
\begin{equation}\label{assumption}
\begin{split}
\text{For every $t \in (-1,1)$ and any compact set $K_t\subset \partial B_1\cap \{u^t>0\}$ there exists $c_{K_t}>0$ such that}\\
\inf_{x\in K_t} \big( u^{t'}(x)- u^t(x)  \big) \ge  c_{K_t} (t'-t) , \qquad \text{for all }-1 <t< t'< 1.\qquad\qquad
\end{split}
\end{equation}
This condition rules out the existence of regions that remain stationary as we increase the parameter $t$.
In case that $u^t$ is continuously differentiable with respect to $t$, then such condition is equivalent to saying that~$\partial_t u^t>0$ inside $\{u^t>0\}$.

We shall also assume that $(-1,1)\ni t \mapsto u^t|_{\partial B_1} \in L^\infty(\partial B_1)$ is continuous with respect to $t$. Note that, by the maximum principle, this implies that $(-1,1)\ni t \mapsto u^t \in L^\infty(B_1)$ is  continuous.
Under this assumption, we prove the following:

\begin{theorem}\label{thm-Schaeffer-intro}
Let $\{u^t\}_{t\in(-1,1)}$ be a monotone and continuous family of solutions to \eqref{obstacle} in $B_1\subset \R^n$ satisfying \eqref{assumption}, and let $\Sigma^t\subset \partial \{u^t>0\}\cap B_1$ be the set of singular points for $u^t$.
Then
\[\mathcal H^{n-4}(\Sigma^t)=0\qquad \text{for a.e. }t \in (-1,1).\]
In particular, Schaeffer's conjecture holds for $n\le 4$.
\end{theorem}

We remark that very few results are known in this direction for elliptic PDE, and most of them deal only with simpler situations (for instance the obstacle problem in $\R^2$ \cite{M03}), or when  the singular set is known to be very small (as in the case of area-minimizing hypersurfaces in $\R^8$ \cite{Smale}).
\\

As a particular family of solutions to which our Theorem~\ref{thm-Schaeffer-intro} applies,
one can consider the solution $u^t$ to the obstacle with boundary data $u^t|_{\partial B_1}=g+t$ (similarly to what was done in \cite{M03}), but many other choices are possible.

In particular, due to the general character of our assumption \eqref{assumption}, we can apply Theorem~\ref{thm-Schaeffer-intro} (more precisely, some of the results behind its proof) to study the \emph{Hele-Shaw flow}.
This is a well-known 2D model which describes a flow between two parallel flat plates following Darcy law \cite{HS,CJK}.
After a transformation of the type $u(x,t)=\int_0^tp(x,\tau)d\tau$ ---where $p(x,t)$ is the pressure--- the problem becomes
\begin{equation} \label{Hele-Shaw}
\begin{split}
\Delta u &=\chi_{\{u>0\}}\quad \textrm{in}\quad K^c\times (0,T)\subset\R^2\times \R \\
u & = t \hspace{14mm} \textrm{in}\quad K\times (0,T)\subset\R^2\times \R  \\
u & \geq 0,
\end{split}
\end{equation}
where $K\subset \R^2$ is a given compact set, and $K^c:=\R^2\setminus K$.
Since the singular set is  closed inside the free boundary (see for instance Lemma~\ref{lem:EG2B1bis}(a)),
as a consequence of our fine analysis of singular points, we can also show the following:

\begin{theorem}\label{cor-Hele-Shaw-intro}
Let $K\subset \R^2$ be any compact set, and $u(x,t)$ be any solution to the Hele-Shaw flow~\eqref{Hele-Shaw}.
Let $\Sigma^t\subset K^c$ be the set of singular points of $\partial \{u(\,\cdot\,,t)>0\}$, and let $\mathcal S:=\{t\in(0,T) : \Sigma^t\neq\varnothing\}$ be the  set of singular times.
Then  $\mathcal S$ is relatively closed inside $(0,T)$ and
\[{\rm dim}_{\mathcal H}(\mathcal S)\leq \frac14.\]
In particular, the free boundary is $C^\infty$ for a.e. time $t\in(0,T)$.
\end{theorem}

Prior to our result, it was an open question to decide whether singularities in such model could persist in time or not.
Theorem~\ref{cor-Hele-Shaw-intro} answers this question, and provides for the first time an estimate on the set of singular times.

\vspace{3mm}

\subsection{Higher-order expansions at most singular points}
\label{sect:intro higher order}

\ \\A key tool in the proof of Theorem~\ref{thm-Schaeffer-intro} is a very fine understanding of singular points, as explained next.

For the obstacle problem \eqref{obstacle}, a classical result of Caffarelli \cite{C-obst2} states that at every singular point $x_\circ$ we have an expansion of the form
\begin{equation}\label{Caffarelli}
u(x)=p_{2,x_\circ}(x-x_\circ)+o(|x-x_\circ|^2),
\end{equation}
where $p_{2,x_\circ}$ is a nonnegative, homogeneous, quadratic polynomial satisfying $\Delta p_{2,x_\circ}\equiv 1$.

In dimension $n=2$ this estimate was improved in \cite{W99} by replacing $o(|x-x_\circ|^2)$ with $O(|x-x_\circ|^{2+\alpha})$ for some $\alpha>0$, and in arbitrary dimensions it was shown in \cite{CSV17} that $o(|x-x_\circ|^2)$ can be replaced by $O(|x-x_\circ|^2|\log|x-x_\circ||^{-\epsilon})$, for some $\epsilon>0$.
More recently, it was proved by the first and third authors \cite{AlessioJoaquim} that, in every dimension $n$, one actually has
\[u(x)=p_{2,x_\circ}(x-x_\circ)+O(|x-x_\circ|^3),\]
possibly outside a set of ``anomalous'' singular points whose Hausdorff dimension is at most $n-3$.

Here, in order to prove our main result, we need to improve substantially the understanding of singular points, establishing a new higher order expansion at most singular points for monotone families of solutions to the obstacle problem.
Here and in the sequel, ${\rm dim}_{\mathcal H}$ will denote the Hausdorff dimension (see Section~\ref{sec:GMT} for a definition).

\begin{theorem}\label{thm-expansion-obstacle}
Let $\{u^t\}_{t\in(-1,1)}$ be a family of solutions to \eqref{obstacle} in $B_1\subset \R^n$ which is continuous and nondecreasing in $t$ (in particular, they could be independent of $t$).
Let $\Sigma^t\subset\partial\{u^t>0\}\cap B_1$ be the set of singular points of $u^t$, and $\hat\Sigma:=\cup_{t \in (-1,1)} \Sigma^t\subset B_1$.

Then there exists a set $E\subset\hat\Sigma$, with ${\rm dim}_{\mathcal H}(E)\leq n-2$, such that for every $t_\circ \in (-1,1)$ and every $x_\circ\in \Sigma^{t_\circ}\setminus E$ we have
\begin{equation}\label{expansion-obstacle}
u^{t_\circ}(x) = P_{4,x_\circ,t_\circ}(x-x_\circ)+O\bigl(|x- x_\circ|^{5-\zeta}\bigr)
\end{equation}
for all $\zeta>0$, where
$P_{4,x_\circ,t_\circ}$ is a fourth order polynomial satisfying $\Delta P_{4,x_\circ,t_\circ}\equiv 1$.
\end{theorem}

An important point here is that the dimension $n-2$ of the ``bad'' set $E$ is sharp.
Indeed, by well known examples in $\R^2$ (see e.g. \cite{Sak93}), one can construct solutions $u$ whose singular set contains a $(n-2)$-dimensional subset $E$ for which \eqref{expansion-obstacle} does not hold at any point in $E$.

As the reader will see from the proof, when $p_{2,x_\circ,t_\circ}(x)=\frac12 (x\cdot e)^2$ for some unit vector $e \in \R^n$ then the expansion \eqref{expansion-obstacle} can be written alternatively as 
\[u(x_\circ+x)=\frac12\big(e\cdot x+p(x)\big)^2+O\bigl(|x|^{5-\zeta}\bigr),\]
for a certain polynomial $p$ of degree 3 with no linear or constant terms.
Geometrically, this expansion ---together with a Lipschitz estimate that we will establish later--- yields that, around most singular points, the contact set is contained inside a set of the form $\bigl\{|e\cdot x+p(x)|\leq C|x|^{4-\zeta}\bigr\}$.
Thus, if the free boundary has a cusp, then at most points this cusp must be very thin.
It is worth noticing that the expression of $p$ (or equivalently, of $P_{4,x_\circ,t_\circ}$) is related to the curvature of the free boundary near a singular point.
In particular, whenever the solution is even with respect to the hyperplane $\{e\cdot x=0\}$, then $p\equiv0$ (and thus $P_{4,x_\circ,t_\circ}(x)\equiv \frac 1 2 (e\cdot x)^2$), since there are no curvature terms.
\\

To establish Theorem~\ref{thm-Schaeffer-intro} we need to introduce a variety of new ideas, combining Geometric Measure Theory tools, PDE estimates, several dimension reduction arguments, and new monotonicity formulas.
This is explained in more detail next.

\vspace{3mm}

\subsection{On the proofs of the main results}
\label{sect:idea pf}

\ \\ Let us give an overview of the main ideas introduced in this paper.

\subsubsection{From expansion to cleaning: a Sard-type approach}
The starting idea to prove Theorem~\ref{thm-Schaeffer-intro} is the following: denote by $\Sigma^t\subset\partial\{u^t>0\}\cap B_1$ the set of singular points of $u^t$. 
Assume that, for some fixed $t_\circ$, we have $x_\circ \in \Sigma^{t_\circ}$ and  $u^{t_\circ}$ has an expansion of the form
\begin{equation}\label{haiohoaih}
u^{t_\circ}(x_\circ+x)=P(x)+O(|x|^\lambda)
\end{equation}
for some  $\lambda\geq 2$  and some polynomial $P$ such that $\Delta P\equiv 1$.
Note that, since $u^{t_\circ}\geq 0$, the expansion above implies that $P(x)\geq -O(|x|^\lambda)$.
Hence, for any $r>0$, $P+Cr^\lambda$ is a solution to the obstacle problem in~$B_r$ with an empty contact set,
and $u^{t_\circ}(x_\circ+\cdot )$ is $O(r^\lambda)$-close to it.
This suggests that, thanks to the monotonicity assumption \eqref{assumption}, by slightly increasing the value of the parameter $t$ the contact set of $u^t$ inside $B_r(x_\circ)$ will disappear.

To make this argument quantitative we need to introduce a series of
delicate barrier constructions which actually depend on the fine structure of the singular point $x_\circ$, see Section~\ref{sec:ESC}.
In this way we are able to prove that, for a increment of $t$ of size $\delta t\sim r^{\lambda-1}$, the contact set of $u^{t+\delta t}$ is $B_r$ disappears:
more precisely we can show that, for some $C>0$,
\begin{equation}
\label{eq:cleaning type}
\Sigma^{t_\circ+Cr^{\lambda-1}}\cap B_{r}(x_\circ)=\varnothing\qquad \text{for $r>0$ sufficiently small}.
\end{equation}
As we will explain better below,  we can prove an expansion as in \eqref{haiohoaih} for $\lambda$ belonging a discrete set $\Lambda$ (see \eqref{setoffreq} for a definition of this set).

Hence, for $t_\circ \in (-1,1)$ and $x_\circ \in \Sigma^{t_\circ}\subset \partial \{u^{t_\circ}>0\}$, we define $\lambda_{x_\circ,t_\circ}$ to be  the maximal $\lambda\in \Lambda$  for which we can prove an expansion as in \eqref{haiohoaih}.
Then, for each $\lambda \in \Lambda$ and $t \in (-1,1)$, we 
define $\Sigma^{t,\lambda}$ as the set of $x_\circ\in \Sigma_{t}$ for which  $\lambda_{x_\circ,t} =\lambda$.  In other words, $\Sigma^{t,\lambda}$ is defined as the set of points at which we have a polynomial expansion up to order $\lambda\in \Lambda$ 
but not better. Then, a covering argument ``\`a la Sard'' yields
\begin{equation}\label{naoibgiaug}
\dim_{\mathcal H}\big(\{t\in (-1,1)\,:\, \Sigma^{t,\lambda}\neq \varnothing\} \big) \leq \frac{\dim_{\mathcal H}\big(\cup_{t \in (-1,1)}\Sigma^{t,\lambda}\big)}{\lambda-1}
\end{equation}
(see Proposition~\ref{prop:GMT4}(a) for a more refined statement).
In particular, if the right hand side is strictly less than $1$, then $\Sigma^{t,\lambda}=\varnothing$ for a.e. $t$.
On the other hand, if the right hand side is larger or equal to~$1$, then a coarea-type argument allows us to show that $\Sigma^{t,\lambda}$ is very small for a.e. $t$ (see Proposition~\ref{prop:GMT4}(b)).

In view of the given description of our approach, our goals are the following:
\begin{itemize}
\item[(1)] given a singular point, prove an expansion up to order $O(|x|^\lambda)$ with $\lambda\in \Lambda$ as large as possible;
\item[(2)] given $\lambda \in \Lambda$, estimate the dimension of $\cup_{t \in (-1,1)} \Sigma^{t,\lambda} $, i.e., the set of points  where the expansion stops at $\lambda$.
\end{itemize}

\subsubsection{A higher-order expansion at singular points: the case of a single solution}
\label{subsect:higher}
To understand these questions in a simplified situation, one can first look at the problem without the parameter $t$.
So, given a solution $u$ to the obstacle problem and a singular point $x_\circ$, we want to obtain a Taylor expansion around $x_\circ$ at the highest possible order.
This will require several steps, described below.

\subsubsection*{(a) Second blow-up: a cubic expansion at most points}
This first part is essentially contained in \cite{AlessioJoaquim}.
Recall first that, as proven in \cite{C-obst2}, for any singular point $x_\circ$ we have \eqref{Caffarelli}, that is, $p_{2,x_\circ}$ is tangent up to second order to $u(x_\circ +\,\cdot \,)$ at $0$.
Equivalently
$$
\frac{u(x_\circ+r\,\cdot\,)}{r^2}\to p_{2,x_\circ}\qquad \text{as }r\to 0,
$$
and $p_{2,x_\circ}$ is called the ``first blow-up'' of $u$ at $x_\circ$.

One can then catalog singular points according to the dimension of the kernel of $p_{2,x_\circ}$: given $m \in \{0,\ldots,n-1\}$, we say
 $$
x_\circ \in \Sigma_{m} \qquad \Longleftrightarrow\qquad {\rm dim}(\{p_{2,x_\circ}=0\})=m.
 $$
 We then consider the ``second blow-ups'', namely, the possible limits of the functions
 $$
 \tilde w_{2,r}(x):=\frac{w_2(rx)}{\|w_2(r\,\cdot\,)\|_{L^2(\partial B_1)}},\qquad \text{where }w_2:=u(x_\circ + \,\cdot\,)-p_{2,x_\circ}
 $$
 as $r \to 0.$
 As shown in \cite{AlessioJoaquim},
 $r\mapsto \|\nabla \tilde w_{2,r}\|_{L^2(B_1)}$ is monotonically increasing (equivalently, the so-called Almgren frequency formula is monotone on $w_2$).
 Thanks to this fact, setting $\lambda_2:=\lim_{r\to 0}\|\nabla \tilde w_{2,r}\|_{L^2(B_1)}$,
one can characterize all the blow-ups (namely, the accumulation points of $\{\tilde w_{2,r}\}$ as $r\to 0$): they are $\lambda_2$-homogeneous functions $q$\\
- either satisfying
  $$
  \Delta q=0 \quad \mbox{ in }  \R^n \,, \quad \mbox{ if } x_\circ \in \Sigma_m\,\text{ with }\, m \leq n-2,
  $$
 - or solving the Signorini problem
  $$
  \Delta q \leq 0,\quad q\Delta q \equiv 0,\quad \Delta q|_{\R^n\setminus L}=0, \quad  q| _L \ge 0 \, , \quad  \mbox{if }L := \{p_{2,x_\circ}=0\} \  \mbox{is a hyperplane (i.e., $x_\circ \in \Sigma_{n-1}$)}.
  $$
  
In the first case (i.e., $m\leq n-2$), since $q$ is harmonic in $\R^n$ it must be  $\lambda_2 \in \{2,3,4,\ldots\}$.
Also, following \cite{AlessioJoaquim}, one can show that $\lambda_2 \geq 3$ up to an ``anomalous set'' of dimension $m-1$ inside $\Sigma_m$.
This implies that, for $x_\circ$ outside this anomalous set,  we  have $u= p_{2,x_\circ} (x-x_\circ)+ O(|x-x_\circ|^3)$. 
Note that applying  this result to $u = u^{t_\circ}$ gives an expansion as in \eqref{haiohoaih} for $\lambda=3$.  
As we will explain in Subsection \ref{subsect:extra}(a) below, when $m\leq n-2$
we are able to improve \eqref{eq:cleaning type} so that this expansion suffices for proving our main theorem.

The real challenge is to understand the set $\Sigma_{n-1}$. In this case, since $q$ solves the Signorini problem, thanks to a classification result for 2D solutions and a dimension reduction argument, we can
show that
\begin{equation}\label{setoffreq}
  \lambda_2 \in \Lambda: = \{2,3,4,\ldots\}\cup\Bigl\{\frac{7}{2}, \frac{11}{2},\frac{15}{2},\ldots\Bigr\}
\end{equation}
outside a set of singular points of dimension $n-3$. Also it is easy to prove that, in this case, $\lambda_2\neq 2$; thus, except for a small set, the lowest possible value for $\lambda_2$ is $3$. The main challenge is now to improve this cubic expansion to higher order.

\subsubsection*{(b) Third blow-up: a delicate dichotomy}
From now on, we focus on points of $\Sigma_{n-1}$ where $\lambda_2=3$, i.e., $q$ is a 3rd-order homogeneous solution of Signorini (as one can see from the coming argument, the other cases can be considered as a particular case of this taking $\lambda_2=3$ and $p_{3,x_\circ }\equiv 0$ in the definition of $w_{3}$ below).
Two possibilities arise, depending  whether some accumulation point $q$ of $\{\tilde w_{2,r}\}$ as $r\to 0$ is harmonic or not. These two cases have to be analyzed separately.

{\it - The third blow-up is not harmonic: a new uniqueness result.}
By another dimension reduction argument, we can prove that the set where  $q$ is not harmonic has dimension $n-2$.
However this is not enough, and 
here comes one of the key arguments introduced in this paper: as explained in Subsection \ref{subsect:extra}(b), in order to obtain Schaeffer's conjecture in $\R^4$ we need to prove that the limit $q$ of $\tilde w_{2,r}$ is unique, and that this set is $(n-2)$-rectifiable. To accomplish, in Section~\ref{cubic-sec} we introduce new differential formulae, compactness and barrier arguments, and a delicate ODE-type lemma, that allow us to obtain the uniqueness of blow-ups (taking quotients of suitable qualities) even if we lack a monotonicity formula.

{\it - The third blow-up is harmonic: a monotonicity argument at nondegenerate points.}
Assume that there exists a harmonic accumulation point $q$. Then (thanks to a Monneau-type monotonicity formula) we can show that the limit $\lim_{r\downarrow 0} \tilde w_{2,r}$ exists (i.e., all accumulation points coincide) and that $u(x_\circ+\cdot) = p_{2,x_\circ} + p_{3,x_\circ} + o(|x|^3)$ for some 3-homogeneous harmonic polynomial $p_{3,x_\circ}$ ($p_{3,x_\circ}$ being a multiple of $q$). This suggests to iterate the previous blow-up procedure by defining
 $$
 \tilde w_{3,r}(x):=\frac{w_3(rx)}{\|w_3(r\,\cdot\,)\|_{L^2(\partial B_1)}},\qquad \text{where }\,w_3:=u-p_{2,x_\circ}-p_{3,x_\circ},
 $$
 and try to mimicking  the argument described above.
 Unfortunately,
in this case it is not true anymore that $r\mapsto \|\nabla \tilde w_{3,r}\|_{L^2(B_1)}$ is monotonically increasing.
Still, by a delicate bootstrapping argument (cf. Lemma~\ref{lem:E3B1}) we can prove that\footnote{More precisely, Lemma~\ref{lem:E3B1} gives the monotonocity of a ``truncated frequency function'', where the size of  $w_3(r\,\cdot\,)$ is corrected by adding a term $r^{\gamma}$ with $\gamma \in (3,4)$. In order to relate the (almost) monotonicity of this truncated frequency function to the one of $\|\nabla \tilde w_{3,r}\|_{L^2(B_1)},$ one needs to know that the size of  $w_3(r\,\cdot\,)$ dominates $r^{\gamma}$ for some $\gamma \in (3,4)$.}
$$
\text{$r\mapsto \|\nabla \tilde w_{3,r}\|_{L^2(B_1)}$ is almost increasing, provided $\|w_3(r\,\cdot\,)\|_{L^2(\partial B_1)} \gtrsim r^{4-\varepsilon}$ for some $\varepsilon>0$}.
$$
Therefore, under this nondegeneracy assumption, we can consider accumulation points of $\tilde w_{3,r}$ and prove that they are $\lambda_3$-homogeneous solution of Signorini for some $\lambda_3\in [3,4)$.
Then, by a dimension reduction argument (based again on the classification of 2D solutions), we can  prove that 
$\lambda_3 \in \{3,7/2\}=\Lambda\cap [3,4)$ in a set of dimension $n-2$, and the remaining points are of codimension~$3$.
So, to summarize:
\begin{itemize}
\item[(i)] for most points in $\Sigma_{n-1}$ where the limit of $\tilde w_{2,r}$ is harmonic,  the assumption $\|w_3(r\,\cdot\,)\|_{L^2(\partial B_1)} \gtrsim r^{4-\varepsilon}$ fails for every $\varepsilon>0$, except perhaps in a set of dimension $n-2$;
\item[(ii)]  if $\|w_3(r\,\cdot\,)\|_{L^2(\partial B_1)} \gtrsim r^{4-\varepsilon}$ holds, then the blow-up is $\lambda_3$-homogeneous for $\lambda_3 \in \{3,7/2\}$, except for a set of dimension at most $n-3$.
\end{itemize}
In order to prove Schaeffer's conjecture in $\R^4$, it is important for us to rule out the possibility that $\lambda_3=3$ in case (ii). This is highly nontrivial, and follows from the analysis performed in Section~\ref{cubic-sec}
to understand points where $\tilde w_{2,r}$ converges to a non-harmonic function.

\subsubsection*{(c) Fourth blow-up: monotonicity via a new ansatz and proof of \eqref{expansion-obstacle}}
To go further in our analysis and prove our main theorem, we now need to investigate the set of points where case (i) happens, namely $\|w_3(r\,\cdot\,)\|_{L^2(\partial B_1)} \gtrsim r^{4-\varepsilon}$ fails for every $\varepsilon>0$.
In this case  Almgren's monotonicity formula fails on $w_{3,r}$, and therefore a new approach needs to be found.

The key idea here is to replace $w_3=u(x_\circ+\cdot)-p_{2,x_\circ}-p_{3,x_\circ}$ with a much more refined Ansatz $W_3:=u(x_\circ+\cdot)-\anz_{x_\circ}$ which takes into account both the curvature of the free boundary and the non-negativity of the solution ---this 
is done in Definition~\ref{ansatz}. Doing so, and defining $\tilde W_{3,r}$ in analogy to what done before, we can show (again after a bootstrap argument) that 
$$
\text{$r\mapsto \|\nabla \tilde W_{3,r}\|_{L^2(B_1)}$ is almost increasing, provided $\|W_3(r\,\cdot\,)\|_{L^2(\partial B_1)} \gtrsim r^{5-\varepsilon}$ for some $\varepsilon>0$}
$$
(see Lemma \ref{lem:E4B1}).
Let us note that obtaining this almost monotonicity is much more involved than in the previous case (when we had $4-\varepsilon$ instead of $5-\varepsilon$). The reason is technical and rather delicate: we need to show that the size of $W_{3,r}$ controls the one of its gradient (see Lemma~\ref{lem:E3B5}), and this follows from a semiconvexity estimate along some rotational derivatives.

Once this almost monotonicity is proved, we can consider accumulation points of $\tilde W_{3,r}$ at all points where $\|W_3(r\,\cdot\,)\|_{L^2(\partial B_1)} \gtrsim r^{5-\varepsilon}$, and we prove that (up to a codimension 2 set) the only possible limit is a harmonic polynomial $p_{4,x_\circ}$ of degree $4$.
This leads to an expansion of the form
$$
w_4:=u(x_\circ+\cdot)-\anz_{x_\circ}-p_{4,x_\circ},\qquad w_4(x)=o(x^4).
$$
Hence, to finally obtain \eqref{expansion-obstacle} with $P_{4,x_\circ}=\anz_{x_\circ}+p_{4,x_\circ}$, we only need to prove that $w_4(x)=O(|x|^{5-\zeta})$ for all $\zeta>0$,  up to a set of dimension $n-2$.
This is again nontrivial: indeed, we can show that
 $r\mapsto \|\nabla \tilde w_{4,r}\|_{L^2(B_1)}$ is almost increasing   provided that $p_{4,x_\circ}$ vanishes on $\{p_{2,x_\circ}=0\}$ and $\|w_4(r\,\cdot\,)\|_{L^2(\partial B_1)} \gtrsim r^{5-\varepsilon}$ for some $\varepsilon>0$. Hence we need to ensure that these assumptions are satisfied in a large enough set, and for this we exploit some recent results on the size of the zero set of harmonic functions (see \cite{NV17}) and another dimension reduction argument. In this way, we conclude the validity of  \eqref{expansion-obstacle}.
 \smallskip 

It is important to remark here that the expansion \eqref{expansion-obstacle} up to order $5-\zeta$ is exactly what we need in order to prove Theorem~\ref{thm-Schaeffer-intro}.
Even if one could improve such an expansion to a higher order, the estimate on the singular set in Theorem~\ref{thm-Schaeffer-intro} would not change.
Indeed, the bounds on the sizes of the sets where the expansions stop before $5-\zeta$ would not improve, and the conclusion in Theorem~\ref{thm-Schaeffer-intro} would remain exactly the same (see also Remark \ref{remark:sharp}).

\subsubsection{From one solution to a monotone 1-parameter family of solutions}
Note that the analysis performed above holds only for {one solution}. 
If now we have an increasing {family of solutions} $\{u^t\}_{t\in(-1,1)}$, we need to understand for each $t$ the size of points where the expansion stops at some fixed order $\lambda \in \Lambda$. 

If one simply applies the previous analysis to each solution $u^t$ one would not be able to conclude. Indeed,  if for each solution the set $\Sigma^{t,\lambda}$ has dimension bounded by some $s\geq 0$, then a simple argument (using the structure of our problem) would show that their union over $t\in (-1,1)$ has dimension bounded by $s+1$. Unfortunately, this estimate would be absolutely too weak for our scope. Indeed, in order to prove our result, we need to show the following: if the analysis performed on a single solution implies that a set $\Sigma^{t,\lambda}$ has dimension bounded by $s$, then also $\cup_{t \in (-1,1)}\Sigma^{t,\lambda}$ has dimension bounded by $s$. In other words, the bound on the union should be exactly the same as the one obtained for each single set!

To achieve this, we have to exploit the fact that we have an increasing family $u^t$ of solutions to obtain very refined estimates on the possible blow-ups of a fixed solution at a free boundary point.
More precisely, the idea is the following: if a sequence of singular points $x_k\in \Sigma^{t_k}$ converges to $0 \in \Sigma^0$,
and if both solutions $u^{t_k}$ and $u^0$ have a Taylor expansion up to the same order $\lambda$ at these points, then this implies some extra information on the possible Taylor expansion of $u^0$ at $0$ (more precisely, this implies some symmetry properties on its higher order term). This analysis is performed in Section~\ref{sec:EG2B} and it introduces a complete series of new ideas and techniques with respect to \cite{AlessioJoaquim}, where only one fixed solution was considered. We want to emphasize that, with respect to the case of only one solution (where one can still deduce symmetry properties of blow-ups as a consequence of being an accumulation point of other singular points), in this case this analysis is made particularly delicate by the fact that we do not have any equation in $t$ relating the solutions: we only know that they are ordered and strictly increasing with respect to $t$.
Still, we are able to deduce some strong symmetry properties of blow-up at all points where the Almgren's frequency is continuous (see the results in Section~\ref{sec:EG2B}), and from these properties we obtain a very precise description of the structure of singular points.
This is then combined with a series of covering and dimension reduction arguments (see Sections~\ref{sec:GMT} and~\ref{sec:DRB}) to estimate the size of the singular points where blow-ups have few symmetries,
which allow us to show the desired dimensional bounds on $\cup_{t \in (-1,1)}\Sigma^{t,\lambda}$. {To our knowledge, this is the first dimension reduction argument for a 1-parameter family of solutions to elliptic PDEs, and we expect these ideas and techniques to be useful in many other problems.}

\subsubsection{Extra comments}
\label{subsect:extra}
The previous description syntheses well the main ideas behind our strategy, and what explained until now suffices for proving the Schaeffer conjecture in $\R^3$.
However, for the $\R^4$ case, other extra ideas (that we only briefly mentioned before) are required.
In particular, we need a more  refined analysis that depends on the type of singular point, having to distinguish two cases:

\subsubsection*{(a) The case of the ``lower dimensional strata''} For singular points in $\cup_{t\in(-1,1)}\Sigma_m^{t}$ with $m \leq n-2$, we know that the expansion \eqref{haiohoaih} stops at $\lambda=2$ at  ``anomalous points'', and these points have dimension bounded by $m-1$.
Hence, denoting this set of anomalous points by $\cup_{t\in(-1,1)}\Sigma_m^{t,2}$, and the remaining $m$-dimensional set (where the expansion stops at $\lambda=3$) by $\cup_{t\in(-1,1)}\Sigma_m^{t,3}$, applying \eqref{naoibgiaug} for $n=4$ and $m=2$ we get the trivial bound
 $$\dim_{\mathcal H}\big(\{t\in (-1,1)\,:\, \Sigma_m^{t,\lambda}\neq \varnothing \}\big)\le \frac{\dim_{\mathcal H}(\cup_{t\in(-1,1)}\Sigma_m^{t,\lambda})}{\lambda-1}=1,\qquad \text{for } \lambda=2,3.$$
To improve this estimate, we refine our barrier arguments and show that \eqref{eq:cleaning type} can be improved to $\Sigma^{t_\circ+r^{\lambda-\ep},i}_m\cap B_{r}(x_0)=\varnothing$ for any $\ep>0$, for $\lambda=2,3$. This increased speed at which the contact set clears near singular points for  ``lower strata'' is not difficult to prove, but is fundamental to establish  Schaeffer's conjecture in $\R^4$.

\subsubsection*{(b) The case of the ``top stratum''} A big difficulty arises from points in $\cup_{t\in (-1,1)}\Sigma_{n-1}^t$ where the expansion stops at $\lambda =3$.
More precisely, as mentioned  in Subsection \ref{subsect:higher}(b) when discussing the case where the third blow-up is not harmonic, there may exist a set $\cup_{t\in (-1,1)}\Sigma^{t,3}_{n-1}$ of dimension $n-2$ at which the expansion   $u^{t_\circ}(x) = p_{2,x_\circ,t_\circ}(x-x_\circ)+ O(|x-x_\circ|^3)$ is sharp ---i.e., the third derivatives of $u^{t_\circ}$ do not exist at those points. Then, if we use \eqref{naoibgiaug}  for $\lambda=3$ and $n=4$, we obtain  the (trivial) estimate $$\dim_{\mathcal H}\big(\{t\in (-1,1)\,:\, \Sigma_{n-1}^{t,3}\neq \varnothing \}\big)\le \frac{n-2}{\lambda-1}=1,$$ while to establish Schaeffer's conjecture in $\R^4$ we  need to prove $\HH^1\big(\{t\in (-1,1)\,:\, \Sigma^{t,3}\neq \varnothing \}\big) =0$. 
Since it is possible to construct examples where the set $\cup_{t\in (-1,1)}\Sigma^{t,3}_{n-1}$ is $(n-2)$-dimensional, there is no hope to improve its dimensional bound. 
In addition, unlike in the ``lower strata'', one can see that  \eqref{eq:cleaning type} for $\lambda=3$ is sharp at these points. 
Consequently, we  need a completely different  argument to conclude. 

Here the idea consists in  taking negative increments of $t$ ---which make  the contact set become thicker instead of thinner--- and use  barrier arguments to show that all free boundary points in a neighborhood become regular at a slightly enhanced speed (see Lemma~\ref{lem:cleaningpast3}). 
However, we can only take advantage of this improvement if we can prove that the set $\cup_{t\in (-1,1)}\Sigma_{n-1}^{t,3}$ is $(n-2)$-rectifiable (the information that its Hausdorff dimension is bounded by $n-2$ is not sufficient). 
This rectifiability result is crucial, and its proof relies on the existence of  $\lim_{r\downarrow 0} \tilde w_{2,r}$ in the non-harmonic case. As mentioned before, this fact requires completely new ideas with respect to the classical tools known in this kind of problems, and it is the focus of 
Section~\ref{cubic-sec}. 
It is worth observing that such arguments lead to new interesting results even when applied to the Signorini problem, see Appendix \ref{apb}.

\vspace{3mm}

\subsection{Organization of the paper}

\ \\ The paper is organized as follows.
In Section~\ref{sec:EFF} we introduce a series useful functionals that will be used in the proof of some of our monotonicity formulae.
In Section~\ref{sec:EMF1} we present some preliminary results that will be needed throughout the paper.
Then, in Section~\ref{sec:E3B} we develop our higher order analysis of singular points.
In Section~\ref{cubic-sec} we study in detail singular points at which the blow up of $u(x_\circ+\,\cdot\,)-p_{2,x_\circ}$ is 3-homogeneous (cf. Subsection~\ref{subsect:higher}).
In Section~\ref{sec:EG2B} we consider a 1-parameter family of solutions to the obstacle problem and study symmetry properties of their blow-ups.
In Section~\ref{sec:GMT} we prove a series of lemmas of geometric measure theory-type, and in Section~\ref{sec:DRB} we establish several dimension reduction results. Finally, 
in Section~\ref{sec:ESC} we prove our main result, Theorem~\ref{thm-Schaeffer-intro}, as well as Theorem \ref{cor-Hele-Shaw-intro}.
In addition, at the end of the paper we provide two appendices on the Signorini problem: one with some basic results that are needed throughout the paper, and a second one with new results (uniqueness and nondegeneracy of blow-ups at odd frequencies) that are a consequence of our analysis in Section~\ref{cubic-sec} and that we believe to be of independent interest.

\section{Useful functionals and formulae} \label{sec:EFF}

In all this section $r \in (0,1),$ and $w: B \rightarrow \R$ denotes an arbitrary $C^{1,1}$ function defined in a ball $B\subset \R^n$ (specified in each statement).
Throughout the paper we shall use the following dimensionless quantities:
\[
D(r,w) := r^{2-n} \int_{B_r} |\nabla w|^2 = r^2\int_{B_1} |\nabla w|^2(r\,\cdot\,),
\]
\[
H(r,w) := r^{1-n}\int_{\partial B_r} w^2 = \int_{\partial B_1} w^2(r\, \cdot\,).
\]
We also introduce here a useful notation for rescaling and normalization. Given $w: B \rightarrow \R$  and $r>0$ we define $w_r$ and $\tilde w_r$ as
\begin{equation}\label{defwr_tildewr}
w_r(x) : = w(rx) \qquad \mbox{and}\qquad \tilde w_r(x) : = \frac{w_r}{H(1,w_r)^{\frac 1 2}}  = \frac{w(r\,\cdot\, )}{H(r,w)^{\frac 1 2}}  .
\end{equation}
We start by computing the derivatives of $H$ and $D$.

\begin{lemma} \label{lem:EFF1}Let $w\in C^{1,1}(B_2)$. Then
\begin{equation}\label{thirda}
\left.\frac{d}{dr}\right|_{r=1}H(r,w) = 2\int_{\partial B_1}  w w_\nu =  2 \int_{B_1} w\Delta w + 2\int_{B_1}|\nabla w|^2.
\end{equation}
\end{lemma}

\begin{proof}
This is a standard computation, that can be found for instance in \cite{AlessioJoaquim}.
\end{proof}

\begin{lemma}\label{lem:EFF2} Let $w\in C^{1,1}(B_2)$. Then
\begin{equation}\label{first}
\left.\frac{d}{dr}\right|_{r=1} D(r,w) = 2\int_{\partial B_1} w_\nu ^2 - 2\int_{B_1} \Delta w \,(x\cdot \nabla w).
\end{equation}
\end{lemma}

\begin{proof}
For convenience, we set $D(r):=D(r,w)$. It holds
\[
\begin{split}
D'(1) &=  \sum_{i,j} \int_{B_1} 2w_i x_j  w_{ij}  + 2D(1) = \sum_{i,j}\int_{\partial B_1}2w_i x_j w_{j} \nu_i  -\sum_{i,j} \int_{B_1}  2(w_{i}x_j)_i w_j +2D(1)
\\
&= 2\int_{\partial B_1} w_\nu ^2 - 2\int_{B_1} \Delta w \,(x\cdot \nabla w) - 2\int_{B_1} |\nabla w|^2 +2D(1)
= 2\int_{\partial B_1} w_\nu ^2 - 2\int_{B_1} \Delta w \,(x\cdot \nabla w).
\end{split}
\]
\end{proof}

Let us introduce the functions
\begin{equation}
\label{eq:Almgren}
\phi(r,w)  : = \frac{D(r,w)}{H(r,w)} , \qquad  \phi^\gamma(r,w)  : = \frac{D(r,w) + \gamma r^{2\gamma}}{H(r,w)+ r^{2\gamma}}.
\end{equation}
The quantity $\phi$ is often known as the Almgren frequency function.
Instead, $\phi^\gamma$ is a new truncated frequency function, that to our knowledge has never been introduced before.
It will be used throughout the paper and will be extremely useful in our arguments, as it will allow us to deal with the cases when $H$ may be too small.\footnote{In the past, other kind of truncations have been introduced (see in particular \cite{CSS08}), but they do not work in our case due to the fact that $D$ is not equal to (a multiple of) the derivative of $H$, as it is instead the case for the Signorini problem.}
The choice of the constant $\gamma$ in front of $r^{2\gamma}$ in the numerator is important to make the following lemma work.

\begin{lemma} \label{lem:EFF4} Let $w\in C^{1,1}(B_1)$. Then for $r\in (0,1)$ we have
\[
 \frac{d}{dr} \phi(r,w) \ge \frac{2}{r}  \frac{ \big( r^{2-n} \int_{B_r}w\Delta w\big)^2  + E(r,w) }{ \big( H(r,w)\big)^2 }
\]
and
\[
 \frac{d}{dr} \phi^\gamma(r,w) \ge \frac{2}{r}  \frac{ \big( r^{2-n} \int_{B_r}w\Delta w\big)^2  + E^\gamma(r,w) }{ \big( H(r,w) + r^{2\gamma}\big)^2 },
\]
where
\begin{equation}\label{wq:Erw}
E(r,w) :=   \left(r^{2-n} \int_{B_r} w\Delta w\right) D(r,w) - \left(r^{2-n} \int_{B_r} (x\cdot \nabla w)\Delta w\right)H(r,w)
\end{equation}
and
\begin{equation}\label{eq:Erwgamma}
E^\gamma(r,w) :=   \left(r^{2-n} \int_{B_r} w\Delta w\right) \big( D(r,w)+  \gamma r^{2\gamma}  \big) - \left(r^{2-n} \int_{B_r} (x\cdot \nabla w)\Delta w\right) \big(H(r,w) +r^{2\gamma} \big).
\end{equation}
\end{lemma}

\begin{proof}
We first observe that, for $r\in (0,1)$, the formula for $\phi$ can be deduced from the one for $\phi^\gamma$ letting $\gamma\uparrow+\infty$.

By scaling it is enough to compute, for $a  >0$,
\[
\left.\frac{d}{dr}\right|_{r=1} \log  \phi^{\gamma,a}(r,w),   \qquad\mbox{for}\quad \phi^{\gamma,a}(r,w):=   \frac{D(r,w) + \gamma (a r)^{2\gamma}}{ H(r,w) + (a r)^{2\gamma} }.
\]
Using Lemmas~\ref{lem:EFF1} and~\ref{lem:EFF2} we obtain
\[
 \frac{\frac{d}{dr}|_{r=1} \big(D(r,w) +\gamma(a r)^{2\gamma} \big)}{D(1,w) +\gamma a^{2\gamma} } =  2\frac{ \int_{\partial B_1} w_\nu ^2 - \int_{B_1} (x\cdot \nabla w)  \Delta w +  \gamma^2a^{2\gamma}}{  \int_{\partial B_1} w w_\nu - \int_{B_1} w \Delta w + \gamma a^{2\gamma}  }
\]
and
\[
 \frac{ \frac{d}{dr}|_{r=1}\big( H(r,w) + a r^{2\gamma} \big)}{ H(1,w) +  a^{2\gamma} } = 2 \frac{  \int_{\partial B_1} ww_\nu   + \gamma a^{2\gamma} }{  \int_{\partial B_1} w^2+ a^{2\gamma}}.
\]
Therefore,
\[
\begin{split}
\frac{d}{dr} \log \phi^{\gamma,a}(1,w)& =  \frac{\frac{d}{dr}|_{r=1} \big(D(r,w) +\gamma (a r)^{2\gamma} \big)}{D(1,w) + \gamma a^{2\gamma} } -  \frac{ \frac{d}{dr}|_{r=1}\big( H(r,w) + ar^{2\gamma} \big)}{ H(1,w) +  a^{2\gamma} }
\\
&=2\frac{X^2 +  {\rm rest}   }{ \big(D(1,w) + \gamma a^{2\gamma}\big)\big( H(1,w) + a^{2\gamma}\big) },
\end{split}
\]
where
\[
X^2 :=   \bigg(\int_{\partial B_1} w_\nu^2  +  \gamma^2a^{2\gamma} \bigg)\bigg(\int_{\partial B_1} w^2+ a^{2\gamma} \bigg) - \bigg( \int_{\partial B_1} ww_\nu+ \gamma a^{2\gamma} \bigg)^2
\ge 0
\]
(the non-negativity of $X^2$ follows by H\"older inequality), 
and
\[
{\rm rest} := \left(\int_{ B_1} w\Delta w\right)\left( \int_{\partial B_1} ww_\nu +  \gamma a^{2\gamma}  \right) - \left(\int_{ B_1} (x\cdot\nabla w)\Delta w\right)\left(  \int_{\partial B_1} w^2 +a^{2\gamma} \right).
\]
Using again Lemma~\ref{lem:EFF1} we have
\[
{\rm rest} =  \left(\int_{B_1} w\Delta w\right)^2  + \left(\int_{B_1} w\Delta w\right) \big( D(1,w)+  \gamma a^{2\gamma}  \big) -
\left( \int_{B_1} (x\cdot \nabla w)\Delta w\right) \big(H(1,w) +a^{2\gamma} \big).
\]
By scaling (applying the previous formulas with $w$ replaced by $w_r =w(r\,\cdot\,)$ and $a$ replaced by~$r$) we obtain
\[
 \frac{d}{dr}\log \phi^\gamma(r,w) \ge \frac{2}{r}  \frac{ \big( r^{2-n} \int_{B_r} w\Delta w\big)^2  + E^\gamma(r,w) }{ \big(D(r,w) +\gamma r^{2\gamma}\big)\big( H(r,w) + r^{2\gamma}\big) },
\]
where $E^\gamma$ is defined in \eqref{eq:Erwgamma}. 
Since $\frac{d}{dr}\log \phi^\gamma(r,w)=\phi^\gamma(r,w)^{-1}\frac{d}{dr} \phi^\gamma(r,w)$,
the lemma follows immediately by recalling the definition of $\phi^\gamma$ in \eqref{eq:Almgren}.
\end{proof}

\section{Preliminaries: First and second blow-up analysis} \label{sec:EMF1}

In this section we collect some known results and basic tools that will be used throughout the paper.
Let $u: B_1  \rightarrow \R$ be a solution to the obstacle problem
\begin{equation}\label{eq:UELL}
\Delta u = \chi_{\{u>0\}}\quad \text{and}\quad u \geq 0\qquad   \mbox{in } B_1.
\end{equation}
By the classical theory of Caffarelli on the obstacle problem \cite{C-obst, C-obst2},  any solution $u$ of \eqref{eq:UELL} with $0\in \partial \{u>0\}$ satisfies
\begin{equation}\label{optimalreg+nondeg}
\|u\|_{C^{1,1}(B_{1/2})} \le C   \qquad \mbox{and}\qquad \sup_{B_r(0)} u\ge cr^2 \quad \forall \,r\in (0,{\textstyle \frac12}),
\end{equation}
where $C,c>0$ are positive dimensional constants. Moreover, points of the {\em free boundary} $\partial \{u>0\}$ can be split into two classes:
\begin{itemize}

\item  {\em Regular points}: $x_\circ\in \partial \{u>0\}$ is a regular point if
\[\lim_{r\to 0} r^{-2}u(x_\circ+ rx) = \frac 1 2 \big(\max\{0, \boldsymbol e\cdot x\}\big)^2\] for some $\boldsymbol e\in \mathbb S^{n-1}$.

\vspace{1mm}

\item {\em Singular points}: $x_\circ\in \partial \{u>0\}$ is a singular point if
\[
p_{2,x_\circ}(x) := \lim_{r\to 0} r^{-2}u(x_\circ+ rx)
\]
exists and  $p_{2,x_\circ}$ is a quadratic polynomial belonging to the set
\[\mathcal P := \big\{ \mbox{convex $2$-homogeneous polynomials $p$ with $\Delta p \equiv 1$} \big\}.\]
\end{itemize}

It is well known that the free boundary is analytic in a neighborhood of regular points. So, the main goal is to understand the structure of singular points.

When $x_\circ= 0$ is a singular point, we will simplify the notation $p_{2,x_\circ}$ to $p_2$.
A well known result due to Caffarelli is the following estimate at singular points.

\begin{lemma}\label{lem:EG2B1}
There exists a modulus of continuity $\omega:\R^+\to \R^+$, depending only on the dimension $n$, such that if $u$ is a solution of the obstacle problem \eqref{eq:UELL} and $0$ is a singular free boundary  point, then
\[
\|u-p_2\|_{L^\infty(B_r)} \le r^2 \omega(r) , \qquad \forall \,r\in (0,1).
\]
\end{lemma}

\begin{proof}
This result, with an abstract (dimensional) modulus of continuity $\omega$, is contained in \cite[Theorem 8]{C-obst2}. 
A stronger quantitative version of the estimate (with independent proofs) giving an explicit $C|\log r|^{-\ep}$ modulus of continuity is given in \cite{CSV17, AlessioJoaquim}.
\end{proof}

\begin{remark}\label{rem:Esigna}  Let $p\in \mathcal P$. Since  $\Delta u = \Delta p= 1$  in $\{u>0\}$, we have
\begin{equation}\label{eq:wlapw}
(u-p)\Delta (u-p)   =  p\chi_{\{u=0\}} \ge 0.
\end{equation}
Similarly,
\begin{equation}\label{eq:ZwLapw}
x\cdot \nabla (u-p)\, \Delta (u-p)= x\cdot \nabla p\chi_{\{u=0\}} = 2p\chi_{\{u=0\}}  \ge 0.
\end{equation}
\end{remark}

We recall Weiss' monotonicity formula (proved in \cite{W99} for $\lambda=2$, and in \cite{AlessioJoaquim} in the general case) and a useful consequence of~it.

\begin{lemma}[Weiss' formula]\label{LOClem:Weiss}
Let $u:B_1 \to [0,\infty)$ be a solution of  \eqref{eq:UELL}, and let $0$ be a singular point. Given $p \in \mathcal P$, set $w:= u-p$. Also, for $\lambda>0$ set
\[W_{\lambda} (r, w) :=   r^{-2\lambda}  \big( D(r,w)-   \lambda H(r,w)\big).\] Then:

{\rm (a)} For all $\lambda\ge 2$
\[
\frac{d}{dr} W_\lambda(r,w) \ge2 r^{-2\lambda-1} \int_{\partial B_r} (x\cdot \nabla w -\lambda w)^2 \ge0.
\]

{\rm (b)}  $W_2(0^+,w) =0$ and
\begin{equation}\label{eq:Dge2H}
D(r,w)- 2H(r,w) \ge 0, \quad\forall \,r\in (0,1).
\end{equation}
\end{lemma}

\begin{proof}
$(a)$ By scaling it is enough to compute  $\frac{d}{dr}W_{\lambda}(r,w)$ at $r=1$. Using Lemmas~\ref{lem:EFF1}-\ref{lem:EFF2}, we obtain
\[
\begin{split}
 W_{\lambda}'(1,w) &=    \big( D'(1,w)-  \lambda H'(1,w)\big)  - 2\lambda D(1,w) + 2\lambda^2 H(1,w)
\\ &=  2\int_{\partial B_1} w_\nu ^2 - 2\int_{B_1} \Delta w \,(x\cdot \nabla w) -2\lambda \int_{\partial B_1}  w w_\nu   - 2\lambda D(1,w)  + 2\lambda^2  H(1,w)
\\ & =  2\int_{\partial B_1} w_\nu ^2  + 2\int_{B_1} (\lambda w-x\cdot \nabla w) \Delta w  - 4\lambda \int_{\partial B_1}  w w_\nu   + 2\lambda^2   \int_{\partial B_1}  w^2
\\& =  2 \int_{\partial B_1} (w_\nu -\lambda w)^2 +   2\int_{B_1} (\lambda w-x\cdot \nabla w)\Delta w
\\& =  2 \int_{\partial B_1} (w_\nu -\lambda w)^2 +   2\int_{\{u(r\,\cdot\,)=0\}\cap B_1} (\lambda p-x\cdot \nabla p).
\end{split}
\]
One concludes noticing that, for $\lambda\ge2$, it holds  $(\lambda p-x\cdot \nabla p)= (\lambda-2)p\ge0$.

$(b)$ Since $0$ is a singular point then $w_r(x) = (u-p)(r x)  = (p_2-p) (rx)+ o(r^2)$, thus
 \[W_2(0^+, w) = \lim_{r\downarrow 0} W(1, r^{-2}w_r) =  W_2(1, p_2-p) =0.\]
As a consequence,  \eqref{eq:Dge2H} follows integrating $(a)$ for $\lambda=2$ between $0$ and $r$.
\end{proof}

We recall from  \cite{AlessioJoaquim} that  the frequency function $\phi$ applied to the function $w= u-p$, with $p\in \mathcal P$, is monotone increasing in $r$. More precisely we have the following:

\begin{proposition}[Frequency formula]\label{prop:EMF1}
Let $u:B_1 \to [0,\infty)$ be a solution of  \eqref{eq:UELL},  and  let $0$ be a singular point.  Given $p \in \mathcal P$, set $w:= u-p$.  Then
\[
\frac{d}{dr} \phi(r, w) \ge \frac 2 r  \frac{ \big( r^{2-n}  \int_{B_r }w\Delta w \big)^2 }{  H(r,w)^2  } \ge 0, \quad \forall \,r\in (0,1).
\]
\end{proposition}

\begin{proof}
By  Lemma~\ref{lem:EFF4} we just need to show that
\[
E(r,w) :=   \left(r^{2-n} \int_{B_r} w\Delta w\right) D(r,w) - \left(r^{2-n} \int_{B_r} (x\cdot \nabla w)\Delta w\right)H(r,w) \ge 0.
\]
Using Remark~\ref{rem:Esigna} for $w= u-p$ we have
\[\left(r^{2-n} \int_{B_r} (x\cdot \nabla w)\Delta w\right) = 2  \left(r^{2-n} \int_{B_r} w\Delta w\right),\] 
thus
\[
E(r,w)  =  \left(r^{2-n} \int_{B_r} w\Delta w\right)  \big(D(r,w) -2H(r,w)\big),
\]
which by \eqref{eq:Dge2H}  and Remark~\ref{rem:Esigna} is nonnegative.
\end{proof}

The following observation, also contained in \cite{AlessioJoaquim}, follows immediately from \eqref{eq:Dge2H}.

\begin{lemma}\label{lem:EMF2}
Let $u:B_1 \to [0,\infty)$ be a solution of  \eqref{eq:UELL},  and  let $0$ be a singular point.  Given $p \in \mathcal P$, set $w:= u-p$.
Then $\phi(0^+,w) \geq 2$.
\end{lemma}

A new important estimate that we will use throughout the paper is the following:

\begin{lemma}\label{lem:EMF3}
Let $u:B_1 \to [0,\infty)$ be a solution of  \eqref{eq:UELL},  and  let $0$ be a singular point.   Given $p \in \mathcal P$, set $w:= u-p$. Suppose  that for  $0<r<R<1$ we have $\underline \lambda \le \phi(r, w) \le \phi(R,w)\le \overline \lambda$. Then,  for any given $\delta >0$ we have
\[
\left(\frac{R}{r} \right)^{2\underline \lambda}  \le \frac{H(R,w)}{H(r,w)} \le  C_\delta  \left(\frac{R}{r} \right)^{2\overline \lambda+\delta},
\]
where $C_\delta$ depends only on $n$, $\overline \lambda$, $\delta$.
\end{lemma}

\begin{proof}
Define
\[
F(r)  := \frac{r^{2-n}  \int_{B_r}  w\Delta w }{H(r,w)}.
\]
By Proposition~\ref{prop:EMF1} we have
\begin{equation}\label{derphia}
\frac{d}{dr} \phi(r,w) \ge   \frac{2}{r} \big(F(r)\big)^2.
\end{equation}
On the other hand, thanks to Lemma~\ref{lem:EFF1},
\begin{equation}\label{derHa}
\begin{split}
\frac{\frac{d}{dr} H(r,w)}{H(r,w)} &= \frac  2 r \, \frac{r^{2-n}  \int_{B_r}  w\Delta w + r^{2-n}\int_{B_r} |\nabla w|^2 }{H(r,w)}
=   \frac 2 r   \phi(r,w) +  \frac{2}{r} F(r).
\end{split}
\end{equation}
Integrating \eqref{derphia} and using Cauchy-Schwartz inequality, since $\underline \lambda \le \phi(\rho, w)\le \overline \lambda$  for all $\rho\in (r,R)$
we get
\[
\begin{split}
\big(\overline\lambda- \underline \lambda\big)^{1/2} \big(\log (R/r)\big)^{1/2}&\ge \left(\int_r^R  {\textstyle \frac{d}{d\rho}}\phi(\rho,w)   d\rho \right)^{1/2}   \left(\int_r^R  \frac{d\rho}{\rho} \right)^{1/2}
\\
 &=  \left(\int_r^R  \frac{1}{\rho} \big(F(\rho)\big)^2   d\rho \right)^{1/2}   \left(\int_r^R \frac{d\rho}{\rho} \right)^{1/2}
\ge\int_r^R  \frac{1}{\rho} F(\rho)   d\rho \ge0.
\end{split}
\]
Hence, integrating \eqref{derHa}, we obtain
\[
\log \frac{H(R,w)}{H(r,w)} \le \int_r^R \frac{2}\rho \big(\overline \lambda + F(\rho)  \big)\, d\rho \le \log \big( (R/r)^{2\overline \lambda} \big) + C\big(\log (R/r)\big)^{1/2} \le  \log \big((R/r)^{2(\overline \lambda+\delta)}\big) +C,
\]
where $C$ depends only on $n$, $\overline \lambda$, and $\delta$.

For the lower bound we recall that, since $w\Delta w \ge 0$, we have $F(\rho)\ge 0$. Therefore, integrating \eqref{derHa} over $[r,R]$,
\[
\log \frac{H(R,w)}{H(r,w)} \ge 2\underline\lambda \int_r^{R} \frac{d\rho}{\rho} = \log( R/r)^{2\underline \lambda}.
\]
\end{proof}

We will also need the following result, which allows us to control the $L^\infty$ norm of the difference of two solutions with the $L^2$ norm.

\begin{lemma}\label{lem:u-v}
Let $u:B_1 \to [0, \infty)$ and $v:B_1 \to [0, \infty)$ be solutions of the obstacle problem \eqref{eq:UELL}.  Then
\[
\| u-v \|_{L^\infty(B_{1/2})}    \le C(n)\| u-v \|_{L^2(B_{1})}.
\]
\end{lemma}

\begin{proof}
On the one hand, from 
\[
\Delta(u-v) = 1- \Delta v \ge  0 \quad \mbox{in } \{u>0\} \qquad \mbox{and}\qquad    u-v =  -v  \le 0 \quad  \mbox{in } \{u=0\} 
\]
we obtain that $(u-v)_+ = \max(u-v, 0)$ is  subharmonic in $B_1$. Exchanging the role of $u$ and $v$, the same argument shows that $(v-u)_+ = (u-v)_-$ is subharmonic.
Thus also $|u-v|= (u-v)_++ (u-v)_-$ is subhamonic in $B_1$, and using the mean value property we obtain, for $x\in B_{1/2}$,
$$
|u-v|(x) \le \ave_{B_{1/2}(x)} |u-v| \le  C(n) \|u-v\|_{L^1(B_1)} \le  C(n) \|u-v\|_{L^2(B_1)}. 
$$
\end{proof}

We now start investigating the structure of possible second blow-ups.
The next few results are a small modification of those in \cite{AlessioJoaquim}.

The following Lipschitz estimate for the rescaled difference $u-p,$ with $p \in \mathcal P$, will be useful in the sequel. We recall that $w_r$ and $\tilde w_r$ have been defined in \eqref{defwr_tildewr}.
\begin{lemma}\label{lem:E2B1}
Let  $u:B_1 \to [0,\infty)$ be a solution of  \eqref{eq:UELL} with $u\not\equiv p_2$,  and  let $0$ be a singular point.   Given $p \in \mathcal P$, set $w:= u-p$. Let $R\ge1$, and $r\in \left(0,\frac{1}{10R}\right)$. Then
\begin{equation}\label{semic}
\partial_{\boldsymbol e \boldsymbol e} \tilde w_r \ge -C \quad \mbox{in }B_{R},  \qquad  \forall\, \boldsymbol e\in \mathbb S^{n-1}\cap\{p=0\} \end{equation}
where $C$ depends only on $n$, $R$, and $\phi(1/2, u-p)$.
In addition, if ${\rm dim}(\{p=0\})=n-1$, then 
\begin{equation}
\label{eq:wk Lip0}
|\nabla \tilde w_r| \le C  \qquad \mbox{in } B_{R/2},
\end{equation}
where $C$ depends only on $n$, $R$, and $\phi({\textstyle \frac12}, u-p)$.
\end{lemma}

\begin{proof}
This proof is essentially contained in Step 3 from the Proof of Proposition 2.10 in \cite{AlessioJoaquim}.  However, since some small changes are needed in our setting, we reproduce the main steps for the convenience of the reader.

Given a function $f: \R^n \to \R$, a vector $\boldsymbol e \in \mathbb S^{n-1}$, and $h>0$, let
\[\delta^2_{\boldsymbol e,h} f:= \frac{f(\,\cdot\,+h\boldsymbol e)+f(\,\cdot\,-h\boldsymbol e)-2f}{h^2}\]
 denote a second order incremental quotient.  For $\boldsymbol e\in  \{p=0\}\cap \mathbb S^{n-1}$ we have $\delta^2_{\boldsymbol e,h} p\equiv 0$. Thus, since $\Delta u =1$ outside of  $\{u=0\}$ and $\Delta u\le 1$ in $B_1,$ we have
\[
\Delta \big(\delta^2_{\boldsymbol e,h} w \big)
 =  \frac{\Delta u\big(\,\cdot\,+h\boldsymbol e\big)+\Delta u\big(\,\cdot\,-h\boldsymbol e\big) - 2\Delta u }{h^2}
\le 0  \quad\mbox{ in } B_R \setminus \{u=0\}.
\]
On the other hand, since $u\ge 0$ we have
\[
\delta^2_{\boldsymbol e,h} w = \delta^2_{\boldsymbol e,h}  u(\,\cdot\,) \ge 0  \quad \mbox{ in } \{u=0\}.
\]
As a consequence, the negative part of the second order incremental quotient $(\delta^2_{\boldsymbol e,h} \tilde w_r)_-$ is subharmonic, and so is its limit $(\partial_{\boldsymbol e\boldsymbol e}^2\tilde w_r)_-$
(recall that $u$ is semiconvex,  and thus $(\delta^2_{\boldsymbol e,h} \tilde w_r)_- \to( \partial_{\boldsymbol e\boldsymbol e}^2\tilde w_r)_-$ a.e. as $h \to 0$).
Therefore, by weak Harnack inequality (see for instance  \cite[Theorem 4.8(2)]{CC95}) there exists $\ep=\ep(n) \in (0,1)$ such that
\[
\|(\partial_{\boldsymbol e\boldsymbol e}^2\tilde w_r)_- \|_{L^\infty(B_{R})}
 \le C(n) \biggl(\ave_{B_{R}} (\partial_{\boldsymbol e\boldsymbol e}\tilde w_r)^\ep_-  \biggr)^{1/\ep}\leq C(n) \biggl(\ave_{B_{2R}} |\partial_{\boldsymbol e\boldsymbol e}\tilde w_r|^\ep  \biggr)^{1/\ep}.
\]
Also, by standard interpolation inequalities, the $L^\ep$ norm (here we use $\ep<1$) can be controlled by the weak $L^1$ norm, namely
$$
\biggl(\int_{B_{2R}} |\partial_{\boldsymbol e\boldsymbol e}\tilde w_r|^\ep \biggr)^{1/\ep} \leq C(n,R)\,\sup_{\theta >0} \,\theta \bigl|\bigl\{|\partial_{\boldsymbol e\boldsymbol e}\tilde w_r| >\theta\bigr\}\cap B_{2R }\bigr|.
$$
Furthermore, by Calder\'on-Zygmund theory, the right hand side above is controlled by
\[
\|\Delta \tilde w_r\|_{L^1(B_{3R})}+
\|\tilde w_r\|_{L^1(B_{3R})}.
\]
Thus, since $ \Delta w_r \leq 0$ in $B_{3R}$, $\|\Delta \tilde w_r\|_{L^1(B_{R})}$ is controlled by the $L^1$ norm of $\tilde w_r$ inside $B_{4R}$: indeed, if $\chi$ is a smooth nonnegative cut-off function that is equal to $1$ in $B_{3R}$ and vanishes outside $B_{4R}$, then
\begin{equation}
\label{eq:control Delta w}
\|\Delta \tilde w_r\|_{L^1(B_{3R} )}\leq -\int_{B_{4R}}\chi\,\Delta \tilde w_r =
-\int_{B_{4R}} \Delta \chi\, \tilde w_r\leq C(n,R)\int_{B_{4R}\setminus B_{3R}}|\tilde w_r|.
\end{equation}
Also, for $8rR<1$, as a consequence of Lemma~\ref{lem:EMF3} we have
\[
 H(4R, \tilde w_r) \le C R^{2\phi(1,u-p)+1}   H(1, \tilde w_r) = C R^{2\phi(\frac 1 2,u-p)+1}
\]
and thus
\[
 \|\tilde w_r\|_{L^1(B_{4R})} \le C\big(n, R, \phi(\textstyle \frac 1 2 ,u-p)\big).
\]
In conclusion, we obtain
$$
\|(\partial_{\boldsymbol e\boldsymbol e}\tilde w_r)_- \|_{L^\infty(B_R)}\leq C(n,R)
  \|\tilde w_r\|_{L^1(B_{4R})} \leq  C\big(n, R, \phi(\textstyle \frac 1 2 ,u-p)\big).
$$
Finally note that,  when $\{p=0\}$ is $(n-1)$-dimensional,  as a consequence of \eqref{semic}  the tangential Laplacian of $\tilde w_r$ (in the directions of $\{p=0\}$) is uniformly bounded from below. Thus, since
 $\Delta \tilde w_r  \le 0 $, we have
\[
\partial_{\boldsymbol e'\boldsymbol e'} \tilde w_r \le C  \qquad \mbox{in } B_R ,\quad  \mbox{ for } \boldsymbol e'\in \{p=0\}^\perp \mbox{ with } |\boldsymbol e'|=1,
\]
where, as before, $C= C\big(n, R, \phi(\frac 1 2 ,u-p)\big)$.
Thanks to these semiconvexity and semiconcavity estimates, we deduce the Lipschitz bound \eqref{eq:wk Lip0}.
\end{proof}

The next result corresponds to \cite[Proposition 2.10]{AlessioJoaquim}.
However the statement there has a small mistake (that anyhow does not affect any of the arguments in \cite{AlessioJoaquim}) and for convenience we correct it here.

\begin{proposition}
\label{prop:E2B2}
Let $u:B_1 \to [0,\infty)$ be a solution of  \eqref{eq:UELL} with $u\not\equiv p_2$, let $0$ be a singular point, and set  $w:=u-p_2$.
Denote  $m: = {\rm dim }(\{p_2=0\}) \in \{0,1,2,\dots n-1\}$, and  $\lambda^{2nd}:=\phi(0^+,w)$.

Then, for every sequence $r_k\downarrow 0$ there is a subsequence $r_{k_\ell}$ such that
$\tilde  w_{r_{k_\ell}} \rightharpoonup  q$ as $\ell \to \infty$  in $W^{1,2}_{\rm loc}(\R^n)$, where
$q\not\equiv 0$ is a $\lambda^{2nd}$-homogeneous function satisfying the following:

\begin{enumerate}
\item[(a)]
If $0\le m\le n-2$ then $q$ is a harmonic polynomial, and in particular $\lambda^{2nd} \in\{2,3,4, \dots\}$.
In addition, if $\lambda^{2nd}=2$, then in some appropriate coordinates we have
\begin{equation}\label{howisD2q}
p_2(x)= \frac{1}{2} \sum_{i=m+1}^n \mu_i x_i^2 \qquad \mbox{and}\qquad  q(x)= \nu \sum_{i=m+1}^n  x_i^2  - \sum_{j=1}^m \nu_j x_j^2,
\end{equation}
where $\mu_{i},\nu>0$, and they satisfy $\sum_{i=m+1}^{n}\mu_i=1$,$(n-m)\nu = \sum_{j=1}^m \nu_j$, and $|\nu_j|\leq \nu$ for any $j=1,\ldots,m$.

\item[(b)] If $m=n-1$ then we have $\tilde  w_{r_{k_\ell}} \rightarrow  q$ in $C^0_{\rm loc}(\R^n)$ and we have $\lambda^{2nd} \ge\ 2+\alpha_\circ$, where $\alpha_\circ >0$ is a dimensional constant. In addition, $q$ is a global solution of the Signorini problem:
\begin{equation}\label{ETOP}
\begin{cases}
\Delta q\le 0\quad \text{and} \quad q\Delta q=0 \quad & \mbox{in }\R^n
\\
\Delta q=0  &\mbox{in }\R^n\setminus  \{p_2=0\}
\\
\,q\ge 0 &\mbox{on } \{p_2=0\}.
\end{cases}
\end{equation}
\end{enumerate}
\end{proposition}

\begin{proof}
The statement here is almost identical to that of  \cite[Proposition 2.10]{AlessioJoaquim}. The only differences are the following:\\
(1) In \cite{AlessioJoaquim} it is uncorrectly stated that $\nu_j>0.$ Instead, \cite[Lemma 2.12]{AlessioJoaquim} proves that $\nu$ is the largest eigenvalue of $D^2q$, hence the correct conclusion is that $|\nu_j|\le \nu$ for each $j=1,\ldots,m$.\\
(2) In the above statement we said that 
$\tilde w_k\rightharpoonup q$ weakly in $W^{1,2}_{\rm loc}(\R^n)$, while  \cite[Proposition 2.10]{AlessioJoaquim} states the convergence only in  $W^{1,2}(B_1)$. The reason why we may replace $B_1$ by any larger ball is Lemma~\ref{lem:EMF3}, as it allows us to control $H(R,\tilde w_r)$ by $C(n, \phi(1,w), R )  H(1,\tilde w_r)$ for any $r<1/R$. Hence, using a diagonal argument, the proof of \cite[Proposition 2.10]{AlessioJoaquim} yields the desired result.
\end{proof}

We now recall another important estimate from \cite{AlessioJoaquim}:

\begin{proposition}\label{prop:E2B3}
Let  $u:B_1 \to [0,\infty)$ be a solution of  \eqref{eq:UELL} with  $u\not\equiv p_2$, let $0$ be a singular point, and set $w:=u-p_2$,
$\lambda^{2nd}:=\phi(0^+,w)$. Let $\lambda\in (0, \lambda^{2nd}]$.
 Then
 $$
 \frac{|\{u =0\}\cap B_r|}{|B_r|}\leq Cr^{\lambda-2}\qquad \forall\,r \in (0,1/2).
 $$
Moreover, if  ${\rm dim}(\{p_2=0\}) =n-1$  then
\[
\{u=0\}\cap B_r \subset \big\{ x \, : \, {\rm dist}(x, \{p_2=0\}) \le C r^{\lambda-1}\big\}.
\]
The constant $C$ depends only on $n$ and $\lambda$.
\end{proposition}

\begin{proof}
The first part is exactly  \cite[Proposition 2.13]{AlessioJoaquim}.
The second part on $\Sigma_{n-1}$ follows by the argument in \cite[Remark 2.14]{AlessioJoaquim}, as a consequence of the Lipschitz estimate in Lemma~\ref{lem:E2B1}.
\end{proof}

Following the notation introduced in \cite{AlessioJoaquim}, we denote
\begin{equation} \label{ahoiah0}
\Sigma_m :=  \big\{ x_\circ \mbox{ singular points with } {\rm dim}\big( \{p_{2,x_\circ} =0 \}\big) =m    \big\}\,, \quad 0\le m \le n-1
\end{equation}
and, for $m\leq n-2$,
\begin{equation} \label{ahoiah7}
\Sigma^a_m := \big\{ x_\circ\in  \Sigma_m \mbox{ such that } \phi\big(0^+, u(x_\circ+\,\cdot\,)-p_{2,x_\circ}\big) =2 \big\}\,,  \quad 0\le m \le n-2.
\end{equation}
Moreover, for $m=n-1$ we introduce some further notation: we define
\begin{equation} \label{ahoiah8}
\Sigma_{n-1}^{<3} := \big\{ x_\circ\in  \Sigma_{n-1} \mbox{ such that } \phi\big(0^+, u(x_\circ+\,\cdot\,)-p_{2,x_\circ}\big) < 3 \big\},
\end{equation}
and
\begin{equation}\label{ahoiah1}
\Sigma_{n-1}^{\ge 3} := \Sigma_{n-1}\setminus \Sigma_{n-1}^{<3}.
\end{equation}

Finally, we will need the following:

\begin{definition}\label{def:Sigma3rd}
Let $u:B_1\to[0,\infty)$ solve \eqref{eq:UELL}.
 For $1\le m \le n-1$ we denote by $\Sigma_{m}^{3rd}$ the set of points $x_\circ\in \Sigma_m$  such that, for  $w:= u(x_\circ+\,\cdot\,)-p_{2, x_\circ}$, the following two conditions hold:
 \begin{itemize}
 \item[(i)] $\phi(0^+,w)\ge3$;
\item[(ii)] there exists some sequence $r_k\downarrow 0$ along which  $r_k^{-3} w(r_k\,\cdot\,)$ converges, weakly in $W^{1,2}_{\rm loc}(\R^n)$, to some 3-homogeneous harmonic polynomial vanishing on $\{p_{2,x_\circ}=0\}$ ---possibly the polynomial zero.
 \end{itemize}
\end{definition}
 
Notice that, by Proposition~\ref{prop:E2B2}(a), for $m\leq n-2$ we have $\Sigma_m\setminus\Sigma_m^a=\Sigma_{m}^{3rd}=\Sigma_{m}^{\ge3}$.
On the other hand, this is not true for $m=n-1$, and later on in the paper we will need to understand the set $\Sigma^{\ge 3}_{n-1}\setminus\Sigma^{3rd}_{n-1}$.

We conclude this section by recalling that if $0 \in \Sigma^{3rd}_{m}$ then the limit 
\begin{equation}\label{cubic-lim}
\lim_{r\downarrow 0} r^{-3} (u-p_2)(r\,\cdot\,) 
\end{equation}
exists.
Indeed, as shown in \cite{AlessioJoaquim}, this is a consequence of the following Monneau-type monotonicity formula.

\begin{lemma}\label{lem:EGB6}
Let $u:B_1 \to [0,\infty)$ satisfy \eqref{eq:UELL}, and let $0\in \Sigma_m^{\geq3}$ for some $0\leq m\leq n-1$.  
Let $w:= u-p_2 - P$, where $P$ is any 3-homogeneous harmonic polynomial vanishing on  $\{p_2=0\}$.
Then
\begin{equation}\label{eq:EGB6a}
D(r,w)\ge 3 H(r,w) \qquad \forall \,r\in (0,1)
\end{equation}
and
\begin{equation}\label{eq:EGB6b}
\frac{d}{dr} \big( r^{-6}H(r,w) \big) \ge - C \sup_{\partial B_1} \frac{P^2}{p_2},
\end{equation}
where $C$ is some dimensional constant.
\end{lemma}

\begin{proof}
The proof is contained in \cite[Lemma 4.1]{AlessioJoaquim}.
\end{proof}

The next result provides the existence of a unique limit in \eqref{cubic-lim} for all points in $\Sigma^{3rd}_m$, which follows immediately from Lemma~\ref{lem:EGB6} (see \cite[Proposition 4.5]{AlessioJoaquim}):

\begin{lemma}\label{lem:EGB7}
Let  $u:B_1\to[0,\infty)$ solve \eqref{eq:UELL}. 
Then, for all $x_\circ$ in $\Sigma_m^{3rd}$ with $0\leq m\leq n-1$, the limit
\[
p_{3, x_\circ} := \lim_{r\downarrow 0} \frac{1}{r^3} \big( u(x_\circ+r \,\cdot\,) - p_{2, x_\circ}(r\,\cdot\,) \big)
\]
exists, and $p_{3, x_\circ}$ is a $3$-homogeneous harmonic polynomial vanishing on  $\{p_{2, x_\circ}=0\}$.
\end{lemma}

%

When $x_\circ= 0$  we simplify the notation $p_{3,0}$ to $p_3$.

\section{Higher order blow-ups on the maximal stratum} \label{sec:E3B}

As explained in Section~\ref{sect:intro higher order}, in order to prove the main result of this paper (Theorem~\ref{thm-Schaeffer-intro}) we need to obtain ---among other things--- an expansion up to order $O(|x|^{5-\zeta})$ at ``most points'' of $\Sigma_{n-1}^{3rd}$.
This requires a very detailed analysis of such set, which is the goal of this section.
From now on, we will only study the points of  $\Sigma_{n-1}$ (hence, $m = n-1$).

In order to study the higher regularity properties of the set $\Sigma_{n-1}^{3rd}$, we need a new frequency function for $u-p_2-p_3$.
The following lemma is a more flexible version of Lemma~\ref{lem:EMF3}.
It will be useful later in order to prove the almost monotonicity of $\phi^\gamma(r,w)$ for a suitable $\gamma$, where $w$ will be the difference between $u$ and its polynomial expansions at singular points.

\begin{lemma}\label{lem:E3B1c}
Let $R\in (0,1)$,  and let  $w:B_R\to  [0,\infty)$ be a $C^{1,1}$ function.
Assume that for some $\kappa\in (0,1)$ we have
\[
\frac{d}{dr} \phi^\gamma(r,w) \ge \frac 2 r  
\frac{\left(r^{2-n} \int_{B_r} w\Delta w\right)^2}{\big(H(r,w) + r^{2\gamma}\big)^2}  - r^{\kappa-1} \qquad \forall \,r\in (0,R).
\]
Then the following holds:

{\rm (a)} Suppose  that $0<\underline \lambda \le \phi^\gamma(r, w)\le \overline \lambda$  for all $r\in (0,R)$. Then,  for any given $\delta >0$ we have
\[
\frac{1}{C_\delta} \left(\frac{R}{r} \right)^{2\underline \lambda-\delta}  \le \frac{H(R,w)+R^{2\gamma}}{H(r,w)+r^{2\gamma}} \le  C_\delta  \left(\frac{R}{r} \right)^{2\overline \lambda+\delta}\qquad \textrm{for all}\quad r\in(0,R),
\]
where $C_\delta$ depends only on $n$, $\gamma$, $\kappa$, $\overline \lambda$, $\delta$.

{\rm (b)} Assume in addition that
\[
\frac{ r^{2-n} \int_{B_r} w\Delta w }{H(r,w) + r^{2\gamma}}  \ge -r^{\kappa}\qquad \forall \,r\in (0,R).
\]
Then, for $\lambda_*  : = \phi^\gamma(0^+,w)$, we have
\[
e^{-\frac{4}{\kappa^2}}   \left(\frac{R}{r}\right)^{2 \lambda_*} \le \frac{H(R, w) +R^{2\gamma}}{H(r,w) +r^{2\gamma}}.
\]
\end{lemma}

\begin{proof}
(a) Define
\[
F(r)  := \frac{r^{2-n}  \int_{B_r}  w\Delta w }{H(r,w) + r^{2\gamma}}
\]
so that, by assumption, we have
\begin{equation}\label{derphia1}
\frac{d}{dr} \phi^\gamma(r,w)  + r^{\kappa-1} \ge   \frac{2}{r} \big(F(r)\big)^2.
\end{equation}
It follows by Lemma~\ref{lem:EFF1}  that
\begin{equation}\label{derHa1}
\frac{\frac{d}{dr} (H(r,w) +r^{2\gamma})}{ (H(r,w)+r^{2\gamma})} = \frac  2 r \, \frac{r^{2-n}  \int_{B_r}  w\Delta w + r^{2-n}\int_{B_r} |\nabla w|^2  +\gamma r^{2\gamma} }{H(r,w) +r^{2\gamma}}
= \frac 2 r   \phi^\gamma(r,w) +  \frac{2}{r} F(r).
\end{equation}
Integrating \eqref{derphia1} and using Cauchy-Schwarz inequality, since $0 \le \phi^\gamma(r, w)\le \overline \lambda$  for all $r\in (0,R)$, we get
\begin{equation}\label{whiohg3oghg}
\begin{split}
\left|\int_r^R  \frac{1}{\rho} F(\rho)   d\rho \right| &\le  \left(\int_r^R  \frac{1}{\rho} \big(F(\rho)\big)^2   d\rho \right)^{1/2}   \left(\int_r^R \frac{d\rho}{\rho} \right)^{1/2}
\\
& \le \left(\int_r^R  \big( {\textstyle \frac{d}{dr}}\phi^\gamma(\rho,w)  + \rho^{\kappa-1} \big) d\rho  \right)^{1/2}   \left(\int_r^R  \frac{d\rho}{\rho} \right)^{1/2}
\\
& \le \left(\overline\lambda +\frac 1 \kappa (R^\kappa- r^{\kappa})\right)^{1/2}\big( \log (R/r) \big) ^{1/2}
\\
&\le C\big(\log (R/r)  \big)^{1/2}.
\end{split}
\end{equation}

Hence, integrating \eqref{derHa1} between $r$ and $R$ (recall $0<r<R<1$) and using  $\phi^\gamma(\rho,w) \le \overline \lambda$ and \eqref{whiohg3oghg} we obtain
\[
\log \frac{H(R,w) + R^{2\gamma} }{H(r,w) +r^{2\gamma}} \le \int_r^R \frac{2}\rho \big(\overline \lambda + F(\rho)  \big)\, d\rho \le 2\overline \lambda\log (R/r) + C\big(\log (R/r)\big) ^{1/2}
\le  (2\overline \lambda+\delta) \log (R/r) +C_\delta.
\]
Similarly (now using $\phi^\gamma(\rho,w) \ge \underline \lambda$) we obtain
\[\begin{split}
\log \frac{H(R,w) + R^{2\gamma} }{H(r,w) +r^{2\gamma}} &\ge   (2\underline \lambda-\delta) \log (R/r)  -C_\delta.
\end{split}
\]

(b) In this case we have $ F(r)\ge -r^\kappa$ for all $r\in (0,R)$. Hence, integrating  \eqref{derHa1} between $0$ and $\rho\in (0,R)$ we obtain
\[
\phi^\gamma(\rho,w) - \lambda_* \ge  -\frac 1 \kappa \rho^\kappa.
\]
Thus,
\[
\log \frac{H(R,w) + R^{2\gamma} }{H(r,w) +r^{2\gamma}} \ge \int_r^R \frac{2}\rho \big(\phi^\gamma(\rho,w) + F(\rho)  \big)\, d\rho
 \ge 2\lambda_* \int_r^R   \frac{d\rho}{\rho} - \left(\frac{2}{\kappa}  +1\right) \int_r^R   \rho^{\kappa-1}\,d\rho
\ge 2\lambda_*\log (R/r) - \frac{4}{\kappa^2}.
\]
\end{proof}

\begin{remark}\label{rem:phi gamma}
An interesting consequence of Lemma~\ref{lem:E3B1c}(a) is the following. If $w$ is as in of Lemma~\ref{lem:E3B1c}(a)  then $\phi^\gamma(0^+,w) \leq \gamma$.
Indeed, otherwise we would have $\phi^\gamma(r,w)\geq \underline\lambda :=\gamma+\delta>\gamma$ for all $r>0$ small, and Lemma~\ref{lem:E3B1c}(a) would imply that
$H(r,w)+r^{2\gamma}\leq Cr^{2\underline\lambda-\delta}=Cr^{2\gamma+\delta}$ for $r\ll 1$, impossible.
\end{remark}

The following lemma gives the (approximate) monotonicity of $\phi^\gamma$ when applied to $w:= u-p_2-P$, where $P$ is any $3$-homogeneous harmonic polynomial vanishing on $\{p_2=0\}$.

\begin{lemma}\label{lem:E3B1}
Let $u:B_1\to  [0,\infty)$ be a solution of \eqref{eq:UELL}, with $0\in \Sigma_{n-1}^{\ge3}$. 
Let $w : = u - p_2- P$, where $P$ is a $3$-homogeneous harmonic polynomial vanishing on $\{p_2=0\}$.  Then, given $\gamma\in(3,4)$, for all $r\in (0,1/2)$ we have
\[
\frac{d}{dr} \phi^\gamma(r,w) \ge   \frac{2}{r}  \frac{ \big( r^{2-n} \int_{B_r}w\Delta w\big)^2  }{ \big( H(r,w) + r^{2\gamma}\big)^2 } - C r^{3-\gamma}\qquad \text{and}\qquad
\frac{r^{2-n} \int_{B_r}w\Delta w  }{ H(r,w) + r^{2\gamma} } \ge - C  r^{4-\gamma},
\]
where $C$  depends only on $n$, $\gamma$,  $\|P\|_{L^2(B_1)}$.

In particular,  assuming $0\in \Sigma^{3rd}_{n-1}$, the previous inequalities hold for $w:=u-p_2-p_3$.
\end{lemma}

\begin{proof}
The proof will rely on a iteration argument where one enlarge the value of $\gamma$, starting from $3$ and increasing it up to the desired $\gamma \in (3,4).$
We split the proof in two steps.

\smallskip

\noindent $\bullet$ \emph{Step 1}. We first show that
\begin{equation}\label{goal100}
 \frac{d}{dr} \phi^\gamma(r,w)  \ge\frac{2}{r}  \frac{ \big( r^{2-n} \int_{B_r}w\Delta w\big)^2  }{ \big( H(r,w) + r^{2\gamma}\big)^2 }  -  Cr^{3-\gamma}\phi^\gamma(r,w) g^\gamma(r)
\end{equation}
and
\begin{equation}\label{wlap1xx}
\frac{r^{2-n} \int_{B_r}w\Delta w  }{ H(r,w) + r^{2\gamma} }   \ge - Cr^{4-\gamma} g^\gamma(r),
\end{equation}
where
\[
g^\gamma (r) :=  \frac{ \|w_r\|_{L^2(B_2\setminus B_1)}}{(H(r,w) + r^{2\gamma})^{1/2}}
\]
and $C$ depends only on $n$ and $\|P\|_{L^2(B_1)}$.

Indeed, by Lemma~\ref{lem:EFF4} we have
\[
 \frac{d}{dr} \phi^\gamma(r,w) \ge \frac{2}{r}  \frac{ \big( r^{2-n} \int_{B_r}w\Delta w\big)^2  }{ \big( H(r,w) + r^{2\gamma}\big)^2 } + \overline E ^\gamma(r,w),
\]
where
\[
\overline E^\gamma(r,w) :=
  \frac{ 2}{r}\frac{r^{2-n} \int_{B_r} (\lambda_r w - x\cdot \nabla w) \Delta w }{ H(r,w) + r^{2\gamma}}
\qquad
\text{and}
\qquad
\lambda_r = \phi^\gamma(r,w) = \frac{D(r,w) + \gamma r^{2\gamma}} {H(r,w) + r^{2\gamma} }.
\]
Note that  $\Delta w = \Delta(u-p_2-p_3) = \chi_{\{u>0\}} - 1-0 = -\chi_{\{u=0\}}$. Also, since $\lambda_r\ge 3$  by  \eqref{eq:EGB6a}, using the inequality $p_2 + P \ge -\frac{P^2}{2p_2}$, since  $\frac{P^2}{p_2}$ is homogeneous of degree $4$ we have
\[
(\lambda_r  - x\cdot \nabla)(p_2 +P)   \ge (\lambda_r-2)p_2 +  (\lambda_r-3) P \ge  (\lambda_r-3) (p_2+P)
\ge - \frac{(\lambda_r-3)}{2}\biggl(\sup_{\partial B_1} \frac{P^2}{p_2} \biggr)|x|^4.
\]
Therefore we obtain
\begin{equation}\label{ashbon}
\overline E^\gamma(r,w) = \frac{ 2}{r}  \frac{r^{2-n} \int_{B_r\cap\{u=0\}} (\lambda_r  - x\cdot \nabla)(p_2 +P)}{H(r,w) + r^{2\gamma}}
 \ge  - (\lambda_r-3) r^{3-\gamma}  \biggl(\sup_{\partial B_1} \frac{P^2}{p_2} \biggr)   \frac{ r^{2-n} |B_r\cap\{u=0\}|}{ (H(r,w) + r^{2\gamma})^{1/2} }.
\end{equation}
Also, since $\Delta w = -\chi_{\{u=0\}}\le 0$, choosing $\chi\in C^\infty_c(B_2)$ a nonnegative cut-off satisfying $\chi=1$ in $B_1$,
integrating by parts we obtain
\[
r^{2-n}  \big|\{u=0\}\cap B_r\big| = \int_{B_1}  -\Delta w_r \le  \int_{B_2} -\Delta w_r\, \chi = -\int_{B_2} w_r \,\Delta\chi  \le C \int_{B_2\setminus B_1} |w_r| \le C\|w_r\|_{L^2(B_2\setminus B_1)} .
\]
Thus, since $\lambda_r=\phi^\gamma(r,w)$ and recalling \eqref{ashbon}, we have shown that
\[
\overline E^\gamma(r,w) \ge -  Cr^{3-\gamma} \phi^\gamma(r,w) g^\gamma(r),
\]
and  \eqref{goal100} follows. Note that, since by assumption $P^2$ is divisible by $p_2,$ we have that $\sup_{\partial B_1} \frac{P^2}{ p_2}$ is bounded by a constant depending only on $n$ and $\|P\|_{L^2(B_1)}$, and thus the constant $C$ above depends only on $n$ and $\|P\|_{L^2(B_1)}$.

Similarly, using  again $p_2 + P \ge - \frac{P^2}{2p_2}$, we obtain
\[
\frac{r^{2-n} \int_{B_r}w\Delta w  }{ H(r,w) + r^{2\gamma} }  =  \frac{r^{2-n} \int_{\{u=0\} \cap B_r} (p_2 +P)  }{ H(r,w) + r^{2\gamma} }  \ge  - Cr^{4-\gamma}   \frac{ r^{2-n} \big|\{u=0\} \cap B_r\big|}{ (H(r,w) + r^{2\gamma})^{1/2} },
\]
which gives \eqref{wlap1xx}.

\smallskip

\noindent $\bullet$ \emph{Step 2.} Next we show that \eqref{goal100} implies that, for all $\gamma<4$,
\begin{equation}\label{boundphi00}
\phi^\gamma(r, w)   \le  C_\gamma \quad \mbox{and} \quad  g^\gamma(r)   \le  C_\gamma,
\end{equation}
with $C_\gamma$ depending only on
$n$, $\gamma$, and $\|P\|_{L^2(B_1)}$.

We prove \eqref{boundphi00} for all $\gamma\in [3,4)$ by iteratively increasing the value of $\gamma$ at each iteration, starting from $\gamma=3$, in order to always have (along the iteration) a uniform bound on $\phi^\gamma(r,w)$ and $g^\gamma(r)$.

First,  we observe that since $0\in \Sigma^{\geq 3}_{n-1}$ we have $\phi(0^+, u-p_2)\ge 3$, hence $|u-p_2| \le C|x|^3$.
Therefore $|u-p_2-P| \le C|x|^3$, which immediately implies that $g^3(r)\leq C_3$, and then it follows by \eqref{goal100} that $\log\left(\phi^3(\cdot,w)\right)$ is almost monotonically increasing, so in particular it is uniformly bounded.

We show next that if \eqref{boundphi00} holds for some $\gamma\ge3$  then \eqref{boundphi00}  holds also with $\gamma$ replaced by $\gamma+\beta$ for any $\beta>0$ such that  $3\beta<4-\gamma$.

Indeed, we can bound
\[
    \phi^{\gamma+\beta}(r,w)  = \frac{D(r,w) + (\gamma+\beta)r^{2\gamma+2\beta}}{H(r,w) + r^{2\gamma+2\beta}} \le \frac{2}{r^{2\beta}}  \frac{D(r,w) + \gamma r^{2\gamma}}{H(r,w) + r^{2\gamma}}  \le  \frac{\phi^\gamma(r,w)}{r^{2\beta}}\le 2\frac{C_\gamma}{r^{2\beta}},
\]
and similarly
\[
 g^{\gamma+\beta}(r)  = \frac{ \|w_r\|_{L^2(B_2\setminus B_1)}}{ (H(r,w) + r^{2\gamma+2\beta})^{1/2}} \le \frac{1}{r^{\beta}}  \frac{ \|w_r\|_{L^2(B_2\setminus B_1)}}{(H(r,w) + r^{2\gamma})^{1/2}}  = \frac{g^\gamma(r)}{r^{\beta}}\le \frac{C_\gamma}{r^{\beta}}.
\]
Then \eqref{goal100} ---with $\gamma$ replaced by $\gamma+\beta$--- yields
\begin{equation}\label{goal1yy00}
 \frac{d}{dr} \phi^{\gamma+\beta}(r,w)  \ge\frac{2}{r}  \frac{ \big( r^{2-n} \int_{B_r}w\Delta w\big)^2  }{ \big( H(r,w) + r^{2\gamma}\big)^2 }  -  C C_\gamma^2 r^{4-\gamma-1-3\beta}.
\end{equation}
Noticing that  $r^{4-\gamma-1-3\beta}$ is integrable over $r \in (0,1)$ provided $3\beta<4-\gamma$,  \eqref{goal1yy00} implies that $\phi^{\gamma+\beta}(r,w)$ is almost monotonically increasing. In particular
 $\phi^{\gamma+\beta}(r,w)\le C'$ for  $r \in (0,1/2)$, where $C'$ is a constant depending only on $n$ and $\gamma+\beta$.

In addition,  \eqref{goal1yy00} and Lemma~\ref{lem:E3B1c}(a) imply that
\[
\frac{H(R,w) +R^{2\gamma+2\beta}}{H(r,w) +r^{2\gamma+2\beta}} \le  C'  \qquad \forall \,R\in (r,2r),
\]
thus
\[
\big(g^{\gamma+\beta} (r) \big)^2 = \frac{\|w_r\|^2_{L^2(B_2\setminus B_1)}}{H(r,w) +r^{2\gamma+2\beta}} \le  C \ave_{r}^{2r} \frac{H(R,w)\,dR}{H(r,w) +r^{2\gamma+2\beta}} \le C  C',
\]
and therefore \eqref{boundphi00} holds for $\gamma$ replaced by $\gamma+\beta$.

Having proven that if \eqref{boundphi00} holds for some $\gamma\ge3$  then \eqref{boundphi00}  holds also with $\gamma$ replaced by $\gamma+\beta$ for any $\beta>0$ such that  $3\beta<4-\gamma$, iterating finitely many times we conclude that   \eqref{boundphi00} holds for any $\gamma<4$, as claimed.

Combining this information with \eqref{goal100} we obtain
\[
\frac{d}{dr} \phi^\gamma(r,w) \ge \frac 2 r  \frac{\left(r^{2-n} \int_{B_r} w\Delta w\right)^2}{\left(H(r,w) + r^{2\gamma}\right)^2}  -C C_{\gamma}^2 r^{3-\gamma} \qquad \forall \,r\in (0,1/2).
\]
Also, it follows by  \eqref{wlap1xx}  and  \eqref{boundphi00} that
$
\frac{ r^{2-n} \int_{B_r} w\Delta w }{H(r,w) + r^{2\gamma}} \ge - C C_{\gamma} r^{4-\gamma}.
$
Finally, when $0\in \Sigma^{3rd}_{n-1}$ we can take $P=p_3$ (since, by definition of $\Sigma^{3rd}_{n-1}$,  $p_3$ vanishes on $\{p_2=0\}$).
\end{proof}

We have proved that in $\Sigma_{n-1}^{3rd}$ we have almost monotonicity of the new truncated frequency function $\phi^\gamma(r,u-p_2-p_3)$ for any $\gamma\in(3,4)$.
This means that $\phi^\gamma(0^+,u-p_2-p_3)$ exists, and satisfies  $3\leq \phi^\gamma(0^+,u-p_2-p_3)\leq \gamma$ (see Remark~\ref{rem:phi gamma}).

Hence, we can now introduce the following:

\begin{definition}\label{def:Sigma>3}
Let $u:B_1\to[0,\infty)$ solve \eqref{eq:UELL}.
We denote by $\Sigma_{n-1}^{>3}$ the set of points $x_\circ\in \Sigma_{n-1}^{3rd}$  such that, for $w:= u(x_\circ+\,\cdot\,)-p_{2, x_\circ}-p_{3, x_\circ}$, we have $\phi^\gamma(0^+,w)>3$ for any  $\gamma\in(3,4)$.
\end{definition}

Moreover, we will need the following:

\begin{definition}\label{ansatz}
Let $u:B_1\to[0,\infty)$ solve \eqref{eq:UELL}, and assume that $0\in \Sigma^{3rd}_{n-1}$. Choose a coordinate system such that
\begin{equation}\label{eq:p2p3}
p_2(x) =\frac 1 2 x_n^2 \qquad \mbox{and}\qquad p_3(x) =  \sum_{\alpha=1}^{n-1} \frac{a_\alpha}{2} x_\alpha^2x_n + \frac{a_n}{6} x_n^3.
\end{equation}
(Here we used that $p_3$ is harmonic and vanishes on $\{x_n=0\}$.)
We define the  \emph{fourth order polynomial Ansatz} at $0$, and we denote it by $\anz$, as\footnote{The formula for the ansatz can found by inspection, by trying to find the coefficients that guarantee the validity of Lemma~\ref{lem:E3B4}.}
\begin{equation}\label{eq:ansatzell}
\begin{split}
\anz(x) :=
\frac 1 2 x_n^2 + p_3 + \frac 1 2 \left(\frac{p_3}{x_n}\right)^2 + x_nQ
= \frac 1 2 \left(x_n  + \frac{p_3}{x_n}  + Q\right)^2 + O(|x|^5).
\end{split}
\end{equation}
Here $Q$ is a $3$-homogeneous polynomial which  depends only on $p_2$ and $p_3$, and is defined as follows
\begin{equation}\label{eq:defQ}
Q(x): = \sum_{\alpha = 1}^{n-1} \left(a_\alpha^2 - \frac{a_\alpha a_n}{3}\right) \left(\frac{x_n^3}{12} -\frac{x_\alpha^2 x_n}{2}\right).
\end{equation}
When $u$ is a solution of   \eqref{eq:UELL} and  $x_\circ\in \Sigma^{3rd}_{n-1}$, we define $\anz_{x_\circ}$ to be the fourth oder polynomial Ansatz
at $0$ of $u(x_\circ + \,\cdot\,)$  (note that $\anz_{x_\circ}$ depends only on $p_{2,x_\circ}$ and $p_{3,x_\circ}$).
In addition, for $\alpha\in \{1,2,\dots, n-1\}$ we define the \emph{osculating rotation vector fields} at $0$ as
\begin{equation}\label{eq:Xalpha}
\Xalpha :=  (1 + a_\alpha x_n) \boldsymbol e_\alpha - a_\alpha x_\alpha  \boldsymbol e_n, \qquad \mbox{where } \boldsymbol e_i = (0, \dots, 0, \place{i}{1}, 0, \dots, 0),
\end{equation}
where $a_\alpha$ are as in \eqref{eq:p2p3}.
\end{definition}

We will use the following notation throughout the section. Given $f\in C^1(\R^n) $, we denote by  $\Xalpha f$ the derivative of $f$ in the direction of $\Xalpha$, namely
$\Xalpha f =(1+a_\alpha x_n)\partial_\alpha f - a_\alpha x_\alpha \partial_n f$.

\begin{lemma}\label{lem:E3B4}
Given $p_2$ and $p_3$ as in \eqref{eq:p2p3},  define $Q$ as \eqref{eq:defQ}.
Then, $\anz$ given by \eqref{eq:ansatzell} satisfies
\[
\Delta \anz = 1
\qquad \mbox{and}\qquad
\Xalpha\Xalpha \anz = O(|x|^3) \quad \forall \alpha \in \{1,2,\dots, n-1\}.
\]
\end{lemma}

\begin{proof}
Let $p_2$, $p_3$, $\anz$, and $\Xalpha$, be as in \eqref{eq:p2p3}, \eqref{eq:ansatzell}, \eqref{eq:defQ}, \eqref{eq:Xalpha}.
We compute
\[
\Xalpha\Xalpha  \left( \frac{ x_n^2}{2} \right) = -a_\alpha x_n + a^2_\alpha (x_\alpha^2-x_n^2),
\]
\[
\Xalpha\Xalpha   p_3  = a_\alpha x_n  +  2a_\alpha^2 (x_n^2-x_\alpha^2) - \sum_{\beta=1}^{n-1}\frac{ a_\alpha a_\beta}{2} x_\beta^2  -\frac{a_\alpha a_n}{2} x_n^2  + O(|x|^3),
\]
\[
\Xalpha\Xalpha  \biggl( \frac 1 2  \left(\frac{ p_3}{x_n}\right)^2 \biggr) =   \sum_{\alpha=1}^{n-1} \frac{a_\alpha a_\beta}{2} x_\beta^2 + \frac{a_\alpha a_n}{6} x_n^2  +  a_\alpha^2 x_\alpha^2 + O(|x|^3),
 \]
\[
\Xalpha\Xalpha (x_n Q)  = - \left(a_\alpha^2 - \frac{a_\alpha a_n}{3}\right) x_n^2  + O(|x|^3),
\]
and thus adding them we obtain
\[
\Xalpha\Xalpha  \anz = O(|x|^3).
\]
Similarly, using that $\sum_{\alpha=1}^{n-1}a_\alpha=-a_n$ (as a consequence of the fact that $p_3$ is harmonic), a direct (but tedious) computation shows that
\[
\Delta\biggl( \frac 1 2 \left(\frac{p_3}{x_n}\right)^2 + x_nQ\biggr) =0,
\]
therefore $\Delta \anz \equiv 1$.
\end{proof}

We will need the following semiconvexity estimate in the spirit of Lemma~\ref{lem:E2B1}.

\begin{lemma}\label{lem:E3B5}
Assume that  $u:B_1 \to [0,\infty)$ is a solution of  \eqref{eq:UELL}  and that $0\in \Sigma^{3rd}_{n-1}$. Let $w:= u-\anz -P$, where $P$ is a 4-homogeneous harmonic polynomial vanishing on $\{p_2=0\}$.
 Then
\[
\inf_{B_r} r^2 \Xalpha\Xalpha w \ge - C(P) \,\big( \|w_r\|_{L^2(B_5\setminus B_1)}+ r^5\big)
\]
and
\[
\sup_{B_{r/2}}  r |\nabla w| \le  C(P) \,\big(  \|w_r\|_{L^2(B_5\setminus B_1)} + r^5\big)
\]
for all $r\in (0,1/5)$.
\end{lemma}

\begin{proof}
Let $p_2$, $p_3$, $\anz$, and $\Xalpha$, be as in \eqref{eq:p2p3}, \eqref{eq:ansatzell}, \eqref{eq:defQ}, \eqref{eq:Xalpha}, and fix $\alpha\in \{1,2,\dots, n-1\}$.

\smallskip

\noindent $\bullet$ \emph{Step 1}. 
For $r>0$ small, we consider the rescaled vector field $\Xalpha^r = \Xalpha(r\,\cdot\,)$, and denote $w_r = w(r\,\cdot\,)$ and $v: = \Xalpha^r \Xalpha^r  w_r$.
We consider
\[
\bar v (x) : = \min \big\{ v(x) \, ,\,   -C(P) r^5\big\}
\]
for some constant $C(P)>0$ depending on $P$, to be chosen. We claim that
\[
\Delta \bar v \le 0 \quad \mbox{in  }B_5.
\]
Since  $w_r$ is $C^{1,1}$ (for fixed $r>0$) but not $C^2$, to prove that  $\bar v$ is  subharmonic we need to proceed similarly to the proof of Lemma~\ref{lem:E2B1}, now taking second order ``rotational'' incremental quotients.
More precisely, let $\phi^h_{\Xalpha^r}$ denote the integral flow of the vector field  $\Xalpha^r$ at time $h$, and
define
\[v^h:=  \frac{w_r\circ \phi^{h}_{\Xalpha^r} + w_r\circ \phi^{-h}_{\Xalpha^r}  -2w_r}{h^2}.\]
On the one hand, since $\phi^h_{\Xalpha^r}$ is a rotation (and thus it commutes with $\Delta$), noticing that $\Delta w = \chi_{\{u>0\}}-1 \le 0$ in $B_1$ and $\Delta w = 0$ in $\{u>0\}$ we obtain
\begin{equation}\label{asivobhasoibn}
\Delta v^h \le 0  \quad \mbox{in }\{u(r\,\cdot\,)>0\}\cap B_{1/r}.
\end{equation}
On the other hand, we claim that
\begin{equation}\label{abnoinaw}
v^h \ge  -C(P)\,r^5  \quad \mbox{in }\{u(r\,\cdot\,)=0\}\cap B_5.
\end{equation}
Indeed, recalling that $\Xalpha\Xalpha\anz = O(|x|^3)$  (by Lemma~\ref{lem:E3B4}) and since $\Xalpha\Xalpha P = \partial_{\alpha\alpha} P + O(|x|^3)$,  we obtain
\[
\Xalpha\Xalpha (-\anz -P) \ge  -\partial_{\alpha\alpha}P + O(|x|^3).
\]
In addition, since  $P$  is 4-homogeneous  and vanishes on $\{x_n=0\}$, we wave $\partial_{\alpha\alpha} P = x_n\ell(x')$ where $\ell$  is some linear function,
thus
\[
|\partial_{\alpha\alpha}P | \le C(P) |x||x_n|.
\]
Therefore, combining all these estimates, we get
\begin{equation}\label{abnaeoihw}
|\Xalpha\Xalpha (\anz +P)| \ge  C(P) \,\big( r^3 + r|x_n| \big) \quad \mbox{in }B_{5r}.
\end{equation}
In addition, thanks to Proposition~\ref{prop:E2B3}  (recall that $\lambda^{2nd} \ge 3$, since by assumption $0\in \Sigma_{n-1}^{3rd}$) we have
\begin{equation}\label{eq:u0}
\{u=0\} \cap B_{1/2} \subset  \{|x_n| \le C |x'|^2 \},
\end{equation}
thus
\[
(x',x_n) \in \{u=0\}\cap B_{5r} \quad \Rightarrow \quad |x'|\le 5r, \ |x_n| \le Cr^2
\quad\Rightarrow \quad  \big| \phi^{rh}_{\Xalpha} (x) \cdot \boldsymbol e_n \big| \le C r^2  \quad \forall\, h\in(0,2).
\]
As a consequence, combining this bound with \eqref{abnaeoihw} we obtain 
\[
\frac{\big|(\anz +P)\circ \phi^{rh}_{\Xalpha} + (\anz +P)\circ \phi^{-rh}_{\Xalpha}  -2(\anz +P)\big|}{(rh)^2} \le C(P)r^3  \quad \mbox{in }\{u=0\}\cap B_{5r} \quad \forall \,h\in(0,1).
\]
Rescaling this estimate one gets \eqref{abnoinaw}, that combined with \eqref{asivobhasoibn} implies that
\[
\bar v^h := \min \big\{ v^h(x) \, ,\,   -C(P) r^5\big\}  \quad \mbox{is superharmonic,}
\]
Therefore, since  $\bar v = \lim_{h\downarrow 0} \bar v^h$ a.e. then the function $\bar v$ is superharmonic too, as claimed.

\smallskip

\noindent $\bullet$ \emph{Step 2}. Note that if we consider $V: = \partial_{\alpha\alpha} w_r$  instead of rotational derivatives (as we did in the proof of Lemma~\ref{lem:E2B1}), then the same argument as the one above gives
\[
|\partial_{\alpha\alpha} (\anz +P)| \le  C(P) r^2 \quad \mbox{in }B_{5r}
\]
(cp. \eqref{abnaeoihw}), from which one deduces that
the function $\bar V: =  \min \big\{ V \, ,\,   -C(P) r^4\big\}$   is superharmonic (notice the difference in the power of $r$ in the definitions of $\bar v$ and $\bar V$).

\smallskip

\noindent $\bullet$ \emph{Step 3}.
Now, as in the proof of Lemma~\ref{lem:E2B1}, by weak Harnack inequality, interpolation, and the Calder\'on-Zygmund theory, we have
\[
\| \bar v \|_{L^\infty(B_1)}\leq  C(n) \biggl(\ave_{B_{3/2}} |\bar v|^\ep \biggr)^{1/\ep}\hspace{-2mm} \le  C (\|  w_r \|_{W^{2,1}_{weak} (B_2)}  + r^5)
\le   C \big( \|  w_r \|_{L^{1} (B_3 \setminus B_2)}  +  \|  \Delta w_r \|_{L^{1} (B_3)} +r^5  \big).
\]
On the other hand, since $\Delta w_r \le 0$ (because $\Delta w   = \Delta u- \Delta(\anz-P) = \Delta u - 1 \le 0$), reasoning as in the proof of Lemma~\ref{lem:E2B1} (cf. \eqref{eq:control Delta w}) we obtain
\[
 \|  \Delta w_r \|_{L^{1} (B_3)}  \le  C(n) \|  w_r \|_{L^{1} (B_4\setminus B_3)}.
\]
Hence
$
\| \bar v \|_{L^\infty(B_1)}\leq C\|  w_r \|_{L^{1} (B_4\setminus B_3)} ,
$
which yields
\begin{equation}\label{aniohaiohaoi}
\inf_{B_1} \Xalpha\Xalpha w_r  \ge  - C(\|  w_r \|_{L^{1} (B_4\setminus B_3)}  + r^5),
\end{equation}
and the first part of the lemma (semiconvexity)  follows easily by scaling.

In order to prove the Lipschitz bound, we 
note that if we repeat the same reasoning with $\bar V$  instead of $\bar v$, we find instead the semiconvexity estimates
\[ 
\inf_{B_1} \partial_{\alpha\alpha} w_r  \ge  - C(\|  w_r \|_{L^{1} (B_4\setminus B_3)}  + r^4), \quad \mbox{for } 1\le\alpha \le n-1.
\]
Although this estimate is less precise (it has an error of size $r^4$ instead of $r^5$) it is still useful. Indeed, using that $\Delta w_r \le 0$, it implies the semiconcavity estimate
\[
\sup_{B_1} \partial_{nn} w_r  \le   C(\|  w_r \|_{L^{1} (B_4\setminus B_3)}  + r^4).
\]
Combined together, these semiconvexity and semiconcavity estimates imply a bound on the Lipschitz constant of $w_r$
in terms of its $L^\infty$ norm in $B_4\setminus B_3$ and its semiconvexity/semiconcavity constants, that is
\begin{equation}\label{ajhioahioughqiuowb}
\| w_r  \|_{{\rm Lip} (B_{3/4})} \le   C(\|  w_r \|_{L^{\infty} (B_4\setminus B_3)}  + r^4).
\end{equation}
Although this is not the desired bound, this will be useful in the next step to obtain the sharp bound.

\smallskip

\noindent $\bullet$ \emph{Step 4}.
To conclude the proof, we need to improve \eqref{ajhioahioughqiuowb} and get Lipschitz bound for $w_r$ in $B_{1/2}$ with an error $O(r^5)$. For this, we note that for each $\alpha=1,\dots,n-1$ the unit vector  field  $E_\alpha  = \Xalpha^r/|\Xalpha^r|$ satisfies $|E_\alpha -\boldsymbol e_\alpha| \le Cr$ in $B_1$, and thus we can complete $\{E_\alpha\}_{\alpha=1}^{n-1}$ to obtain a orthonomal moving frame by adding a vectorfield $E_n$  satisfying $|E_n -\boldsymbol e_n|\le Cr$ in $B_1$.

Note that, since  $\nabla_{\Xalpha^r} {\Xalpha^r} = O(r)$  and $\nabla_{E_\alpha} {E_\alpha} = O(r)$, we can choose $E_n$ satisfying $\nabla_{E_n} {E_n} = O(r)$. Hence, using these bounds and $\Delta w_r \le 0$,  we obtain 
\[
\begin{split}
E_n E_n w_r - Cr \|w_r\|_{\rm Lip (B_{3/4})}  &\le  D^2 w_r(E_n, E_n) \le  -\sum_{\alpha=1}^{n-1} D^2 w_r(E_\alpha, E_\alpha)
\\
& \le   -\sum_{\alpha=1}^{n-1}  \Xalpha^r\Xalpha^r  w_r  +Cr \|w_r\|_{\rm Lip (B_{3/4})} \quad \mbox{ in }B_{3/4}.
\end{split}
\]
Thus, recalling \eqref{aniohaiohaoi} and \eqref{ajhioahioughqiuowb}, we get
\begin{equation}\label{qnioqoinqoi}
\sup_{B_{3/4}} E_n E_n w_r \le C(\|  w_r \|_{L^{1} (B_4\setminus B_3)}+r\|  w_r \|_{L^{\infty} (B_4\setminus B_3)}  +r^5)
\end{equation}
It follows by \eqref{aniohaiohaoi}  (resp. \eqref{qnioqoinqoi}) that the restiction of $w_r$ to the integral curves of $\Xalpha$ (resp. $E_n$) is semiconvex (resp. semiconcave), and hence Lipschitz along these curves.
Since the directions of these curves span $\R^n$, this yields
\[
\sup_{B_{1/2} }  |\nabla w_r|  \le C\sup_{B_{1/2}} \bigg( \sum_{\alpha =1}^{n-1} |\nabla w_r \cdot \Xalpha |  +   |\nabla w_r \cdot E_n|  \bigg)  \le C(\|  w_r \|_{L^{\infty} (B_4\setminus B_3)} +r^5).
\]
Finally, to conclude the proof, it suffices to show that the $L^\infty(B_4\setminus B_3)$-norm above can be replaced by $\|  w_r \|_{L^{2} (B_5\setminus B_2)} +C(P)r^5$. 
Indeed, recalling that $\Delta w=- \chi_{\{u=0\}}$, the function $w$ is superharmonic everywhere, and harmonic outside $\{u=0\}$. In particular, this gives the desired control on the $L^\infty$ norm of $w_-$. 
In addition, since $\anz\geq O(|x|^5)$ and $P$ is a 4-th order polynomial vanishing on $\{x_n=0\}$,
it follows from \eqref{eq:u0} that
$$w=u-\anz-P \leq P \leq C(P) r^5\qquad \text{inside }\{u=0\}\cap B_r,$$
 hence the function $\max\{  w_+, C(P)r^5\}$ is subharmonic.
Thus, the mean value inequalities allow us to control the $L^\infty$ norm of $(w_r)_\pm$ with the $L^1$ (or $L^2$) norm of $|w_r|+Cr^5$ in a slightly larger domain, concluding the proof.
\end{proof}

\begin{remark}\label{rem:E3B3}
We note that Lemma~\ref{lem:EFF4} can be rewritten as
\[
 \frac{d}{dr} \phi^\gamma(r,w) \ge \frac{2}{r}  \frac{ \big( r^{2-n} \int_{B_r}w\Delta w\big)^2  }{ \big( H(r,w) + r^{2\gamma}\big)^2 }  + \frac{2}{r}  \int_{B_1} (\lambda_r \hat w_r^{(\gamma)} -x\cdot\nabla  \hat w_r^{(\gamma)}) \Delta  \hat w_r^{(\gamma)},
\]
where
\[
\lambda_r := \phi^\gamma(r, w) \qquad \mbox{and} \qquad \hat w_r^{(\gamma)}: =  \frac{w(r\,\cdot\,)}{\big( H(r,w) + r^{2\gamma}\big)^{1/2}}.
\]
Also, we observe that Lemma~\ref{lem:E3B5} yields $\|\nabla \hat w_r^{(\gamma)}\|_{L^\infty}\le C$ for all $\gamma\le 5$ when $w:=u-\anz-P$. 
This will be crucial in the proof of Lemma~\ref{lem:E4B1} below.
\end{remark}

\begin{lemma}
\label{lem:E4B1}
Let $u:B_1 \to [0,\infty)$ be a solution of  \eqref{eq:UELL}, and assume that $0\in  \Sigma^{3rd}_{n-1}$. Set $w:=u-\anz-P$, where $\anz$ is defined in \eqref{eq:ansatzell}, and  $P$ is a $4$-homogeneous harmonic polynomial vanishing on $\{p_2=0\}$.
Then, given $\gamma\in (4,5)$, for all $r\in (0,1/2)$ we have
\[
\frac{d}{dr} \phi^\gamma(r,w) \ge \frac 2 r  \frac{\left(r^{2-n} \int_{B_r} w\Delta w\right)^2}{\left(H(r,w) + r^{2\gamma}\right)^2}  - Cr^{4-\gamma}
\qquad
\text{and}
\qquad
\frac{ r^{2-n} \int_{B_r} w\Delta w }{H(r,w) + r^{2\gamma}}  \ge -Cr^{5-\gamma },
\]
where $C$ is a constant depending only on $n$, $\gamma$, and $\|P\|_{L^2(B_1)}$.
\end{lemma}

\begin{proof}
With no loss of generality, we assume that $\{p_2=0\}=\{x_n=0\}$.
By Remark~\ref{rem:E3B3} we have
\begin{equation}\label{000}
 \frac{d}{dr} \phi^\gamma(r,w) \ge \frac{2}{r}  \frac{ \big( r^{2-n} \int_{B_r}w\Delta w\big)^2  }{ \big( H(r,w) + r^{2\gamma}\big)^2 }  + \frac{2}{r}  \int_{B_1} (\lambda_r \hat w_r^{(\gamma)} -x\cdot\nabla  \hat w_r^{(\gamma)}) \Delta  \hat w_r^{(\gamma)}.
\end{equation}

\smallskip

\noindent $\bullet$ \emph{Step 1.} We show that for some $C= C(n, P)$ we have
\begin{equation}\label{goal1xx}
 \frac{d}{dr} \phi^\gamma(r,w)  \ge\frac{2}{r}  \frac{ \big( r^{2-n} \int_{B_r}w\Delta w\big)^2  }{ \big( H(r,w) + r^{2\gamma}\big)^2 }  -  Cr^{4-\gamma}\phi^\gamma(r,w) \big( g^\gamma(r) +1\big),
\end{equation}
where
\[
g^\gamma(r) :=  \frac{ \|w_r\|^2_{L^2(B_5\setminus B_1)}}{H(r,w)+r^{2\gamma}} .
\]
Indeed, recall that by   Lemma~\ref{lem:E3B5}  we have
\[
r|\nabla w| \le C (\|w_r\|_{L^2(B_5\setminus B_1)} + r^5) \quad  \mbox{in }B_r.
\]

Now, notice that (recall that $p_3$ and $P$ are divisible by $x_n$)
\[\frac 1 2(x_n + p_3/x_n + Q + P/x_n)^2 =\anz +P + O(|x|^5),\]
we obtain
\[ 
\begin{split}
 (x_n + p_3/x_n + Q + P/x_n)(1+O(|x|))&= (x_n + p_3/x_n + Q + P/x_n)\partial_n(x_n + p_3/x_n + Q + P/x_n)
 \\&=  |\partial_n(\anz +P)|  +O(|x|^4).
\end{split}
\]

Therefore, since $Q$ is 3-homoegenous and also divisible by $x_n$ and  $|x_n|\le C|x|^2$ in $\{u=0\}$ (cf. \eqref{eq:u0}) we obtain 
\[ 
\begin{split}
{\textstyle \frac{1}{2}} |x_n + p_3/x_n + P/x_n| \le     \partial_n(\anz +P)  +O(|x|^4) \quad \mbox{in} \{u=0\}.
\end{split}
\]

Thus, since $w := u-\anz-P$,  inside $B_r\cap\{u=0\}$ (for $r$ small) we have
\begin{equation}
\label{111}
 r \textstyle{ \frac{1}{2}}\big|x_n + p_3/x_n +P/x_n \big| \le  r\big| \partial_n(\anz +P)\big|  +O(r^5)=  r|\partial_n w|  + O(r^5)
  \le C ( \|w_r\|_{L^2(B_5\setminus B_1)} + r^5)
\end{equation}

Therefore
\begin{equation}
\label{222}
\begin{split}
|w| &  =   |- (\anz +P)|   = \big|x_n + p_3/x_n +P/x_n \big|^2 + O(r^5) \le  2 r^2 |x_n + p_3/x_n +P/x_n|   + O(r^5)
 \\
 & \le C (r \|w_r\|_{L^2(B_5\setminus B_1)} + r^5) \qquad \mbox{ in  } B_r\cap\{u=0\}
\end{split}
\end{equation}
and, for $\alpha=1,2,\dots, n-1$,
\begin{equation}
\label{333}
\begin{split}
r|\partial_\alpha w| & =  r|- \partial_\alpha(\anz +P)| = r| \partial_\alpha (p_3/x_n) | \, \big|x_n + p_3/x_n +P/x_n \big|  + O(r^5)
\\
&\le   Cr^2\big|x_n + p_3/x_n +P/x_n \big|  + O(r^5) \le C (r  \|w_r\|_{L^2(B_5\setminus B_1)} + r^5) \qquad \mbox{ in  } B_r\cap\{u=0\},
\end{split}
\end{equation}
from which it follows that
\begin{equation}
\label{444}
\begin{split}
|\lambda_r \hat w_r^{(\gamma)} -x\cdot\nabla  \hat w_r ^{(\gamma)}| & \le  \frac{C(1+\lambda_r)  (r \|w_r\|_{L^2(B_5\setminus B_1)} + r^5)}{  (H(r,w) + r^{2\gamma})^{1/2}}
\\
& \le C (1+\lambda_r) (r g^\gamma(r)^{1/2} + r^{5-\gamma})
\\
& \le C(1+\lambda_r)  r^{5-\gamma} \big(g^\gamma(r)^{1/2} +1\big)   \qquad \mbox{ in  } B_1\cap\{u_r=0\}.
\end{split}
\end{equation}
Also, since $\Delta \anz \equiv 1$ and $\Delta P\equiv 0$ we have $\Delta w = -\chi_{\{u=0\}} \le 0$. Hence, arguing as in \eqref{eq:control Delta w},
\begin{equation}
\label{555}
0 \ge  \int_{B_1} \Delta \hat w_r^{(\gamma)} \ge - C \|\hat w_r^{(\gamma)}\|_{L^1(B_5\setminus B_1)}  \ge - C \frac{  \| w_r\|_{L^2(B_5\setminus B_1)} }{ (H(r,w) + r^{2\gamma})^{1/2} } \ge -C g^\gamma(r)^{1/2}.
\end{equation}
Combining this information with \eqref{444} we obtain
\[
  \left|  \int_{B_1} (\lambda_r \hat w_r^{(\gamma)} -x\cdot\nabla  \hat w_r^{(\gamma)}) \Delta  \hat w_r^{(\gamma)}  \right| \le   C(1+\lambda_r)  r^{5-\gamma} \big( g^\gamma(r) +1\big),
\]
that together with \eqref{000} yields \eqref{goal1xx}.

\smallskip

\noindent $\bullet$ \emph{Step 2.} Next we show that \eqref{goal1xx} implies that, for all $\gamma<5$, we have
\begin{equation}\label{boundphi}
\phi^\gamma(r, w)   \le  C_\gamma \quad \mbox{and} \quad  g^\gamma(r)   \le  C_\gamma,
\end{equation}
where $C_\gamma$  depends only on
$n$, $\gamma$,  and  $\|P\|_{L^2(B_1)}$.

We prove \eqref{boundphi} for all $\gamma\in [3,5)$ by iteratively increasing the value of $\gamma$ at each iteration, starting from $\gamma=3$, in order to always have (along the iteration) a uniform bound on $\phi^\gamma(r,w)$ and $g^\gamma(r)$.

First,  we observe that since $0\in \Sigma^{3rd}_{n-1}$ we have $\phi(0^+, u-p_2)\ge 3$, hence $|u-\anz - P| \le Cr^3$ in $B_r$, with a bound depending only on $n$ and $P$.
This immediately implies that $g^3(r)\leq C_3$, and then it follows by \eqref{goal1xx} that $\phi^3(\cdot,w)$ is almost monotonically increasing, so in particular it is uniformly bounded.

Then, by the very same argument as the one used in Step 2 in the proof of Lemma~\ref{lem:E3B1} (using \eqref{goal1xx} in place of \eqref{goal100}) we deduce that if \eqref{boundphi} holds for some $\gamma\ge3$,  then \eqref{boundphi}  holds also with $\gamma$ replaced by $\gamma+\beta$ for any $\beta>0$ such that  $4\beta<5-\gamma$.
Thanks to this fact, with finitely many iterations we conclude that  \eqref{boundphi} holds for any  $\gamma<5$, as claimed.

Combining  \eqref{boundphi} with \eqref{goal1xx}, we obtain
\[
\frac{d}{dr} \phi^\gamma(r,w) \ge \frac 2 r  \frac{\left(r^{2-n} \int_{B_r} w\Delta w\right)^2}{\left(H(r,w) + r^{2\gamma}\right)^2}  -C C_{\gamma}^2 r^{4-\gamma} \qquad \forall \,r\in (0,1).
\]
Moreover, recalling  \eqref{222} and \eqref{555}, we conclude that
\[
\frac{ r^{2-n} \int_{B_r} w\Delta w }{H(r,w) + r^{2\gamma}}  =  \int_{B_1\cap \{u_r=0\}} \hat w_r\Delta \hat w_r  \ge  - C r^{5-\gamma}  g^\gamma(r)  \ge - C C^\gamma r^{5-\gamma}.
\]
\end{proof}

Thanks to Lemma~\ref{lem:E4B1} we know that the truncated frequency function $\phi^\gamma$ is almost monotone for $\gamma<5$, and we can use this to study finer properties for points in $\Sigma_{n-1}^{3rd}$.
In particular, we introduce the following:

\begin{definition}\label{def:Sigma4th}
Let $u:B_1\to[0,\infty)$ solve \eqref{eq:UELL}.
 We denote by $\Sigma_{n-1}^{4th}$ the set of points $x_\circ\in \Sigma^{3rd}_{n-1}$  such that the following holds:\\
 Set $w= u(x_\circ+\,\cdot\,)-\anz_{x_\circ}$, where $\anz_{x_\circ}$ is defined as in \eqref{eq:ansatzell} starting from $p_{2,x_\circ}$ and $p_{3,x_\circ}$. Then there exists some sequence $r_k\downarrow 0$ along which  $r_k^{-4} w(r_k\,\cdot\,)$ converges, weakly in $W^{1,2}_{\rm loc}(\R^n)$, to some 4-homogeneous harmonic polynomial vanishing on $\{p_{2,x_\circ}=0\}$ ---possibly the polynomial zero.
\end{definition}

We can now prove the existence of a unique 4-th order limit at points of  $\Sigma_{n-1}^{4th}$.

\begin{lemma}\label{lem:E4B2}
Let $u:B_1 \to [0,\infty)$ solve  \eqref{eq:UELL}, and let $0\in\Sigma^{\ge4}_{n-1}$ (see Definition \ref{def:Sigma>4} below).  Set $w:= u-\anz - P$, where $\anz$ is defined in \eqref{eq:ansatzell}, and $P$  is a 4-homogeneous harmonic polynomial vanishing on  $\{p_2=0\}$.
Then, for any $\gamma\in (4,5)$ we have
\begin{equation}\label{eq:E4B2b}
\frac{d}{dr} \log \big( r^{-8}(H(r,w) + r^{2\gamma}) \big) \ge - C r^{4-\gamma},
\end{equation}
where $C$ is a constant depending only on $n$, $\gamma$, and $\|P\|_{L^2(B_1)}$.

As a consequence,  for all $x_\circ\in\Sigma^{4th}_{n-1}$ the limit
\[
p_{4, x_\circ} := \lim_{r\downarrow 0} \frac{1}{r^4} \big( u(x_\circ+r \,\cdot\,) - \anz_{x_\circ}(r\,\cdot\,) \big)
\]
exists, and it is a $4$-homogeneous harmonic polynomial vanishing on  $\{p_{2, x_\circ}=0\}$.
\end{lemma}

\begin{proof}
For every 4-homogeneous harmonic polynomial $P$ vanishing on $\{p_2=0\}$, we have
\[
\frac{d}{dr} \log \big( r^{-8}(H(r,w) + r^{2\gamma} )\big)  = \frac{2}{r}\big(  \phi^\gamma(r,w) -4\big) + \frac{ r^{2-n} \int_{B_r} w\Delta w }{H(r,w) + r^{2\gamma}}
\ge -C r^{4-\gamma},
\]
where we used Lemma~\ref{lem:E4B1}. This proves \eqref{eq:E4B2b}.

Now, if $0\in \Sigma^{4th}_{n-1}$  then we have that, for some some $r_k\downarrow0$ and some $P$ which is $4$-homogeneous harmonic and vanishes  on $\{p_2=0\}$,
\[\log \big( r_k^{-8}(H(r_k,w) + r_k^{2\gamma}) \big)  \to -\infty.\]
Thus, thanks to \eqref{eq:E4B2b} we have
\[
\lim_{r\downarrow 0} \log \big( r^{-8}(H(r,w) + r^{2\gamma}) \big) = -\infty,
\]
which implies that $r^{-4} (u -\anz)(rx) \to P =: p_{4,0}$.
\end{proof}

When $x_\circ=0$  we will simplify the notation $p_{4,0}$ to $p_4$.

We can now prove an enhanced version of Proposition~\ref{prop:E2B2} for higher-order blow-ups in $\Sigma^{3rd}_{n-1}$ and $\Sigma^{4th}_{n-1}$.

\begin{proposition}
\label{prop:E3B5}
Let $u:B_1\to [0, \infty)$ solve \eqref{eq:UELL}, and let $\anz$ be as in Definition~\ref{ansatz}.

{\rm (a)}  Let $0\in\Sigma^{3rd}_{n-1}\setminus \Sigma^{4th}_{n-1}$ and set $w:=u-\anz$.
Then the limit $\lambda^{3rd}:=  \lim _{r\downarrow 0}\phi(r,w)$ exists and satisfies $\lambda^{3rd}\in [3,4]$. Moreover,  for every sequence $r_k\downarrow 0$ there is a subsequence $r_{k_\ell}$ such that
$\tilde  w_{r_{k_\ell}} \to  q$ as $\ell \to \infty$  in $C^0_{\rm loc}(\R^n)$ and weakly in $W^{1,2}_{\rm loc}(\R^n)$, where
$q\not\equiv 0$ is a global $\lambda^{3rd}$-homogeneous solution of the Signorini problem \eqref{ETOP}.  
In addition, if $\lambda^{3th} <4$ then $q$ is even with respect to $\{p_2=0\}$.

 {\rm (b)}  Let $0\in \Sigma^{4th}_{n-1}$ and set $w:=u-\anz-p_4$.
Then:

\begin{itemize}
\item[(b1)] either  $H(r, w)^{1/2}\le C_\zeta r^{5-\zeta}$  for all $\zeta\in (0,1)$, for some $C_\zeta$ depending on $\zeta$;
\item[$\,$(b2)] or the limit $\lambda^{4th}:=  \lim _{r\downarrow 0}\phi(r,w)$ exists and satisfies $\lambda^{4th}\in [4,5)$. Moreover, for every sequence $r_k\downarrow 0$ there is a subsequence $r_{k_\ell}$ such that
$\tilde  w_{r_{k_\ell}} \rightharpoonup  q$ as $\ell \to \infty$  in $C^0_{\rm loc}(\R^n)$ and weakly in $W^{1,2}_{\rm loc}(\R^n)$, where
$q\not\equiv 0$ is a $\lambda^{4th}$-homogeneous solution of \eqref{ETOP}, even with respect to $\{p_2=0\}$.
\end{itemize}
\end{proposition}

\begin{proof}
(a)  Let $0\in \Sigma^{3rd}_{n-1}\setminus \Sigma^{4th}_{n-1}$,  and $w:=u-\anz$. 
Let $\gamma := 5-\ep \in (4,5)$.

We note that $\phi^\gamma(0^+, w)\le 4$.
Indeed, if by contradiction   $\phi^\gamma(0^+, w)>4$ then by Lemma~\ref{lem:E3B1c} (which can be applied thanks to Lemma~\ref{lem:E4B1}) we would have
\[
H(r,w) + r^{2\gamma}   \le   Cr^{\phi^\gamma(0^+, w)}  (H(1,w) + 1)  \ll r^8 \quad \mbox{as } r\downarrow 0,
\]
and hence we would have $r^{-4}w_r \to 0$ and in particular $0\in \Sigma^{4th}_{n-1}$ (with $p_4\equiv 0$), contradicting our assumption.

Note now that  since $\gamma>4$ we have
\begin{equation}\label{<gamma1aa}
\phi^\gamma(0^+, w) = \lim_{r\downarrow 0} \frac{D(r,w) + \gamma r^{2\gamma}}{ H(r,w) +r^{2\gamma}} \leq 4<\gamma.
\end{equation}
Also, using Lemma \ref{lem:E3B1c},
\[
\mbox{\eqref{<gamma1aa}} \quad \Rightarrow \quad \frac{r^{2\gamma}}{H(r, w)}  \downarrow 0 \quad  \Rightarrow  \quad \lambda^{3rd}:=\phi^\gamma(0^+, w) = \phi(0^+,w).
\]
Thus, the limit $\phi(0^+, w)$ exists and equals again $\lambda^{3rd}\leq 4$.
In addition, by Lemmas~\ref{lem:E3B1c} and~\ref{lem:E4B1} we have
\[
\| \tilde w_r\|_{W^{1,2}(B_R)} \le C(R) \qquad \forall \,r>0
\]
for each $R\ge 1$, which gives weak compactness in $W^{1,2}_{\rm loc}(\R^n)$ of $\tilde w_{r}$ as $r \downarrow 0$.

We now show that any ``accumulation point''
\[ 
q := \lim_k \tilde w_{r_k}
\]
satisfies
\begin{equation}\label{01aa}
c r^{-\lambda^{3rd}}H(r,q)^{1/2} \le H(1,q)^{1/2} =1 \qquad  \forall \,r\in (0,1)
\end{equation}
and, for all $\delta >0$,
\begin{equation}\label{01bb}
c_\delta R^{-\lambda^{3rd}-\delta}H(R,q)^{1/2} \le H(1,q)^{1/2} =1 \qquad  \forall \,R\in (1,\infty),
\end{equation}
where $c$, $c_\delta$ are positive constants.

Indeed, since $\phi^\gamma(0^+,w)= \lambda^{3rd}$, Lemma~\ref{lem:E3B1c} gives (for $0<r<R\ll 1$)
\[
c(R/r)^{2\lambda^{3rd}} \le  \frac{H(R,w) + R^{2\gamma}}{H(r,w)  + r^{2\gamma}} =  \frac{ H(R/r_k,\tilde w_{r_k})   + \frac{R^{2\gamma}}{H(r_k w)}   }{  H(r/r_k,\tilde w_{r_k})  + \frac{r^{2\gamma}}{H(r_k, w)}  }.
\]
In particular, replacing $R$ by $r_k$ and $r$ by $r_kr$   (with $r<1$)  we obtain
\[
cr^{-2\lambda^{3rd}} \le    \frac{ H(1,\tilde w_{r_k})   + \frac{r_k^{2\gamma}}{H(r_k w)}   }{  H(r,\tilde w_{r_k})  + \frac{(r_kr)^{2\gamma}}{H(r_k, w)}  }
\le \frac{ H(1,\tilde w_{r_k})   + \frac{r_k^{2\gamma}}{H(r_k ,w)}   }{  H(r,\tilde w_{r_k})  } \to  \frac{ H(1,q)  }  {  H(r,q)  },
\]
proving  \eqref{01aa}.

Similarly,  by the other inequality in  Lemma~\ref{lem:E3B1c} we have (for $0<r<R\ll 1$)
\[
c_\delta(R/r)^{2\lambda^{3rd} +\delta} \ge  \frac{H(R,w) + R^{2\gamma}}{H(r,w)  + r^{2\gamma}} =  \frac{ H(R/r_k,\tilde w_{r_k})   + \frac{R^{2\gamma}}{H(r_k w)}   }{  H(r/r_k,\tilde w_{r_k})  + \frac{r^{2\gamma}}{H(r_k, w)}  }.
\]
In particular, replacing $R$ by $r_kR$ and $r$ by $r_k$   (with $R>1$)  we obtain
\[
c_\delta R^{2\gamma}  \ge    \frac{ H(R,\tilde w_{r_k})   +   R^{2\gamma} \frac{r_k^{2\gamma}}{H(Rr_k w)}   }{  H(1,\tilde w_{r_k})  + \frac{(r_k)^{2\gamma}}{H(r_k, w)}  }
\ge \frac{ H(R,\tilde w_{r_k})    }{  H(1,\tilde w_{r_k})  +\frac{(r_k)^{2\gamma}}{H(r_k, w)} } \to  \frac{ H(R,q)  }  {  H(1,q)  },
\]
which proves \eqref{01bb}.

We now note that $\Delta w =-\chi_{\{u=0\}}$ implies  that  $\Delta \tilde w_k$ and $\Delta q$ are nonpositive measures. Also, the Lipschitz estimate in Lemma~\ref{lem:E3B5} (with $P\equiv 0$) implies that
\[
\|\tilde w_{r_k} \|_{{\rm Lip}(B_R)} \le C(R), \qquad\mbox{and thus} \qquad    \|q\|_{{\rm Lip}(B_R)} \le C(R),
\]
for all $R\ge1$. As a consequence, the convergence of $\tilde w_{r_k}$ to $q$ is uniform on compact sets.
Furthermore, by \eqref{eq:ansatzell} we have $w =u-\anz \ge - \anz \ge -O(|x|^5)$ on $\{x_n = p_3/x_n\}$, so by uniform convergence we obtain
\[q\ge 0\quad \mbox{ on }\{x_n=0\}.\]
In addition, by Lemma~\ref{lem:E4B1} we have
\[
\int_{B_R}  \tilde w_{r_k} \Delta \tilde w_{r_k}  \downarrow 0,
\]
and since $\tilde w_{r_k} \to q$ in $C^0$ and $0\ge \Delta w_k \rightharpoonup^* \Delta q$ (up to extracting a further subsequence) we obtain
\[
\int_{B_R} q\Delta q = 0\qquad \forall \,r\ge1.
\]
But since $\Delta \tilde w_{r}$ is supported in $\{u_r=0\} \subset \{|x_n|\le o(1)\}$ as  $r\downarrow 0$, the support the nonpositive measure $\Delta q$ is contained on $\{x_n=0\}$ where $q\ge 0$.
We have thus shown that $q: \R^n \rightarrow \R$ is a solution of the Signorini problem \eqref{ETOP}.

Finally, recalling \eqref{01aa} and \eqref{01bb} we have that $\phi(0^+, q) \ge \lambda^{3rd}$ while  $\phi(+\infty, q) \le \lambda^{3rd}+\delta$ for all $\delta >0$. This implies that $\phi(r, q) =  \lambda^{3rd}$ for all $r>0$, and thus
$q$ must be $\lambda^{3rd}$-homogeneous (see Lemma~\ref{lemAP:Alm}).

Note also that, by Lemma~\ref{lem:E3B1c}, we have   $H(w,r) \gg r^{\lambda^{3rd}+\delta}$ for every $\delta>0$. Thus, since by definition of $\Sigma_{n-1}^{3rd}$ we have $\phi(0^+, u-p_2)\ge 3$, this implies
$|u-p_2|  \le C|x|^3$ and thus $|u-\anz|  \le C|x|^3$. 
Therefore  it must be $\lambda^{3rd}\ge 3$.

To conclude part (a), we need to show that if $\lambda^{3rd}<4$ then $q$ is even.
For this, notice that if one writes $q$ as the sum of its even and odd part, then the odd part is harmonic.
Thus, if $\lambda^{3rd}\in(3,4)$ then any $\lambda^{3rd}$-homogeneous solution of the Signorini problem is even (since the homogeneity of a harmonic function is always an integer), so we only need to understand the case $\lambda^{3rd}=3$. 

Assume $\lambda^{3rd}=3$, and let us show that $q$ is even.
We have (see Lemma~\ref{lem:E3B1c}) that
$H(w,r)\gg r^{3+\delta}$ for all $\delta>0$ as $r\downarrow 0$ therefore
\[
q := \lim_k \tilde w_{r_k} =   \lim_k \frac{(u-p_2-p_3)(r_k \cdot )}{H(r_k, u-p_2-p_3)^{1/2} }.
\]
Moreover, using \eqref{eq:EGB6b} from Lemma~\ref{lem:EGB6} we obtain (note that $w$ in this proof and in Lemma~\ref{lem:EGB6} are different)
\begin{equation}\label{quadratictrick}
\int_{\partial B_1} \left( \frac{(u-p_2-p_3)_r}{r^3} +P\right)^2  + C(P)r \ge  \lim_{r\downarrow 0}\int_{\partial B_1} \left(\frac{(u-p_2-p_3)_r}{r^3} +P\right)^2  \ge
 \int_{\partial B_1} P^2
\end{equation}
for all $P$ harmonic $3$-homogeneous vanishing on $\{p_2=0\}$, therefore
\[
 -C(P)r \le \int_{\partial B_1} \biggl[\left( \frac{(u-p_2-p_3)_r}{r^3} +P\right)^2 - P^2\biggr].
\]
 As a consequence, expanding the square and dividing by
 \[\ep_r :=   \left(\int_{\partial B_1} \left( \frac{(u-p_2-p_3)_r}{r^3}\right)^2 \right)^{1/2}  = o(1),\]
since $r^{\delta} \ll \ep_r$  as $r\downarrow 0$ we obtain
\[
 -C(P)\frac{r_k}{\ep_{r_k}}  \le \int_{\partial B_1} \biggl[\ep_r  \left(\frac{(u-p_2-p_3)_r}{H(r_k, u-p_2-p_3)^{1/2} }\right)^2  + 2  \frac{(u-p_2-p_3)_r}{H(r_k, u-p_2-p_3)^{1/2} }P\biggr],
\]
and in the limit as $r\downarrow 0$ we get
\[
0\le 2\int_{\partial B_1} q P
\]
for every odd harmonic $3$-homogeneous polynomial $P$. 
Since if $P$  is an odd harmonic $3$-homogeneous polynomial then so is $-P$, we deduce that $q$ must be orthogonal to all odd harmonic polynomials, hence $q$ is even.

\smallskip

(b) 
We assume that (b1) fails and we prove (b2). If (b1) fails then  there exist $\zeta \in (0,1)$ and a sequence $r_k\to 0$ such that $H(r_k, w)^{1/2} \ge  C_\zeta r_k^{5-\zeta}$.
In particular, there exists some $\gamma\in (4,5)$ such that
\begin{equation}\label{<gamma}
\phi^\gamma(0^+, w) = \lim_{r\downarrow 0} \frac{D(r,w) + \gamma r^{2\gamma}}{ H(r,w) +r^{2\gamma}} <\gamma.
\end{equation}
Thus, as in the proof of (a),
\[
\mbox{\eqref{<gamma}} \quad \Rightarrow \quad \frac{r^{2\gamma}}{H(r, w)}  \downarrow 0 \quad  \Rightarrow  \quad \lambda^{4th}:=\phi(0^+, w) = \phi^\gamma(0^+,w), \quad \forall\,\gamma\in (\lambda^{4th},5).
\]
In addition, combining Lemmas~\ref{lem:E3B1c} and~\ref{lem:E4B1} we obtain that
\[
\| \tilde w_r\|_{W^{1,2}(B_R)} \le C(R) \qquad \forall \,r>0
\]
for each $R\ge 1$, which gives compactness of sequences $\tilde w_{r_k}$ as $r_k \downarrow 0$ ---they converge weakly  in $W^{1,2}(B_R)$ for every $R$ up to extracting a subsequence.
Also, as in (a), it follows by Lemma~\ref{lem:E3B1c} that any ``accumulation point''
\[
q := \lim_k \tilde w_{r_k}
\]
satisfies \eqref{01aa} and  \eqref{01bb} with $\lambda^{3rd}$ replaced by $\lambda^{4th}$.
Also, exactly as in (a)  we have that  $\Delta \tilde w_k$ and $\Delta q$ are nonpositive measures, and
\[
\|\tilde w_{r_k} \|_{{\rm Lip}(B_R)} \le C(R) \qquad\mbox{and thus} \qquad    \|q\|_{{\rm Lip}(B_R)} \le C(R)
\]
for all $R\ge1$. As a consequence, the convergence is uniform on compact sets.
Furthermore \eqref{eq:ansatzell} yields $w =u-\anz -p_4 \ge - \anz -p_4 \ge -O(|x|^5)$ on $\{x_n = p_3/x_n\}$, so by uniform convergence  of $\tilde w_{r_k}$ to $q$ we obtain
\[q\ge 0\quad \mbox{ on }\{x_n=0\}.\]
Also,  by Lemma~\ref{lem:E4B1} we obtain
$\int_{B_R}  \tilde w_{r_k} \Delta \tilde w_{r_k}  \downarrow 0$ from which we deduce that  $\int_{B_R} q\Delta q = 0$ for all $R>1$.
As a consequence, $q$ is a solution of the Signorini problem \eqref{ETOP}.
Finally, recalling \eqref{01aa} and \eqref{01bb} we have that $\phi(0^+, q) \ge \lambda^{4th}$ while  $\phi(+\infty, q) \le \lambda^{4th}+\delta$ for all $\delta >0$. 
This implies that $\phi(r, q) =  \lambda^{4th}$ for all $r>0$, and thus
$q$ must be $\lambda^{4th}$-homogeneous.

Note also that by Lemma~\ref{lem:E3B1c} we have   $H(w,r) \gg r^{\lambda^{4th}+\delta}$ for every $\delta>0$. 
Thus, since by definition of $\Sigma_{n-1}^{4th}$ we have
$|u-\anz|  \le C|x|^4$, it must be $\lambda^{4th}\ge 4$.

Finally, we prove that $q$ must be even.  
As in (a), we only need to understand the case $\lambda^{4th}=4$.
When $\lambda^{4th}=4$  then we have (see Lemma~\ref{lem:E4B1})
$H(w,r)\gg r^{4+\delta}$ for all $\delta>0$ as $r\downarrow 0$. On the other hand, by definition of $p_4$ it follows that $H(r_k, u-\anz -p_4)^{1/2}= r_k ^4 \ep_k$, where $r_k^\delta  \ll \ep_k =o(1)$. Then, using Lemma~\ref{lem:E4B2} we obtain
\[
 r^{-8}H\big(r,u-\anz-P \big) \ge - C(P)\,r^{\delta} + \lim_{r\downarrow 0}    r^{-8}H\big(r,u-\anz-P\big) =  -C(P)\,r^\delta  + H(1,p_4-P),
\]
for all $P$ quartic vanishing on $\{x_n=0\}$.
Therefore, similarly to (a), we deduce that $q$  is orthogonal to every odd harmonic 4-homogeneous polynomial. Since $q$ is a solution of Signorini this implies that its odd part (which is harmonic) must vanish, concluding the proof.
\end{proof}

We can now introduce the following:

\begin{definition}\label{def:Sigma>4}
Let $u:B_1\to[0,\infty)$ solve \eqref{eq:UELL}, and recall the definition of $\anz_{x_\circ}$ in \eqref{eq:ansatzell}.

We denote by $\Sigma_{n-1}^{\geq4}$ the set of points $x_\circ\in \Sigma_{n-1}^{3rd}$  such that, for $w:= u(x_\circ+\,\cdot\,)-\anz_{x_\circ}$, we have $\phi^\gamma(0^+,w)\geq4$ for every $\gamma\in(4,5)$.

We denote by $\Sigma_{n-1}^{>4}$ the set of points $x_\circ\in \Sigma_{n-1}^{4th}$  such that, for $w:= u(x_\circ+\,\cdot\,)-\anz_{x_\circ}-p_{4,x_\circ}$, we have $\phi^\gamma(0^+,w)>4$  for every $\gamma\in(4,5)$.

Furthermore, for fixed $\zeta\in (0,1)$ we denote by $\Sigma_{n-1}^{\geq 5-\zeta}$ the set of points $x_\circ\in \Sigma_{n-1}^{4th}$ such that, for $w:= u(x_\circ+\,\cdot\,)-\anz_{x_\circ}-p_{4,x_\circ}$,
we have $\phi^\gamma(0^+,w)\ge 5-\zeta$ for any $\gamma\in(5-\zeta,5)$.
\end{definition}

Our last goal of this section is to show that $\Sigma_{n-1}^{>4}=\Sigma_{n-1}^{4th}$.
For this, we need a new monotonicity formula.

\begin{lemma} \label{lem-Sigma4th=Sigma>4-mon}
Let $u:B_1 \to [0,\infty)$ solve  \eqref{eq:UELL}, and let $0\in\Sigma^{4th}_{n-1}$.  Let $w:= u-\anz - p_4$, 
where $\anz$ is defined in \eqref{eq:ansatzell}, and let $P$ 
be any $4$-homogeneous harmonic polynomial such that $P\geq0$ on  $\{p_2=0\}$.
Then
\[\frac{d}{dr}  \bigg( r^{-4} \int_{\partial B_1} w_r P\bigg) \le C,\]
where $C$ is a constant depending only on $n$ and $\|P\|_{L^2(B_1)}$.
\end{lemma}

\begin{proof}
After a rotation, we may assume $p_2 = \frac 1 2 x_n^2$.
We have 
\[
\begin{split}
\frac{d}{dr} \int_{\partial B_1}  w_r P  &=  \int_{\partial B_1}  \frac x r \cdot \nabla w_r P   =  \frac{1}{r} \int_{\partial B_1}  \partial_\nu w_r P  =  \frac{1}{r} \int_{B_1}  {\rm div} ( \nabla w_r P) = \frac{1}{r} \bigg( \int_{B_1}  \nabla w_r \nabla P   + \int_{B_1}  \Delta w_r P\bigg)
\\
&= \frac{1}{r} \bigg( \int_{\partial B_1}   w_r \partial_\nu P   -\int_{B_1} w_r \Delta P  + \int_{B_1}  \Delta w_r P\bigg)
= \frac{1}{r} \bigg( 4\int_{\partial B_1}   w_r P + \int_{B_1}  \Delta w_r P\bigg),
\end{split}
\]
where we used that $\partial_\nu P=4P$ on $\partial B_1$, and that $\Delta P=0$.
Now, since $\Delta w_r=-r^2\chi_{\{u_r=0\}}$, we deduce that 
\[
\frac{d}{dr}  \bigg( r^{-4} \int_{\partial B_1} w_r P\bigg) = -\frac{1}{r^3}\int_{B_1\cap \{u_r=0\}} P.\]
Finally notice that  \eqref{111} rescaled implies, using  $\|w_r\|_{L^2(B_5\setminus B_1)} \le Cr^4$ since $0\in\Sigma^{4th}_{n-1}$,
\[
\{u_r =0\}\cap B_1 \subset \{ |x_n+rp_3/x_n|\le Cr^2\} \quad \mbox{and thus} \quad \big|\{u_r=0\}\cap B_1\big|\le Cr^2
\]
Moreover,  since $P\geq0$ on $\{x_n=0\}$, we have $P \ge -C|x_n|$ in $B_1$.
Hence we obtain
\[-\int_{B_1\cap \{u_r=0\}} P \leq Cr\,\big|\{u_r=0\}\cap B_1\big| \leq Cr^3,\]
and the lemma follows.
\end{proof}

We can now show the following:

\begin{proposition} \label{lem-Sigma4th=Sigma>4}
Let $u:B_1 \to [0,\infty)$ solve  \eqref{eq:UELL}.
Then $\Sigma^{4th}_{n-1} = \Sigma^{>4}_{n-1}$.
\end{proposition}

\begin{proof}
Assume by contradiction that $0\in \Sigma^{4th}_{n-1} \setminus \Sigma^{>4}_{n-1}$.
Then, by Proposition~\ref{prop:E3B5}(b), there is a sequence $r_k\to 0$ along which $\tilde w_{r_k}\to q$ locally uniformly in $\R^n$, where $q$ is a 4-homogeneous even solution of the Signorini problem \eqref{ETOP}.
Then, by \cite[Lemma 1.3.4]{GP09}, $q$ is a harmonic polynomial.

Let $w:= u-\anz - p_4$.
Since $r^{-4} w_r\to0$ (by definition of $p_4$), it follows by Lemma~\ref{lem-Sigma4th=Sigma>4-mon} that
\[  \int_{\partial B_1} r^{-4}w_r P \leq Cr\]
for any $4$-homogeneous harmonic polynomial $P$ vanishing on  $\{p_2=0\}$.

Set now $\tilde w_r = w_r/H(1,w_r)^{1/2}$ and $\varepsilon_r:= r^{-4} H(1,w_r)^{1/2}$, and notice that, since $0\notin \Sigma^{>4}_{n-1}$, for any $\delta>0$ we have $\varepsilon_r\gg r^{\delta}$ for $r>0$ small enough.
Hence, 
\[ Cr \ge \int_{\partial B_1} r^{-4}w_r P = \int_{\partial B_1} \varepsilon_r \tilde w_r P.\]
Dividing by $\varepsilon_r$, and letting $r=r_k\to0$, we deduce that 
\[ 0 \geq \int_{\partial B_1} q P.\]
Taking $P=q$, this provides the desired contradiction.
\end{proof}

\section{Uniqueness and nondegeneracy of non-harmonic cubic blow-ups}  \label{cubic-sec}

The goal of this section is to study the set $\Sigma_{n-1}^{\geq3}\setminus\Sigma_{n-1}^{>3}$,
namely the set of singular points where blow-ups are 3-homogeneous and non-harmonic\footnote{More precisely, $\Sigma_{n-1}^{\geq3}\setminus\Sigma_{n-1}^{3rd}$ is the set in which any second blow-up (for $u-p_2$) is 3-homogeneous and non-harmonic, while $\Sigma_{n-1}^{3rd}\setminus\Sigma_{n-1}^{>3}$ is the set in which the third blow-up (for $u-p_2-p_3$) is 3-homogeneous and non-harmonic.}.
As explained in the introduction, this study is crucial for our proof of Theorem~\ref{thm-Schaeffer-intro}.

We will prove that $\Sigma_{n-1}^{\geq3}\setminus\Sigma_{n-1}^{3rd}$ is contained in a countable union of $(n-2)$-dimensional Lipschitz manifolds, and that $\Sigma_{n-1}^{3rd}\setminus\Sigma_{n-1}^{>3}=\varnothing$.
For this, we will need to establish the uniqueness and nondegeneracy of blow-ups at these points.

We start by classifying all $\lambda$-homogeneous solutions of the Signorini problem in $\R^n$, with $\lambda$ odd.

\begin{lemma}\label{lem:signoriniodd}
Let $q:\R^n \rightarrow \R$ be a $\lambda$-homogenous solution of the Signorini problem
\begin{equation}\label{signorininormal}
\begin{cases}
\Delta q\le 0\quad \text{and} \quad q\Delta q=0 \quad & \mbox{in }\R^n
\\
\Delta q=0  &\mbox{in }\R^n\setminus  \{x_n=0\}
\\
\,q\ge 0 &\mbox{on } \{x_n=0\}.
\end{cases}
\end{equation}
with  homogeneity $\lambda=2m+1$, $m\in \mathbb N$.
Then $q\equiv 0$ on $\{x_n=0\}$.
\end{lemma}

\begin{proof}
Using complex variables (so $i$ denotes the imaginary unit), for $\alpha \in \{1,2,\dots, n-1\}$ define
\[
\psi (x) :=
\begin{cases}
i^{1-\lambda} {\rm Re}\bigl[(x_n + i x_\alpha)^\lambda\bigr] \quad & x_n \ge 0
\\
-i^{1-\lambda}{\rm Re}\big[(x_n + i x_\alpha)^\lambda\bigr]  		& x_n \le 0.
\end{cases}
\]
Note that
\[
\psi(x' ,x_n)= \psi(x' ,-x_n)\quad \mbox{and}\quad \psi (x) = 0  \quad \mbox{on } \{x_n =0\}.
\]
In addition,  on $\{x_n =0\}$ we have
$
\partial_n \psi (x',0^+) =    \lambda |x_\alpha|^{\lambda-1}
$ (recall that $\lambda-1$  is even), therefore
\[
\Delta \psi =  2\lambda |x_\alpha|^{\lambda-1}  \HH^{n-1}|_{\{x_n =0\}}.
\]
On the other hand, since both $\psi$ and $q$ are $\lambda$-homogeneous we have $(x\cdot \nabla q)  \psi  =  q(x\cdot \nabla \psi)  = \lambda q\psi$. Thus $\int_{\partial B_1}  (q_\nu \psi - q\psi_\nu)=0$, and an integration by parts gives
\[
\int_{B_1}  \Delta q \psi  =  \int_{B_1}   q\Delta \psi.
\]
Since $\Delta q$ is concentrated on $\{x_n=0\}$ where $\psi$ vanishes, combining all together we get
\[
0 = \int_{B_1}   q\Delta \psi  = 2\lambda\int_{B_1 \cap \{x_n =0\}}  q |x_\alpha|^{\lambda-1}.
\]
Since $q\ge 0$ on $\{x_n=0\}$ and the previous equality holds for all $\alpha \in \{1,2,\dots, n-1\}$, we conclude that $q$ must vanish on $\{x_n=0\}$.
\end{proof}

\begin{lemma}\label{lem:signorinieven3}
Assume that $q:\R^n \rightarrow \R$ is a $3$-homogenous  even solution of the Signorini problem~\eqref{signorininormal}.
Then, after a suitable rotation that leaves the hyperplane $\{x_n=0\}$ invariant, we have
\[
q(x) =b |x_n|^3 -3 |x_n| \left( \sum_{\alpha=1}^{n-1} b_\alpha x_\alpha^2\right),
\]
where $b, b_\alpha\ge 0$ and $b = \sum_{\alpha=1}^{n-1} b_\alpha$.
\end{lemma}

\begin{proof}
By Lemma~\ref{lem:signoriniodd} $q$ must vanish everywhere on $\{x_n=0\}$.  Thus, $q$ is a $3$-homogenous harmonic function in $\{x_n>0\}$ vanishing on $\{x_n =0\}$, so its odd extension is a 3-homogeneous harmonic polynomial. This implies, after a rotation, that
\[
q(x) =b x_n^3 -3 x_n \left( \sum_{\alpha=1}^{n-1} b_\alpha x_\alpha^2\right)\qquad \text{for }x_n>0,
\]
where $b, b_\alpha\in \R$ and $b = \sum_{\alpha=1}^{n-1} b_\alpha$.
Finally, since $q$ is an even solution of Signorini, it follows that $\partial_n q \le 0$ on $\{x_n=0\}$.
This implies that $b_\alpha\ge 0$ (and thus $b\ge 0$), concluding the proof.
\end{proof}

In order to continue our analysis, we introduce a
new monotonicity formula:

\begin{lemma}\label{lem:derprod}
Let  $u:B_1\to[0,\infty)$ solve \eqref{eq:UELL}, and let $0 \in \Sigma^{\ge3}_{n-1} \setminus \Sigma^{>3}_{n-1}$. 
Set $w := u-p_2$ and 
$w_r:= w(r\,\cdot\, )$. 
Then, for fixed $\varrho\in (0,1)$ and for any 3-homogeneous solution $q$ of the Signorini problem \eqref{ETOP}, we have
\[
\frac{d}{dr} \int_{\partial B_\varrho}  w_r q = \frac{3}{r} \int_{\partial B_\varrho} w_r q - \frac{\varrho}{r} \int_{B_\varrho} w_r \Delta q + O(r^3).
\]
In particular
\[\frac{d}{dr}\left(\frac{1}{r^3}\int_{\partial B_1} w_r q\right) \geq -C.\]
\end{lemma}

\begin{proof}
We have 
\[
\begin{split}
\frac{d}{dr} \int_{\partial B_\varrho}  w_r q  &=  \int_{\partial B_\varrho}  \frac x r \cdot \nabla w_r q   =  \frac{\varrho}{r} \int_{\partial B_\varrho}  \partial_\nu w_r q   =  \frac{\varrho}{r} \int_{B_\varrho}  {\rm div} ( \nabla w_r q)= \frac{\varrho}{r} \bigg( \int_{B_\varrho}  \nabla w_r \nabla q   + \int_{B_\varrho}  \Delta w_r q\bigg)
\\
&= \frac{\varrho}{r} \bigg( \int_{\partial B_\varrho}   w_r \partial_\nu q   -\int_{B_\varrho} w_r \Delta q  + \int_{B_\varrho}  \Delta w_r q\bigg).
\end{split}
\]
Now, since  $q$ is 3-homogeneous, we find that $\varrho \int_{\partial B_\varrho}  w_r \partial_\nu q = 3\int_{\partial B_\varrho}  w_r q$. To complete the proof of the Lemma we only need to show that 
$\int_{B_\varrho}  \Delta w_r q = O(r^4)$.

With no loss of generality, assume that $p_2 = \frac 1 2 x_n^2$. Then it follows by Proposition~\ref{prop:E2B3} that $\{u(r\,\cdot\,) =0\}\cap B_1 \subset \{ |x_n|\le Cr\}$, and  $|q| \le C|x_n|$ in $B_1$ (by Lemma~\ref{lem:signorinieven3}).
Thus, since $\Delta w_r =   -r^2 \chi_{\{u(r\,\cdot\,) =0\}}$, we get $\int_{B_\varrho}  \Delta w_r q = O(r^4)$.

Finally, taking $\varrho=1$ and using that $-w_r\Delta q\geq0$ in $\R^n$ (since $w_r=u(r\,\cdot\,)\geq 0$ on $\{x_n=0\}$), we obtain
\[
\frac{d}{dr} \left(\frac{1}{r^3}\int_{\partial B_1} w_r q\right)  =  \frac{1}{r^4} \bigg( -\int_{\partial B_1} w_r \Delta q  + \int_{B_1}  \Delta w_r q\bigg)\geq \frac{1}{r^4}\int_{B_1}  \Delta w_r q\geq -C,
\]
as desired.
\end{proof}

As a consequence of the previous lemma, we deduce the uniqueness of blow-ups in $\Sigma_{n-1}^{\geq3} \setminus \Sigma_{n-1}^{3rd}$.
Notice that this is quite surprising, since even in the (simpler) case of the Signorini problem it was not known if cubic blow-ups are unique at every point (see Appendix \ref{apb}).

\begin{proposition}\label{prop-uniq-cubic-blowups}
Let $u:B_1\to[0,\infty)$ solve \eqref{eq:UELL}, and let $0 \in \Sigma^{\ge3}_{n-1}\setminus \Sigma_{n-1}^{3rd}$. 
Then the limit
\[
\tilde q := \lim_{r\downarrow 0} \frac{(u-p_2)(r\,\cdot\,)}{r^3}
\]
exists, and it is a 3-homogeneous (non-harmonic) solution of Signorini.
\end{proposition}

\begin{proof}
Let $w:=u-p_2$, and $w_r=w(r\,\cdot\,)$.
Assume that
\[
q^{(i)} = \lim_{r^{(i)}_k\downarrow 0} \frac{1}{(r_k^{(i)})^3}  w_{r_k^{(i)}} ,\qquad i=1,2,
\]
are two accumulation points along different sequences $r_k^{(i)}$. Then, give a $3$-homogeneous solution of Signorini $q$, we can apply Lemma~\ref{lem:derprod} to deduce that  $r\mapsto \frac{1}{r^3} \int_{\partial B_1}  w_r q $ has a unique limit as $r\to 0.$
In particular this implies that 
\begin{equation}\label{haiohaohaoih}
\int_{\partial B_1} q^{(1)} q = \int_{\partial B_1} q^{(2)} q. 
\end{equation}
Choosing $q=q^{(1)}-q^{(2)}$ we obtain 
$$
\int_{\partial B_1} \big(q^{(1)}-  q^{(2)}\big)^2=0,
$$
hence $q^{(1)}\equiv q^{(2)}$, as desired.
\end{proof}

The next step consists in showing that if $0 \in\Sigma^{3rd}_{n-1}$ then $\phi(0^+, u-p_2-p_3 )> 3$.
This is a kind of nondegeneracy property, which implies that $\Sigma^{3rd}_{n-1}\setminus \Sigma^{>3}_{n-1}$ is empty. 
This highly non-trivial fact  is essential in order to establish Schaeffer conjecture in $\R^4$, and it is the core of this section. 
Its proof require a barrier and ODE-type arguments obtained below.

\begin{lemma}\label{lem:compcubic}
Let $u:B_1\to[0,\infty)$ solve \eqref{eq:UELL}, and let $0 \in\Sigma^{3rd}_{n-1}$. 
Set $w := u-p_2-p_3$, and let $w_r$ and $\tilde w_r$ be defined as in \eqref{defwr_tildewr}. 
Assume that $\{p_2=0\}=\{x_n=0\}$, and
given $x= (x_1, \dots, x_{n}) \in \R^n$ let $x':= (x_1, \dots, x_{n-1}) \in \R^{n-1}$.

For any $\eta>0$ there exists  $\delta = \delta(n,\eta)$  such that  if
\[
\| \tilde w_r -q\|_{L^\infty(B_2)} \le \delta \qquad \mbox{for }\quad  q = |x_n| \bigg(\frac{a}{3} x_n^2 - x' \cdot A x'  \bigg), \quad A\in \R^{(n-1)\times (n-1)},\,A\ge0,\  a = {\rm trace}(A)
\]
then
\[
u(r\,\cdot\,)= O(r^4) \quad \mbox{on}\quad \{x_n=0\} \cap  (B_1\setminus B_{1/2} )\cap \big\{x' \cdot A x'  \ge \eta\big\}.
\]
\end{lemma}

\begin{proof}
Let  $z= (z',0)\in (B_1\setminus B_{1/2} )$ satisfy $z' \cdot A z'  \ge \eta$, and given $c>0$, denote 
\[
\phi_{z,c}(x) := (p_2 + p_3)(rz+rx)  -r^3(n-1)|x_n|^2 + r^3 |x'|^2 + c.
\]
Note that, since $q$ is uniformly close to $\tilde w_r$, the constant $a$ and the matrix $A$ appearing in the definition of $q$ are universally bounded. Hence, there exists  $\varrho>0$ small, depending only on $n$ and  $\eta$, such that
\[ 
-n |x_n|^2   \ge   |x_n| \bigg(\frac{a}{3} x_n^2  - (z'+x') \cdot A  (z'+x')  \bigg)\qquad \mbox{for }|x|<\varrho.
\] 
Thus, denoting $h_r : = H(r, w)^{1/2} = o(r^3)$, we have
\begin{equation}\label{haiohaoih}
\begin{split}
 \phi_{z,c} &\ge  (p_2 + p_3)(rz+rx)    -r^3(n-1)|x_n|^2 + r^3 |x'|^2 +c
 \\ 
  &> (p_2 + p_3)(rz+rx)  +h_r q(rz+rx) + r^3( |x_n|^2+|x'|^2)
  \\
  & \ge  (u(rz+rx) -\delta h_r) + r^3 |x|^3  \qquad \mbox{for }|x|<\varrho .
\end{split}
\end{equation}
We now compare the two functions $\hat u_z(x) :=  u(rz +rx)$ and $\phi_{z,c}$ in $B_\varrho(0)$. 
Two cases arise:\\
(1) either $\phi_{z,c}\geq u_z$ for each $c>0$, which implies that $0\leq u(rz)=u_z(0)\leq \phi_{z,0}(0)=0$ (since $p_2$ and $p_3$ vanish on $\{x_n=0\}$);\\
(2) or there exists $c_*> 0$  such that $\phi_{z,c_*}$ touches from above $\hat u_z$ at some point $y =(y',y_n)\in \overline{B_\rho}$.
Note that $\Delta \phi_{z,c_*} = r^2$ in $B_\rho$, and  $\Delta  \hat u_{z}  = r^2 \chi_{\{\hat u_{z}>0\}}$ in $B_\rho$.
Also, since $h_r : = H(r, w)^{1/2} = o(r^3)$, for $r$ small enough we have $\phi_{z,c} \ge \hat u_z(x)$ on $\partial B_{\varrho}$ (by \eqref{haiohaoih}). 
Thus, it follows by the maximum principle that the point $y$ must belong to $\{\hat u_{z}=0\}\cap B_\rho \subset \{|x_n| \le Cr\}\cap B_\rho$, therefore  
\[
0 = \hat u_z(y)  = \phi_{z,c_*}(y) =   (p_2 + p_3)(rz'+ry', ry_n) -r^3(n-1)|y_n|^2 + r^3 |y'|^2 +c_*
 \ge -C r^4 +c_*.
\]
Thus $c_*\leq Cr^4$, and as a consequence
\[0 \le u(rz) = \hat u_z(0) \le   \phi_{z,c_*}(0) = c_* \le Cr^4.\]
This proves that in both cases $0\leq u(rz)\leq Cr^4$, and since $z\in  \{x_n=0\} \cap  (B_1\setminus B_{1/2} )\cap \big\{x' \cdot A x'  \ge \eta\big\}$ is arbitrary, the result follows.
\end{proof}

Another key tool is the following ODE-type formula.

\begin{lemma}\label{lem:ODEalessio}
Let $u:B_1\to[0,\infty)$ satisfy \eqref{eq:UELL}, and  $0 \in\Sigma^{3rd}_{n-1}$. 
Set $w := u-p_2-p_3$, let $w_r$ and $\tilde w_r$ be defined as in \eqref{defwr_tildewr}, and set $h(r): = H(r,w)^{1/2}$. 
Assume that $\{p_2=0\}=\{x_n=0\}$,
and given a symmetric  $(n-1)\times (n-1)$ matrix $A\geq0$, we define its ``associated  solution of the Signorini problem''
 \begin{equation}\label{haisohaioh}
 q_A(x) := |x_n| \bigg( \frac {{\rm trace}(A)} 3 x_n^2 -  x'\cdot Ax'\bigg),\qquad x=(x',x_n) \in \R^{n-1}\times \R,
 \end{equation}
and we introduce the quantity 
\begin{equation}\label{psi-alessio}
\psi(r;A): =  \int_{\partial B_1}  \tilde w_r q_A  - 2\int_{\partial B_{1/2}}  \tilde w_r q_A .
\end{equation}
Then
\[
\frac{d}{dr} \psi(r;A)  = -\theta(r)  \psi(r;A)  - \frac{1}{r} \int_{B_1\setminus B_{1/2}} \tilde w_r \Delta q_A + O\big(r^3/h(r)\big),
\]
where 
\[
\theta(r) : =  \bigg(\frac{h'(r)}{h(r)} +\frac 3 r\bigg) = \big(  \log(h(r)/r^3)\big)'.
\]
\end{lemma}

\begin{proof}
As in the proof of Lemma~\ref{lem:derprod}, we obtain
\[
\frac{d}{dr} \int_{\partial B_\varrho}  w_r q_A = \frac{3}{r} \int_{\partial B_\varrho} w_r q_A - \frac{\varrho}{r} \int_{B_\varrho} w_r \Delta q_A + O(r^3).
\]
Now, since $\tilde w_r = w_r/h(r)$ we deduce that
\[
\frac{d}{dr} \int_{\partial B_\varrho}  \tilde w_r q_A = \bigg( - \frac {h'(r)}{h(r)} + \frac{3}{r}\bigg) \int_{\partial B_\varrho} \tilde w_r q - \frac{\varrho}{r} \int_{B_\varrho} \tilde w_r \Delta q_A + O\big(r^3/h(r)\big)
\]
and the lemma follows by combining the identities for $\varrho =1$ and  $\varrho =1/2$.
\end{proof}

We shall also need the following formula:

\begin{lemma}\label{lem:scalarprod}
Given  $A,\bar A\geq 0$ be two symmetric  $(n-1)\times (n-1)$ matrices, let $q_{A}$ and $q_{\bar  A}$ as defined in \eqref{haisohaioh}. 

Then
\[
 \int_{\partial B_\varrho}  q_{A} q_{\bar  A}  = \frac{4  \varrho^{n+5} |\partial B_1|}{n(n+2)(n+4)} \bigg\{ {\rm trace }\big(  A\cdot \bar A\big) + \frac 1 3 {\rm trace }(A)\,  {\rm trace} (\bar A) \bigg\}.
\]
\end{lemma}

\begin{proof}
Let $A=(a_{\alpha\beta})_{\alpha,\beta=1}^{n-1}$, $\bar A=(\bar a_{\alpha\beta})_{\alpha,\beta=1}^{n-1}$, $a={\rm trace}(A) = \sum_\alpha a_{\alpha\alpha}$, $\bar a ={\rm trace}(\bar A)= \sum_\alpha \bar a_{\alpha\alpha}$. Denote for brevity  $q= q_A$, $\bar q=q_{\bar  A}$.
Then 
\[
\begin{split}
 \int_{\partial B_\varrho} q\bar q =   \sum_{\alpha,\beta,\gamma,\delta=1}^{n-1}
  \int_{\partial B_\varrho}  x_n^2 &\bigg(a \frac{x_n^2}{3} - a_{\alpha\beta}x_\alpha  x_\beta\bigg) \bigg(\bar a\frac{x_n^2}{3} - \bar a_{\gamma \delta}x_\gamma  x_\delta\bigg) .
\end{split}
\]
Up to a rotation in the $\{x_n=0\}$ plane, we may assume that $a_{\alpha\beta}$ is diagonal.  
Noting that  $\int_{\partial B_\varrho} x_n^4x_\gamma x_\delta=\int_{\partial B_\varrho} x_n^2x_\alpha^2x_\gamma x_\delta=0$ for $\gamma \neq\delta$, we have
\[
 \int_{\partial B_\varrho} q\bar q  = \int_{\partial B_\varrho}\bigg( \frac{a\bar a}{9} x_n^6   +\sum_{\alpha} \bigg\{- \bigg( \frac{\bar a}{3}  a_{\alpha\alpha} + \frac{a}{3} \bar a_{\alpha \alpha}\bigg) x_n^4x_\alpha^2 +  a_{\alpha\alpha} \bar a_{\alpha \alpha} x_n^2 x_\alpha^4\bigg\}    +  \sum_{\alpha \neq \gamma} a_{\alpha\alpha} \bar a_{\gamma \gamma}   x_n^2 x_\alpha^2  x_\gamma^2\bigg)
\]
We observe that 
\[
\int_{\partial B_1} x_i^4 =  \frac 1 4 \int_{\partial B_1} \partial_\nu(x_i^4) =  \frac 1 4 \int_{B_1} \Delta(x_i^4) =  3 \int_{B_1} x_i^2 = \frac{3}{n+2} \int_{\partial B_1} x_i^2 =   \frac{3}{n(n+2)} |\partial B_1|.
\]
Similarly,
\[
\int_{\partial B_1} x_i^6  =  \frac 1 6 \int_{B_1} \Delta(x_i^6) =  5 \int_{B_1} x_i^4 = \frac{5}{n+4} \int_{\partial B_1} x_i^4 =\frac{15}{n(n+2)(n+4)} |\partial B_1| ,
\]
\[
\int_{\partial B_1} x_i^2x_j^2 =  \frac 1 4 \int_{B_1} \Delta( x_i^2x_j^2)  = \frac{2}{4(n+2)} \int_{\partial B_1}  2x_i^2 = \frac{1}{n(n+2)} |\partial B_1|,  
\]
\[
\int_{\partial B_1} x_i^4x_j^2 =  \frac 1 6 \int_{B_1} \Delta( x_i^4x_j^2)  = \frac{1}{6(n+4)} \int_{\partial B_1} (12 x_i^2x_j^2 + 2 x_i^4)  = \frac{3}{n(n+2)(n+4)} |\partial B_1|,
\]
and
\[
\int_{\partial B_1} x_i^2x_j^2x_k^2 =  \frac 1 6 \int_{B_1} \Delta( x_i^2x_j^2 x_k^2)  = \frac{3}{6(n+4)} \int_{\partial B_1}  2x_i^2x_j^2  =  \frac{1}{n(n+2)(n+4)} |\partial B_1| .
\]
Thus, calling $c_n := \frac{|\partial B_1|}{n(n+2)(n+4)}$  and using that $\sum_{\alpha} a_{\alpha\alpha} = a$ and  $\sum_{\alpha} \bar a_{\alpha\alpha} = \bar a$, we obtain
\[
 \int_{\partial B_\varrho} q\bar q  = \varrho^{n+5} \bigg(  \frac{a\bar a}{9} 15 c_n    - \frac{a\bar a}{3}   6c_n  +  \sum_{\alpha }a_{\alpha\alpha} \bar a_{\alpha \alpha}  3c_n    +  \sum_{\alpha \neq \gamma} a_{\alpha\alpha} \bar a_{\gamma \gamma} c_n\bigg).
\]
Finally, since $\sum_{\alpha}\sum_{\gamma}a_{\alpha\alpha} \bar a_{\gamma\gamma} =  \big( \sum_{\alpha }a_{\alpha\alpha} \big) \big( \sum_{\gamma }\bar a_{\gamma\gamma} \big) = a\bar a$ and recalling that $a_{\alpha \beta}$ is diagonal, we 
get
\[
 \int_{\partial B_\varrho} q\bar q  =   2 \varrho^{n+5} c_n \sum_{\alpha }a_{\alpha\alpha} \bar a_{\alpha \alpha}   = 2c_n\varrho^{n+5}\left(2\, {\rm trace }\big(  (a_{\alpha\beta})\cdot(\bar a_{\gamma\delta}) \big) + \frac  2 3 {a\bar a}\right),
\]
as claimed.
\end{proof}

We can now finally prove the following fundamental result, which implies that $\Sigma_{n-1}^{3rd}\setminus\Sigma_{n-1}^{>3}=\varnothing$:

\begin{proposition}\label{prop:nobadpoints}
Let $0 \in\Sigma^{3rd}_{n-1}$, and set $w := u-p_2-p_3$.  
Then $\phi(0^+, w)>3$. 
\end{proposition}

\begin{proof}
Without loss of generality, we can assume that $\{p_2=0\}=\{x_n=0\}$.

Suppose by contradiction that $\phi(0^+, w)=3$. 
Then we know that  the accumulation points  of $\tilde w_r$ as $r\downarrow 0$ must be 3-homogeneous even solutions of the Signorini problem, that is,
of the form $q_A$ for some symmetric matrix $A\geq0$ (see \eqref{haisohaioh}). 
Note that, by construction, $\|q_A\|_{L^2(\partial B_1)} =1$ and thus the matrix $A$ must have at least one positive eigenvalue.

Let us define the quantity 
\begin{equation}\label{themax}
\Psi(r): = \max\big\{   \psi(r;A) \ : \  \|q_A \|_{L^2(\partial B_1)}=1\big\},
\end{equation}
where $\psi$ is given by \eqref{psi-alessio}.
Let $A_r^*$ be the matrix for which the previous maximum is attained. 
Then, as a consequence of Lemma~\ref{lem:ODEalessio}, we have 
\[
\frac{d}{dr} \Psi(r)  = \theta(r)  \Psi(r)   - \frac{1}{r} \int_{B_1\setminus B_{1/2}} \tilde w_r \Delta q_{A_r^*} + O\big(r^3/h(r)\big), \quad \mbox{ for a.e. } r>0 .
\]
On the other hand, if we define $\Phi(r) : = \psi(r, {\rm Id})$,
then
\begin{equation}\label{reugheirhn}
\frac{d}{dr} \Phi(r)  = \theta(r)  \Phi(r)   - \frac{1}{r} \int_{B_1\setminus B_{1/2}} \tilde w_r \Delta q_{\rm Id} + O\big(r^3/h(r)\big).
\end{equation}
We now claim that
\[
\Psi(r)  \asymp  \Phi(r)  \asymp \frac{\Psi(r) }{\Phi(r)}\asymp  1 \quad \mbox{as } \quad r\downarrow 0,
\]
where $X\asymp Y$ is a short notation for $X\le C(n) Y$ and $Y \le C(n) X$.
Indeed, the accumulation points of $\tilde w_r$ (as $r\downarrow 0$ and in the $C^0_{\rm loc}(\R^n)$ topology) are of the form $q_A$ (and have unit norm) and thus for every $r>0$ we have $w_r -q_{A_r} =o(1)$ for some $A_r$.
Hence, by definition of $\Psi$,
\[
\begin{split}
\Psi(r) \ge  \psi(r;A_r) &=  \int_{\partial B_1}\tilde w_r q_{A_r}  - 2\int_{\partial B_{1/2}}  \tilde w_r q_{A_r}
=   \int_{\partial B_{1}}  q_{A_r}^2  - 2\int_{\partial B_{1/2}}  q_{A_r}^2 +o(1) 
\\
& =   (1-2^{-n-4})\int_{\partial B_{1}}  q_{A_r}^2 + o(1)  \ge c(n)>0.
\end{split}
\]
Note that the above computation shows also that
$\psi(r;A^*_r) = (1-2^{-n-4})\int_{\partial B_{1}} q_{A_r} q_{A_r^*}+o(1)$  (as $r\downarrow 0$).
Thus since by defintion of $A^*_r$ we have $\psi(r;A^*_r)\ge\psi(r;A_r)$  it follows 
\[
\int_{\partial B_{1}} q_{A_r^*} q_{A_r}    \ge \int_{\partial B_{1}} q_{A_r} ^2 +o(1)
\]
Since$\int_{\partial B_{1}} q_{A_r^*}^2   =  \int_{\partial B_{1}} q_{A_r}^2=1$ is follows that $q_{A_r^*} = q_{A_r} + o (1)$ and hence
\[
A_r^* = A_r+o(1)\quad \mbox{ as $r\downarrow 0$}.
\]

Similarly,  using   Lemma~\ref{lem:scalarprod},
\[
\begin{split}
\Phi(r) &= \int_{\partial B_{1}}  \tilde w_r q_{\rm Id}  - 2\int_{\partial B_{1/2}}  \tilde w_r q_{\rm Id} 
=   \int_{\partial B_{1}}  q_{A_r}q_{\rm Id}  - 2\int_{\partial B_{1/2}}  q_{A_r}q_{\rm Id} +o(1)
\\
& =    \frac{(1-2^{-n-4}) 4 |\partial B_1| }{n(n+2)(n+4)}  \bigg\{  {\rm trace}(A_r) +  \frac 1 3 {\rm trace}(A_r) (n-1) \bigg\} +o(1)\ge c(n)>0.
\end{split}
\]
Since $\Psi(r)$ and $\Phi(r)$ are bounded by above, the claim follows.

Now notice that, using the expressions for $\frac{d}{dr}\Psi$ and $\frac{d}{dr}\Phi$, we find
\[
\frac{d}{dr}  \bigg( \frac{\Psi(r)}{\Phi(r)}\bigg)  =   \frac{1}{r}  \frac{ - \Phi(r) \int_{B_1\setminus B_{1/2}} \tilde w_r \Delta q_{\rm Id} + \Psi(r) \int_{B_1\setminus B_{1/2}} \tilde w_r \Delta q_{A_r^*} }{\Phi(r)^2}+ O\big(r^3/h(r)\big),
\]
We claim that, given $\varepsilon>0$, for $r$ sufficiently small it holds
\begin{equation}\label{ggooaall}
\left|\int_{B_1\setminus B_{1/2}} \tilde w_r \Delta q_{A_r^*}\right| \leq \varepsilon \left|\int_{B_1\setminus B_{1/2}} \tilde w_r \Delta q_{\rm Id}\right| + Cr^4/h(r).
\end{equation}
Indeed, it follows by  Lemma~\ref{lem:compcubic} that, for any $\eta>0$, if $r>0$ is sufficiently small so that $\|\tilde w_r - q_{A_r^*}\|_{L^\infty(B_2)} \le \delta (n,\eta)$ then
(here we use the notation $B_r':=B_r\cap\{x_n=0\}$)
\[\begin{split}
-\int_{B_1\setminus B_{1/2}}  w_r \Delta q_{A_r^*} & =2  \int_{B'_1\setminus B'_{1/2}}  u(rx',0) \, (x'\cdot A_r^* x')\,dx' \\
&\leq 2 \eta \int_{(B'_1\setminus B'_{1/2})\cap \{x'\cdot Ax'\leq \eta\}} u(rx',0)\,dx'+\int_{(B'_1\setminus B'_{1/2})\cap \{x'\cdot Ax'\geq \eta\}} Cr^4 \,dx' \\
& \le  2\eta \int_{B'_1\setminus B'_{1/2}}  u(rx',0)  +  Cr^4
\end{split}\]
(here we used that $w_r\equiv u(r\,\cdot\,)$ on $\{x_n=0\}$), while 
\[
-\int_{B_1\setminus B_{1/2}}  w_r \Delta q_{\rm Id} = c_n \int_{B'_1\setminus B'_{1/2}}  u(r x',0)\,  |x'|^2\,dx'   \ge c_n  \int_{B'_1\setminus B'_{1/2}}  u(rx',0)\,dx',
\]
where $c_n>0$.
Dividing by $h(r)$, we obtain 
\[
0\leq -\int_{B_1\setminus B_{1/2}}  \tilde w_r \Delta q_{A_r^*} \le -4\eta \int_{B_1\setminus B_{1/2}}  \tilde w_r \Delta q_{\rm Id} + Cr^4/ h(r),
\]
and thus \eqref{ggooaall} follows.

Hence, thanks to \eqref{ggooaall}, we have that
\[\begin{split}
\frac{d}{dr} \bigg( \frac{\Psi(r)}{\Phi(r)}\bigg)  & =   \frac{1}{r}  \frac{  - \Phi(r) \int_{B_1\setminus B_{1/2}} \tilde w_r \Delta q_{\rm Id} +  \Psi(r) \int_{B_1\setminus B_{1/2}} \tilde w_r \Delta q_{A_r^*}}{\Phi(r)^2}+ O\big(r^3/h(r)\big) \\
&=-\frac{a(r)}{r} \int_{B_1\setminus B_{1/2}} \tilde w_r \Delta q_{\rm Id} + O\big(r^3/h(r)\big),\qquad a(r)\asymp 1.
\end{split}
\]
Choosing $r_0$ so that $C^{-1} \leq a(r)\leq C$ over $[0,r_0]$, 
we can integrating the above ODE over $[\hat r,r_0]$ for any $\hat r>0$.
Then, since the integrals of 
$\frac{d}{dr} \big(\frac{\Psi(r)}{\Phi(r)}\big)$ and $r^3/h(r)$ are both uniformly bounded independently of $\hat r$, so must be the integral of the negative term $\frac{a(r)}r \int_{B_1\setminus B_{1/2}} \tilde w_r \Delta q_{\rm Id}$.
Hence, this proves that
$$
\int_0^{r_0}\biggl|\frac{1}r \int_{B_1\setminus B_{1/2}} \tilde w_r \Delta q_{\rm Id}\biggr|\,dr<\infty.
$$
Since $\Phi(r)\asymp 1$ and $\theta(r)= \frac{d}{dr} \log (h(r)/r^3)$, it follows from \eqref{reugheirhn} that
$$
\frac{d}{dr}\log \Phi(r)= \frac{d}{dr} \log (h(r)/r^3)+g(r),\qquad \text{with $g \in L^1([0,r_0])$}.
$$
Integrating over $[\hat r,r_0]$ and using again that $\Phi(r)\asymp 1$, we deduce that $\log (h(\hat r)/\hat r^3)$ is uniformly bounded as $\hat r \to 0$,
therefore $h(r)\asymp r^3$.
However, since $0 \in \Sigma_{n-1}^{3rd}$ we know that $h(r)=o(r^3)$, contradiction.
\end{proof}

\section{Symmetry properties of blow-ups for 1-parameter family of solutions} \label{sec:EG2B}

As explained in the introduction, to establish generic regularity results, we shall consider 1-parameter monotone family of  solutions. For this, we shall use the parameter $t$ (over which solutions are indexed) as a second variable for our solution $u$ (one may think of $t$ as a ``time'' variable, although there is no equation in $t$).

So, let $u: \overline{B_1}\times [-1,1]  \rightarrow \R$, $u\ge 0$, be a monotone  1-parameter family of  solutions of the obstacle problem, namely
\begin{equation}\label{eq:UELL+t}
\Delta u(\cdot,t) = \chi_{\{u(\cdot,t)>0\}}\quad \textrm{and} \quad    0\leq u(\,\cdot\,,t)  \le u(\,\cdot\,, t') \quad  \mbox{in } B_1,\qquad \mbox{for } -1\le t\leq t'\le1.
\end{equation}
We will assume in addition that $u\in C^0\big(\overline{B_1}\times [-1,1]  \big)$ (this continuity property in $t$ follows by the maximum principle whenever $u\in C^0\big(\partial {B_1}\times [-1,1]  \big)$).

Note that by, the regularity results for the obstacle problem,  $u(\,\cdot\,, t)$ is of class $C^{1,1}$ inside  $B_1$ for each $t\in (-1,1)$.
Moreover, for each fixed $t\in (-1,1)$, we can apply the results of the previous sections, and define the different blow-ups at singular points.

So, following the previous sections,
we say that $(x_\circ, t_\circ)$ is a singular point of $u$ if $x_\circ$ is a singular point of $u(\,\cdot\, , t_\circ )$. 
Given a singular free boundary point $(x_\circ, t_\circ)$, we denote
\[ p_{2,x_\circ, t_\circ}(x) := \lim_{r\to 0} r^{-2}u(x_\circ+ rx, t_\circ).\]
Note that $p_{2,x_\circ, t_\circ}$ is a convex $2$-homogeneous polynomials  with $\Delta p_{2,x_\circ, t_\circ} =1$.
When $(x_\circ,t_\circ) = (0,0)$, we simplify the notation to $p_2$.

From now on, using the notation introduced in the previous sections, we set:
\begin{equation} \label{eq:Sigmaall}
\begin{array}{rcl}
\mathbf{\Sigma} &:=& \{ (x_\circ, t_\circ) \mbox{ singular points in } B_{1}\times[-1,1] \}, \vspace{1mm}\\
\mathbf{\Sigma}_m &:=&  \big\{ (x_\circ, t_\circ): x_\circ\in {\Sigma}_m \mbox{ for } u(\,\cdot\,,t_\circ)\big\},  \quad 0\le m \le n-1, \vspace{1mm}\\
\mathbf{\Sigma}^a_m &:=& \big\{ (x_\circ, t_\circ): x_\circ\in {\Sigma}_m^a \mbox{ for } u(\,\cdot\,,t_\circ)\big\},  \quad 0\le m \le n-2, \vspace{1mm}\\
\mathbf{\Sigma}_{n-1}^{<3} &:=& \big\{ (x_\circ, t_\circ): x_\circ\in {\Sigma}_{n-1}^{<3} \mbox{ for } u(\,\cdot\,,t_\circ)\big\}, \vspace{1mm}\\
\mathbf{\Sigma}_{n-1}^{\ge 3} &:=& \big\{ (x_\circ, t_\circ): x_\circ\in {\Sigma}_{n-1}^{\ge 3} \mbox{ for } u(\,\cdot\,,t_\circ)\big\}, \vspace{1mm}\\
\mathbf{\Sigma}_{n-1}^{3rd} &:=& \big\{ (x_\circ, t_\circ): x_\circ\in {\Sigma}_{n-1}^{3rd} \mbox{ for } u(\,\cdot\,,t_\circ)\big\}, \vspace{1mm}\\
\mathbf{\Sigma}_{n-1}^{>3} &:=& \big\{ (x_\circ, t_\circ): x_\circ\in {\Sigma}_{n-1}^{>3} \mbox{ for } u(\,\cdot\,,t_\circ)\big\}, \vspace{1mm}\\
\mathbf{\Sigma}_{n-1}^{4th} &:=& \big\{ (x_\circ, t_\circ): x_\circ\in {\Sigma}_{n-1}^{4th} \mbox{ for } u(\,\cdot\,,t_\circ)\big\}, \vspace{1mm}\\
\mathbf{\Sigma}_{n-1}^{>4} &:=& \big\{ (x_\circ, t_\circ): x_\circ\in {\Sigma}_{n-1}^{>4} \mbox{ for } u(\,\cdot\,,t_\circ)\big\}, \vspace{1mm}\\
\mathbf{\Sigma}_{n-1}^{\ge 5-\zeta} &:=& \big\{ (x_\circ, t_\circ): x_\circ\in {\Sigma}_{n-1}^{\ge 5-\zeta} \mbox{ for } u(\,\cdot\,,t_\circ)\big\},\quad \zeta \in (0,1).
\end{array}
\end{equation}
Recall that ${\Sigma}_m$, ${\Sigma}_m^a$,  ${\Sigma}_{n-1}^{<3}$, and ${\Sigma}_{m}^{\ge3}$ were defined in \eqref{ahoiah0}-\eqref{ahoiah1}, while ${\Sigma}_{m}^{3rd}$, $\Sigma_{n-1}^{>3}$, ${\Sigma}_{m}^{\ge4}$, ${\Sigma}_{n-1}^{4th}$, ${\Sigma}_{n-1}^{>4}$, and ${\Sigma}_{n-1}^{\ge5-\zeta}$ were defined in Definitions~\ref{def:Sigma3rd},~\ref{def:Sigma>3},~\ref{def:Sigma4th},~\ref{def:Sigma>4}, respectively.

\begin{remark}
\label{rmk:Sigma 3rd}
Note that, as a consequence of Proposition~\ref{prop:nobadpoints}, $\mathbf{\Sigma}_{n-1}^{3rd}=\mathbf{\Sigma}_{n-1}^{>3}.$
\end{remark}

For $(x_\circ,t_\circ)\in \mathbf{\Sigma}_{m}^{3rd}$ we define
\begin{equation}
\label{eq:def p3 t}
p_{3,x_\circ, t_\circ}(x) := \lim_{r\to 0} r^{-3}\bigl(u(x_\circ+ rx, t_\circ)-p_{2,x_\circ, t_\circ}(rx)\bigr),
\end{equation}
and for  $(x_\circ,t_\circ)\in \mathbf{\Sigma}_{m}^{4th}$ we define $\anz_{x_\circ,t_\circ}$ as the fourth order Ansatz of $u(x_\circ+\,\cdot\,,t_\circ)$ at 0 (cf. \eqref{eq:ansatzell}), and
\begin{equation}
\label{eq:def p4 t}
 p_{4,x_\circ, t_\circ}(x) := \lim_{r\to 0} r^{-4}\bigl(u(x_\circ+ rx, t_\circ)-\anz_{x_\circ, t_\circ}(rx)\bigr).
\end{equation}

We begin with a simple lemma.

\begin{lemma}\label{lem:EG2B1bis}
Let $u\in C^0\big(\overline{B_1}\times [-1,1]  \big)$ solve \eqref{eq:UELL+t}. 
Then:

{\rm (a)} The singular set  is closed ---more precisely $\mathbf{\Sigma}\cap \overline B_\varrho\times [-1,1]$ is closed for any $\varrho<1$. Moreover,
\[ \mathbf{\Sigma} \cap \overline B_\varrho\times [-1,1]\ni(x_k,t_k)\to (x_\infty, t_\infty) \quad \Rightarrow \quad p_{2,x_k,t_k} \to p_{2,x_\infty, t_\infty}.\]

{\rm (b)} The frequency function \[\mathbf{\Sigma} \ni (x_\circ, t_\circ) \mapsto \phi(0^+, u(x_\circ + \,\cdot\, \,, t_\circ)- p_{2,x_\circ, t_\circ} )\] is upper semi-continuous.

{\rm (c)} If $(x_\circ, t_1)$ and $(x_\circ, t_2)$ belong both to $\mathbf{\Sigma}$ and $t_1< t_2$, then there exists $r>0$ such that $u(x,t)$ is independent of $t$ for all $(x,t)\in B_r(x_\circ)\times [t_1, t_2]$.
\end{lemma}

\begin{proof}
(a) We first show that if $(x_k,t_k)$ are singular points and $(x_k,t_k)\to (x_\infty, t_\infty)$ then the limit point is also singular.
Indeed, by Lemma~\ref{lem:EG2B1} we have
\[
\|u(x_k + \, \cdot\,, t_k)-p_{2, x_k t_k} \|_{L^\infty(B_r)} \le r^2 \omega(r)  \qquad \forall \,r>0.
\]
Hence, since  $u(x_k + \, \cdot\,, t_k) \rightarrow u(x_\infty + \, \cdot\,, t_\infty)$ in $C^0$  as $k\to \infty$ and (after taking a subsequence) $p_{2, x_k t_k} \to P$ for some convex $2$-homogeneous polynomials  with $\Delta P=1$,  we obtain
\begin{equation}\label{noname}
\|u(x_\infty + \, \cdot\,, t_\infty)- P \|_{L^\infty(B_r)} \le r^2 \omega(r)  \qquad \forall \,r>0.
\end{equation}
Thus  $(x_\infty, t_\infty)\in \mathbf{\Sigma}$  and $p_{2, x_\infty t_\infty} = P$. 
A posteriori, we deduce that  for any other subsequence it must be $p_{2, x_k t_k} \to p_{2, x_\infty t_\infty} $ since there is only one $P$ for which \eqref{noname} holds, namely, $p_{2, x_\infty t_\infty}$.

(b) The upper semicontinuity follows from the fact that the map $r\mapsto \phi(r, u(x_\circ +\, \cdot \,, t_\circ ) - p_{2, x_\circ t_\circ} ) $ is increasing, and that for $r>0$ fixed the map $(x_\circ, t_\circ)\mapsto \phi(r, u(x_\circ +\, \cdot \,, t_\circ ) - p_{2, x_\circ t_\circ} ) $ is continuous on $\mathbf{\Sigma}$ ---using (a) and the fact that $ u(x_\circ +\, \cdot \,, t_\circ )$ satisfies uniform $C^{1,1}$ estimates.

(c) 
As in (a), we have, for $i=1,2$,
\[
\|u(x_\circ + \, \cdot\,, t_i)- p_{2,x_\circ, t_i} \|_{L^\infty(B_r)} \le r^2 \omega(r)  \qquad \forall \,r>0.
\]
Since  $u(x_\circ + \, \cdot\,, t_1) \le  u(x_\circ + \, \cdot\,, t_2)$ then it must be $p_{2,x_\circ, t_1} \equiv p_{2,x_\circ, t_2} =:  P$.
Also, after a change of coordinates, we can assume that $\{P=0\}  \subset  \{x_n=0\}$.

Take $r>0$ small enough, and set $v: = u(x_\circ + r\, \cdot\,, t_2)-  u(x_\circ + r\, \cdot\,, t_1) \ge 0$. Then
\[
\Delta v= 0 \quad \mbox{in } \{u(x_\circ + r\, \cdot\,, t_1)>0\} .
\]
Also,  as a consequence of Lemma~\ref{lem:EG2B1}, given $\ep>0$, for $r>0$ small enough we have
\[
\mathcal C_\ep : = \big\{y \, :\;  {\rm dist}\big({\textstyle \frac{y}{|y|}}, \{x_n=0\}\big) > \ep\big\} \subset  \{u(x_\circ + r\, \cdot\,, t_1)>0\}.
\]
Consider now the first eigenfunction of
\[
-\Delta_{\mathbb S^{n-1}}  \Psi = k_\ep \Psi \ \mbox{ in }\  {\mathbb S^{n-1}} \cap \mathcal C_{2\ep} , \qquad \Psi = 0 \  \mbox{ in }  {\mathbb S^{n-1}} \cap  \partial \mathcal C_{2\ep}.
\]
Then, setting $\psi(x) := |x|^{\lambda_\ep} \Psi(x/|x|)$ with $k_\ep=(n-2+\lambda_\ep)\lambda_\ep$, we have that $\psi$ is a positive $\lambda_\ep$-homogeneous harmonic function in $\mathcal C_{2\ep}$ which vanishes on the boundary.
Note that as $\ep\downarrow 0$ we have $\mathbb S^{n-1} \cap\partial  \mathcal C_{2\ep} \downarrow \{x_n=0\}$ and $\lambda_0=1$  (this corresponds to the solution $|x_n|$). Thus, by continuity, for $\ep>0$ small enough, the function $\hat\psi(x) := |x|^{3/2}\Psi(x/|x|)$ is subharmonic and vanishes on  $\partial \mathcal C_{2\ep}$. Hence using $\hat\psi$ as lower barrier and the standard Harnack inequality on $v$, we obtain that if $v>0$ somewhere then  $v\ge c \hat \psi(x)$ in $B_1$ for some $c>0$. This implies
\[
 u(x_\circ + r\, \cdot\,, t_2) \ge c \hat \psi(x),
\]
which is impossible since $u(x_\circ +\, \cdot\,  , t_2)  = P + o(|x|^2) = O(|x|^2)$, while
$\hat\psi$ is positive in some cone and  $3/2$-homogeneous.
This proves that $u(\cdot,t_1)\equiv u(\cdot ,t_2)$ inside $B_r(x_\circ)$, which implies the result.
\end{proof}

We now prove some relations between $p_2$ and singular points close to $(0,0)$.

\begin{lemma}\label{lem:EG2B1bisbis}
Let $u\in C^0\big(\overline{B_1}\times [-1,1]  \big)$ solve \eqref{eq:UELL+t},  let $(x_k,t_k) \in \mathbf{\Sigma}$, $(0,0)\in \mathbf{\Sigma}$,  and assume that $x_k\to 0$.
Set  $p_{2,k}:= p_{2,x_k,t_k}$.
Then $p_{2,k} \to p_2$ and we have
\[
\textstyle \left\|p_{2,k} -p_2\left(\frac{x_k}{|x_k|}+\,\cdot\,\right) \right\|_{L^\infty(B_1)} \le C \omega(2|x_k|) \quad \mbox{and}\quad \|p_{2,k} -p_2 \|_{L^\infty(B_1)} \le  C \omega(2|x_k|).
\]
In addition,
\[
\textstyle {\rm dist}\left ( \frac{x_k}{|x_k|} , \{p_2=0\}\right ) \to 0 \qquad \mbox{as $k\to \infty$}.
\]
\end{lemma}

\begin{proof}
We observe first that $p_{2, x_k,t_k} \to p_2$. 
Indeed, if $t_k\to t_\infty$ then (up to a subsequence) by Lemma~\ref{lem:EG2B1bis} we have $p_{2, x_k,t_k} \to p_{2, 0,t_\infty}$ and $p_{2, 0,t_\infty }\equiv p_2$, as desired.

Now, set $r_k:=|x_k|$. By Lemma~\ref{lem:EG2B1} we have
\[
\| r_k^{-2} u(x_k +r_k x, t_k)  - p_{2,k}(x) \|_{L^\infty(B_{2})}  \le 4\omega(2r_k)
\qquad
\text{and}
\qquad
\| r_k^{-2} u(r_k x, 0)  - p_{2}(x) \|_{L^\infty(B_{2})}  \le 4\omega(2r_k).
\]
Thus, defining $y_k:=x_k/|x_k|$, for all $x\in B_2$ we have the following: if $t_k \leq 0$ then
\[
  - 4 \omega(2r_k) + p_{2,k}(x) \le r_k^{-2}u(x_k +r_k x, t_k ) \le r_k^{-2} u(x_k +r_k x, 0)  \le    4 \omega(2r_k) + p_{2}(y_k +x),
\]
while if $t_k \ge 0$ then
\[
  4 \omega(2r_k) + p_{2,k}(x) \ge r_k^{-2}u(x_k +r_k x, t_k ) \ge r_k^{-2} u(x_k +r_k x, 0)  \ge    -4 \omega(2r_k) + p_{2}(y_k +x).
\]
In both cases, since  $p_{2,k}$ and $p_2$ are nonnegative $2$-homogeneous polynomials vanishing at $0$ and with Laplacian $1$, then $p_{2,k}-p_{2}(y_k +\,\cdot\,)$ is a harmonic quadratic polynomial which vanishes at some point of the segment joining $0$ and $y_k$, where $y_k:=x_k/|x_k|$.
Moreover, $|p_{2,k}-p_{2}(y_k +\,\cdot\,)|$ is bounded from above by $  8\omega(2r_k) $ in $B_2$. 
Using the mean value formula and  the fact that all norms are comparable on polynomials, we obtain
\[
\|p_{2,k} -p_2(y_k+\,\cdot\,) \|_{L^\infty(B_1)} \le C\|p_{2,k} -p_2(y_k+\,\cdot\,) \|_{L^2(\partial B_1)} \le C \omega(2r_k).
\]
By orthogonality of spherical harmonics with different homogeneities (or by a direct computation) we then obtain
\[
\|p_{2,k} -p_2 \|_{L^2(\partial B_1)}^2  + \|p_2 -p_2(y_k+\,\cdot\,) \|^2_{L^2(\partial B_1)}  =   \|p_{2,k} -p_2(y_k+\,\cdot\,) \|^2_{L^2(\partial B_1)}  \le C \omega(2r_k) ^2.
\]
In particular $\|p_2 -p_2(y_k+\,\cdot\,) \|_{L^2(\partial B_1)} \to 0$, and therefore ${\rm dist}(y_k, \{p_2=0\})\to 0$.
\end{proof}

We prove next two key lemmas that will allow us to perform some dimension reduction arguments  needed to control the spatial projection (i.e., $\pi_1: (x,t)\mapsto x$) of some ``bad'' subsets of $\mathbf{\Sigma}\subset B_1\times [-1,1]$. 
Note that the spatial version  of these first two lemmas (i.e., when considering $u(\cdot, t_\circ)$ with $t_\circ$ fixed) was proven in \cite{AlessioJoaquim}. 
Here we need stronger results valid for a one-parameter monotone family of  solutions to the obstacle problem. To our best knowledge, this is the first dimension reduction argument applicable to a one-parameter family of solutions to an elliptic equation, and it will involve several new and non-standard techniques.

We recall that, given $w: \R^n \rightarrow \R$, the rescaled functions $w_r$ and $\tilde w_r$ have been defined in \eqref{defwr_tildewr}.

The first lemma concerns the intermediate strata of the singular set $\mathbf{\Sigma}_m$ with $0\le m \le n-2$.

\begin{lemma}\label{lem:EG2Bm}
Let $u\in C^0\big(\overline{B_1}\times [-1,1]  \big)$ solve \eqref{eq:UELL+t}, let
$(0,0)\in \mathbf{\Sigma}_m$ with $0\le m \le n-2,$ and assume that $u(\,\cdot\,,0)\not\equiv p_2$. Let $(x_k, t_k)\in \mathbf{\Sigma}_m$  satisfy $|x_k|\le r_k$ with $r_k\downarrow 0$, and suppose that
\begin{equation}\label{extraassump}
\tilde  w_{r_{k}} \rightharpoonup  q \mbox{  in $W^{1,2}_{\rm loc}(\R^n)$} \qquad \mbox{for}\quad  w:= u-p_2  \quad \mbox{and} \quad   y_k :=\frac{x_k}{r_k} \to y_\infty.
\end{equation}
Then   $y_\infty\in \{p_2=0\}$   and $q(y_\infty)=0$.
\end{lemma}

\begin{proof}
Let  us define
\[
w_k  := u(x_k  +r_k\,  \cdot\, , t_k  ) -   p_{2}(r_k\,  \cdot\,  ) = w_k ^{(1)} +w_k ^{(2)}  +w_k ^{(3)},
\]
where
\[
\begin{split}
w_k ^{(1)}  &:= u(x_k +r_k\,  \cdot\,  , t_k  )  - u(x_k +r_k\,  \cdot\,  ,0 ),
\\
w_k ^{(2)}  &:=  u(x_k + r_k\,  \cdot\,  , 0 )  - p_{2}(x_k +r_k\,\cdot\,),
\\
w_k ^{(3)}  &:=   p_{2} (x_k +r_k \,\cdot\,) -p_{2} \big(r_k\,\cdot\,).
\end{split}
\]
We divide the proof into three steps.

\smallskip

\noindent $\bullet$ \textit{Step 1.}
We prove that
\[
\tilde w_k := \frac{w_k}{\| w_k\|_{L^2(\partial B_1)}} \rightharpoonup Q \quad \mbox{in } W^{1,2}_{\rm loc} (\R^n)
\]
for some harmonic function $Q$ with polynomial growth.

Indeed, since $u\in C^0(\overline{B_1} \times [-1,1])$, by the monotonicity of $\phi$  there exist $r_\circ>0$ and $k_\circ \in \N$ such that, for $M := \phi\big(0^+, u(\,\cdot\,, 0)-p_2\big)+1$, we have
\begin{equation}\label{freqbda}
\phi\left( r \,, u(x_k + \,\cdot\,, t_k ) -p_{2}\right) \le M  \qquad \forall \,r\in (0, r_\circ), \ \forall \,k \ge k_\circ,
\end{equation}
or equivalently
 \begin{equation}\label{freqbd11aaa}
\phi(r , w_k )  \le M \qquad \forall \,r\in (0, r_\circ/r_k), \ \forall\, k \ge k_\circ.
\end{equation}
Then, applying  Lemma~\ref{lem:EMF3} to  $w_k$,  we obtain the following polynomial growth control for $\tilde w_k$:
\begin{equation}\label{growthcona}
H\big(R , \tilde w_k \big) \le  CR^{2M+1} H\big(1, \tilde w_k\big) = CR^{2M+1}  \quad \forall \,R\in [1, r_\circ/r_k), \ \forall\, k \ge k_\circ.
\end{equation}
Note that \eqref{freqbd11aaa}  is equivalent to $\phi(r_k , \tilde w_k )\le M$, which combined with \eqref{growthcona} implies  that
\begin{equation}\label{bound!}
\| \tilde w_k\|_{W^{1,2}(B_R)} \le C(R).
\end{equation}
This gives compactness of the sequence $\tilde w_k$ and hence (up to a subsequence)
\[\tilde w_k  \rightharpoonup Q \quad \mbox{in $W^{1,2}_{\rm loc}(\R^n)$}\]
for some  $Q:\R^n \rightarrow \R$. Let us prove next that $Q$ is harmonic.

Indeed, on the one hand we have
\begin{equation}\label{eq:lapwk}
\Delta w_k = -r_k^2\chi_{ \{u(x_k +r_k \,\cdot\,, \,t_k  )=0\} } \le 0 \quad \mbox{in }  B_{\frac{1}{2r_k}}.
\end{equation}
On the other hand, Lemmas~\ref{lem:EG2B1} and~\ref{lem:EG2B1bisbis} imply that, for  $R\ge 1$, 
\[
x\in B_R \cap \{u(x_k +r_k x,t_k  )=0\} \quad \Rightarrow \quad p_{2,x_k,t_k} (x) \le R^2\omega(Rr_k )
\quad\Rightarrow  \quad p_{2} (x) \le CR^2\omega(Rr_k );
\]
thus, since $p_2$ grows quadratically away from its zero set,
\begin{equation}\label{eq:distcont}
B_R \cap \{u(x_k +r_k \,\cdot\, , t_k  )=0\} \subset \left\{ y \, : \, {\rm dist}\big(y, \{p_2=0\}\big) \le CR \sqrt{ \omega(Rr_k)} \right\}.
\end{equation}
Note that, for any fixed $R\ge 1$,  we have $CR \sqrt{ \omega(Rr_k)}\downarrow 0$ as $k\to \infty$. We have thus shown
\[\sup \big\{ {\rm dist}(x, \{p_2=0\}) \ :\ x\in B_R\cap\{u(x_k +r_k \,\cdot\, , t_k  )=0\}  \big\} \downarrow 0 \qquad \textrm{as}\quad k\to\infty.\]
Therefore, the weak limit of the sequence of nonpositive measures $\Delta \tilde w_k$ will be supported on $\{p_2=0\}$. 
But then, recalling \eqref{bound!}, we have shown that $Q$ is a locally $W^{1,2}$ function whose Laplacian is supported in linear space of dimension $m= {\rm dim} (\{p_2=0\} )\le n-2$ and thus of zero harmonic capacity. It follows\footnote{The proof of this implication is standard. 
We want to prove that
$\int \nabla Q\cdot\nabla \xi =0$ for all $\xi\in C^1_c(\R^n)$. But since $\{p_2=0\}$ has zero harmonic capacity, any given $\xi$ can be approximated in $W^{1,2}$ norm by functions $\xi_k$ which vanish on $\{p_2=0\}$, and thus for which $\int \nabla Q \cdot\nabla \xi_k = - \int \Delta Q \xi_k = 0$. 
The desired conclusion follows by taking the limit as $k\to \infty$.} that $Q$ must be harmonic.

Moreover, since $x_k$ is a singular point,  Lemma~\ref{lem:EMF3} yields
$$
H(\rho , w_k) \le \rho ^4   H(1 , w_k) \quad \mbox{for all } \rho \in (0,1),
$$
and thus in the limit we find
\begin{equation}\label{quadraticsep}
H(\rho , Q)^{1/2} \le \rho ^2   \quad \mbox{for all } \rho \in (0,1).
\end{equation}

\smallskip

\noindent $\bullet$ \textit{Step 2.} We now want to prove that
\begin{equation}\label{abcd}
\frac{w_k^{(2)}}{ \| w_k^{(2)} \|_{L^2(\partial B_1)}} \rightharpoonup \frac{q (y_\infty +\,\cdot\,)}{ \|q (y_\infty +\,\cdot\,) \|_{L^2(\partial B_1)}} \quad \mbox{in } W^{1,2}_{\rm loc} (\R^n)
\end{equation}
(with $q$ defined in \eqref{extraassump}), and
\begin{equation}\label{dcba}
 \frac{ w_k^{(3)} }{  \| w_k^{(3)} \|_{L^2(\partial B_1)}}    \rightarrow  \nabla p_2 \cdot \boldsymbol e  \quad \mbox{in } W^{1,2}_{\rm loc} (\R^n)
\end{equation}
for some (nonzero) $\boldsymbol e\in \{p_2=0\}^\perp$.

Note that, since  $y_\infty\in \{p_2=0\}$ (by Lemma~\ref{lem:EG2B1bisbis}),
\[
\frac{w_k^{(2)}}{ \| w_k^{(2)} \|_{L^2(\partial B_1)}}  = \frac{w_{r_k}(y_k + \cdot\,)}{\| w_{r_k}(y_k + \cdot\,)\|_{L^2(\partial B_1)}} =  \tilde w_{r_k}(y_k + \cdot\,)  \frac{\| \tilde w_{r_k} \|_{L^2(\partial B_1)} }{\| \tilde w_{r_k}(y_k + \cdot\,)\|_{L^2(\partial B_1)}}.
\]
Thus, noticing that $\|\tilde w_{r_k} \|_{L^2(\partial B_1)} \to\| q \|_{L^2(\partial B_1)}$ and $\| \tilde w_{r_k}(y_k + \cdot\,)\|_{L^2(\partial B_1)}\to \|q(y_\infty + \cdot\,)\|_{L^2(\partial B_1)}$, since $q$ is a nonzero quadratic harmonic polynomial (see Proposition~\ref{prop:E2B2}) both limits are nonzero and universally bounded. Thus \eqref{abcd} follows.

To prove \eqref{dcba}, we set $\ep_k := \| p_2(y_k +\,\cdot\,) -p_2\|_{L^2(\partial B_1)}\to 0$.
Then, if $y_k^*$ denotes the projection of $y_k$ onto $\{p_2=0\}$, we have  $p_2(y^*_k + \,\cdot\,)\equiv p_2$ and $y_k^*-y_k\to y_\infty^* -y_\infty =0$. Thus, up to taking a further subsequence, we obtain
\[
 \lim_k \frac{ w_k^{(3)} }{  \| w_k^{(3)} \|_{L^2(\partial B_1)}}   =  \lim_k \frac{ p_2(y_k +\,\cdot\,) -p_2} { \ep_k } =  \lim_k \frac{ p_2(y_k-y_k^* +\,\cdot\,) -p_2 }{ \ep_k }=  c\nabla p_2 \cdot \lim_k \frac{y_k-y_k^*}{|y_k-y_k^*|}  = \nabla p_2 \cdot \boldsymbol e
\]
for some nonzero $\boldsymbol e \in \{p_2=0\}^\perp$. Note that the limit in $k$ exists (up to subsequence) and is nonzero, since $\frac{ w_k^{(3)} }{  \| w_k^{{3}} \|_{L^2(\partial B_1)}}$ is a sequence of linear functions with  unit $L^2$ norm.

\smallskip

\noindent $\bullet$ \textit{Step 3}. We finally prove that $q(y_\infty)=0$.

Let us consider
\[
\hat \ep_k := \sum_{i=1,2,3}\| w_k^{(i)}\|_{L^2(\partial B_1)}  \qquad \mbox{and } \qquad  \hat w_k : = \frac{w_k}{\hat \ep_k}\,.
\]
By Step 1 we have
\begin{equation}\label{bjfiobi}
\hat w_k \to  \hat Q = a Q\quad \mbox{ for some }a\in [0,1].
\end{equation}
Moreover, by Step 2, 
\[
\hat Q^{(2)} := \lim _k w_k^{(2)}/\hat \ep_k = b q(y_\infty+\,\cdot\,),
\qquad
\hat Q^{(3)} := \lim _k w_k^{(3)}/\hat \ep_k  = c\nabla p_2 \cdot \boldsymbol e,
\]
for some constant $b,c\ge0$. (Above, the convergences are weak in $W^{1,2}_{\rm loc}(\R^n)$.)

Then, it is well defined
\[\hat Q^{(1)} := \lim_k w_k^{(1)}/\hat \ep_k =  \lim_k w_k/\hat \ep_k -\lim _k w_k^{(2)}/\hat \ep_k- \lim _k w_k^{(3)}/\hat \ep_k,\]
and we observe that $Q^{(1)}$ is either nonpositive or nonnegative (since $w_k^{(1)}= u(x_k +r_k\,  \cdot\,  , t_k  )  - u(x_k +r_k\,  \cdot\,  ,0 ) $ is so, depending on the sign of $t_k$). Moreover, since $\hat Q$, $\hat Q^{(2)}$, and
$\hat Q^{(3)}$ are harmonic, so is $\hat Q^{(1)}$ and thus it must be constant by Liouville Theorem.
Hence, we have
\[
\hat Q = C + b q(y_\infty+\,\cdot\,) +c\nabla p_2 \cdot \boldsymbol e.
\]
Note now that, by definition of $\hat \ep_k$, we have  $\sum_{i=1,2,3}  \|\hat Q^{(i)}\|_{L^2(\partial B_1)}=1$. 
Moreover, since the homogeneity of $q$ at the origin is at least two,  the three functions  $\hat Q^{(i)}$ are linearly independent and hence their sum $\hat Q$ cannot be zero (equivalently, in \eqref{bjfiobi} it must be $a>0$). 
Note also that it must be $b>0$ since \eqref{quadraticsep} implies that $\hat Q$ is at least quadratic and hence it can not be equal to the constant $\hat Q^{(1)}$ plus the linear function, $\hat Q^{(3)}$.
Finally, \eqref{quadraticsep} implies $\nabla Q(0) =0$.

But then, since $y_\infty \in \{p_2=0\}$, and $q$ is a homogeneous polynomial of degree $\phi(q,1)$,
\[
0 = y_\infty\cdot \nabla \hat Q (0) = y_\infty \cdot \nabla \hat Q^{(1)} (0) +  b  y_\infty \cdot \nabla  q(y_\infty) + cy_\infty\cdot \nabla ( \nabla p_2 \cdot \boldsymbol e) (0)
=  0 +b\phi(q,1)\, q(y_\infty) + 0,
\]
which proves that $q(y_\infty) =0$.
\end{proof}

The next lemma concerns the maximal stratum $\mathbf{\Sigma}_{n-1}$.
This case is more involved, since blow-ups are not necessarily harmonic functions as in the previous lemma.
In particular, in this situation we will need to assume that the frequency is continuous along the sequence that we consider.

\begin{lemma}\label{lem:EG2Bn-1}
Let $u\in C^0\big(\overline{B_1}\times [-1,1]  \big)$  solve \eqref{eq:UELL+t}, let
$(0,0)\in \mathbf{\Sigma}_{n-1},$ and assume that $u(\,\cdot\,,0)\not\equiv p_2$.  Let $(x_k, t_k) \in \mathbf{\Sigma}_{n-1}$ satisfy $|x_k|\le r_k$ with $r_k\downarrow 0$, assume that \eqref{extraassump} holds, and  that $\lambda^{2nd}_k\to \lambda^{2nd}$, where
\[\lambda^{2nd}_{k} := \phi\big(0^+,  u(x_k +\,\cdot\,, t_k  )-p_{2,x_k,t_k}\big)\quad \mbox{and}\quad  \lambda^{2nd} := \phi(0^+,  u-p_2).\]
Then $y_\infty\in \{p_2=0\}$  and $q^{even}$ is translation invariant in the direction $y_\infty$. (Here $q^{even}$ denotes the even symmetrisation of $q$ with respect to the hyperplane $\{p_2=0\}$.)
\end{lemma}

\begin{proof}
Let  us define
\[
w_k  := u(x_k  +r_k\,  \cdot\, , t_k  ) -   p_{2,x_k, t_k} (r_k\,  \cdot\,  ) = w_k ^{(1)} +w_k ^{(2)}  +w_k ^{(3)},
\]
where
\[
\begin{split}
w_k ^{(1)}  &:= u(x_k +r_k\,  \cdot\,  , t_k  )  - u(x_k +r_k\,  \cdot\,  ,0 ),
\\
w_k ^{(2)}  &:=  u(x_k + r_k\,  \cdot\,  , 0 )  - p_{2}(x_k +r_k\,\cdot\,),
\\
w_k ^{(3)}  &:=   p_{2} (x_k +r_k \,\cdot\,) -p_{2,x_k, t_k} \big(r_k\,\cdot\,).
\end{split}
\]
We divide the proof into three steps.

\smallskip

\noindent $\bullet$ \textit{Step 1.}
Exactly as in  Lemma~\ref{lem:EG2Bm},
\[
\tilde w_k := \frac{w_k}{\| w_k\|_{L^2(\partial B_1)}} \rightharpoonup Q \quad \mbox{in } W^{1,2}_{\rm loc} (\R^n)
\]
for some $Q\in W^{1.2}_{\rm loc} (\R^n) $  with polynomial growth. We claim that $Q$ is a $\lambda^{2nd}$-homogeneous solution of the Signorini problem \eqref{ETOP}.

Indeed, by the upper-semicontinuity property in  Lemma~\ref{lem:EG2B1bis}(b) and the assumption  $\lambda^{2nd}_k \to \lambda^{2nd}$, given $\delta >0$ there exist $r_\delta>0$ and  $k_\delta$ such that
\begin{equation}\label{freqbdb}
\phi\left( r , u(x_k + \,\cdot\,, t_k ) -p_{2, x_k, t_k}\right) \in (\lambda^{2nd}-\delta, \lambda^{2nd}+\delta)  \qquad \forall \,r\in (0, r_\delta), \ \forall\, k \ge k_\delta,
\end{equation}
or equivalently
 \begin{equation}\label{freqbd11bbb}
\phi(r, w_k )  \in (\lambda^{2nd}-\delta, \lambda^{2nd}+\delta)  \qquad \forall \,r\in (0, r_\delta/r_k), \ \forall\, k \ge k_\delta.
\end{equation}
Then, applying  Lemma~\ref{lem:EMF3} to  $w_k$  we obtain the following polynomial growth control for $\tilde w_k$:
\begin{equation}\label{growthconyy}
H\big(R , \tilde w_k \big) \le   C_\delta R^{2\lambda^{2nd} + 3\delta}  \qquad \forall \,R\in [1, r_\delta/r_k), \ \forall\, k \ge k_\circ,
\end{equation}
and the decay estimate
\begin{equation}\label{sepsep}
H\big(\varrho , \tilde w_k\big) \le  C\varrho^{2(\lambda^{2nd} - \delta)}  \qquad \forall \,\varrho\in (0, 1], \ \forall\, k \ge k_\circ.
\end{equation}
In addition, the Lipschitz estimate in Lemma~\ref{lem:E2B1} gives
\[
\|\tilde w_k\|_{{\rm Lip}(B_R)} \le C(R).
\]
Hence $\tilde w_k  \rightarrow Q$ in $C^0_{\rm loc}(\R^n)$ (up to a further subsequence).

Note that, using \eqref{eq:lapwk} and  \eqref{eq:distcont} in our context, one deduces that $\Delta Q$ is a nonpositive measure supported on $\{p_2=0\}$.
Moreover, since $w_k(y_k +\,\cdot\,) = u(x_k + r_k\, \cdot\, ) -p_{2,x_k t_k}(r_k\,\cdot\, )$, it follows that
 $\tilde w_k (y_k +\,\cdot\,) \ge 0$ on $\{p_{2,x_k t_k}=0\}$ and  thus, by uniform convergence,  $Q\ge 0$  on $\{p_2=0\}$.
 
 On the other hand \eqref{eq:lapwk} and the fact that $\tilde w_k (y_k +\,\cdot\,) \le 0$ on  $\{u(x_k +r_k \,\cdot\,, t_k  )=0\}$ imply that  $\tilde w_k\Delta \tilde w_k \ge 0$,
 and since $\Delta \tilde w_k  \rightharpoonup \Delta Q$ weakly as measures and $\tilde w_k  \to Q$  in $C^0$, we obtain
 $Q\Delta Q \ge 0$ in $\R^n$.
But since $\Delta Q\leq $ is nonpositive  and supported on $\{p_2=0\}$ where $Q\ge 0$,   it must be  $Q\Delta Q\le 0$.
This implies that $Q$ is a solution of the Signorini problem \eqref{ETOP}.

Finally, taking the limit in \eqref{growthconyy} and \eqref{sepsep} we obtain that, for any given $\delta>0$,
\begin{equation}\label{growthconzz}
H\big(R , Q \big) \le   C_\delta R^{2\lambda^{2nd} + 3\delta}  \quad \forall \,R\in [1, \infty)
\end{equation}
and
\begin{equation}\label{sepsepzz}
H\big(\varrho , Q \big) \le  C\varrho^{2(\lambda^{2nd} - \delta)}  \quad \forall\,\varrho\in (0, 1].
\end{equation}
Since $\delta>0$ is arbitrary and 
 $Q$ is a global solution of Signorini, it follows by  Lemma~\ref{lemAP:Alm} that
\[
\lambda^{2nd}\leq \phi(0^+,Q)\leq \phi(+\infty,Q) \leq \lambda^{2nd}.
\]
Hence $\phi(r,Q)=\lambda^{2nd}$ for all $r>0$, from which (using Lemma~\ref{lemAP:Alm} again) it follows that $Q$ is a $\lambda^{2nd}$-homogeneous.

\smallskip

\noindent $\bullet$ \textit{Step 2.} We now want to prove that
\begin{equation}\label{abcd2}
\frac{w_k^{(2)}}{ \| w_k^{(2)} \|_{L^2(\partial B_1)}} \rightharpoonup \frac{q (y_\infty +\,\cdot\,)}{ \|q (y_\infty +\,\cdot\,) \|_{L^2(\partial B_1)}} \quad \mbox{in } W^{1,2}_{\rm loc} (\R^n)
\end{equation}
and
\begin{equation}\label{dcba2}
 \lim_k \frac{ w_k^{(3)} }{  \| w_k^{{3}} \|_{L^2(\partial B_1)}}  \rightarrow  ( \boldsymbol e\cdot x)  + ( \boldsymbol e'\cdot x)(\boldsymbol e\cdot x) \not\equiv 0 \quad \mbox{in } W^{1,2}_{\rm loc} (\R^n).
\end{equation}
for some $\boldsymbol e \in \{p_2=0\}^\perp$ and  $\boldsymbol e' \in \{p_2=0\}$.

Indeed, the proof of \eqref{abcd2} is identical to the one of \eqref{abcd} in the proof of Lemma~\ref{lem:EG2Bm}.

To show \eqref{dcba2},  denote $\ep_k := \| p_2(y_k +\,\cdot\,) -p_{2,x_k, t_k} \|\to 0$.  
Recall that (by Lemma~\ref{lem:EG2B1bisbis}) we have $y_\infty\in \{p_2=0\}$ and hence, if $y_k^*$ denotes the projection of $y_k$ onto $\{p_2=0\},$ then   $p_2(y^*_k + \,\cdot\,)\equiv p_2$ and $y_k^*-y_k\to y_\infty^* -y_\infty =0$. 
Thus,  up to taking a further subsequence, if $\{p_2=0\}=\{ \hat {\boldsymbol e}\cdot x=0\}$ and $\{p_{2,x_k,t_k}=0\}=\{ \hat{\boldsymbol  e}_k\cdot x=0\}$ with $\hat {\boldsymbol e}, \hat {\boldsymbol e}_k\in \mathbb S^{n-1}$, then
\[
\begin{split}
 \lim_k \frac{ w_k^{(3)} }{  \| w_k^{(3)} \|_{L^2(\partial B_1)}}   &=  \lim_k \frac{ p_2(y_k +\,\cdot\,) -p_{2, x_k, t_k}} { \ep_k }
 =  \lim_k \frac{ p_2(y_k-y_k^* +\,\cdot\,) -p_2 } { \ep_k } +  \lim_k \frac{p_2 -p_{2, x_k, t_k}} { \ep_k }
\\
&=  c_1\nabla p_2 \cdot \lim_k \frac{y_k-y_k^*}{|y_k-y_k^*|} + c_2\lim_k \frac{ (\hat {\boldsymbol e}\cdot x)^2 - (\hat {\boldsymbol e}_k \cdot x)^2 }{2|\hat {\boldsymbol  e}- \hat {\boldsymbol e}_k|}
\\
&= ( \boldsymbol e\cdot x)  +  ( \boldsymbol e'\cdot x)(\boldsymbol e\cdot x),
\end{split}
\]
where $\boldsymbol e \in \{p_2=0\}^\perp$ and  $\boldsymbol e' \in \{p_2=0\}$. Note that the previous limit in $k$ must exist (up to subsequence) and will be nonzero, since $w_k^{(3)}/ \| w_k^{(3)} \|_{L^2(\partial B_1)}$ is a sequence of quadratic polynomials  with  unit $L^2$ norm.

\smallskip

\noindent $\bullet$ \textit{Step 3}. We finally prove that $q$ is translation invariant in the direction $y_\infty$. 
Consider
\[
\hat \ep_k := \sum_{i=1,2,3}\| w_k^{(i)}\|_{L^2(\partial B_1)}  \qquad \mbox{and } \qquad  \hat w_k : = \frac{w_k}{\hat \ep_k}\,.
\]
By Step 1 we have
\[
\hat w_k \to  \hat Q = a Q\quad \mbox{ for some }a\in [0,1].
\]
Moreover, by Step 2
\[
\hat Q^{(2)} := \lim _k w_k^{(2)}/\hat \ep_k = b q(y_\infty+\,\cdot\,)
\]
and, after choosing some appropriate coordinate frame (so that, in particular, $\{p_2=0\}=\{x_n=0\}$),
\[
\hat Q^{(3)} := \lim _k w_k^{(3)}/\hat \ep_k  =c_1 x_n + c_2x_n x_{n-1}\]
for some $b,c\ge0$. (Above, the convergences are weak in $W^{1,2}_{\rm loc}(\R^n)$.)

Then, it is  well defined
\[\hat Q^{(1)} := \lim_k w_k^{(1)}/\hat \ep_k =  \lim_k w_k/\hat \ep_k -\lim _k w_k^{(2)}/\hat \ep_k- \lim _k w_k^{(3)}/\hat \ep_k,\]
and we observe that $Q^{(1)}$ is either nonpositive or nonnegative (since the functions $w_k^{(1)}$ are so).
Hence, we have
\[
\hat Q = \hat Q^{(1)} + b q(y_\infty+\,\cdot\,) + c_1 x_n + c_2x_n x_{n-1}.
\]
Note now that, by definition of $\hat \ep_k$, we have  $\sum_{i=1,2,3}  \|\hat Q^{(i)}\|_{L^2(\partial B_1)}=1$. Moreover, since the homogeneity of $q$ at the origin is at least $2+\alpha_\circ$ (see Proposition~\ref{prop:E2B2}),  the three functions  $\hat Q^{(i)}$ are  linearly independent\footnote{Note again that $\hat Q^{(1)}$ has a sign, $\hat Q^{(2)}$ is (the translation of) a $\lambda^{2nd}$-homogeneous solution of Signorini with $\lambda^{2nd}>2$, and $\hat Q^{(3)}$ is a odd quadratic harmonic polynomial, and thus they are linearly independent.} and thus their sum $\hat Q$  cannot be zero.

Let us show next that $b>0$ and that $\hat Q \equiv  bq$. Indeed, since both $q$ and $\hat Q$ are $\lambda^{2nd}$-homogeneous with $\lambda^{2nd}\ge 2+\alpha_\circ$, if $Q^{(1)}\ge 0$ (resp. $\le$) then
\[
\hat Q = \lim_{R \to \infty} \frac{ \hat Q(R \,\cdot\,) }{ R^{\lambda^{2nd}} }  = \lim_{R \to \infty} \frac{Q^{(1)}(R\,\cdot\,) + b q(y_\infty +R\,\cdot\,) + Q^{(3)}(R\,\cdot\,)}{R^{\lambda^{2nd} }} \ge b q  \quad \mbox{(resp. $\le$)},
\]
where we used that $Q^{(3)}$ is 2-homogeneous.
Hence, $\hat Q$ and $bq$ are two ordered solutions  of Signorini  with homogeneities greater than 1 at the origin and thus they must be equal by Lemma~\ref{lemAP:ordered}.

Therefore, we have shown that
\begin{equation}\label{idqSodd}
\hat Q =  \hat Q^{(1)} +  b q(y_\infty+\,\cdot\,) + x_n (c_1 x_{n-1}+ c_2 )= b q.
\end{equation}
In particular, since $\hat Q$ has unit $L^2(\partial B_1)$ norm this implies that $b>0$.

Now, taking the even parts, if $Q^{(1)}\ge 0$ (resp. if  $Q^{(1)}\le 0$) we obtain
\begin{equation}\label{idqSodd2}
b q^{even}(y_\infty+\,\cdot\,)  \le b q^{even} \qquad (\mbox{resp. }\ge ).
\end{equation}
Hence it follows by homogeneity that,  for all $s>0$,
\[
b s^{-\lambda^{2nd}}q^{even}(sy_\infty+x) \le b s^{-\lambda^{2nd}} q^{even}(x) \qquad (\mbox{resp. }\ge ).
\]
Therefore, since $b>0$,
\[
q(sy_\infty+x)   \le q(x)\qquad (\mbox{resp. }\ge),
\]
and thus
 \[
y_\infty\cdot  \nabla q^{even} \le 0  \qquad (\mbox{resp. }\ge ).
\]
In summary we obtain  that $\psi:= y_\infty\cdot  \nabla q^{even}$ has constant sign. But then $\psi$ restricted to the sphere  $\mathbb S^{n-1}$ must be a multiple of the first even eigenfunction (since all other eigenfunctions change sign) of
\[
\begin{cases}
-\Delta_{\mathbb S^{n-1}} \psi= k \psi \ &\mbox{in } \mathbb S^{n-1}\setminus \mathcal Z \\
 \psi = 0 &\mbox{on } \mathbb S^{n-1}\cap \mathcal Z,
\end{cases}
\]
where $\mathcal Z : = \{x_n =0\}\cap \{q=0\}$ and   $k := (n-2+\lambda^{2nd}) \lambda^{2nd}$.
Note $\mathcal Z \subset \{x_n=0\}$, and the two extremal cases $\mathcal Z = \varnothing$ and $\mathcal Z = \{x_n=0\}$ correspond respectively to the eigenfunctions $1$ and $|x_n|$ (restricted to the sphere), which have homogeneity $0$ and $1$ respectively. As a consequence of the monotonicity property of the eigenvalues with respect to the domain, for every $\mathcal Z$  we will have  $(n-2+0) 0  \le  k = (n-2+\lambda^{2nd}) \lambda^{2nd} \le  (n-2+1)1$. This leads to $\lambda^{2nd} \le 2$; a contradiction. Therefore, the only possibility is that $\psi = y_\infty\cdot  \nabla q^{even}\equiv 0$. In other words  $q^{even}$ is translation invariant in the direction $y_\infty$.
\end{proof}

The next result will imply that the projection $\pi_1\big(\mathbf{\Sigma}_{n-1}^{\ge3}\setminus \mathbf{\Sigma}^{3rd}_{n-1}\big)$ (recall that $\pi_1(x,t)=x$) is contained in a countable union of $(n-2)$-dimensional Lipschitz manifolds, i.e., it is $(n-2)$-rectifiable.
This will be crucial in our proof of Theorem~\ref{thm-Schaeffer-intro}.

\begin{lemma}\label{dimred}
Let $u\in C^0\big(\overline{B_1}\times [-1,1]  \big)$ solve \eqref{eq:UELL+t}, and let
$(0,0)\in \mathbf{\Sigma}_{n-1}^{\ge3}\setminus \mathbf{\Sigma}^{3rd}_{n-1}$. 
Then there exists a $(n-2)$-dimensional linear subspace $L$ such that the following holds: for any $\ep>0$  there exists $\varrho_\ep>0$ such that
\[
\pi_1\big(\mathbf{\Sigma}_{n-1}^{\ge 3}\big)\cap B_r \subset  L +   B_{\ep r} \qquad \mbox{for all }r \in (0,\varrho_\ep),
\]
where $L+B_{\ep r}:=\{z=(x+y)\,:\,x\in L,\,y \in B_{\ep r}\}$ denotes the sum of sets.
\end{lemma}

\begin{proof}
Let $w : = u(\,\cdot\,, 0)-p_2$, and recall that $w_r(x) = w(rx)$ and $\tilde w_r= w_r/\|w_r\|_{L^2(\partial B_1)}$. 
Recall also that, by Proposition~\ref{prop-uniq-cubic-blowups}, the following limit exists
\[
\tilde q := \lim_{r\downarrow 0} r^{-3} w(r \,\cdot\,),
\]
and (after choosing suitable coordinate system) the even part of $\tilde q$ is of the form
\begin{equation}\label{ahoiahiohai}
\tilde q^{even}(x) =  b |x_n|^3 -3 |x_n| \left( \sum_{\alpha=1}^{n-1} b_\alpha x_\alpha^2\right),
\end{equation}
where $b>0$, $b_\alpha\ge 0$, and $b = \sum_{\alpha=1}^{n-1} b_\alpha$; see Lemma~\ref{lem:signorinieven3}. 
Relabelling if necessary the indices,  we may assume that $b_1\le b_2 \le\cdots\le b_{n-1}$. In particular we must have $b_{n-1}>0$. 

Define $L$ to be the $(n-2)$-dimensional subspace $\{x_n = x_{n-1} =0\}$ in this system of coordinates.
We claim that, for any sequence $(x_k, t_k) \in \mathbf{\Sigma}^{\ge 3}_{n-1}$ such that $x_k \to 0$, we have
\[
{\rm dist} \bigg(\frac{x_k}{|x_k|} , L\bigg) \to 0.
\]
Note that the lemma follows immediately from this claim. 
To prove the claim we observe that
\[
\lambda^{2nd}_{k} := \phi\big(0^+,  u(x_k +\,\cdot\,, t_k  )-p_{2,x_k,t_k}\big)\ge 3 \qquad \mbox{and}\qquad  \lambda^{2nd} := \phi(0^+,  u-p_2) =3.
\]
Thus, since the frequency is upper-semicontinuous, $\lambda^{2nd}_{k}\to \lambda^{2nd}=3$.
This allows us to apply  Lemma~\ref{lem:EG2Bn-1} with $r_k := |x_k|$ and deduce that, if $y_\infty$ is an accumulation point of $\big\{x_k /|x_k|\big\}$, then the even part of
$
q = \frac{\tilde q}{ \|\tilde q\|_{L^2(\partial B_1)} }
$
is translation invariant in the direction $y_\infty$. Thus $\tilde q^{even}$ has the same invariance. But then, recalling 
\eqref{ahoiahiohai} and $b_{n-1}>0$,  we find that  $y_\infty \in \{x_n=x_{n-1}=0\}=L$.
\end{proof}

We next need the following Lipschitz estimate.

\begin{lemma}\label{lem:lipestF3}
Let $u:B_1 \to [0,\infty)$ solve  \eqref{eq:UELL}, and let $0\in {\Sigma}^{3rd}_{n-1}\setminus {\Sigma}^{\ge 4}_{n-1}$. 
Set $w:= u - p_2- P$, where $P$ is a $3$-homogeneous harmonic polynomial vanishing on $\{p_2=0\}$, and let $w_r$ and $\tilde w_r$ be as in \eqref{defwr_tildewr}.
Assume that, for some $r_\circ >0$, $\gamma\in (3,4)$, $\delta_\circ>0$, and $h_\circ>0$,  we have
\begin{equation}\label{boundphi11}
\phi^{\gamma}(r, u-p_2-P) \le \gamma-\delta_\circ \quad \forall \,r\in (0, r_\circ) \qquad \mbox{and}  \qquad H(r_\circ, u-p_2-P) \ge h_\circ.
\end{equation}
Then there exist positive constants $\varrho_\circ$, $\eta_\circ$, and $C$,  depending only on $n$, $\gamma$, $\delta_\circ$, $r_\circ$, $h_\circ$, and $\|P\|_{L^2(B_1)}$,  such that for any given $R\ge 1$ and for all $r\in \big(0 , \textstyle \frac{\varrho_\circ}{10R}\big)$ we have
\begin{equation}
\label{eq:wk Lip}
\|\tilde w_r\|_{{\rm Lip} (B_R)}  \le CR^3   \qquad \mbox{and} \qquad  \tilde w_r \Delta \tilde w_r \ge -Cr^{\eta_\circ} R^4  \Delta \tilde w_r\quad \mbox{in } B_{R}.
 \end{equation}
\end{lemma}

\begin{proof}
With no loss of generality we can assume that $\{p_2=0\}=\{x_n=0\}$.

Since $P$ is some $3$-homogeneous harmonic polynomial vanishing on $\{p_2=0\}$,
for any unit vector $\boldsymbol e$ tangential to $\{p_2=0\}$ we have $|\partial_{\boldsymbol e\boldsymbol e} P | \le C|x_n| \le  Cr^2$ in $B_r\cap \{u=0\}$ (cf. \eqref{eq:u0}).
Thus, arguing as in the proof of Lemma~\ref{lem:E3B5} (see Step 3), we get
\begin{equation}\label{sahfgij}
\inf_{B_r} r^2\partial_{\boldsymbol e\boldsymbol e} w  \ge - C(P) (\|w(r\,\cdot\,)\|_{L^2(B_{5})}  + r^4).
\end{equation}
Also, since $0\in {\Sigma}^{3rd}_{n-1}\setminus {\Sigma}^{\ge4}_{n-1}$, we can apply Lemmas~\ref{lem:E3B1c} and~\ref{lem:E4B1} to deduce that $\phi(0^+, u -\anz)$ exists and is less that $4$ (cf. proof of Proposition~\ref{prop:E3B5}(a)).

We now note that, as a consequence of \eqref{boundphi11}, 
Lemmas~\ref{lem:E3B1} and~\ref{lem:E3B1c} yield that, for any $\delta>0$, $r>0$, and $\varrho\in (r, r_\circ]$,
\[
\frac{ H(\varrho, w) + \rho^{2\gamma}}{H(r, w) + r^{2\gamma} } \le C_\delta \left( \varrho/r\right)^{2(\gamma-\delta) +\delta}.
\]
In particular, for $\delta = 4-\gamma$ and $\varrho = r_\circ$ we obtain
 \begin{equation}
 \label{lower}
H(1, w_r) =  H(r, w)  =   \frac{H(r_\circ, w) + r_\circ^{2\gamma} }{C_\delta} \left(r/ r_\circ\right)^{2(\gamma-\delta) +\delta} -    r^{2\gamma} \ge c _1 r^{2\gamma-\delta},
 \end{equation}
 provided that $r\in (0, r_1)$, where $c_1>0$ and $r_1 \in (0,r_\circ)$ is sufficiently small.
Also, for $r\in (0, r_1)$ and  $\varrho = Rr \le r_\circ$ we get
 \[
 H(Rr, w)  \le C_\delta R^{2\gamma-\delta} (H(r,w) + r^{2\gamma}) \le  C R^8 H(r,w),
 \]
 where $C= C_\delta(1+1/c_1)$ depends only on $n$, $\gamma$, $\delta$, and $h_\circ$
Thus, scaling \eqref{sahfgij}, for $r\in \big(0 , \textstyle \frac{r_1}{10R}\big)$ we obtain
\[
(2R)^{-2}\inf_{B_{2R}}  \partial_{\boldsymbol e\boldsymbol e} w_r  \ge - C(P) \big(\|w(3Rr \,\cdot\,)\|_{L^2(B_{5})}  + (2Rr)^4)
\ge - CR^4 (H(r,w)^{1/2}  + r^4\big) \ge -C H(1,w_r)^{1/2},
\]
where $C$ depends only on $n$, $R$, $\gamma$, $\delta$ and $h_\circ$.

Hence,  given  $R\ge 1$,  for all $r\in \big(0 , \textstyle \frac{r_\circ}{10R}\big)$ we have $\partial_{\boldsymbol e\boldsymbol e} \tilde w_r \ge -C R^2 $ in $B_{2R}$.
Therefore, as in the proof of Lemma~\ref{lem:E2B1}, we obtain  $|\nabla \tilde w_r| \le CR^3$ in $B_R$, where $C$ depends only on $n$, $\gamma$, $\delta$, and $h_\circ$.
This proves the first  part of \eqref{eq:wk Lip}.

For the second part, notice that $|u-p_2-P|=|\frac12(x_n)^2+P|\leq C|x|^4$ inside $\{u=0\}$ ---here we used that $|x_n|\leq C|x|^2$ in $\{u=0\}$ and that $P$ is a cubic polynomial divisible by $x_n$.
Combining this bound with \eqref{lower} and the fact that $\Delta \tilde w_r=0$ inside $\{u_r>0\}$, we find (choosing for instance $\eta_\circ:=4-\gamma)$
\[\tilde w_r \Delta \tilde w_r \geq -\frac{C|rx|^4}{H(1,w_r)^{1/2}}\,\Delta \tilde w_r \geq -\frac{C(rR)^4}{cr^{4-\eta_\circ}}\,\Delta \tilde w_r=-Cr^{\eta_\circ} R^4\,\Delta \tilde w_r\qquad \textrm{in}\, B_r,\]
which proves  \eqref{eq:wk Lip}.
\end{proof}

The following result will be needed in order to bound the Hausdorff dimension of the projection $\pi_1(\mathbf{\Sigma}_{n-1}^{>3}\setminus \mathbf{\Sigma}^{\ge 4}_{n-1})$.
Although the argument is very similar to the one used in the proof of Lemma~\ref{lem:EG2Bn-1}, we repeat the proof in detail since there are differences that require a detailed analysis. Recall that $p_3=p_{3,0,0}$ is defined in \eqref{eq:def p3 t}.

\begin{lemma}\label{lem:E3B2d}
Let $u\in C^0\big(\overline{B_1}\times [-1,1]  \big)$ solve \eqref{eq:UELL+t}, let $(0,0)$ and $(x_k,t_k)$ belong to  $ \mathbf{\Sigma}_{n-1}^{>3}\setminus \mathbf{\Sigma}^{\ge 4}_{n-1}$,
and suppose that  $|x_k|\le r_k\downarrow 0$.
Assume in addition that
\begin{equation} \label{ashoasdh}
\tilde  w_{r_{k}} \rightharpoonup  q \mbox{  in } W^{1,2}_{\rm loc}(\R^n) \qquad \mbox{for } w :=  u-p_2-p_3\quad \mbox{and} \quad y_k :=\frac{x_k}{r_k} \to y_\infty,
\end{equation}
 and that   $\lambda^{3rd}_k\to \lambda^{3rd}$, where
\[\lambda^{3rd}_{k} := \phi\big(0^+,  u(x_k +\,\cdot\,, t_k  )-p_{2,x_k,t_k} -p_{3,x_k,t_k}\big)\quad \mbox{and}\quad  \lambda^{3rd} := \phi(0^+,  u-p_{2} -p_{3}).\]
Then $y_\infty\in \{p_2=0\}$, and $q$ is translation invariant in the direction $y_\infty$.
\end{lemma}

\begin{proof}
The fact that  $y_\infty\in \{p_2=0\}$ follows from Lemma~\ref{lem:EG2B1bisbis}.

Since $(0,0)\in \mathbf{\Sigma}^{>3}_{n-1}\setminus \mathbf{\Sigma}^{\ge 4}_{n-1}$,
as in the proof of Lemma~\ref{lem:lipestF3} the limit $ \lim_{r\downarrow 0} \phi(r, u(\,\cdot \,, 0) -p_2 -p_3)$ exists and belongs to $(3,4)$, that is
\[
\lambda^{3rd} : = \phi(0^+, u(\,\cdot \,, 0) -p_2 -p_3)  \in (3, 4).
\]
Similarly, the limits defining $\lambda^{3rd}_{k}$ exist, and by assumption, we have
\begin{equation}
\lambda^{3rd}_{k} := \phi\big(0^+,  u(x_k +\,\cdot\,, t_k  )-p_{2,x_k,t_k} -p_{3,x_k,t_k}\big) \to \lambda^{3rd}.
\end{equation}
We define
\[ \mathfrak p := p_2 + p_3\qquad \mbox{and}\qquad  \mathfrak p_k := p_{2,x_k, t_k} + p_{3,x_k, t_k}\]
and consider
\begin{equation}\label{defwk00}
\begin{split}
w_k  &:= u(x_k +r_k\,  \cdot\, , t_k  ) -   \mathfrak p_{k} \big( r_k\,\cdot\, \big) = w_k ^{(1)} +w_k ^{(2)}  +w_k ^{(3)},
\\
 w_k ^{(1)} &:=  u(x_k +r_k\,  \cdot\,  , t_k  )  - u(x_k + r_k\,  \cdot\,  ,0 ),
 \\
 w_k ^{(2)} &:=  u(x_k +r_k\,  \cdot\,  , 0 )  - \mathfrak p(x_k +r_k\,\cdot\,),
\\
 w_k ^{(3)} &:=   \mathfrak p (x_k+ r_k \,\cdot\,) -\mathfrak p_{k} \big(r_k\,\cdot\,).
\end{split}
\end{equation}
Recall that $y_k : = x_k /r_k$ and define
\begin{equation}\label{tildewk}
\tilde w_k : = \frac{w_k}{\|w_k\|_{L^2(\partial B_1)} }.
\end{equation}

\smallskip

\noindent $\bullet$ \emph{Step 1. } Throughout the proof we fix $\gamma\in (\lambda^{3rd},4)$.
Thanks to Lemma~\ref{lem:E3B1}, for any given  $\delta>0$ we have
\begin{equation}\label{continuity}
\big|\phi^\gamma\big(r,  u(x_k +\,\cdot\,, t_k  )-p_{2,x_k,t_k} -p_{3,x_k,t_k}\big) - \lambda^{3rd}\big| \le \delta \qquad \forall \,r\in (0,r_\delta), \ \forall\, k\ge k_\delta.
\end{equation}
Hence, we may fix  positive constants $\delta_\circ$ and $r_\circ$ such that, for $k\ge k_\circ$ large enough, we have 
\begin{equation}\label{continuity2}
\phi^\gamma\big(r,  u(x_k +\,\cdot\,, t_k  )-  \mathfrak p_k \big)  \le \gamma -3\delta_\circ \qquad \forall \,r\in (0,r_\circ),
\end{equation}
and Lemma~\ref{lem:lipestF3}  ---applied to the function $u(x_k +\,\cdot\, , t_k)$ and  with $r= r_k$--- yields
\begin{equation}
\label{eq:wk Lip11}
\| \tilde w_k\|_{{\rm Lip} (B_R)} \le C(R)  \qquad \mbox{in } B_{R}
 \end{equation}
and $\tilde w_k \Delta\tilde w_k \ge -C(R) r_k ^{\eta_\circ} \Delta\tilde w_k  $, where $\eta_\circ>0$ and $C(R)$ are independent of $k$.
 Then, similarly to the proof of Lemma~\ref{lem:EG2Bn-1}, the (locally uniformly bounded) nonpositive measures $\Delta \tilde w_k$ converge weakly to  $\Delta Q \le 0,$  and  since  $r_k^{\eta_\circ} \Delta \tilde w_k \rightharpoonup 0$  and $\tilde w_k \to Q$ locally uniformly, we have $\tilde w_k\Delta \tilde w_k \to Q\Delta Q \ge 0$. 
Furthermore, since $w_k = u(x_k +r_k\,\cdot\,, t_k  ) \ge 0$ on $\{p_{2,x_k,t_k} =0\}$  and $p_{2,x_k,t_k} \to p_2$, we obtain that  $Q\ge 0$ on $\{p_2=0\}$.
Therefore, we proved that $Q$ is a solution of the Signorini problem \eqref{ETOP}.
Finally,  arguing as in the proof of Lemma~\ref{lem:EG2Bn-1}, it follows by \eqref{continuity} that 
the function $Q$ is $\lambda^{3rd}$-homogeneous.

Note that, by the same reasoning, also $q$ is a  $\lambda^{3rd}$-homogeneous of the Signorini problem \eqref{ETOP}.
\smallskip

\noindent $\bullet$ \textit{Step 2.}  Recall that  $y_\infty\in \{p_2=0\}$.
In addition  by Proposition~\ref{prop:E3B5}(a) we have
\[
\frac{w_k^{(2)}}{ \| w_k^{(2)} \|_{L^2(\partial B_1)}} \rightharpoonup \frac{q (y_\infty +\,\cdot\,)}{ \|q (y_\infty +\,\cdot\,) \|_{L^2(\partial B_1)}} \quad \mbox{in } W^{1,2}_{\rm loc} (\R^n),
\]
and, by construction, $w_k^{(3)}$ is a cubic hamonic polynomial. 

We claim that, for each $k$, there exists a point $\bar y_k$ in the segment $\overline{0\,y_k}$ such that
\begin{equation}\label{ahsgoih}
|w_k ^{(3)}  (\bar y_k)| \le Cr_k^4.
\end{equation}
Indeed, note that $p_2 +p_3  \ge -\frac{p_3^2}{2p_2}$ and thus we have  $\mathfrak p (r_k \,\cdot\,) \ge -Cr_k^4$ and   $\mathfrak p_k (r_k \,\cdot\,) \ge -Cr_k^4$ in $B_1$.
Hence, since $\mathfrak p (0) = \mathfrak p_k (0)=0$,
\[
w_k ^{(3)}(0)   =  \mathfrak p (0) -\mathfrak p_k \big(-r_k y_k)  \ge -Cr_k^4
\qquad \text{and}\qquad
w_k ^{(3)}(y_k)  =  \mathfrak p (r_k y_k) -\mathfrak p_k \big(0)  \le Cr_k^4,
\]
so \eqref{ahsgoih} follows.

\smallskip

\noindent $\bullet$ \textit{Step 3}.  Let us consider
\[
\hat \ep_k := \sum_{i=1,2,3}\| w_k^{(i)}\|_{L^2(\partial B_1)} \qquad \mbox{and } \qquad  \hat w_k : = \frac{w_k}{\hat \ep_k}\,.
\]
Recalling that $\phi^{\gamma} (0^+, u(\,  \cdot\, , 0 ) -   \mathfrak p ) =\lambda^{3rd} <\gamma<4$,
it follows by Lemma~\ref{lem:E3B1c}  that, for any given $\delta>0$,
\[
\hat \ep_k  \ge\| w_k^{(2)}\|_{L^2(\partial B_1)} = \big \|  (u-p_2-p_3)\big(r_k(y_k +\,\cdot\,) \big)\big\|_{L^2(\partial B_1)} \gg r_k^{\lambda^{3rd}+\delta}\quad \mbox{ as $k\to \infty$}.
\]
Thus, by Step 1, we have
\[
\hat w_k \to  \hat Q = a Q\quad \mbox{ for some }a\in [0,1].
\]
Moreover, by Step 2,
\[
\hat Q^{(2)} := \lim _k w_k^{(2)}/\hat \ep_k = b q(y_\infty+\,\cdot\,)
\qquad \text{and}\qquad
\hat Q^{(3)} := \lim _k w_k^{(3)}/\hat \ep_k  =  \big[\mbox{degree 3 hamonic polynomial}\big] \]
for some $b\ge0$. 
(Above, the convergences are weak in $W^{1,2}_{\rm loc}(\R^n)$.) Thus, it is 
well defined
\[\hat Q^{(1)} := \lim_k w_k^{(1)}/\hat \ep_k =  \lim_k w_k/\hat \ep_k -\lim _k w_k^{(2)}/\hat \ep_k- \lim _k w_k^{(3)}/\hat \ep_k,\]
and we observe that $Q^{(1)}$ is either nonpositive or nonnegative (since so is $w_k^{(1)}$).
Hence, we have
\[
\hat Q = \hat Q^{(1)} + b q(y_\infty+\,\cdot\,) + \hat Q^{(3)}.
\]
Moreover, it follows by \eqref{ahsgoih} that  the polynomial $\hat Q^{(3)}$ vanishes at some point $\bar y$ in the segment $\overline{0\,y_\infty}$.
Hence, since $\hat Q^{(3)}$ is harmonic, we see that it cannot have constant sign  (unless it is identically zero).

Note now that, by definition of $\hat \ep_k$, we have  $\sum_i  \|\hat Q^{(i)}\|_{L^2(\partial B_1)}=1$. Hence, since $q$ is a $\lambda^{3rd}$-homogeneous solution of Signorini with $\lambda^{3rd}>3$,  $\hat Q^{(1)}$ has constant sign, and  $\hat Q^{(3)}$ is a cubic harmonic polynomial that does not have constant sign, we deduce that the three functions  $\hat Q^{(i)}$ are  linearly independent and their sum $\hat Q$ cannot be zero.

We show next that $b>0$ and that $\hat Q \equiv  bq$. 
Indeed, since both $q$ and $\hat Q$ are $\lambda^{3rd}$-homogeneous, if $\hat Q^{(1)}\ge 0$ (resp. $\le$) then
\[
\hat Q = \lim_{R \to \infty} \frac{ \hat Q(R \,\cdot\,) }{ R^{\lambda^{3rd}} }  = \lim_{R \to \infty} \frac{\hat Q^{(1)}(R\,\cdot\,) + b q(y_\infty +R\,\cdot\,) + \hat Q^{(3)}(R\,\cdot\,)}{R^{\lambda^{3rd}}} \ge b q  \quad \mbox{(resp. $\le$).}
\]
But then $\hat Q$ and $bq$ are two solution ordered solutions  of Signorini  with homogeneities $>1$ at the origin, and thus they must be equal by Lemma~\ref{lemAP:ordered}.

Therefore, we have shown
\[
b \big(q - q(y_\infty+\,\cdot\,)\big)  = \hat Q^{(1)} + \hat Q^{(3)}.
\]
Now,  using homogeneity, we obtain that  for all $s>0$
\[
 \hat Q^{(1)}(s^{-1}x)  +b s^{-\lambda^{3rd}}q(sy_\infty+x)  + \hat Q^{(3)}(s^{-1}x)  = b s^{-\lambda^{3rd}} q(x).
\]
If  $Q^{(1)}\ge 0$ (resp. if $Q^{(1)}\le 0$), we obtain
\begin{equation}\label{ashgiodashp}
b\frac{q(sy_\infty+x)   - q(x)}{s}  \le s^{\lambda^{3rd}-1}  \hat Q^{(3)}(s^{-1}x)   \quad (\mbox{resp. }\ge ).
\end{equation}
Note that, since $q$ is a solution of \eqref{ETOP} (and so it is Lipschitz continuous, see for instance \cite{ACS08}), the absolute value of the left hand side of   \eqref{ashgiodashp} is bounded as $s\downarrow 0$. 
Hence, since $\lambda^{3rd} \in (3,4)$, the cubic coefficients of $\hat Q^{(3)}$  (recall that $\hat Q^{(3)}$ is a cubic harmonic polynomial) must vanish as otherwise the right hand side would be unbounded. Thus, the cubic coefficients of  $\hat Q^{(3)}$ vanish and therefore right hand side  converges to zero.

Thus, since $b>0$, we have shown that
 \[
y_\infty\cdot  \nabla q \le 0  \quad (\mbox{resp. }\ge 0).
\]
Hence, reasoning as in Step 3 of the proof of Lemma~\ref{lem:EG2Bn-1}, we obtain that $\psi:= y_\infty\cdot  \nabla  q$ restricted to $\mathbb S^{n-1}$  must be a multiple of the first eigenfunction of a certain elliptic problem, and this easily leads to a contradiction because the homogeneity of $q$ is greater than 2.
\end{proof}

Our next goal is to prove a variant of Lemma~\ref{lem:E3B2d} for points in $\mathbf{\Sigma}_{n-1}^{>4} \setminus \mathbf{\Sigma}_{n-1}^{\ge5-\zeta} $.
For that, we need the following Lipschitz estimate.

\begin{lemma}\label{lem:lipestF4}
Let $u:B_1 \to [0,\infty)$ solve  \eqref{eq:UELL}, and let $0\in {\Sigma}^{>4}_{n-1}\setminus {\Sigma}^{\ge 5-\zeta}_{n-1}$. Set $w:= u - \anz -P$, where $P$ is some $4$-homogeneous harmonic polynomial vanishing on $\{p_2=0\}$.
Assume that, for some $r_\circ >0$, $\gamma\in (4,5)$, $\delta>0$, and $h_\circ>0$,
\begin{equation}\label{boundphi22}
\phi^{\gamma}(r, u-\anz-P) \le \gamma-\delta_\circ \quad \forall \,r\in (0, r_\circ) \qquad \mbox{and}  \qquad H(r_\circ, u-\anz-P) \ge h_\circ.
\end{equation}
Then there exist positive constants $\varrho_\circ$, $\eta_\circ$, and $C$,  depending only on $n$, $\gamma$, $\delta_\circ$, $r_\circ$, and $h_\circ$, such that  for any given $R\ge 1$ and  for all $r\in \big(0 , \textstyle \frac{\varrho_\circ}{10R}\big)$ we have
\begin{equation}
\label{eq:wk Lip5}
\|\tilde w_r\|_{{\rm Lip} (B_R)}  \le CR^4  \quad \mbox{and} \quad   \tilde w_r \Delta \tilde w_r \ge -Cr^{\eta_\circ} R^5  \Delta \tilde w_r\quad \mbox{in } B_{R} .
 \end{equation}
\end{lemma}

\begin{proof}
The proof is analogous to the one of Lemma~\ref{lem:lipestF3}, using Lemma~\ref{lem:E3B5} instead of \eqref{sahfgij} and Lemma~\ref{lem:E4B1} instead of Lemma~\ref{lem:E3B1}.
\end{proof}

Recalling that $p_4=p_{4,0,0}$ is defined in \eqref{eq:def p4 t},
we now prove the following:

\begin{lemma}\label{lem:E3B2b}
Let $u\in C^0\big(\overline{B_1}\times [-1,1]  \big)$ solve \eqref{eq:UELL+t}, let $(0,0)$ and $(x_k,t_k)$ belong to  $\mathbf{\Sigma}_{n-1}^{>4} \setminus \mathbf{\Sigma}_{n-1}^{\ge5-\zeta} $ for some $\zeta\in (0,1)$, and suppose that $|x_k|\le r_k\downarrow 0$. Assume in addition that 
$$
\tilde  w_{r_{k}} \rightharpoonup  q \mbox{  in } W^{1,2}_{\rm loc}(\R^n) \qquad \mbox{for } w :=  u-\anz-p_4\quad \mbox{and} \quad y_k :=\frac{x_k}{r_k} \to y_\infty,
$$
and that $\lambda^{4th}_k\to \lambda^{4th}$, where
\[\lambda^{4th}_{k} := \phi\big(0^+,  u(x_k +\,\cdot\,, t_k  )-\anz_{x_k,t_k}-p_{4,x_k,t_k}\big)\quad \mbox{and}\quad  \lambda^{4th} := \phi(0^+,  u-\anz-p_4).\]
Then $y_\infty\in \{p_2=0\}$, and $q$ is translation invariant in the direction $y_\infty$.
\end{lemma}

\begin{proof}
The proof is very similar to that of Lemma~\ref{lem:E3B2d}, with some appropriate modifications.
As before,  the fact that $y_\infty\in \{p_2=0\}$ follows from Lemma~\ref{lem:EG2B1bisbis}.

Also, since $(0,0)\in \mathbf{\Sigma}^{>4}_{n-1}\setminus \mathbf{\Sigma}^{\ge 5-\zeta}_{n-1}$, as in the proof of Lemma~\ref{lem:lipestF3} the limit $ \lim_{r\downarrow 0} \phi(r, u(\,\cdot \,, 0) -\anz -p_4)$ exists and belongs to $(4,5-\zeta)$, that is
\[
\lambda^{4th} : = \phi(0^+, u(\,\cdot \,, 0) -\anz -p_4)  \in (4, 5-\zeta).
\]
Similarly, the limits defining $\lambda^{4th}_{k}$ exist, and by assumption we have
\[
\lambda^{4th}_{k} := \phi\big(0^+,  u(x_k +\,\cdot\,, t_k  ) -\anz_{x_k,t_k} -p_{4,x_k,t_k}\big) \to \lambda^{4th}.
\]
We define
\[ \mathfrak p := \anz +p_4 \qquad \mbox{and}\qquad  \mathfrak p_k := \anz_{x_k,t_k} +p_{4,x_k,t_k},\]
and consider $w_k:= u(x_k + r_k\, \cdot\,, t_k ) - \mathfrak p_k(r_k\,\cdot\,) =  w_k^{(1)}+w_k^{(2)}+ w_k^{(3)} $  as in \eqref{defwk00}.
Recall that $y_k : = x_k /r_k$ and define $\tilde w_k$ as in \eqref{tildewk}.

\smallskip

\noindent $\bullet$ \emph{Step 1. } 
Here we argue as in Step 1 in the proof of Lemma~\ref{lem:E3B2d}.
More precisely, using Lemma~\ref{lem:E3B1} in place of Lemma~\ref{lem:lipestF3},
by the very same argument we deduce that $\tilde w_k$ converges locally uniformly to $Q$,
and that both $q$ and $Q$ are $\lambda^{4th}$-homogeneous solutions of \eqref{ETOP}.

\smallskip

\noindent $\bullet$ \textit{Step 2.}  By Proposition~\ref{prop:E3B5}(a), we have
\[
\frac{w_k^{(2)}}{ \| w_k^{(2)} \|_{L^2(\partial B_1)}} \rightharpoonup \frac{q (y_\infty +\,\cdot\,)}{ \|q (y_\infty +\,\cdot\,) \|_{L^2(\partial B_1)}} \quad \mbox{in } W^{1,2}_{\rm loc} (\R^n)
\]
and, by construction, $w_k^{(3)}$ is a quartic harmonic polynomial.
In addition, arguing as in Step 2 of the proof of Lemma~\ref{lem:E3B2d} we obtain that,  for each $k$, there exists a point $\bar y_k$ in the segment $\overline{0\,y_k}$ such that
\begin{equation}\label{ahsgoih5}
|w_k ^{(3)}  (\bar y_k)| \le Cr_k^5.
\end{equation}

\smallskip

\noindent $\bullet$ \textit{Step 3}.  Considering
\[
\hat \ep_k := \sum_{i=1,2,3}\| w_k^{(i)}\|_{L^2(\partial B_1)} \qquad \mbox{and } \qquad  \hat w_k : = \frac{w_k}{\hat \ep_k}\,,
\]
as in Step 3 of the proof of Lemma~\ref{lem:E3B2d} we have 
\[
\hat w_k \to  \hat Q = a Q,\quad \hat Q^{(2)} := \lim _k w_k^{(2)}/\hat \ep_k = b q(y_\infty+\,\cdot\,),\quad \hat Q^{(3)} := \lim _k w_k^{(3)}/\hat \ep_k  =  \big[\mbox{degree 4 harmonic pol.}\big] ,
\]
where $a \in [0,1]$, $b\ge0$, and all the convergences hold weakly in $W^{1,2}_{\rm loc}(\R^n)$.
Hence
\[
\hat Q = \hat Q^{(1)} + b q(y_\infty+\,\cdot\,) + \hat Q^{(3)},
\]
where $\hat Q^{(1)} := \lim_k w_k^{(1)}/\hat \ep_k $ has constant sign.
Since $q$ is a $\lambda^{4th}$-homogeneous solution of Signorini with $\lambda^{4th}>4$,  $\hat Q^{(1)}$ has constant sign, and  $\hat Q^{(3)}$ is a forth order harmonic polynomial that does not have constant sign (as a consequence of \eqref{ahsgoih5}), we deduce that the three functions  $\hat Q^{(i)}$ are  linearly independent and their sum $\hat Q$ cannot be zero.

Also, exactly as in Step 3 of the proof of Lemma~\ref{lem:E3B2d},
 $b>0$ and $\hat Q \equiv  bq$, therefore
 \[
b \big(q - q(y_\infty+\,\cdot\,)\big)  = \hat Q^{(1)} + \hat Q^{(3)}.
\]
Now,  using homogeneity, if  $Q^{(1)}\ge 0$ (resp. if $Q^{(1)}\le 0$) we obtain 
$$
\frac{q(sy_\infty+x)   - q(x, t)}{s}  \le s^{\lambda^{4th}-1}  \hat Q^{(3)}(s^{-1}x)   \quad (\mbox{resp. }\ge ),
$$
for all  $s>0$. As in Step 3 of the proof of Lemma~\ref{lem:E3B2d}, this is possible only if 
the quartic coefficients of $\hat Q^{(3)}$ vanishes, and letting $s\to0$ we get

 \[
y_\infty\cdot  \nabla q \le 0  \quad (\mbox{resp. }\ge 0).
\]
Reasoning now as in Step 3 of the proof of Lemma~\ref{lem:EG2Bn-1} (see also Step 3 of the proof of Lemma~\ref{lem:E3B2d}), we obtain that $\psi:= y_\infty\cdot  \nabla  q$ restricted to $\mathbb S^{n-1}$  must be a multiple of the first eigenfunction of a certain elliptic problem, and this easily leads to a contradiction.
\end{proof}

Before proving the last result of this section, we introduce a definition:
\begin{definition}\label{setofpolsP4}
We denote by $\mathcal P_{4,\ge}^{even}$ the set of $4$-homogeneous harmonic polynomials $p= p(x_1, \dots, x_n)$, such that,  for some $\boldsymbol e\in \mathbb S^{n-1},$ we have:
\begin{itemize}
\item $p$ is even with respect to $\{\boldsymbol e\cdot x =0\}$, that is, $p(x)= p(x -2(\boldsymbol e\cdot x )\boldsymbol e )$;
\item $p\ge 0$ on $\{\boldsymbol e\cdot x =0\}$;
\item $\|p\|_{L^2(\partial B_1)} =1$.
\end{itemize}
Given $p\in  \mathcal P_{4,\ge}^{even}$, we denote $\boldsymbol S(p,\ep )\subset \R^n$ the set
\[\boldsymbol S(p,\ep): = \{ |\boldsymbol e\cdot x| \le \ep\} \cap \{p \le \ep\} \cap \overline{B_2}.\]
\end{definition}
We now show the following result, which will be used later to bound the Hausdorff dimension of $\pi_1(\mathbf{\Sigma}_{n-1}^{\ge 4} \setminus  \mathbf{\Sigma}^{4th}_{n-1})$.

\begin{lemma}\label{lem:Sigma4aux}
Let  $u:B_1 \to [0,\infty)$ solve \eqref{eq:UELL}, and let  $0\in {\Sigma}_{n-1}^{\ge 4} \setminus  {\Sigma}^{4th}_{n-1}$.  
Let $\mathcal P_{4,\ge}^{even}$ and $\boldsymbol S (p,\ep)$ be as in Definition~\ref{setofpolsP4}.
Then,  given $\ep>0$,  there exists  $\varrho_\ep>0$ such that, for all $r\in (0,\varrho_\ep)$,
\begin{equation}\label{u=0closetozp}
\{u=0\}\cap \overline{B_r} \subset r \boldsymbol S (p_r,\ep)\qquad \text{for some $p_{r}\in \mathcal P_{4,\ge}^{even}$.}
\end{equation}
\end{lemma}

\begin{proof} Consider the set of ``accumulation points'' $ \overline{\mathcal Q}$ defined as
\[
 \overline{\mathcal Q}  : = \big\{ q \ : \ \exists \,r_k \downarrow 0 \mbox{ s.t. } r_k^{-4} (u -\anz)(  r_k\, \cdot\, ) \rightarrow q \big\}.
\]
Note that, for all $\eta>0$, there exists $\varrho_\eta>0$ such that for any $r\in (0, \varrho_\eta)$ we have
\begin{equation}\label{12345}
\big\|u - \anz-q_{r}  \big\|_{L^\infty(B_r)}   \le \eta  r^4\qquad \text{for some $q_r \in \overline{\mathcal Q}$}.
\end{equation}
Thanks to Proposition~\ref{prop:E3B5}(a) and  \cite[Lemma 1.3.4]{GP09}, $\overline{\mathcal Q}$ is a closed set of $4$-homogeneous harmonic polynomials.
Also, using Lemma~\ref{lem:E4B2}  with $P\equiv 0$ and $\gamma\in (4,5)$ fixed, we see that $\|q\|_{L^2(\partial B_1)}\le C$ for all $q\in  \overline{\mathcal Q}$.
This implies that set  $\overline{\mathcal Q}$ is compact. 

Now, since by assumption $0\in{\Sigma}^{\ge 4}_{n-1} \setminus {\Sigma}^{4th}_{n-1}$, then $q^{even} \not\equiv 0$  for all $q\in  \overline{\mathcal Q}$
(recall Definition~\ref{def:Sigma4th}).
Thus, by compactness of $\overline{\mathcal Q}$, we deduce that
\[
0< c_\circ \| q^{even} \|_{L^2(\partial B_1)}   \le  \| q \|_{L^2(\partial B_1)}   \le C  \quad \forall \, q\in\overline{\mathcal Q}.
\]
Now, for $r>0$ and $q_r$ as in \eqref{12345}, we define
\[p_r : =  \frac{ q_r^{even}}{ \| q_r^{even} \|_{L^2(\partial B_1)}   }, \]
and note that $p_r\in \mathcal P_{4,\ge}^{even}$. We claim that \eqref{u=0closetozp} holds true provided that $r\in (0, \rho_\ep)$, with $\rho_\ep>0$ small.

Indeed, assume with no loss of generality that $\{p_2=0\}= \{x_n=0\}$. Then (since $q_r$ solves \eqref{ETOP}) every $p_r$ is a $4$-homogeneous harmonic polynomial, even in the variable $x_n$, nonnegative on $\{x_n=0\}$, and with unit $L^2(\partial B_1)$ norm.

We recall that
 \begin{equation}\label{eq:bound anz}\anz(x) \ge (x_n + p_3/x_n +Q)^2  -C|x|^5.\end{equation}
Now, by definition of $\boldsymbol S(p,\ep)$, it follows in particular that, fixed $\theta>0$,
\[
 y\in B_2\setminus \boldsymbol S(p_r,\ep) \qquad \Rightarrow \qquad \mbox{either}\quad (\,p_r(y)> \ep \mbox{ and } |y_n|\le \theta \ep\,) \quad \mbox{or} \quad ( \,|y_n|>\theta \ep\, ).
\]
We now observe that, if $p_r(y)> \ep \mbox{ and } |y_n|\le \theta \ep$, since
$q_r^{odd}$ vanishes on $\{x_n=0\}$ we get
\[
q_r(y) = p_r(y)  \| q^{even} \|_{L^2(\partial B_1)}  +  q_r^{odd}(y) \ge  c_\circ \ep - C |y_n| \ge  (c_\circ- C\theta) \ep \ge \frac12 c_\circ \ep >0
\]
provided we choose $\theta: = \frac{c_\circ}{2C}$ small enough. Thus, recalling \eqref{12345} and \eqref{eq:bound anz}, if $r>0$ is sufficiently small (so that we can take $\eta\ll \ep$) we get
\[  u(ry) \ge  \anz(ry) + q_r(ry)  - \eta r^4   \ge  -Cr^5 + \frac12 c_\circ \ep  r^4 -\eta r^4 >0.\]
On the other hand, if $|y_n|>\theta \ep$,
using again \eqref{eq:bound anz} we obtain, for $r>0$ sufficiently small,
\[  u(ry) \ge  \anz(ry) + q_r(ry)  - \eta r^4   \ge  (\theta \ep r -Cr^2)^2 - C r^5 - Cr^4 -\eta r^4 >0.\]
Therefore, we have proven that
\[
 y\in B_2\setminus \boldsymbol S(p_r,\ep) \quad \Rightarrow \quad  u(ry)>0,
\]
which gives \eqref{u=0closetozp}.
\end{proof}

\section{Hausdorff measures and covering arguments} \label{sec:GMT}

As already explained in the introduction, to prove our main results we will need some auxiliary results from geometric measure theory.
Before stating them, we recall some classical definitions.

Given $\beta >0$ and
$\delta \in (0,\infty]$,
the Hausdorff premeasures $\mathcal H^{\beta}_\delta(E)$ of a set $E$ are defined as follows:\footnote{In many textbooks, the definition of $\mathcal H^{\beta}_\delta$ includes a normalization constant\ chosen so that the Hausdorff measure of dimension $k$ coincides with the standard $k$-dimensional volume on smooth sets. However  such normalization constant is irrelevant for our purposes, so we neglect it.}
\begin{equation}
\label{eq:def Haus}
\mathcal H^{\beta}_\delta(E):=\inf\biggl\{\sum_i {\rm diam}(E_i)^\beta\,:\,E\subset \bigcup_i E_i,\,{\rm diam}(E_i)<\delta \biggr\},
\end{equation}
where the index $i$ goes through a finite or countable set.
Then, one defines the $\beta$-dimensional Hausdorff measure of $E$ as $\mathcal H^\beta(E):=\lim_{\delta\to 0^+}\mathcal H^{\beta}_\delta(E)$.

The Hausdorff dimension can be defined in terms of $\mathcal H^{\beta}_\infty$ as follows:
\begin{equation}
\label{eq:def dim}
{\rm dim}_{\mathcal H}(E):=\inf\{\beta >0\,:\,\mathcal H^{\beta}_\infty(E)=0\}
\end{equation}
(this follows from the fact that $\mathcal H^{\beta}_\infty(E)=0$ if and only if $\mathcal H^{\beta}(E)=0$, see for instance \cite[Section~1.2]{Sim83}).

%

We now state (and prove, for completeness) a couple of standard results.

\begin{lemma}\label{lem:GMT1}
Let $E\subset \R^n$, and $f:E\to \R$.
Define
$$
F:=\{x\in E\,:\,\exists \,x_k\to x,\, x_k\in E,\, \text{ s.t. } f(x_k)\to f(x)\}.
$$
Then $E\setminus F$ is at most countable.
\end{lemma}

\begin{proof}
Let $G := \big \{(x,f(x)): x\in E\big\}\subset \R^n \times \R$ be the graph of $f$. We note that $x\in E\setminus F$ if a only if $(x,f(x))$ is a isolated point of $G$. In particular $E\setminus F$  is the projection of a discrete (and hence countable) set.
\end{proof}

From now on, by convention, whenever we say that a set $E$ can be covered by a number $M>0$ of balls that it is not necessarily an integer, we mean that it can be covered by $\lfloor M\rfloor$ balls, where $\lfloor M\rfloor$ denotes the integer part of $M$.

\begin{lemma}
\label{LOC.lem:GMT2}
Let $B_r(x)\subset \R^n$ be an open ball, and $\Pi$ be a $m$-dimensional plane. 
Let $\beta_1>m$. Then there exists $\hat\ep=\hat\ep(m,\beta_1)>0$ such that the following holds:
Let $E\subset \R^n$ satisfy
$$
E\subset B_r(x) \cap \{y\,:\,{\rm dist}(y,\Pi)\leq \ep r\},\qquad \text{for some  $0<\ep\leq \hat \ep$, $x \in \R^n$, $r>0$.}
$$
Then $E$ be covered with $\gamma^{-\beta_1}$ balls of radius $\gamma r$ centered at points of $E$, where $\gamma:=5\ep$.
\end{lemma}

\begin{proof}
Up to a scaling and a translation, it suffices to prove the result when $r=1$ and $B_r(x)$ is the unit ball $B_1$ centered at the origin.
Consider the $m$-dimensional set
$
B_{1}\cap \Pi,
$
and given $\ep>0$ small consider  the covering of $E\subset B_1\cap  \{y\,:\,{\rm dist}(y,\Pi)\leq \ep\}$ given by the closed balls
$
\{\overline{B_{\ep}(x)}\}_{x\in E}.
$
By Vitali Covering Lemma, there exists a disjoint family $\{\overline{B_{\ep}(x_i)}\}_{i\in \mathcal I}$ such that
$$
\bigcup_{i\in \mathcal I}\overline{B_{5\ep}(x_i)}\supset \bigcup_{x \in E}\overline{B_{\ep}(x)}\supset E.
$$
Note that
\[ \overline{B_{\ep}(x_i)} \subset \mathcal N_{2\ep}(\Pi) : =   \{x\in B_2  \ : \ {\rm dist}(x,\Pi) \le 2\ep\}.\]
Since
 $ \HH^n\big (\mathcal N_{2\ep}(\Pi)  \big)  \le C(n) \ep^{n-m},$
denoting by $\omega_n$ the volume of the $n$-dimensional unit ball we have
$$
\omega_n\,\ep^n \,\#\mathcal I \le\sum_{i\in \mathcal I}\HH^n(B_{\ep}(x_i))\leq \HH^n(\mathcal N_{2\ep}(\Pi) )=C(n) \ep^{n-m},
$$
which proves that $\#\mathcal I \leq C(n)  \ep^{-m}$.
Set $\gamma:=5\ep$. 
Then,
since $\beta_1>m$,
choosing $\ep$ sufficiently small we have
$C(n)  \ep^{-m}=C(n)5^m \gamma^{-m}\leq \gamma^{-\beta_1}$, proving that
 $ E$ can be covered by $\gamma^{-\beta_1}$ open balls of radius $\gamma$ centered at points of $E$, as desired.
\end{proof}

The following Reifenberg-type result will be used later choosing as function $f$ the frequency function, and it will allow us to perform our dimension reduction arguments only at continuity points of the frequency.

\begin{proposition}\label{prop:GMT2}
Let $E\subset \R^n$, and $f:E\to \R$.
Assume that, for any $\ep>0$ and $x\in E$ there exists $\varrho= \varrho(x,\ep)>0$ such that, for all $r\in (0, \varrho)$, we have
$$
E\cap \overline{B_r(x)}\cap f^{-1}([f(x)-\varrho,f(x)+\varrho])\subset \{y\,:\,{\rm dist}(y,\Pi_{x,r})\leq \ep r\},
$$
for some $m$-dimensional plane $\Pi_{x,r}$ passing through $x$ (possibly depending on $r$).
Then ${\rm dim}_\HH(E)\leq m$.
\end{proposition}

\begin{proof}
We need to prove that, given $\beta>m$, we have $\HH^\beta(E)=0$.

Let $\ep>0$ be a small constant to be fixed later, and for any $k>1$ and $j\in\mathbb Z$ define
$$
E_{k,j}:=\left\{x\in E\,:\, \varrho(x,\ep)>1/k, \,f(x)\in \bigl[{\textstyle \frac{j}{2k} ,\frac{j+1}{2k} }\bigr)\right\}.
$$
Since $E=\cup_{k,j} E_{k,j}$, it suffices to prove that $\HH^\beta(E_{k,j})=0$ for each $k,j$.
So, we fix $k>1$ and $j \in \mathbb Z$.
Similarly, it suffices to prove that for all $R\ge 1$ we have $\HH^\beta(E_{k,j}^R)=0$, where $E_{k,j}^R : = B_R\cap E_{k,j}$.

By assumption, for every $x \in E_{k,j}^R$ and $r\in(0, 1/k]$, there exists
a $m$-dimensional plane $\Pi_{x,r}$ such that
$$
E_{k,j}^R\cap \overline{B_r(x)} \subset \{y\,:\,{\rm dist}(y,\Pi_{x,r})\leq \ep r\}.
$$
So, we consider the covering $\{\overline{B_{1/k}(x)}\}_{x\in E_{k,j}^R}$, and since $E_{k,j}^R\subset B_R$ we extract a finite subcovering of closed balls $B_1^{(0)},\ldots,B_M^{(0)}$. (Indeed, by Vitali's lemma there is a covering $\{\overline{B_{1/k}(x_\ell)}\}$ such that the balls  $\{\overline{B_{1/(5k)}(x_\ell)}\}$ are disjoint, and hence there is a finite number of them.)
Inside each ball $B_i^{(0)}$ we have, by assumption,
$$
E_{k,j}^R \cap B_i^{(0)}\subset \{y\,:\,{\rm dist}(y,\Pi_{B_i^{(0)}})\leq \ep/k\}.
$$
Choose $\beta_1:=\frac{m+\beta}{2}\in (m,\beta)$.
Applying Lemma~\ref{LOC.lem:GMT2} with $\ep=\hat\ep(m,\beta_1)$ we deduce that, for each fixed $i,j,k, R$,  the set $E_{k,j}^R\cap B_i^{(0)}$ can be covered with $\gamma^{-\beta_1}$ closed balls $\hat B^{(1)}_{1},\ldots ,\hat B^{(1)}_{\gamma^{-\beta_1}}$ of radius $\gamma/k$ centered at points of $E^R_{k,j}\cap B_i^{(0)}$, where $\gamma=5\ep$. Using our assumption again, in each of these balls we have
$$
E_{k,j}^R \cap {\hat B^{(1)}_{\ell}}\subset \{y\,:\,{\rm dist}(y,\Pi_{x_{\ell}^{(1)}})\leq \ep\gamma/k\},
$$
where $x_{\ell}^{(1)}$ is the centre of $\hat B^{(1)}_{\ell}$.
We then apply again Lemma~\ref{LOC.lem:GMT2} so that, for each $\ell\in \{1,\ldots,\gamma^{-\beta_1}\}$, we can cover the set $E_{k,j}^R\cap \hat B^{(1)}_{\ell}$ with $\gamma^{-\beta_1}$ closed balls
of radius $\gamma^2/k$. This gives a new covering of $E_{k,j}^R\cap {B_i^{(0)}}$ with  $\gamma^{-2\beta_1}$ closed balls $\hat B^{(2)}_{1},\ldots ,\hat B^{(2)}_{\gamma^{-2\beta_1}}$ of radius $\gamma^2/k$ centered at points of $E^R_{k,j}$.
Iterating this construction, we conclude that
$
E_{k,j}^R\cap B_i^{(0)}
$
can be covered by $\gamma^{-N\beta_1}$ closed balls $\{\hat B^{(N)}_{\ell}\}$ of radius $\gamma^{N}/k$ for any $N\geq 1$,
which implies that
$$
\mathcal H_\infty^{\beta}(E_{k,j}^R\cap B_i)\leq C_{n,m}\sum_{\ell}{\rm diam}(\hat B^{(N)}_{\ell})^\beta
\leq C_{n,m}\gamma^{-N\beta_1}(\gamma^N/k)^{\beta} \leq C\gamma^{N(\beta-\beta_1)}.
$$
Since $\beta_1 \in (m,\beta)$, letting $N \to \infty$ we conclude that
$$
\mathcal H_\infty^{\beta}(E_{k,j}^R\cap B_i^{(0)})=0\qquad \mbox{for all } i,j,k,R,
$$
concluding the proof.
\end{proof}

In our study of $4$-homogeneous blow-ups, we will need a variant of the previous results involving zero sets of 4-homogeneous harmonic polynomials instead of hyperplanes (recall Definition~\ref{setofpolsP4}).

\begin{lemma}
\label{LOC.lem:GMT2bis}
Given $\beta_1>n-2$, there exists $\hat\ep=\hat\ep(n,\beta_1)>0$ 
 such that the following holds:
Let $E\subset \R^n$ satisfy
$$
E\subset B_1 \cap \boldsymbol S(p,\ep),\qquad \text{for some  $0<\ep\leq \hat \ep$, $p \in\mathcal P_{4,\ge}^{even}$.}
$$ 
Then $E$
can be covered with $\gamma^{-\beta_1}$ balls of radius $\gamma$ centered at points of $E$, for  some $\gamma=\gamma(n,\beta_1) \in (0,1)$.
\end{lemma}

\begin{proof}
Given $t,\ep>0$ small, consider  the covering of $E\subset B_1\cap  \{y\,:\,{\rm dist}(y,\boldsymbol S(p,\ep))\leq t\}$ given by
$
\{\overline{B_{t}(x)}\}_{x\in E}.
$
By Vitali's lemma, there exists a disjoint family $\{\overline{B_{t}(x_i)}\}_{i\in \mathcal I}$ such that
$$
\bigcup_{i\in \mathcal I}\overline{B_{5t}(x_i)}\supset \bigcup_{x \in E}\overline{B_{t}(x)}\supset E.
$$
Note that, since $x_i \in E$,
\[ \overline{B_{t}(x_i)} \subset \mathcal N_{2t}(\boldsymbol S(p,\ep)) : =   \{x\in B_2  \ : \ {\rm dist}(x,\boldsymbol S(p,\ep)) \le 2t\}.\]
We claim that there exists a dimensional constant $C(n)$ such that, for any given $t\in(0,1)$, there is $\ep_t>0$ such that
\begin{equation}\label{ahfoih}
 \HH^n\big ( \mathcal N_{2t}(\boldsymbol S(p,\ep))    \big) \le C(n) t^2\qquad \forall\,\ep \in (0, \ep_t).
 \end{equation}
Indeed, if not, then for arbitrarily large $M$ there would exist some $t_M\in (0,1)$  and sequences $\ep_k\downarrow 0$ and $p_k\in \mathcal P_{4, \ge }^{even}$ such that
\begin{equation}\label{hsagb}
 \HH^n\big ( \mathcal N_{2t_M}(\boldsymbol S(p_k,\ep_k))    \big)  \ge M t_M^2\qquad \forall\,k \geq 1.
\end{equation}
Now, 
given $p\in P_{4, \ge }^{even}$ which is even with respect to the hyperplane $\{ \boldsymbol e \cdot x =0\} $ and nonnegative on it, let us define
 \[
 \boldsymbol z(p) : =  \{p=0\}\cap\{ \boldsymbol e \cdot x =0\}.
 \]
Notice that, by definition of  $\boldsymbol S(p,\ep)$, for all  $p\in P_{4, \ge }^{even}$  we have
\begin{equation}
\label{eq:Sp z}
 \boldsymbol S(p,\ep) \downarrow  \boldsymbol z(p)  \quad \mbox{as}\quad  \ep\downarrow 0.
\end{equation}
In addition, for all  $x\in \boldsymbol z(p)$, we have that  $\boldsymbol  e\cdot \nabla p(x)=0$ (since $p$ is even with respect $\{ \boldsymbol e \cdot x =0\}$). Furthermore, the tangential gradient vanishes at $x\in \boldsymbol z(p)$ (since $p\ge0$  on $\{ \boldsymbol e \cdot x =0\}$  and $p(x)=0$).
Hence, this proves that
\begin{equation}
\label{eq:z p Dp}
 \boldsymbol z(p) \subset \{p=|\nabla p| =0\}.
\end{equation}
Let $p_k$ be even with respect to $\boldsymbol e_k \in \mathbb S^{n-1}$, and assume without loss of generality that $p_k\to p_\infty \in \mathcal P_{4, \ge }^{even}$ and that $\boldsymbol e_k\to \boldsymbol e_\infty$.
Then it follows by \eqref{eq:Sp z} that, for all $\delta>0$, there exists $k_\delta$ such that
\[
\mathcal N_{2t_M}(\boldsymbol S(p_k,\ep_k)) \subset   \{x\in B_2  \ :  \ {\rm dist}(x, \boldsymbol z(p_\infty)) \le 2t_M+\delta\}, \quad \forall\, k \ge k_\delta.
\]
Recalling \eqref{hsagb}, this implies that 
 \[ \HH^n\big (  \big\{x\in B_2  \ :  \ {\rm dist}(x, \boldsymbol z(p_\infty)) \le 2t_M  \big\}     \big)  \ge M t_M^2.\]
On the other hand, \cite[Theorem 1.1]{NV17} implies the existence of a dimensional constant $C(n)$ such that
 \[ \HH^n\big (  \big\{x\in B_2  \ :  \ {\rm dist}(x, \{u=|\nabla u| =0\} ) \le 2t   \big\}   \big)  \le C(n) t^2 \quad \forall \,t \in (0,1)\]
 for every nonzero harmonic function $u$ in $B_4$. Recalling \eqref{eq:z p Dp}, we obtain a contradiction by choosing $M>C(n)$.

Now, denoting by $\omega_n$ the volume of the $n$-dimensional unit ball, given $t\in (0,1)$, thanks to \eqref{ahfoih} we have
$$
\omega_n\,t^n\,\#\mathcal I  \le\sum_{i\in \mathcal I}\HH^n(B_{t}(x_i))\leq \HH^n(N_{2t}(\boldsymbol S(p,\ep) )=C(n) t^{2} \qquad \forall\, \ep\in (0,\ep_t),
$$
which proves that $\#\mathcal I \leq C(n)  t^{2-n}$.
Set $\gamma := 5t$.
Since $\beta_1>n-2$,
choosing $t$ sufficiently small we have
$C(n)  t^{2-n}=C(n)5^{n-2} \gamma^{2-n}\leq \gamma^{-\beta_1}$, proving that
$E$ can be covered by $\gamma^{-\beta_1}$ open balls of radius $\gamma$ centered at points of $E$ whenever $\ep<\hat \ep:= \ep_t$.
\end{proof}

\begin{proposition}\label{prop:GMT2bis}
Let $E\subset \R^n$ be a measurable set, and $\tau: E\rightarrow \R $ a lower-semicontinuous function.
Assume that, for any $\ep>0$ and $x\in E$, there exists $\varrho= \varrho(x,\ep)>0$ such that, for all $r\in (0,\varrho)$, we have
$$
E\cap \overline{B_r(x)}\cap \tau^{-1}\big([\tau(x),+\infty) \big)\subset  \big\{x+ ry \, :\, y\in \boldsymbol S(p_{x,r},\ep) \big\}
$$
for some  $p_{x,r}\in \mathcal P_{4,\ge}^{even}$.
Then ${\rm dim}_\HH(E)\leq n-2$.
\end{proposition}

\begin{proof}
Given $\beta>n-2,$ we need to prove that $\HH^\beta(E)=0$.
Let $\ep>0$ be a small constant to be fixed later, and for any $k>1$ define
$$
E_k:=\{x\in E\,:\,\rho(x) \geq 1/k\}.
$$
Since $E=\cup_k E_k$, it suffices to prove that $\HH^\beta(E_k)=0$ for each $k$.
So, we fix $k>1$.
Thanks to \cite[Corollary 2.10.23]{Fed69}, it suffices to prove that $\HH^\beta(K)=0$ for any compact set $K$ contained inside $E_k$.

We now claim the following: For each closed ball $\overline{B_r(x)}$ centered at a point $x \in E$ and of radius $r\leq 1/k$, there exist  $\bar x \in  K\cap \overline{B_r(x)}$
and $p_{\bar x,2r} \in \mathcal P_{4,\ge}^{even}$  such that
$$
K\cap \overline{B_r(x)}\subset  \bar x+  r \boldsymbol S (p_{\bar x,r},\ep).
$$
To prove this claim it suffices to observe that it is trivially true if $K\cap \overline{B_r(x)}$ is empty. Otherwise, the lower semicontinuous function $\tau$ attains its minimum at some point $\bar x \in K\cap \overline{B_r(x)}$. 
Then, by the assumption of the lemma,
$$
K\cap \overline{B_r(x)}=K \cap \overline{B_r(x)} \cap \tau^{-1}\big([ \tau(\bar x), \infty) \big) \subset  E\cap \overline{B_{2r}(\bar x)}  \cap \tau^{-1}\big([ \tau(\bar x), \infty) \big)  \subset
 \bar x+  r \boldsymbol S (p_{\bar x,r},\ep),
$$
which proves the claim.

Now,
consider the covering $\{\overline{B_{1/k}(x)}\}_{x\in K}$, and extract a finite subcovering of closed balls $B_1^{(0)},\ldots,B_M^{(0)}$.
Inside each ball $B_i^{(0)}$ we can apply the claim to deduce that
$$
K\cap {B_i^{(0)}}\subset   \bar x_i+  r \boldsymbol S (p_{\bar x_i,r},\ep).
$$
Applying now Lemma~\ref{LOC.lem:GMT2bis} we deduce that, for each fixed $i$, the set  $K\cap {B_i^{(0)}}$ can be covered with $\gamma^{-\beta_1}$ closed balls $\hat B^{(1)}_{1},\ldots,\hat B^{(1)}_{\gamma^{-\beta_1}}$ of radius $\gamma/k$ centered at points of $E$.
 In each of these balls we now reapply the claim to deduce that
$$
K\cap \hat B^{(1)}_{\ell} \subset  \bar x_{\ell}^{(1)} +  \frac \gamma k \boldsymbol S (p_{\bar  x_{\ell}^{(1)},\frac \gamma k},\ep).
$$
Thus we can apply again Lemma~\ref{LOC.lem:GMT2bis} (rescaled) to cover, for each $\ell$, the set $K\cap \hat B^{(1)}_{\ell}$ with $\gamma^{-\beta_1}$ closed balls. In this way we obtain a new covering of $K\cap {B_i^{(0)}}$ by $\gamma^{-2\beta_1}$ closed balls $\hat B^{(2)}_{1},\ldots,\hat B^{(2)}_{\gamma^{-2\beta_1}}$ of radius $\gamma^2/k$ centered at points of $E$.
Iterating this construction, we conclude that
$
K\cap {B_i^{(0)}}
$
can be covered by $\gamma^{-N\beta_1}$ closed balls $\{\hat B^{(N)}_{\ell}\}$ of radius $\gamma^N/k$ for any $N\geq 1$, which implies that
$$
\mathcal H_\infty^{\beta}(K\cap \overline{B_i})\leq C_{n,m}\sum_{\ell}{\rm diam}(\hat B^{(N)}_{\ell})^\beta
\leq C_{n,m}\gamma^{-N\beta_1}\left(\frac{\gamma^N}{k}\right)^{\beta} \leq C\gamma^{N(\beta-\beta_1)}.
$$
Since $\beta_1 \in (n-2,\beta)$, letting $N \to \infty$ we conclude that
$$
\mathcal H_\infty^{\beta}(K\cap B_i^{(0)} )=0\qquad \mbox{for all } i=1,2,\dots,M.
$$
This proves that
$\mathcal H_{\infty}^{\beta}(K)=0$ and therefore $\mathcal H^\beta(K)=0$ (see
\cite[Section 1]{Sim83}), concluding the proof.
\end{proof}

We will also use the following basic result about Hausdorff measures.
We refer to \cite[2.10.19(2)]{Fed69} and \cite[Lemma 3.5]{AlessioJoaquim} for a proof of such result; see also \cite[Lemma 2.4]{Whi97}.

\begin{lemma} \label{lem:GMT3}
Let  $E\subset \R^n$ be a set satisfying $\HH^\beta_\infty(E) >0$ for some $\beta\in (0,n]$.
Then:
\begin{enumerate}
\item[(a)] For $\HH^\beta$-almost every point $x_\circ \in E$, there is a sequence $r_k\downarrow 0$  such that
\begin{equation}\label{denpt}
\lim_{k\to \infty} \frac{\HH^\beta_\infty(E\cap B_{r_k}(x_\circ) )}{r_k^\beta} \ge c_{n,\beta}>0,
\end{equation}
where $c_{n,\beta}$ is a constant depending only on $n$ and $\beta$.  We call these points ``density points''.
\item[(b)] Assume that $0$ is a  ``density point'', let $r_k\downarrow 0$ be a sequence along which \eqref{denpt} holds, and define the ``accumulation set'' for $E$ at $0$ along $r_k$ as
\[
\mathcal A=\mathcal A_{E} := \big\{z\in \overline{B_{1}} \ : \, \exists\,(z_\ell)_{\ell\ge 1},(k_\ell)_{\ell \ge 1} \text{ s.t.  $z_{\ell}\in r_{k_\ell}^{-1}E \cap B_{1}$ and  $z_{\ell}\to z$ } \big\}.
\]
Then $\HH_\infty^\beta(\mathcal A) >0.$
\end{enumerate}
\end{lemma}

The last main result of this section is the following covering-type result that will play a crucial role in the understanding of the generic size of the singular set, and in particular in the proof of Theorem~\ref{thm-Schaeffer-intro}.
Notice that, when $k=1$, $\beta$ is an integer, and $\pi_1(E)$ is $\beta$-rectifiable, then the result follows from the coarea formula; see also Eilenberg's inequality \cite[13.3]{BZ}.

\begin{proposition}\label{prop:GMT4}
Let $E\subset \R^n\times [-1,1]$, let $(x,t)$ denote a point in  $\R^n\times [-1,1]$, and let $\pi_1:(x,t)\mapsto x$ and $\pi_2:(x,t)\mapsto t$ be the standard projections.
Assume that for some  $\beta\in (0,n]$ and $s> 0$ we have:
\begin{itemize}
\item  $\HH^\beta\big(\pi_1(E)\big) < +\infty$;

\item For any $(x_\circ,t_\circ) \in E$ there exists a modulus of continuity  $\omega_{x_\circ, t_\circ}:\R^+\to \R^+$   such that
\[
\big\{ (x,t)\in \R^n\times[-1,1] \ :\   t-t_\circ>  \omega_{x_\circ, t_\circ}(|x-x_\circ|) |x-x_\circ|^{s} \big\}\cap E = \varnothing.
\]
\end{itemize}
Then:
\begin{enumerate}
\item[(a)]  If $\beta\le s$, we have $\HH^{\beta/s} \big( \pi_2(E) \big) =0$.

\item[(b)]  If $\beta> s$, for $\HH^1$-a.e. $t\in \R$ we have $\HH^{\beta-s}\big(  E \cap \pi_2^{-1}(\{t\}) \big) =0$.
\end{enumerate}
\end{proposition}

\begin{proof}
Fix $\ep>0$ be arbitrarily small.
We decompose $E = \bigcup_{i\ge 1} E_i$ as

\[
E_1 : = \big\{ (x_\circ,t_\circ) \in E \ :\  \omega_{x_\circ,t_\circ}(1) <\ep   \big \},
\]

\[
E_i  : = \big\{ (x_\circ,t_\circ) \in E \ :\  \omega_{x_\circ,t_\circ}(2^{-i+1} ) <\ep \le  \omega_{x_\circ,t_\circ}(2^{-i+2} )  \big \}, \quad \mbox{for } i\ge 2. 
\]
Fix $i \geq 1$ and note that, if $(x_1, t_1)$ and $(x_2, t_2)$ belong to $E_i$, then
\begin{equation}\label{XR}
\{ (x,t)\in B_{2^{-i+1}(x_j)} \times(-1,1) \ :\   t-t_j> \ep |x-x_j|^{s} \}\cap E = \varnothing, \quad j=1,2.
\end{equation}
Hence, by triangle inequality,
\begin{equation}\label{aiaiai}
x_1, x_2\in E_i, \ |x_1-x_2|\le 2^{-i} \quad \Rightarrow \quad  |t_1-t_2| \le \ep |x_1-x_2|^{s}.
\end{equation}
In particular, since the sets $\{E_i\}$ give a partition of $E$, it follows from \eqref{aiaiai} that the projection $\pi_1 :E \to \R^n$ is injective, and thus the sets $\pi_1(E_i)$ are disjoint.

Now, by definition of  $\HH^\beta \big(\pi_1(E_i)\big)$,  there is  countable collection of balls $\{B_\ell\}$  such that $\pi_1(E_i)\subset \cup_\ell B_\ell$, with
\begin{equation}
\label{eq:Haus Ei}
{\rm diam }(B_\ell) <  2^{-i}  \quad \mbox{and}\quad \sum_\ell {\rm diam }(B_\ell)^{\beta} < \HH^\beta \big(\pi_1(E_i)\big) +2^{-i}.
\end{equation}
Then, thanks to \eqref{aiaiai}, we see that  $E_i$ can be covered by the family of cylinders
\[
 \mathcal F_i:=\big\{ C_{\ell}: =  B_{\ell} \times ( t_\ell - \ep \,{\rm diam }(B_\ell)^{s}  , t_\ell + \ep\, {\rm diam }(B_\ell)^{s} ) \big\}
  \]
  for some suitable $t_\ell\in (-1,1)$.

\smallskip

Let us show $(a)$. Since $\{\pi_2(C_\ell)\}$ is a covering of $\pi_2(E_i)$ made of intervals of length  $2\ep {\rm diam }(B_\ell)^{s}$,  we have
\[
\HH^{\beta/s}_\infty(\pi_2(E_i)) \le  (2\ep)^{\beta/s}  \sum_\ell {\rm diam }(B_\ell)^{\beta} \leq (2\ep)^{\beta/s} \big(\HH^\beta \big(\pi_1(E_i)\big) +2^{-i}\big).
\]
Summing over $i\geq 1$ we obtain 
\[\HH^{\beta/s}_\infty(\pi_2(E)) \le  (2\ep)^{\beta/s}\big(\HH^\beta \big(\pi_1(E)\big) +1\big),\] 
and  (a) follows from the arbitrariness of $\ep$.

\smallskip

To show (b), following the same notation as above, we define the function
\[
N_i(t, j) = \#\big\{ C_{\ell} \in  \mathcal F_i\ : \   {\rm diam }(B_\ell) \in (2^{-j-1}, 2^{-j}), \ t \in  ( t_\ell - \ep {\rm diam }(B_\ell)^{s}  , t_\ell + \ep {\rm diam }(B_\ell)^{s} )   \big \}.
\]
Let $\mathcal I_{i,j}$ denote the set of indices $\ell$ such that $C_\ell \in \mathcal F_i$ and ${\rm diam }(B_\ell) \in (2^{-j-1}, 2^{-j})$.
Then we can rewrite $N_i(t,j)$ as
$$
N_i(t,j)=\sum_{\ell \in \mathcal I_{i,j}}\chi_{( t_\ell - \ep {\rm diam }(B_\ell)^{s}  , t_\ell + \ep {\rm diam }(B_\ell)^{s} )}(t).
$$
Hence, integrating over $[-1,1]$ we get
$$
\int_{-1}^1 N_i(t,j)\,dt \leq \sum_{\ell \in \mathcal I_{i,j}}2\ep {\rm diam }(B_\ell)^{s} \leq 2\ep (2^{-j})^s \,\# \mathcal I_{i,j},
$$
therefore, multiplying this estimate by $(2^{-j})^{\beta -s}$ and summing over $j$, we obtain (recall \eqref{eq:Haus Ei})
\begin{equation}\label{asiobdfsnion}
\begin{split}
\int_{-1}^1  \sum_j (2^{-j})^{\beta -s}  N_i(t, j)  \,dt   = 2 \ep \sum_j (2^{-j})^{\beta}  \,\# \mathcal I_{i,j}
  \le 2^{1+\beta} \ep\sum_{C_\ell \in \mathcal F_i} {\rm diam }(B_\ell)^{\beta} \le 2^{1+\beta}\ep \big(\HH^\beta \big(\pi_1(E_i)\big) +2^{-i}\big) .
\end{split}
\end{equation}
We now consider the functions $f_{i,\ep}(t):=\sum_j (2^{-j})^{\beta-s}  N_i(t, j)$ (note that the covering used to define $N_i(t,j)$ depends on $\epsilon$), and $f_\ep(t):=\sum_i f_{i,\ep}(t)$.
Then, summing \eqref{asiobdfsnion} over $i$, we have
$$
\int_{-1}^1f_\ep(t)\,dt \leq 2^{1+\beta}\ep \big(\HH^\beta \big(\pi_1(E)\big) +1\big),
$$
and it follows  by Chebyshev inequality
 $$
 \HH^1(X^\ep)\leq 2^{1+\beta}\ep^{1/2} \big(\HH^\beta \big(\pi_1(E)\big) +1\big),\qquad \text{where }X^\ep:= \{ t \in  (-1,1) \, :\    f_\ep(t)>\ep^{1/2}  \}.
 $$
Set $X:=  \cap_{M= 1}^\infty X_M$, where $X_M:=\cup_{k=M}^\infty X^{2^{-2k}}$.
Then
$$
\mathcal H^1(X_M)\leq \sum_{k=M}^\infty 2^{1+\beta}2^{-k}\big(\HH^\beta \big(\pi_1(E)\big) +1\big) \leq  2^{1+\beta}2^{1-M}\big(\HH^\beta \big(\pi_1(E)\big) +1\big),
$$
therefore $\mathcal H^1(X)=0$.

Also, for any  $t\in [-1,1]\setminus X$, there exists $M_t$ such that $t \in [-1,1]\setminus X_M \subset [-1,1]\setminus X^{2^{-2M}}$ for any $M \geq M_t$.
Therefore, considering the covering associated to $\ep=2^{-2M}$,  we get
\begin{equation*}
\begin{split}
\HH^{\beta-s}_\infty\big(\pi_1(E)\cap\pi_2^{-1}(\{t\})\big)  &\leq \sum_i \HH^{\beta-s}_\infty\big(\pi_1(E_i)\cap\pi_2^{-1}(\{t\})\big)\\
&  \leq
\sum_i \sum_j (2^{-j})^{\beta-s}  N_i(t, j) = f_{2^{-2M}}(t)\leq 2^{-M}\qquad \forall\, M \geq M_t.
\end{split}
\end{equation*}
This proves that $\HH^{\beta-s}_\infty\big(\pi_1(E)\cap\pi_2^{-1}(\{t\})\big) =  0$ for all $t\in  [-1,1]\setminus X$, as wanted.
\end{proof}

As an immediate consequence of Proposition~\ref{prop:GMT4}, we get:

\begin{corollary}\label{prop:GMT4b}
Let $E\subset \R^n\times [-1,1]$, let $(x,t)$ denote a point in  $\R^n\times [-1,1]$, and let $\pi_1:(x,t)\mapsto x$ and $\pi_2:(x,t)\mapsto t$ be the standard projections.
Assume that, for some  $\beta\in (0,n]$ and $s> 0,$ we have:
\begin{itemize}
\item  ${\rm dim}_\HH\big(\pi_1(E)\big) \le \beta$;

\item For all $(x_\circ,t_\circ) \in E$ and $\ep>0$, there exists  $\varrho= \varrho_{x_\circ, t_\circ, \ep}>0$   such that
\[
\big\{ (x,t)\in B_\varrho(x_\circ)\times[-1,1] \ :\   t-t_\circ>   |x-x_\circ|^{s-\ep} \big\}\cap E = \varnothing.
\]
\end{itemize}
Then:
\begin{enumerate}
\item[(a)]  If $\beta< s$,  we have ${\rm dim}_\HH \big( \pi_2(E) \big) \le \beta/s$.

\item[(b)]  If $\beta\ge s$, for $\HH^1$-a.e. $t\in \R$ we have ${\rm dim}_\HH\big(  E \cap \pi_2^{-1}(\{t\}) \big) \le \beta-s$.
\end{enumerate}
\end{corollary}

\section{Dimension reduction results} \label{sec:DRB}

This section is concerned with bounding the Hausdorff dimension of the differences of the subsets of $\mathbf{\Sigma}_{n-1}$ defined in \eqref{eq:Sigmaall}.
Note that we have the chain of inclusions
\begin{equation}\label{INCLUSI}
\mathbf{\Sigma}\supset \mathbf{\Sigma}_{n-1} \,\supset\, \mathbf{\Sigma}_{n-1}^{\ge 3} \,\supset\, \mathbf{\Sigma}_{n-1}^{3rd}\,=\, \mathbf{\Sigma}_{n-1}^{> 3}\,\supset\,\mathbf{\Sigma}_{n-1}^{\ge 4} \,\supset\,\mathbf{\Sigma}_{n-1}^{4th} \,=\, \mathbf{\Sigma}_{n-1}^{>4}  \,\supset\, \mathbf{\Sigma}_{n-1}^{\ge 5-\zeta},
\end{equation}
where the two equalities in such chain of inclusions follow from Propositions~\ref{lem-Sigma4th=Sigma>4} and~\ref{prop:nobadpoints}.

For $0\leq m\leq n-2$, we simply consider the sets
\[\qquad\qquad\qquad \mathbf{\Sigma}_m \,\supset\, \mathbf{\Sigma}_m^{\ge 3} \,=\, \mathbf{\Sigma}_m^{3rd}, \qquad 0\leq m\leq n-2,\]
as this suffices for our purposes.
Recall that, by Proposition~\ref{prop:E2B2}(a), we have $\mathbf{\Sigma}_m\setminus\mathbf{\Sigma}_m^a=\mathbf{\Sigma}_{m}^{\ge3}=\mathbf{\Sigma}_{m}^{3rd}$.

\smallskip

Our goal is to show that ${\rm dim}_{\HH}(\pi_1 (\mathbf{\Sigma}\setminus \mathbf{\Sigma}_{n-1}^{\ge 5-\zeta}) ) \le n-2$ for any $\zeta\in(0,1)$, where $\pi_1$ denotes the canonical projection $\pi_1:(x,t)\mapsto x$.
For this, using the tools developed in the previous sections, in the next lemmas we bound the size all the differences between consecutive sets of the previous chain of inclusions.

\begin{proposition}\label{prop:EG2B4}
Let $u\in C^0\big(\overline{B_1}\times [-1,1]  \big)$ solve \eqref{eq:UELL+t}. 
Then:

{\rm (a)} $\dim_{\HH}\big(\pi_1(\mathbf{\Sigma}^a_m) \big)\le m-1$ for $1\le m\le n-2$   ($\pi_1(\mathbf{\Sigma}^a_m)$ is discrete if $m=1$).

{\rm (b)} $\dim_{\HH}\big(\pi_1(\mathbf{\Sigma}_{n-1}^{<3} )\big)\le n-3$  ($\pi_1(\mathbf{\Sigma}_{n-1}^{<3})$ is countable if $n=3$).

{\rm (c)} For any $\varrho\in(0,1)$:
\begin{itemize}
\item[-] if $m \leq n-2$ then $\pi_1\big(\mathbf{\Sigma}_{m}\setminus \mathbf{\Sigma}^a_m\big) \cap \overline{B_\rho}$ is covered by a $C^{1,1}$ $m$-dimensional manifold;
\item[-] $\pi_1\big(\mathbf{\Sigma}_{n-1}^{\ge3} \big) \cap \overline{B_\rho}$ is covered by a $C^{1,1}$ $(n-1)$-dimensional manifold.
\end{itemize}
\end{proposition}

\begin{proof}
(a) We need to prove that, for any $\beta>m-1$, the set $\pi_1(\mathbf{\Sigma}^a_m) $ has zero $\HH^\beta$ measure. Assume by contradiction that
\[ \HH^\beta \big( \pi_1(\mathbf{\Sigma}^a_m) \big) >0.
\]
Then,  by Lemma~\ref{lem:GMT3}, there exists a point $(x_\circ,t_\circ)\in \mathbf{\Sigma}^a_m$ ---which we assume for simplicity to be $(x_\circ,t_\circ)=(0,0)$---, a sequence $r_k\downarrow 0$, and a set $\mathcal A\subset \overline{B_{1}}$ with
$\HH^\beta \big( \mathcal A \big) >0$, such that for every point $y\in \mathcal A$ there is a sequence $(x_k, t_k)$ in  $\mathbf{\Sigma}^a_m$ such that $x_k/r_k \to y$.

Let $w=u(\,\cdot\,, 0)-p_2$, $w_r= w(r\,\cdot\,)$, $\tilde w_r= w_r/H(1,w_r)^{1/2}$. Then, thanks to Proposition~\ref{prop:E2B2}, up to extracting a subsequence we have
\begin{equation}\label{convergence11}
\tilde w_{r_k} \rightharpoonup q\quad \mbox{in }W^{1,2}_{\rm loc}(\R^n),
\end{equation}
where $q$ is $\lambda^{2nd}$-homogeneous harmonic function.  By definition of  $\mathbf{\Sigma}^a_{m}$ we have $\lambda^{2nd}=2$, and thus $q$ is a quadratic harmonic polynomial  satisfying \eqref{howisD2q}.

Thanks to Lemma~\ref{lem:EG2Bm} we have $\mathcal A \subset \{q=0\}\cap \{p_2=0\}$. Therefore, since $\HH^\beta \big( \mathcal A \big) >0 $, the polynomial $q$ vanishes in a subset of dimension $\beta>m-1$ of the $m$-dimensional linear space $\{p_2=0\}$. The only possibility is that $q\equiv 0$ on $\{p_2=0\}$,
and then \eqref{howisD2q} implies $q\equiv 0$; a contradiction since $H(1, q)=1$.

We note that in the case $m=1$ the same proof gives that $\mathbf{\Sigma}_1^a$ cannot have accumulation points,  i.e.,  it must be a discrete set.

\smallskip

(b) We apply  Proposition~\ref{prop:GMT2} to the set $\pi_1(\mathbf{\Sigma}_{n-1}^{<3})$ with the function $f:\pi_1(\mathbf{\Sigma}_{n-1}^{<3}) \rightarrow [0,\infty)$ defined by
$$
f(x_\circ)  := \phi\big(0^+, u(\,\cdot\, , \tau (x_\circ))  -p_{2,x_\circ, \tau (x_\circ)} \big) ,
\qquad \text{with}\quad
\tau (x_\circ)  := \min \big\{t\in [-1,1]  \,: \, ( x_\circ, t)\in \mathbf{\Sigma} \big\}.
$$
Note that, by Lemma~\ref{lem:EG2B1bis}(c), we have $\phi\big(0^+, u(\,\cdot\, , t)  -p_{2,x_\circ, t} \big) = f(x_\circ)$ for every $t$ such that $(x_\circ, t)\in \mathbf{\Sigma}$.
Also, by Proposition~\ref{prop:E2B2} (b) and the  definition of  $\mathbf{\Sigma}_{n-1}^{<3}$, we have $f(x_\circ)\in  [2+\alpha_\circ, 3).$

To obtain the result, thanks to Proposition~\ref{prop:GMT2}, it suffices to show the following property: for all  $x_\circ\in \pi_1(\mathbf{\Sigma}_{n-1}^{<3})$ and for all $\ep >0$ there exists $\varrho= \varrho(x_\circ,\ep) >0$ such that 
\[
B_r(x_\circ) \cap \pi_1(\mathbf{\Sigma}_{n-1}^{<3}) \cap f^{-1}([f(x_\circ)-\varrho,f(x_\circ)+\varrho]) \subset \{y \ :\  {\rm dist}(y, \Pi_{x_\circ,r}) \le \ep r\} \quad \forall \,r\in (0, \varrho),
\]
where $\Pi_{x_\circ,r}$ is a  $(n-3)$-dimensional plane passing through $x_\circ$.

With no loss of generality we can assume that  $(x_\circ, t_\circ) =(0,0)$, and we prove this statement by contradiction. If such $\varrho>0$ did not exist  for some $\ep >0$, then we would have sequences $r_k \downarrow 0$ and $x^{(j)}_k \in \pi_1(\mathbf{\Sigma}_{n-1}^{<3})\cap B_{r_k}$, $1\le j \le n-2$, such that 
\[
y^{(j)}_k := x^{(j)}_k/ r_k \to  y^{(j)}_\infty \in \overline{B_1},\qquad \dim\big({\rm span}( y^{(1)}_\infty, y^{(2)}_\infty, \dots, y^{(n-2)}_\infty)\big) =n-2,\qquad |f(x^{(j)}_k) - f(0)|\downarrow 0.
\]
Let $w=u(\,\cdot\,, 0)-p_2$, $w_r= w(r\,\cdot\,)$, $\tilde w_r= w_r/H(1,w_r)^{1/2}$. It follows by Proposition~\ref{prop:E2B2} that  \eqref{convergence11} holds,  where  $q$ is a $\lambda^{2nd}$-homogeneous solution of the Signorini problem \eqref{ETOP}. Also, since we are supposing that $(0,0)\in \mathbf{\Sigma}^{<3}_{n-1}$, we have $\lambda^{2nd}\in [2+\alpha_\circ, 3)$.

Applying then Lemma~\ref{lem:EG2Bn-1} to the sequences $\big(x^{(j)}_k,\tau(x^{(j)}_k)\big)$ we deduce that $q$ is translation invariant in the  $n-2$ independent directions
\[y^{(1)}_\infty, y^{(2)}_\infty, \dots, y^{(n-2)}_\infty\in \{p_2=0\}.\]
 As a consequence $q$ is a two dimensional $\lambda^{2nd}$-homogeneous solution of Signorini, with $\lambda^{2nd}\in [2+\alpha_\circ, 3)$. 
However, it follows from Lemma~\ref{lemAP:2D} that 2D homogeneous solutions of Signorini have homogeneities $\{1,2,3,4, \dots\}\cup \{1+\frac 1 2 , 3 + \frac 1 2 ,5 + \frac 1 2, 7, + \frac 1 2, \dots\}$, impossible.

 Note finally that, when $n=3$,  the same argument (but using  Lemma~\ref{lem:GMT1} in place of  Proposition~\ref{prop:GMT2}) implies that $\mathbf{\Sigma}_{n-1}^{<3}$ is at most countable.

(c) We prove the statement for the maximal stratum $\mathbf{\Sigma}_{n-1}^{\ge3}$; the proof for  $\mathbf{\Sigma}_{m}\setminus \mathbf{\Sigma}_{m}^a=\mathbf{\Sigma}_{m}^{\ge3}$ is analogous.

Given $x_\circ \in \pi_1(\mathbf{\Sigma}_{n-1}^{\ge3})$, set $P_{x_\circ} := p_{2,x_\circ, \tau(x_\circ)}(\,\cdot\,-x_\circ)$.
We claim that, for every pair $x_\circ, x \in \pi_1(\mathbf{\Sigma}_{n-1}^{\ge3})\cap \overline{B_\varrho}$, we have
\begin{equation}\label{ass.whitney}
|D^k P_{x_\circ}(x)  -  D^k P_{x}(x)|  \le C |x-x_\circ|^{3-k} \quad \mbox{for } k=0,1,2.
\end{equation}
Indeed, note that for all $\hat x\in \pi_1(\mathbf{\Sigma}_{n-1}^{\ge 3})\cap \overline{B_\varrho}$ we have
$\phi(0^+u(\hat x +\,\cdot\, ,  \tau(\hat x)) - p_{2, \hat x, \tau(\hat x)})\ge 3$.  Thus, by Lemma~\ref{lem:EMF3},
\[
\big\| u(\hat x +\,\cdot\, ,  \tau(\hat x)) - p_{2, \hat x, \tau(\hat x)}\big\|_{L^\infty(B_r)} \le C(n, \varrho) r^3 \qquad \forall \,r\in (0, {\textstyle \frac{1- \varrho}{2} }),
\]
therefore, applying this bound both to $\hat x=x_\circ$ and $\hat x=x$, we get
\[
|u(\, \cdot\,,  \tau(x_\circ)) -  P_{x_\circ} | \le C r^3 \quad \mbox{in } B_r(x_\circ)
\qquad 
\text{and}
\qquad
|u(\, \cdot\,, \tau(x) ) -  P_{x} | \le C r^3 \quad \mbox{in } B_r(x).
\]
Choosing $r = 2|x-x_\circ|$, and assuming without loss of generality that  $\tau(x_\circ)\le  \tau(x)$, since $u(\, \cdot\,, \tau(x_\circ)) \le u(\, \cdot\,, \tau(x))$   we obtain
\[
 P_{x_\circ} -   P_{x} \le Cr^3 + u(\, \cdot\,, \tau(x_\circ))- u(\, \cdot\,, \tau(x)) \le Cr^3 \quad \mbox {in }B_r(x_\circ) \cap B_r(x).
\]
Noticing that $P_{x_\circ} -   P_{x}$ is a harmonic quadratic polynomial that vanishes at some point $\hat x$  in the segment  joining $x_\circ$ to $x$, as a consequence of the above upper bound we easily deduce that 
\[
 \|P_{x_\circ} -   P_{x}\|_{ L^\infty(B_{4r}(\hat x))} \le Cr^3,
\]
and since the $L^\infty(B_1)$ and the  $C^3(B_1)$ norm are equivalent on space of quadratic polynomials, \eqref{ass.whitney} holds.

Then, applying Whitney's extension theorem (see \cite{Fef09} or \cite[Lemma 3.10]{AlessioJoaquim}) we obtain a $C^{2,1}$ function $F: B_1\to \R$ satisfying
\[
F(x) = P_{x_\circ}(x) + O(|x_\circ -x|^3)
\]
for all $x_\circ \in \pi_1(\mathbf{\Sigma}_{n-1}^{\ge 3})\cap \overline{B_\varrho}$.
In particular $\pi_1(\mathbf{\Sigma}_{n-1}^{\ge 3}) \subset \{\nabla F = 0\}$ and $D^2 F(x_\circ) = D^2 p_{2,x_\circ, \tau(x_\circ)}(0)$ has rank one
(recall that  $(x_\circ, \tau(x_\circ)) \in \mathbf{\Sigma}_{n-1}$). Hence, by the implicit function theorem, we find that  $\{\nabla F = 0\}$ is a $C^{1,1}$ $(n-1)$-dimensional manifold in a neighborhood of $x_\circ$.
\end{proof}

As a consequence of the previous result, we get the following:

\begin{corollary}\label{cor:EG2B4bis}
Let $n=3$, let $u\in C^0\big(\overline{B_1}\times [-1,1]  \big)$ solve \eqref{eq:UELL+t}, and assume that $u(x,t')>u(x,t)$ whenever $t'>t$ and $u(x,t)>0$. 
Then,  for all  but a countable set of  singular points $(x_\circ,t_\circ)$, we have
\[
\| u(x_\circ + \,\cdot\, ,t_\circ )-p_{2,x_\circ, t_\circ}\|_{L^\infty(B_r)}\le C r^3 \qquad \forall \,r\in  {\textstyle \big(0, \frac{1-|x_\circ|}{2}\big)},
\]
where $C$ depends only on $n$ and $1-|x_\circ|$.
\end{corollary}

\begin{proof}
On the one hand, since $n=3$, Proposition~\ref{prop:EG2B4} implies that  $\mathbf \Sigma_m\setminus \mathbf \Sigma^{\ge 3}_m$ is a countable set for $m=0,1,2$.\footnote{Note that, as a consequence of \cite{C-obst2}, points in $\Sigma_0$ are always
isolated and $u$ is strictly positive in a neighborhood of them.}
On the other hand,  for $(x_\circ,t_\circ)\in    \mathbf \Sigma^{\ge 3}_m$,  setting $\rho= \frac{1-|x_\circ|}{2}$ and applying Lemma~\ref{lem:EMF3} to the function $w= \rho^{-2}u(x_\circ + \rho\,\cdot\, ,t_\circ ) -p_{2,x_\circ, t_\circ}$ (note that then $\phi(0^+, w)\ge 3$) we obtain  
\[
\bigg(\frac{\rho}{r}  \bigg)^6 \le \frac{H(w,\rho)}{H(w,r)}. 
\]
Therefore, using Lemma~\ref{lem:u-v}, we obtain
\[
\| w\|_{L^\infty(B_r)}  \le C(n) H(w, 2r)^{1/2} \le C(n)   \frac{H(w,\rho)^{1/2}}{\rho^3}\,r^3,
\]
as desired.
\end{proof}

\begin{proposition}\label{prop:recti}
Let $u\in C^0\big(\overline{B_1}\times [-1,1]  \big)$ solve \eqref{eq:UELL+t}.  Then $\pi_1\big(\mathbf{\Sigma}_{n-1}^{\ge3}\setminus \mathbf{\Sigma}^{3rd}_{n-1}\big)$ is contained in a countable union of $(n-2)$-dimensional Lipschitz manifolds.
\end{proposition}

\begin{proof}
For any $(x_\circ, t_\circ)\in \mathbf{\Sigma}_{n-1}^{\ge3}\setminus \mathbf{\Sigma}^{3rd}_{n-1}$ we apply~Lemma~\ref{dimred} to $u(x_\circ+ \,\cdot\,, t_\circ+ \,\cdot\,)$ to find a $(n-2)$-dimensional linear subspace  $L_{x_\circ, t_\circ}$ and
$\varrho_{x_\circ, t_\circ}>0$ such that 
\[
\pi_1\big(\mathbf{\Sigma}_{n-1}^{\ge 3}\big)\cap B_r(x_\circ) \subset x_\circ +  L_{x_\circ, t_\circ} +   B_{r} \qquad \mbox{for all }r \in (0,\varrho_{x_\circ, t_\circ}).
\]
Write $\mathbf{\Sigma}_{n-1}^{\ge3}\setminus \mathbf{\Sigma}^{3rd}_{n-1} = \bigcup_j E_j$, where 
\[
E_j := \big\{ (x_\circ, t_\circ)\in\mathbf{\Sigma}_{n-1}^{\ge3}\setminus \mathbf{\Sigma}^{3rd}_{n-1} :\   \varrho_{x_\circ, t_\circ}>1/j \big\} .
\]
Note that, for any $(x_\circ, t_\circ)\in E_j$, the set $\pi_1\big(\mathbf{\Sigma}_{n-1}^{\ge 3}\big)\cap B_{1/j}(x_\circ)$ is contained inside the cone 
\[
\biggl\{x \in  B_{1/j}(x_\circ)\,:\,{\rm dist}\bigg( \frac{x-x_\circ}{|x-x_\circ|} , x_\circ+L_{x_\circ, t_\circ}  \bigg)  \le 1\biggr\},
\]
which implies (by a classical geometric argument) that the set $\pi_1\big(E_j\big)\cap B_{1/2}$ can be covered by a 1-Lipschitz $(n-2)$-dimensional manifold.
The result follows by taking the union of these manifolds over all $j\in \mathbb N$.
\end{proof}

\begin{lemma}\label{lem:Sigmainterm}
Let  $u\in C^0\big(\overline{B_1}\times [-1,1]  \big)$ solve \eqref{eq:UELL+t}. Then:
\begin{enumerate}
\item[(a)]   ${\rm dim}_{\HH} \big(\pi _1(\mathbf{\Sigma}_{n-1}^{>3} \setminus \mathbf{\Sigma}^{\ge 4}_{n-1}) \big)\le n-2$ (countable if $n=2$).
\item[(b)]    ${\rm dim}_{\HH} \big(\pi _1(\mathbf{\Sigma}_{n-1}^{>4} \setminus \mathbf{\Sigma}^{\ge 5-\zeta}_{n-1}) \big)\le n-3$ (countable if $n=3$).
\end{enumerate}
\end{lemma}

\begin{proof}
(a) The proof  is similar to the one of Proposition~\ref{prop:EG2B4}(b).
Indeed, we apply  Proposition~\ref{prop:GMT2} to the set $\pi _1(\mathbf{\Sigma}_{n-1}^{>3} \setminus \mathbf{\Sigma}^{\ge 4}_{n-1})$ with the function $f:\pi _1(\mathbf{\Sigma}_{n-1}^{>3} \setminus \mathbf{\Sigma}^{\ge 4}_{n-1}) \rightarrow [0,\infty)$ defined as 
\begin{equation}\label{deftau}
f(x_\circ)  := \phi\big(0^+, u(\,\cdot\, , \tau (x_\circ))  -\anz_{x_\circ, \tau (x_\circ)} \big) ,
\qquad \text{where} \quad\tau (x_\circ)  := \min \big\{t\in [-1,1]  \,: \, ( x_\circ, t)\in \mathbf{\Sigma} \big\}.
\end{equation}
By Lemma~\ref{lem:EG2B1bis} (c) we have $\phi\big(0^+, u(\,\cdot\, , t)  -\anz_{x_\circ, t} \big) = f(x_\circ)$ for every $t$ such that $(x_\circ, t)\in \mathbf{\Sigma}$.
Moreover, by definition of  $\mathbf{\Sigma}_{n-1}^{>3} \setminus \mathbf{\Sigma}^{\ge 4}_{n-1}$, we have $f(x_\circ)\in  (3,4).$
Then, thanks to Proposition~\ref{prop:GMT2}, it is enough to show that for all  $x_\circ\in \pi_1(\mathbf{\Sigma}_{n-1}^{>3} \setminus \mathbf{\Sigma}^{\ge 4}_{n-1})$ and for all $\ep >0$ there exist $\varrho= \varrho(x_\circ,\ep) >0$, and a $(n-2)$-dimensional plane $\Pi_{x_\circ}$ passing through $x_\circ$,  such that
\[
B_r(x_\circ) \cap \pi_1(\mathbf{\Sigma}_{n-1}^{>3} \setminus \mathbf{\Sigma}^{\ge 4}_{n-1}) \cap f^{-1}([f(x_\circ)-\varrho,f(x_\circ)+\varrho]) \subset \{y \ :\  {\rm dist}(y, \Pi_{x_\circ}) \le \ep r\} \quad \forall \,r\in (0, \varrho).
\]
Assuming $(x_\circ, t_\circ) =(0,0)$ and arguing by contradiction, we find sequences $r_k \downarrow 0$ and $x^{(j)}_k \in \pi_1(\mathbf{\Sigma}_{n-1}^{>3} \setminus \mathbf{\Sigma}^{\ge 4}_{n-1})\cap B_{r_k}$, $1\le j \le n-1$, such that 
\[
y^{(j)}_k := x^{(j)}_k/ r_k \to  y^{(j)}_\infty \in \overline{B_1},\qquad \dim\big({\rm span}( y^{(1)}_\infty, y^{(2)}_\infty, \dots, y^{(n-1)}_\infty)\big) =n-1,\qquad |f(x^{(j)}_k) - f(0)|\downarrow 0.
\]
Setting $w=u(\,\cdot\,, 0)-\anz$, $w_r= w(r\,\cdot\,)$, $\tilde w_r= w_r/H(1,w_r)^{1/2}$, it follows by
 Proposition~\ref{prop:E3B5}(a)  that \eqref{convergence11} holds,    where  $q$ is a $\lambda^{3rd}$-homogeneous solution of the Signorini problem \eqref{ETOP} with $\lambda^{3rd}\in (3,4)$ (recall that $(0,0)\in \mathbf{\Sigma}_{n-1}^{>3} \setminus \mathbf{\Sigma}^{\ge 4}_{n-1}$).
Also, applying Lemma~\ref{lem:E3B2d} to the sequences $\big(x^{(j)}_k,\tau(x^{(j)}_k)\big)$, we deduce that $q$ is translation invariant in the  $n-1$ independent directions
\[y^{(1)}_\infty, y^{(2)}_\infty, \dots, y^{(n-1)}_\infty\in \{p_2=0\}.\]
Thus $q$ is a 1D $\lambda^{3rd}$-homogeneous solution of Signorini, with $\lambda^{3rd}\in (3,4)$,
and this is impossible by Lemma~\ref{lemAP:1D}.

Finally, when $n=2$, the same argument (using  Lemma~\ref{lem:GMT1} instead of  Proposition~\ref{prop:GMT2})  implies that $\mathbf{\Sigma}_{n-1}^{>3} \setminus \mathbf{\Sigma}^{\ge 4}_{n-1}$ is at most countable.

\smallskip

(b) The proof is completely analogous to  the one of part (a), using Lemmas~\ref{lem:E3B2b} and~\ref{lemAP:2D} instead of Lemmas~\ref{lem:E3B2d} and~\ref{lemAP:1D}.
\end{proof}

\begin{remark}\label{rem-7/2}
Notice that the difference between parts (a) and (b) in the previous Lemma comes from the fact that there exist 2D solutions to the Signorini problem with homogeneity $3+\frac12 \in (3,4)$, while there is no such solution with homogeneity in the interval $(4,5)$.
Hence, using the exact same proof as above, one can show that ${\rm dim}_{\HH} \big(\pi _1(\mathbf{\Sigma}_{n-1}^{>3} \setminus \mathbf{\Sigma}^{\ge 7/2}_{n-1}) \big)\le n-3$, where we define  $\mathbf{\Sigma}^{\ge 7/2}_{n-1}$ as the set at which $\phi(0^+,u-\anz)\ge 7/2$.
\end{remark}

With the aid of Lemmas~\ref{LOC.lem:GMT2bis} and~\ref{prop:GMT2bis}, we can next prove the following:

\begin{lemma}\label{lem:Sigma4}
Let $u\in C^0\big(\overline{B_1}\times [-1,1]  \big)$ solve  \eqref{eq:UELL+t}.
Then
\[ {\rm dim}_{\HH} \big(\pi _1(\mathbf{\Sigma}_{n-1}^{\ge 4} \setminus \mathbf{\Sigma}^{4th}_{n-1}) \big)\le n-2.\]
\end{lemma}

\begin{proof}
Define $\tau: \pi_1(\mathbf{\Sigma})\rightarrow [-1,1]$  as in \eqref{deftau} and note that, by Lemma~\ref{lem:EG2B1bis},  it is lower semicontinuous.

Hence, thanks to Lemma~\ref{prop:GMT2bis}, it suffices to prove that, for any given $\ep>0$ and $(x_\circ, \tau(x_\circ))\in\mathbf{\Sigma}_{n-1}^{\ge 4} \setminus \mathbf{\Sigma}^{4th}_{n-1}$, there exists $\varrho= \varrho(x_\circ,\ep)>0$ such that
\begin{equation}\label{goal123}
\mathbf{\Sigma} \cap \overline{B_r(x_\circ)} \times  [\tau(x_\circ),1)  \subset \big\{ x_\circ +r y\,:\,   y\in \boldsymbol S( p_{x_\circ,r} , \ep)  \big\}  \qquad \forall\,r \in (0 , \varrho),
\end{equation}
for some $p_{x_\circ,r}\in \mathcal P_{4,\ge}^{even}$. 
This follows from Lemma~\ref{lem:Sigma4aux} applied to  $u(x_\circ+\cdot\,,\,\tau(x_\circ) \,)$, since by monotonicity
\[
\mathbf{\Sigma} \cap \pi_2^{-1}\big( [\tau(x_\circ),1] \big)   \subset   \{ u(x_\circ+\cdot\,, \tau(x_\circ) ) =0\}.
\]
\end{proof}

We can finally prove the following:

\begin{theorem}\label{thm:E3B3}
Let $u\in C^0\big(\overline{B_1}\times [-1,1]  \big)$ solve \eqref{eq:UELL+t}.  There exists a set $\mathbf{\Sigma}^* \subset \mathbf{\Sigma}_{n-1} \subset \mathbf{\Sigma}$,  with ${\rm dim}_{\HH}\big(\mathcal \pi_1(\mathbf{\Sigma}\setminus \mathbf{\Sigma}^*)\big) \le n-2$, such that for any given $\ep>0$ the following holds:
\[
\big\|u(x_\circ+\,\cdot\,, t_\circ)-\anz_{x_\circ,t_\circ} -p_{4,x_\circ,t_\circ}  \big\|_{L^\infty(B_r)} \le C r^{5-\ep} \qquad  \forall \,r\in \big(0, {\textstyle \frac{1}{2}}\big),\, \forall\,(x_\circ,t_\circ)\in (\mathbf{\Sigma}^* \cap B_{1/2} ) \times(-1,1),
\]
where $C$ depends only on $n$ and $\ep$.
\end{theorem}

\begin{proof}
Recall the chain of inclusions \eqref{INCLUSI}. 
We have:
\vspace{1mm}

\begin{tabular}{llcc@{}c@{}l}
$\bullet$ &  Proposition~\ref{prop:EG2B4} (a) and (c)   \qquad 	&$\Rightarrow$& ${\rm dim}_{\HH} ($ &  $\pi_1 (\mathbf{\Sigma}\setminus \mathbf{\Sigma}_{n-1})$                      &  $)\le n-2,$ \vspace{1mm}
\\
$\bullet$ &  Proposition~\ref{prop:EG2B4} (b)         	&$\Rightarrow$&  ${\rm dim}_{\HH}($   &  $\pi_1(\mathbf{\Sigma}_{n-1} \setminus \mathbf{\Sigma}_{n-1}^{\ge3})$   &  $)\le n-3,$\vspace{1mm}
\\
$\bullet$ &  Proposition~\ref{prop:recti}     		&$\Rightarrow$&  ${\rm dim}_{\HH}($   &  $\pi_1(\mathbf{\Sigma}_{n-1}^{\ge3} \setminus \mathbf{\Sigma}_{n-1}^{3rd})$    &  $)\le n-2$,\vspace{1mm}
\\
$\bullet$ &  Remark~\ref{rmk:Sigma 3rd}  	&$\Rightarrow$&     & $ \pi_1(\mathbf{\Sigma}_{n-1}^{3rd} \setminus \mathbf{\Sigma}_{n-1}^{>3})$  &    $=\varnothing,$\vspace{1mm}
\\
$\bullet$ &  Lemma~\ref{lem:Sigmainterm}(a)   	&$\Rightarrow$&  ${\rm dim}_{\HH}($   &  $\pi_1(\mathbf{\Sigma}_{n-1}^{>3} \setminus \mathbf{\Sigma}_{n-1}^{\ge4})$   &    $)\le n-2,$\vspace{1mm}
\\
$\bullet$ &  Lemma~\ref{lem:Sigma4}       	&$\Rightarrow$&  ${\rm dim}_{\HH}($   &  $\pi_1(\mathbf{\Sigma}_{n-1}^{\ge 4} \setminus \mathbf{\Sigma}_{n-1}^{>4})$  &   $)\le n-2,$\vspace{1mm}
\\
$\bullet$ &  Lemma~\ref{lem:Sigmainterm}(b)     &$\Rightarrow$&  ${\rm dim}_{\HH}($ &  $\pi_1(\mathbf{\Sigma}_{n-1}^{>4} \setminus \mathbf{\Sigma}_{n-1}^{\ge 5-\zeta})$     &  $)\le n-3$.
\end{tabular}
\vspace{2mm}

\noindent Thus, if we define
\[\mathbf{\Sigma}^* := \bigcap_{\ep>0} \mathbf{\Sigma}_{n-1}^{\ge 5-\ep},\]
then  ${\rm dim}_{\HH}(  \pi_1(\mathbf{\Sigma} \setminus \mathbf{\Sigma}^*))\le n-2$.
Fix $\ep>0$, and
let $(x_\circ,t_\circ)\in (\mathbf{\Sigma}^* \cap B_{1/2} ) \times(-1,1)$.
By Lemmas~\ref{lem:E4B1} and~\ref{lem:E3B1c} applied to  $w: = u(x_\circ+\,\cdot\,, t_\circ)-\anz_{x_\circ,t_\circ} -p_{4,x_\circ,t_\circ}$ we obtain
\[
c \left(\frac{1}{r}\right)^{2 (5-\ep)} \le \frac{H\left(1/2, w\right) +  (1/2)^{2(5-\ep)}}{H(r,w) +r^{2(5-\ep)}},
\]
therefore
\[
H(r,w)^{1/2} \le  C \bigg( \int_{B_{1/2}}   \big( u(x_\circ+\,\cdot\,, t_\circ)-\anz_{x_\circ,t_\circ} -p_{4,x_\circ,t_\circ}\big)^2  +(1/2)^{2(5-\ep)}\bigg)^{ 1 /2 }  r^{5-\ep}
\le C(n, \ep) \,r^{5-\ep}.
\]
Combining this bound with the Lipschitz estimate in Lemma~\ref{lem:E3B5}, we easily conclude that
\[
\big\|u(x_\circ+\,\cdot\,, t_\circ)-\anz_{x_\circ,t_\circ} -p_{4,x_\circ,t_\circ}  \big\|_{L^\infty(B_r)}  = \big\|w \big\|_{L^\infty(B_r)} \le C r^{5-\ep} \qquad \forall\, 0<r<1/2.
\]
where $C$ depends only on $n$ and $\ep$.
\end{proof}

\section{Cleaning lemmas and proof of the main results} \label{sec:ESC}
Recall that, in all the previous sections, we only assumed that $u(\cdot,t)$ was nondecreasing in $t$.
Now, in order to conclude the proof of Theorem~\ref{thm-Schaeffer-intro},
we will assume the ``uniform monotonicity'' condition~\eqref{assumption}.
Note that condition \eqref{assumption} rules out the existence of connected components of the complement of the contact set that remain unchanged for some interval of times.

The first bound involves a barrier argument that will play an important role.

\begin{lemma}\label{lem:ESCbar}
Let  $u:B_1 \times (-1,1) \to [0,\infty)$ solve \eqref{eq:UELL+t}, with $(0,0)\in\mathbf{\Sigma}$ and $\{p_2=0\} \subset \{x_n=0\}$. 
Let $\mathfrak p$ be a polynomial satisfying  $\Delta \mathfrak p=1$. 
Assume that, for some $\beta\ge0$, we have
\[
| u(\,\cdot\,, 0) -\mathfrak p | \le Cr^\beta \quad \mbox{in } B_r \qquad  \forall  \,r\in (0,r_\circ),
\]
and define
\[
\psi(x) :=  -\sum_{i=1}^{n-1}x_i^2 +(n-1)x_n^2 + \frac 1 2,\qquad \psi^r(x) := \psi(x/r),  \qquad D_r := \partial B_r\cap \{\psi^r >0\} = \partial B_r\cap \big\{ |x_n| >{\textstyle \frac{r}{\sqrt{2n}} }\}.\]
Then, for all $t\ge 0$ we have
\[
u(\,\cdot\,,t) \ge  \mathfrak p+ \frac{\min_{D_r} [u(\,\cdot\,, t)-u(\,\cdot\,, 0)]}{\max_{\partial B_1} \psi} \psi^r -   Cr^{\beta}  \quad\mbox{in } B_r \qquad  \forall  \,r\in (0,r_\circ).
\]
\end{lemma}

\begin{proof}
It follows by our assumption on $u$ that
\begin{equation}\label{ahgon}
u(\,\cdot\,, 0)-\mathfrak p  \ge -C r^{\beta}  \qquad  \forall \,r\in \big(0, {\textstyle \frac{1}{2}}\big).
\end{equation}
Set
\[
v : = \mathfrak p  +   M \psi^r   - Cr^\beta, \qquad \text{with }M :=  \frac{ \min_{D_r}  [u( \, \cdot\, ,t) - u( \, \cdot\, ,0)] }{\max_{\partial B_1}  \psi}.
\]
We claim that
 $v\le u( \, \cdot\, ,t)$ on $\partial B_r$.
Indeed, since $t\ge 0$, it follows by \eqref{ahgon} that
\[
v\le  \mathfrak p   - Cr^{\beta}  \le u( \, \cdot\, ,0) \le u( \, \cdot\, ,t) \qquad \mbox{on }\partial B_r\cap \{\psi^r\leq 0\}.
\]
On the other hand, since $\max_{\partial B_1}  \psi = \max_{D_r} \psi^r$, we see that
$M\psi^r \le  \min_{D_r}  [u( \, \cdot\, ,r) - u( \, \cdot\, ,0)]$ on $\partial B_r$.
Hence
\[
v=   \mathfrak p +   M \psi^r  - Cr^{\beta}  \le u( \, \cdot\, ,0) + M \psi^r \le u( \, \cdot\, ,t) \qquad \mbox{on } D_r=\partial B_r\cap \{\psi^r>0\},
\]
and the  claim follows.

To conclude the proof it suffices to observe that, since $\psi^r$ is harmonic, we have $\Delta v=1\geq \chi_{ \{u( \, \cdot\, ,t)  >0\}} =
\Delta u( \, \cdot\, ,t)$.
Thus, combining the claim with the maximum principle, we conclude that
\[
\mbox{$v\le u( \, \cdot\, ,t)$ in $B_r$}.
 \]
\end{proof}

The second result gives us a bound on the speed at which $u$ increases in $t$ at singular points.
Note that this speed is much better in the lower strata $ \mathbf{\Sigma}_m$ with $m \leq n-2$ with respect to $\mathbf{\Sigma}_{n-1}$. This is one of the reasons why, in the previous sections, we needed to perform a very refined analysis at points in $\mathbf{\Sigma}_{n-1}$.

\begin{lemma}\label{lem:ESC1}
Let $u:B_1 \times (-1,1) \to [0,\infty)$ satisfy \eqref{eq:UELL+t} and  \eqref{assumption},
 with $(0,0)\in\mathbf{\Sigma}$ and $\{p_2=0\} \subset \{x_n=0\}$. Let $D_r$ be defined as in Lemma \ref{lem:ESCbar}.

{\rm (a)} If $(0,0)\in \mathbf{\Sigma}_m$ with $m\le n-2$,  then for all $\ep>0$ there exist $c_\ep,\rho_\ep>0$ such that
 \[\min_{D_r} [u(\,\cdot\,, t)-u(\,\cdot\,, 0)]
 \ge c_\ep r^{\ep} t,\qquad \forall\,r \in (0,\rho_\ep).\]

{\rm (b)} If $(0,0)\in \mathbf{\Sigma}_{n-1}$,  there exists $c,\rho>0$ such that 
\[\min_{D_r} [u(\,\cdot\,, t)-u(\,\cdot\,, 0)] \ge cr t,\qquad \forall\,r \in (0,\rho).\]
\end{lemma}

\begin{proof} (a)
Note that, by the uniform convergence of $r^{-2}u(r\,\cdot\,,0)$ to $p_2$,
given $\delta>0$ there exists $r_\delta>0$ such that
\[
\{u(\,\cdot\, ,0)=0\}\cap B_{r_\delta} \subset \mathcal C_\delta :=  \left\{ x \in \R^n \  :\  {\rm dist} \big( {\textstyle \frac{x}{|x|} } , \{p_2=0\} \big) \le \delta \right\}.
\]
Denote by $\tilde {\mathcal C}_\delta : = \R^n \setminus \mathcal C_\delta$ the complementary cone, and let
$\psi_{\delta}(x) := |x|^{\mu_\delta} \Psi_\delta(x/|x|)$, where  $\Psi_\delta \ge 0$ is the first eigenfunction of the spherical Laplacian
in $\tilde {\mathcal C}_\delta \cap \mathbb S^{n-1}$ with zero boundary conditions and $\mu_\delta$ is chosen so that the first eigenvalue is $\mu_\delta (n-2+\mu_\delta)$). 
In this way  $\psi_\delta$ is a $\mu_\delta$-homogeneous harmonic function vanishing on the boundary of $\tilde {\mathcal C}_\delta$.

Since $\dim(\{p_2=0\})=m\le n-2$, the set $\{p_2=0\}$ has zero capacity and so $\psi_{\delta}$ converges to a positive constant as $\delta\downarrow 0$. Thus  ${\mu_\delta}\downarrow 0$, and we can choose $\delta = \delta(\ep)>0$ such that $\mu_{2\delta}<\ep$.

We now observe that, for $t\ge 0$, we have
\[
\{u(\,\cdot\,, t)>0\} \supset  \{u(\,\cdot\,, 0)>0\}\supset  \tilde {\mathcal C}_\delta \cap B_{r_\delta} \supset   \tilde {\mathcal C}_{2\delta} \cap B_{r_\delta},
\]
and  $v:= u(\,\cdot\,, t)-u(\,\cdot\,, 0)$ is nonnegative and harmonic in $\{u(\,\cdot\,, t)>0\}$.
Note also that, by the maximum principle, every connected component of $\{u(\,\cdot\,, t)>0\}$ must have a part of its boundary on $\partial B_1$, and thus \eqref{assumption} and the Harnack inequality  (applied to a chain of balls) imply that
\[
v \ge c_\delta t \quad \mbox{in}\quad  \tilde{\mathcal C}_{2\delta} \cap \partial B_{r_\delta},\qquad c_\delta>0.
\]
Hence we can use the function
\[v' := \frac{c_\delta \psi_{2\delta}}{ \| \psi_{2\delta} \|_{L^\infty(\partial B_{r_{\delta}})}  }\,t\]
as lower barrier, and applying the maximum principle we obtain $v-v'\ge 0$  inside the domain  $\tilde{\mathcal C}_{2\delta} \cap B_{r_\delta}$.
Since $D_{r_\delta}\subset \tilde{\mathcal C}_{2\delta} \cap B_{r_\delta}$, this proves that 
\[
\min_{D_r} [u(\,\cdot\,, t)-u(\,\cdot\,, 0)]=\min_{D_r} v \geq \min_{D_r} v'=cr^{\mu_{2\delta}}t \geq  c r^\ep t\qquad \forall\,r\in (0, r_\delta).
\]
\smallskip

(b) After a rotation, we may assume that $\{p_2=0\}=\{x_n=0\}$.
By Propositions~\ref{prop:E2B2}  and~\ref{prop:E2B3}, we have that  $\{ u(\cdot ,0 )>0\} \supset \{|x_n|\leq C|x'|^{1+\alpha_\circ}\}$ in a neighborhood of the origin, where $x=(x',x_n)$ and $\alpha_\circ>0$.
In particular, there exists a $C^{1,\alpha_\circ}$ domain $\Omega$ contained inside  $\{ u(\cdot ,0 )>0\}$ and satisfying $0\in \partial\Omega$.
By monotonicity of $u$ in $t$, the same domain $\Omega$ is contained in $\{ u(\cdot ,t )>0\}$ for $t>0$.

Hence, the function $v:=u(\cdot ,t )-u(\cdot ,0 )$ is positive and harmonic in $\Omega$, and by assumption  \eqref{assumption} we have ---as in the proof of (a)--- that
$v \ge c_1 t>0$ in a small ball $B\subset\subset \Omega$.
Using Hopf's lemma in $C^{1,\alpha}$ domains, we deduce that $\partial_{x_n}v(0)\geq c_2t>0$, and the result follows.
\end{proof}

We can now prove the following key result:

\begin{lemma}\label{lem:ESC2}
Let $u:B_1 \times (-1,1) \to [0,\infty)$ satisfy \eqref{eq:UELL+t} and \eqref{assumption}, let $\alpha_\circ>0$ be given by Proposition~\ref{prop:E2B2}, and let $\mathbf{\Sigma}^*\subset \mathbf{\Sigma}_{n-1}$ be given by Theorem~\ref{thm:E3B3}. 

\vspace{1mm}

{\rm (a)} If $(0,0) \in \mathbf{\Sigma}_m^a$ and $m\le n-2$, then for all $\ep>0$ there exists   $\varrho>0$ such that
\[
\big\{ (x,t)\in B_\varrho \times (0,1) \ :\   t> |x|^{2-\ep} \big\} \cap \{u=0\} = \varnothing.
\]

{\rm (b)} If $(0,0) \in \mathbf{\Sigma}_m\setminus \mathbf{\Sigma}_m^a$, $m\le n-2$, then for all $\ep>0$ there exists   $\varrho>0$ such that
\[
\big\{ (x,t)\in B_\varrho \times (0,1) \ :\   t> |x|^{3-\ep} \big\} \cap \{u=0\} = \varnothing.
\]

{\rm (c)} If $(0,0)\in\mathbf{\Sigma}^{<3}_{n-1}$, then there exist $C,\varrho>0$ such that
\[
\big\{ (x,t)\in B_\varrho\times (0,1) \ :\   t>C |x|^{1+\alpha_\circ} \big\} \cap \{u=0\}  = \varnothing.
\]

{\rm (d)} If $(0,0)\in\mathbf{\Sigma}^{> 3}_{n-1}$,  then  there exist $\delta,\varrho>0$ such that
\[
\big\{ (x,t)\in B_\varrho\times (0,1) \ :\   t> |x|^{2+\delta} \big\} \cap \{u=0\}  = \varnothing.
\]

{\rm (e)} If $(0,0)\in\mathbf{\Sigma}^*$,  then  for all $\ep>0$  there exists $\varrho>0$ such that
\[
\big\{ (x,t)\in B_\varrho\times (0,1) \ :\   t> |x|^{4-\ep} \big\} \cap \{u=0\}  = \varnothing.
\]
\end{lemma}

\begin{proof}
After a rotation, we may assume $\{p_2=0\}\subset \{x_n=0\}$.
In all the following cases we will apply Lemma~\ref{lem:ESCbar} and use that  $\psi^r\ge \frac 1 4$ in $B_{r/2}$.

\smallskip

(a) By Lemma~\ref{lem:ESC1}(a) we have, for any $\ep>0$,
 \begin{equation} \label{agoisnbe}
\min_{D_r} [u(\,\cdot\,, t)-u(\,\cdot\,, 0)] \ge c_\ep r^{\ep/2} t.
 \end{equation}
Also, since $u$ is $C^{1,1}$, $|u(\,\cdot\,,  0)|\le C_0r^2$ in~$B_r$  for all $r\in (0,1/2)$.
Thus, by Lemma~\ref{lem:ESCbar} applied with $\mathfrak p\equiv 0$ and $\beta=2$,
 \begin{equation}\label{anboiafsnboasn}
u(\,\cdot\,,t) \ge  c_1  \min_{D_r}[u(\,\cdot\,, t)-u(\,\cdot\,, 0)] \psi^r -   C_0r^{2}  \qquad \textrm{in}\, B_r, \qquad    \forall  \,r\in (0,1/2).
\end{equation}
Since $\psi^r\ge \frac 1 4$ in $B_{r/2}$,
thanks to \eqref{agoisnbe} we deduce that
\[
u(\,\cdot\,,  t)>0 \quad \mbox{in }B_{r/2} \qquad  \mbox{for }  t \ge (r/2)^{2-\ep},
\]
therefore
\[
\{u=0\}\cap\ \{t> |x|^{2-\ep}\} = \varnothing.
\]
\smallskip

(b) Using again Lemma~\ref{lem:ESC1}(a), it follows that \eqref{agoisnbe} holds. 
Also, since $(0,0)\in\mathbf{\Sigma}_m\setminus \mathbf{\Sigma}_m^a$, it follows from Proposition~\ref{prop:E2B2}(a) that $\lambda^{2nd}\geq 3$. Hence Lemmas~\ref{lem:EMF3} and~\ref{lem:u-v} imply that $| u(\,\cdot\,  ,0) - p_2| \le C_0r^{3}$ in $B_r$, and therefore Lemma~\ref{lem:ESCbar} applied with $\mathfrak p\equiv p_2$ and $\beta=3$ yields
 \[
u(\,\cdot\,,t) \ge p_2+ c_1 \min_{D_r} [u(\,\cdot\,, t)-u(\,\cdot\,, 0)]  \psi^r -   C_0r^{3}     \qquad \textrm{in}\, B_r, \qquad \forall  \,r\in (0,1/2).
\]
Since $p_2\geq0$, one concludes as in the proof of (a).

\smallskip

(c) By Lemma~\ref{lem:ESC1}(b) we have
 \begin{equation} \label{asdhgoh}
\min_{D_r} [u(\,\cdot\,, t)-u(\,\cdot\,, 0)] \ge c r t.
 \end{equation}
Since at the maximal stratum the frequency is at least $2+\alpha_\circ$ (see Proposition~\ref{prop:E2B2}(b)),
using Lemmas~\ref{lem:EMF3} and~\ref{lem:u-v} we have $| u(\,\cdot\,  ,0) - p_2| \le C_0r^{2+\alpha_\circ}$ in $B_r$.
Therefore, it follows from Lemma~\ref{lem:ESCbar} applied with $\mathfrak p\equiv p_2$ and $\beta=2+\alpha_\circ$, that
 \[
u(\,\cdot\,,t) \ge  p_2+ c_1   \min_{D_r} [u(\,\cdot\,, t)-u(\,\cdot\,, 0)]  \psi^r -   C_0r^{2+\alpha_\circ}   \qquad \textrm{in}\, B_r, \qquad  \forall  \,r\in (0,1/2).
\]
Thus, since $\psi^r\ge \frac 1 4$ in $B_{r/2}$,
thanks to \eqref{asdhgoh} 
we obtain
\[
u(\,\cdot\,,  t)>0 \quad  \mbox{in } B_{r/2} \qquad \mbox {for } t \ge C_3 r^{1+\alpha_\circ}.
\]

 \smallskip

(d) Again,  \eqref{asdhgoh} holds as a consequence of Lemma~\ref{lem:ESC1}(b). 
Moreover, since $(0,0)\in\mathbf{\Sigma}^{>3}_{n-1}$, thanks to Lemma~\ref{lem:E3B5} we deduce that  $| u(\,\cdot\,  ,0) - \anz| \le C_0r^{3+2\delta}$ in $B_r$ for some $\delta>0$ (note that $\delta$ may depend on the point $(0,0)$). 
Therefore, Lemma~\ref{lem:ESCbar} applied with $\mathfrak p\equiv \anz$ and $\beta=3+2\delta$ yields
 \[
u(\,\cdot\,,t) \ge \anz+ c_1 \min_{D_r} [u(\,\cdot\,, t)-u(\,\cdot\,, 0)]  \psi^r -   C_0r^{3+2\delta}      \qquad \textrm{in}\, B_r, \qquad  \forall  \,r\in (0,1/2).
\]
Recalling that $\anz  \ge - \bar C|x|^5$, it follows from \eqref{asdhgoh} that,
for  $t> (r/2)^{2+\delta}$ and $r$ sufficiently small, 
\[
u(\,\cdot\,,t) \ge   \anz  + c_3 r t -   C_0r^{3+\delta}  \ge -\bar Cr^5 + c_3 r t  -C_0 r^{3+2\delta} >0\quad \mbox{in }B_{r/2}\cap \{u(\,\cdot\,,0) =0 \}.
\]
Since $u(\,\cdot\,,t)\geq u(\,\cdot\,,0)$, this proves the result.

\smallskip

(e) Again,  \eqref{asdhgoh} holds as a consequence of Lemma~\ref{lem:ESC1}(b). 
Moreover,  by Theorem~\ref{thm:E3B3}, for every $\ep>0$ we have   $| u(\,\cdot\,  ,0) - \anz -p_4| \le C_0r^{5-\ep/2}$ in $B_r$. 
Then, applying Lemma~\ref{lem:ESCbar}  with $\mathfrak p\equiv  \anz +p_4$ and $\beta=5-\ep/2$,
 \[
u(\,\cdot\,,t) \ge   \anz +p_4 + c_1\min_{D_r}  [u(\,\cdot\,, t)-u(\,\cdot\,, 0)] \psi^r -   C_0r^{5-\ep/2}   \qquad \textrm{in}\, B_r \qquad     \forall  \,r\in (0,1/2).
\]
Also,
\[
 \anz + p_4 \ge - \bar C|x|^5 \qquad \mbox{in }  \{u(\,\cdot\,,0) =0 \}\subset \{x_n \le C|x'|^2\}.
\]
Thus \eqref{asdhgoh} yields, for $t> r^{4-\ep}$ and $r$ small,
\[
u(\,\cdot\,,t) \ge   \anz +p_4 + c_3 r t -   C_0r^{5-\ep/2}  \ge -\bar Cr^5 + c_3 r t  -C_0 r^{5-\ep/2} >0\qquad \mbox{in }B_{r/2}
\cap \{u(\,\cdot\,,0) =0 \}.
\]
\end{proof}

The set $\mathbf{\Sigma}^{\ge 3}_{n-1}\setminus \mathbf{\Sigma}^{3rd}_{n-1}$ is treated separately in the following lemma.
Since in this case the 3rd order blow-up is not harmonic, the proof is more involved. In particular, instead of proving that there are no singular points in the ``future'' $t>0$, we show that they do not exists in the past.

\begin{lemma}\label{lem:cleaningpast3}
Let $u:B_1 \times (-1,1) \to [0,\infty)$ satisfy \eqref{eq:UELL+t} and \eqref{assumption}, with $(0,0) \in \mathbf{\Sigma}^{\ge 3}_{n-1}\setminus \mathbf{\Sigma}^{3rd}_{n-1}$. 
Then
\[
\big\{ (x,t)\in B_1\times (-1,0) \ :\   t<-\omega(|x|)|x|^{2} \big\} \cap \mathbf{\Sigma}^{\ge 3}_{n-1} = \varnothing,
\]
for some modulus a continuity $\omega:[0,\infty)\to [0,\infty)$.
\end{lemma}

\begin{proof}
Let $w = u(\,\cdot\,0)-p_2$, $w_r= w(r\,\cdot\,)$. Also, with no loss of generality we assume that $p_2= \frac 12 x_n^2$. 
By Proposition~\ref{prop-uniq-cubic-blowups} we have that
\begin{equation}\label{haiohio2h}
\|r^{-3}w_r -\tilde q\|_{L^\infty(B_4)} \le \delta(r)\downarrow 0 \quad \mbox{for }\quad  \tilde q(x) = |x_n| \bigg(\frac{a}{3} x_n^2 - x' \cdot A x'  \bigg) + x_n\bigg(\frac{b}{3} x_n^2 - x' \cdot B x'  \bigg),
\end{equation}
where $x=(x',x_n)$, and $A\in \R^{n-1}\times \R^{n-1}$ is symmetric, nonnegative definite, and has at least one positive eigenvalue.

Fix $\eta>0$, and assume by contradiction that there exists $r>0$ small and $t \leq -\eta r^2$ such that
$u(\,\cdot\,, t)$ has a singular point in $\mathbf\Sigma^{\ge 3}_{n-1}\cap B_{r}$.
Under this assumption, we claim that 
\begin{equation}\label{iugfdtdxjkhnj}
\bigg\{x_n+ \frac{b}3 x_n^2=x'\cdot Bx'\bigg\} \cap B_{2r} \subset \{u(\,\cdot\,, t/2)=0\},
\end{equation}
where $b$ and $B$ are given by \eqref{haiohio2h}.

Before proving the claim, we show that it leads to a contradiction.
Indeed, thanks to Proposition~\ref{prop:E2B3}, since we are assuming that $u(\,\cdot\,, t)$ has a singular point $x_r\in\Sigma ^{\ge 3}_{n-1}$ with $|x_r|\le r$, then for some $e_r\in \mathbb S^{n-1}$ we have
\begin{equation}\label{eq:trapped}
\{u(\,\cdot\,, t) =0\}\cap B_{\rho} \subset \bigg\{x\in B_\rho\,:\, |e_r\cdot (x -x_r)| \le C\rho^2 \bigg\}, \quad\mbox{for all  }\rho\in [r, 1].
\end{equation}
Note that the hypersurface $\big\{x_n+ b x_n^2=x'\cdot Bx'\big\}  \cap B_{2r}$ separates the ball $B_{2r}$ in two connected components $B_{2r}^+$ and $B_{2r}^-$.
Also, by monotonicity, \eqref{iugfdtdxjkhnj} and \eqref{eq:trapped} hold for $u(\,\cdot\,,t')$ for all $t' \in[t,t/2]$.
Hence, if we define
$$
u^+(x,t'):=\left\{
\begin{array}{ll}
u(x,t')&\text{in }B_{2r}^+,\\
0&\text{in }B_{2r}^-,
\end{array}
\right.
\qquad
u^-(x,t'):=\left\{
\begin{array}{ll}
0&\text{in }B_{2r}^+,\\
u(x,t')&\text{in }B_{2r}^-,
\end{array}
\right.
$$
then both $u^+(\,\cdot\,,t')$ and $u^-(\,\cdot\,,t')$ are solutions to the obstacle problem with a thick contact set at $0$.
Combining this information with \eqref{eq:trapped}, it follows by \cite{C-obst} that 
the free boundaries of $u^+(\,\cdot\,,t')$ and $u^-(\,\cdot\,,t')$ are uniformly smooth hypersurface inside $B_{3r/2}$,
for all $t' \in[t,t/2]$. In addition, by strict monotonicity, these hypersurfaces are disjoint for any $t'<t/2$.
Since the free boundary of $u(\,\cdot\,,t')$ is the union of these hypersurfaces, this proves 
 in particular that the free boundary of $u(\,\cdot\,,t)$ has no singular points, a contradiction.

Thus, we are left with proving~\eqref{iugfdtdxjkhnj}.
First of all we note that, by 
Lemma~\ref{lem:EG2B1bisbis}, we have 
\begin{equation}\label{anklabjklba}
e_r \to \boldsymbol e_n \quad \mbox{and} \quad 
 (r^{-1} x_r)\cdot e_r \to 0\qquad \mbox{as } r \downarrow 0,
\end{equation}
where $e_r$ is the unit vector appearing in \eqref{eq:trapped}.
Furthermore, by the classical barrier argument used in proof of Hopf's Lemma (see for instance \cite[Chapter 6.4.2]{Eva10}), it follows from \eqref{eq:trapped} that 
\begin{equation}
\label{eq:Hopf t}
u(\,\cdot\,, 0)- u(\,\cdot\,, t)  \ge c_1  |t|  \big( |e_r\cdot (x- x_r)|- C|x-x_r|^2\big)_+.
\end{equation}
Now, given $z' \in B_2'\subset \R^{n-1}$ and $c\ge0$, we define the function
\[
\phi_{z',c}(x) :=  
\bigg( \frac 1 {2r}  -n\bigg) \bigg(x_n + \frac{br} 3 x_n^2 - rx'\cdot Bx'\bigg)^2 + (x'-z')^2 +c. 
\]
Note that $\phi_{z',c}\ge c\ge 0$ and
\[
\Delta \phi_{z',c} = \frac 1 r  - 2n + O(r) + 2(n-1) \le \frac 1 r,\qquad \mbox{provided } 0<r\ll 1 .
\]
Also, since $A\ge 0 $ we have $\tilde q(x)\le -x_n\big( x' \cdot B x'  \big) +  C|x_n|^3$, therefore (recall \eqref{haiohio2h})
\[
 r^{-3}u(r x,0) - \frac{1}{2 r} x_n^2= r^{-3}w_r(x)  \le  \tilde q(x)+ \delta(r) \le -x_n\big( x' \cdot B x'  \big)  +  C|x_n|^3 + \delta(r).
\]
Thus, combining the bound above with \eqref{eq:Hopf t}, 
we get
\[
\begin{split}
r^{-3} u(r x,t) - \frac{1}{2 r} x_n^2 &\le    r^{-3} \big(u(r x,t)- u(rx,0) \big)  -x_n\big( x' \cdot B x'  \big)  +  C|x_n|^3 + \delta(r)\\
& \le  -  r^{-3} c_1  |t|  \big(|e_r\cdot (rx- x_r)|- C|rx- x_r|^2\big)_+ -x_n\big( x' \cdot B x'  \big)  +  C|x_n|^3 + \delta(r)
\\
& \leq  -  c_1  \eta  \big(|e_r\cdot (x- \hat x_r)|- Cr|x- \hat x_r|^2\big)_+ -x_n\big( x' \cdot B x'  \big)  +  C|x_n|^3 + \delta(r),
\end{split}
\]
where we used that $|t|\geq \eta r^2$ and we denoted $\hat x_r =: r^{-1}x_r \in B_1$.

Recalling \eqref{anklabjklba}, this implies that
\[
v(x) : = r^{-3}u(r x,t) \le  \frac{1}{2 r} x_n^2 - x_n\big( x' \cdot B x'  \big)   -c_1 \eta |x_n|    +  C|x_n|^3 + \theta(r),
\]
for some modulus of continuity $\theta(r)$.
On the other hand, for any $c\geq0$ we have
\[
\phi_{z',c}(x) \geq \frac{1}{2 r} x_n^2 -nx_n^2 - C|x_n|^3 - x_n\big( x' \cdot B x'  \big) +(x'-z')^2 +O(r).
\]
Let now $z: = (z',z_n)$ satisfy $z_n+ r\frac b  3z_n^2=rz'\cdot Bz'$, and consider a point
$x\in \partial B_s(z)$, where $0<r\ll s\ll 1$.
Then, since $|z_n|=O(r)$, we have $(x'-z')^2 = s^2 - x_n^2 + O(r)$, and therefore
\[
\phi_{z',c}(x) \geq \frac{1}{2 r} x_n^2 -(n+1)x_n^2 - C|x_n|^3 - x_n\big( x' \cdot B x'  \big) +s^2 +O(r) \qquad \textrm{on } \partial B_s(z).
\]
Since, for $r\ll s\ll1$,
\[c_1 \eta |x_n| +s^2 \ge Cx_n^2 + C|x_n|^3 +\theta(r) \qquad \textrm{for}\quad |x_n|\leq s, \]
we deduce that
\[
v(x) = r^{-3}u(r x,t) < \phi_{z',c}(x)  \qquad \textrm{on } \partial B_s(z)
\]
for all $c>0$.

Now, assume there exists $c_*>0$ be such that $\phi_{z',c_*}$ touches $v$ from above at some point $x_\circ\in \overline{B_s(z)}$. Since $v<\phi_{z',c_*}$ on $\partial B_s(z)$, the contact point is inside $B_s(z).$
Also, since $\Delta v= r^{-1}$ in $\{ u(r \, \cdot\, ,t)>0\}$  while $\Delta \phi_{z',c_*}\le r^{-1}$, 
it follows by the maximum principle that $x_\circ \not\in \{ u(r \, \cdot\, ,t)>0\}$.
Thus, 
\[
 0 = r^{-3}u(r x_\circ,t)   =  v(x_\circ) =  \phi_{z',c_*}(x_\circ)\geq c_*>0, 
\]
a contradiction. This proves that $v \leq \phi_{z',c}$ for all $c>0$, and letting $c\to 0$ we obtain
\[
 0 \le  r^{-3}u(r z,t) = v(z)  \le  \phi_{z',0}(z) = 0 .
\]
Since $z'\in  B_2'$ is arbitrary, this proves \eqref{iugfdtdxjkhnj}, and the lemma follows.
\end{proof}

We finally prove: 

\begin{theorem}\label{thm:mainell}
Let $u:B_1 \times (-1,1) \to [0,\infty)$ satisfy \eqref{eq:UELL+t} and \eqref{assumption}. 
Then:

\begin{enumerate}
\item[(a)] In dimension $n=2$ we have $\dim_\HH \big(  \pi_2(\mathbf{\Sigma})\big) \le 1/4$.

\item[(b)] In dimension $n=3$ we have $\dim_\HH \big(  \pi_2(\mathbf{\Sigma})\big) \le 1/2$.

\item[(c)] In dimensions $n\ge 4$,  for $\HH^1$-a.e. $t\in (-1,1)$ we have \[\HH^{n-4} \big(\mathbf{\Sigma}\cap  \pi_2^{-1}(\{t\}) \big) = 0.\]
\end{enumerate}
In particular, for $n\leq 4$, the singular set is empty for a.e. $t$.
\end{theorem}

\begin{proof}
First of all, as in the proof of Theorem~\ref{thm:E3B3}, we have the following:
\begin{itemize}
\item $\dim_\HH \big(  \pi_1(\mathbf{\Sigma}_m^a)\big) \le n-3$ for $0\leq m\leq n-2$; \vspace{1mm}

\item $\dim_\HH \big(  \pi_1(\mathbf{\Sigma}_m\setminus \mathbf{\Sigma}_m^a)\big) \le n-2$ for $0\leq m\leq n-2$;  \vspace{1mm}

\item $\dim_\HH \big(  \pi_1(\mathbf{\Sigma}_{n-1} \setminus \mathbf{\Sigma}_{n-1}^{\ge3})\big) \le n-3$; \vspace{1mm}

\item $\pi_1(\mathbf{\Sigma}_{n-1}^{\geq 3}\setminus \mathbf{\Sigma}_{n-1}^{3rd})$ is contained in a countable union of $(n-2)$-dimensional Lipschitz graphs;

\vspace{1mm}

\item $\pi_1(\mathbf{\Sigma}_{n-1}^{3rd}\setminus \mathbf{\Sigma}_{n-1}^{> 3})=\varnothing$; \vspace{1mm}

\item $\dim_\HH \big(  \pi_1(\mathbf{\Sigma}_{n-1}^{>3}\setminus \mathbf{\Sigma}^*)\big) \le n-2$;

\item $\dim_\HH \big(  \pi_1(\mathbf{\Sigma}^*)\big) \le n-1$. \vspace{1mm}

\end{itemize}
Furthermore, thanks to Lemmas~\ref{lem:ESC2} and~\ref{lem:cleaningpast3}, we have:
\begin{itemize}
\item In $\mathbf{\Sigma}_m^a$ for $0\leq m\leq n-2$, we can use Corollary~\ref{prop:GMT4b} with $\beta=n-3$ and $k=2$; \vspace{1mm}

\item  In $\mathbf{\Sigma}_m\setminus \mathbf{\Sigma}_m^a$ we can use Corollary~\ref{prop:GMT4b} with $\beta=n-2$ and $k=3$; \vspace{1mm}

\item In $\mathbf{\Sigma}_{n-1} \setminus \mathbf{\Sigma}_{n-1}^{\ge3}$ we can use Corollary~\ref{prop:GMT4b} with $\beta=n-3$ and $k=1+\alpha_\circ$;\vspace{1mm}

\item In $\mathbf{\Sigma}_{n-1}^{\geq 3}\setminus \mathbf{\Sigma}_{n-1}^{3rd}$ (up to taking a countable union, and up to reversing time) we can use Proposition~\ref{prop:GMT4} with $\beta=n-2$ and $k=2$;

\vspace{1mm}

\item In $\mathbf{\Sigma}_{n-1}^{>3}\setminus \mathbf{\Sigma}^*$ we can use Proposition~\ref{prop:GMT4} with $\beta=n-2$ and $k=2$; \vspace{1mm}

\item  In $\mathbf{\Sigma}^*$ we can use Corollary~\ref{prop:GMT4b} with $\beta=n-1$ and $k=4$. \vspace{1mm}

\end{itemize}
Hence, combining these information, we deduce that:
\begin{itemize}

\item  $\dim_\HH\big(\mathbf{\Sigma}_m^a\cap  \pi_2^{-1}(\{t\}) \big)\leq n-5$  for $\HH^1$-a.e. $t\in \R$; \vspace{1mm}

\item  $\dim_\HH\big((\mathbf{\Sigma}_m\setminus \mathbf{\Sigma}_m^a) \cap  \pi_2^{-1}(\{t\}) \big)\leq n-5$  for $\HH^1$-a.e. $t\in \R$; \vspace{1mm}

\item  $\dim_\HH\big(\mathbf{\Sigma}_{n-1}^{<3}\cap  \pi_2^{-1}(\{t\}) \big)\leq n-4-\alpha_\circ$  for $\HH^1$-a.e. $t\in \R$; \vspace{1mm}

\item $\HH^{n-4} \big((\mathbf{\Sigma}_{n-1}^{\geq 3}\setminus \mathbf{\Sigma}_{n-1}^{3rd})\cap  \pi_2^{-1}(\{t\}) \big) = 0$  for $\HH^1$-a.e. $t\in \R$;   \vspace{1mm}

\item $\HH^{n-4} \big((\mathbf{\Sigma}_{n-1}^{>3}\setminus \mathbf{\Sigma}^*)\cap  \pi_2^{-1}(\{t\}) \big) = 0$  for $\HH^1$-a.e. $t\in \R$; \vspace{1mm}

\item  $\dim_\HH\big(\mathbf{\Sigma}^* \cap  \pi_2^{-1}(\{t\}) \big)\leq n-5$  for $\HH^1$-a.e. $t\in \R$. \vspace{1mm}

\end{itemize}
Thus, part (c) is proved.
Parts (a) and (b) follow exactly in the same way, but using instead Proposition~\ref{prop:GMT4}(a) and Corollary~\ref{prop:GMT4b}(a).
\end{proof}

\begin{remark}
\label{remark:sharp}
Thanks to Remark~\ref{rem-7/2}, one could actually slightly improve the estimate for the set $\mathbf{\Sigma}_{n-1}^{>3}\setminus \mathbf{\Sigma}^*$ and show that $\dim_\HH\big((\mathbf{\Sigma}_{n-1}^{>3}\setminus \mathbf{\Sigma}^*)\cap  \pi_2^{-1}(\{t\}) \big)\leq n-4-\frac12$.
However, all the other estimates are sharp (at least with respect to the techniques introduced in this paper), and in particular we believe that it is very unlikely that one could prove
a stronger result with these techniques. 
\end{remark}

As a consequence of the previous estimates, we finally obtain our main results:

\begin{proof}[Proof of Theorems~\ref{thm-Schaeffer-intro} and~\ref{cor-Hele-Shaw-intro}]
The results follow immedialy from Theorem~\ref{thm:mainell}.
\end{proof}

\appendix
\section{Some results on the Signorini problem}
For the convenience of the reader,
in this appendix we gather some classical results on the Signorini problem \eqref{signorininormal} that we use several times throughout the paper

\begin{lemma}\label{lemAP:1D}
The only 1D solutions to \eqref{signorininormal} that vanish at the origin are given by $q(x_n)= -c|x_n| +bx_n$, for some $c\ge 0$ and $b\in \R$.
\end{lemma}

\begin{proof}
Since  $q=q(x_n)$, it follows from \eqref{signorininormal} that $q$ must be affine in $\R^n\setminus \{0\}$,
hence $q(x_n)=a -c|x_n| +bx_n$ for some $a,b,c\in \R$.
The condition $\Delta q \leq 0$ implies that $c\geq 0$. Also, since $q(0)=0$ we deduce that $a=0$, as desired.
\end{proof}

\begin{lemma}\label{lemAP:2D}
Let $\lambda>0$, and let $i$ denote the imaginary unit.
The only 2D $\lambda$-homogeneous solutions  of \eqref{signorininormal} (i.e.,  $q=q(x_n,x_{n-1})$ and $q(rx)= r^\lambda q(x)$ for every $r>0$) are given by
$$
\left\{
\begin{array}{ll}
q(x_n, x_{n-1}) = c i^{1-\lambda} {\rm Re}(|x_n| + i x_{n-1})^\lambda   + b  {\rm Re}(x_n + i x_{n-1})^\lambda ,& \text{if }\lambda\in \{1,3,5,...\}\\
q(x_n, x_{n-1}) = c i^{\lambda}  {\rm Re}(x_n + i x_{n-1})^\lambda   + b  {\rm Im}(x_n + i x_{n-1})^\lambda ,& \text{if }\lambda\in \{2,4,6,...\},
\\
q(x_n, x_{n-1}) = c {\rm Re}(x_n + i x_{n-1})^\lambda   ,& \text{if }  {\textstyle \lambda\in \left\{\frac{3}{2}, \frac{7}{2}, \frac{11}{2}, ...\right\}},
\end{array}
\right.
$$
where $c\ge 0$ and $b\in \R$.
In particular the set of possible homogeneities is $ \{1,2,3,4,5,...\}\cup
\left\{\frac{3}{2}, \frac{7}{2}, \frac{11}{2}, ...\right\}$.
\end{lemma}

\begin{proof}
A proof of this result can be found, for instance, in \cite[Proposition A.1]{FS17}; see also \cite[Remark 1.2.7]{GP09}.
\end{proof}

The following result is proved in \cite[Lemma 1]{ACS08}.
\begin{lemma}\label{lemAP:Alm}
Let $q$ be a solution of \eqref{signorininormal} and assume that $q(0)=0$. 
Let $\phi(\cdot,q)$ be as in \eqref{eq:Almgren}. Then $r\mapsto \phi(r, q)$ is monotone nondecreasing. Moreover, if $I\ni r \mapsto \phi(r, q)\equiv \lambda>0$ for some open interval $I\subset \R^+$, then $q$ is $\lambda$-homogeneous.
\end{lemma}

We conclude this section with a uniqueness result.

\begin{lemma}\label{lemAP:ordered}
Let $q_i$, $i=1,2$  be two solutions of \eqref{signorininormal}  satisfying $q_1\ge q_2$ in $B_1$ and $q_i(0)=0$.  Assume that $\phi(0^+, q_2)>1$, or that $q_2\equiv 0$. Then $q_1\equiv q_2$.
\end{lemma}

\begin{proof}
We use coordinates $(x',x_n) \in \R^{n-1}\times \R.$
Assume by contradiction that $q_1\not\equiv q_2.$
Then, applying  Hopf's Lemma at the origin, we deduce that $\partial_{x_n}(q_1-q_2)(0,0^+)>0$. Also, our assumption on $q_2$ implies that $\nabla q_2(0,0)=0$, thus $\partial_{x_n}q_1(0,0^+)>0$.
On the other hand, the distributional Laplacian of $q_1$ on $\{x_n=0\}$ is given by $2\partial_{x_n}q_1(x',0^+)$.
Since $\Delta q_1\leq 0$, this gives the desired contradiction.
\end{proof}

\section{Odd frequency points in the Signorini problem}
\label{apb}

The aim of this section is to show how the arguments developed in this paper (see in particular Section~\ref{cubic-sec})
can be applied in the context of the Signorini problem to prove both uniqueness and nondegeneracy of blow-ups at all points of odd frequency for solutions of the Signorini problem
\begin{equation}\label{signoriniB1}
\begin{cases}
\Delta u\le 0\quad \text{and} \quad u\Delta u=0 \quad & \mbox{in }B_1
\\
\Delta u=0  &\mbox{in }B_1\setminus  \{x_n=0\}
\\
\,u\ge 0 &\mbox{on } B_1\cap\{x_n=0\},
\end{cases}
\end{equation}
see Theorem \ref{unique:blow up} below.
Since this was an open problem in this topic which we expect to be of interest to a wide audience, we prefer to give a complete and self-contained proof (rather than referring to parts of this paper) so that this appendix can become of reference for future results on the Signorini problem.

Note that this appendix extends the results of \cite{GP09} (which were dealing only with even frequencies) to the sets $\Gamma_{2m+1}(u)$, $m\in \mathbb N$ (see \cite{GP09} for an explanation of this notation).\smallskip

Since the odd part of a solution of \eqref{signoriniB1} is harmonic and vanishes on $\{x_n=0\}$, to understand the structure of the solution and the free boundary it suffices to study even solutions, that is,
$u(x',x_n)=u(x',-x_n).$
For this reason, also when studying global homogeneous solutions, we can restrict ourselves to even functions.

\smallskip 

We begin by recalling Lemma~\ref{lem:signoriniodd}: If $q$ is a $\lambda$-homogeneous even solution of  \eqref{signoriniB1} and $\lambda=2m+1$ is an odd integer, then $q\equiv 0$ on $\{x_n=0\}$ (this result was not known before).
As a consequence of this fact and the Liouville Theorem for harmonic functions vanishing on a hyperplane\footnote{More precisely, if we consider the odd reflection of $u|_{\{x_n>0\}}$, then we obtain a global $\lambda$-homogeneous harmonic functions in the whole space. By Liouville Theorem, this functions mush be a $\lambda$-homogeneous harmonic polynomial.}, $q$ must be a harmonic polynomial on each side sides of $\{x_n =0\}$.

Then, since $q|_{\{x_n=0\}}=0$, $q$ is even, and $q$ is harmonic in $\R^n\setminus\{x_n=0\}$, we deduce that
\[
q(x',x_n) = -|x_n| \big(q_0(x') + x_n^2 q_1(x', x_n)  \big),\qquad (x',x_n) \in \R^{n-1}\times \R,
\]
where $q_0$ and $q_1$ are polynomials.  Furthermore, since $\Delta q\le 0$, the polynomial $q_0(x')$ is nonnegative.

In the sequel it will be useful to define the ``trace operator'' $T$ as
\begin{equation}\label{defq0}
  q\mapsto  T[q] : = q_0.
\end{equation}
Since $q_0\equiv 0$ implies that $q\equiv 0$ (as a consequence the harmonicity of $q$ outside of $\{x_n=0\}$), one easily deduces that $T$ is injective.

We will need a monotonicity formula that is the analogue of Lemma~\ref{lem:derprod}. 

\begin{lemma}\label{lem:derprodsig}
Let $u:B_1\to \R$ be an even solution of  \eqref{signoriniB1} with $\phi(0^+,u)= \lambda$, where $\lambda$ is an odd integer, and define $u_r(x):= u(rx)$. Also, let $q$ be any $\lambda$-homogeneous even solution of  \eqref{signoriniB1}.
Then, for any $\varrho\in (0,1)$,
\[
\frac{d}{dr} \int_{\partial B_\varrho}  u_r q = \frac{\lambda}{r} \int_{\partial B_\varrho} u_r q - \frac{\varrho}{r} \int_{B_\varrho} u_r \Delta q.
\]
In particular
\[\frac{d}{dr}\left(\frac{1}{r^\lambda}\int_{\partial B_1} u_r q\right) \geq 0.\]
\end{lemma}

\begin{proof}
We have 
\[
\begin{split}
\frac{d}{dr} \int_{\partial B_\varrho}  u_r q  &=  \int_{\partial B_\varrho}  \frac x r \cdot \nabla u_r q   =  \frac{\varrho}{r} \int_{\partial B_\varrho}  \partial_\nu u_r q   =  \frac{\varrho}{r} \int_{B_\varrho}  {\rm div} ( \nabla u_r q)= \frac{\varrho}{r} \bigg( \int_{B_\varrho}  \nabla w_r \nabla q   + \int_{B_\varrho}  \Delta w_r q\bigg)
\\
&= \frac{\varrho}{r} \bigg( \int_{\partial B_\varrho}   u_r \partial_\nu q   -\int_{\partial B_\varrho} u_r \Delta q  + \int_{B_\varrho}  \Delta u_r q\bigg).
\end{split}
\]
Since  $q$ is $\lambda$-homogeneous, we find that $\varrho \int_{\partial B_\varrho}  u_r \partial_\nu q = \lambda\int_{\partial B_\varrho}  u_r q$.
Also, since $q$ vanishes on $\{x_n=0\}$ (by Lemma~\ref{lem:signoriniodd}) and $\Delta u$ is a measure supported on $\{x_n=0\}$, we have  $\int_{B_\varrho}  \Delta u_r q=0$.
This proves the first statement.

Finally, taking $\varrho=1$ and using that $-u_r\Delta q\geq0$ in $\R^n$ (since $\Delta q\le 0$ is supported on $\{x_n=0\}$, and $u_r\ge0$ there) we obtain
\[
\frac{d}{dr} \left(\frac{1}{r^\lambda}\int_{\partial B_1} u_r q\right)  =  -\frac{1}{r^{\lambda+1}} \int_{\partial B_1} u_r \Delta q  \geq 0.
\]
\end{proof}

As a consequence of the previous result, we deduce the following:

\begin{proposition}\label{uniquelambda}
Let $u:B_1\to \R$ be an even solution of  \eqref{signoriniB1} with $\phi(0^+,u)= \lambda$, where $\lambda=2m+1$ is an odd integer.
Then the limit
\[
\tilde q := \lim_{r\downarrow 0} \frac{u(r\,\cdot\,)}{r^\lambda}
\]
exists and it is a $\lambda$-homogeneous even solution of \eqref{signoriniB1}.
\end{proposition}

\begin{proof}
Let
\[
q^{(i)} = \lim_{r^{(i)}_k\downarrow 0} \frac{1}{(r_k^{(i)})^\lambda}  u_{r_k^{(i)}} ,\qquad i=1,2,
\]
be two accumulation points along different sequences $r_k^{(i)}$. Then, given a $\lambda$-homogeneous solution of Signorini $q$, we can apply Lemma~\ref{lem:derprodsig} to deduce that  $r\mapsto \frac{1}{r^\lambda} \int_{\partial B_1}  u_r q $ has a unique limit as $r\to 0.$
In particular this implies that 
$$
\int_{\partial B_1} q^{(1)} q = \int_{\partial B_1} q^{(2)} q,
$$
therefore, choosing $q=q^{(1)}-q^{(2)}$, we deduce that $q^{(1)}\equiv q^{(2)}$.
\end{proof}

The next step consists in showing the following nondegeneracy property:
if $\phi(0^+,u)= \lambda$, then the limit $\tilde q$  obtained in Proposition~\ref{uniquelambda} cannot be identically zero.
This is the most delicate part of this appendix, and the 
the proof of this fact requires a new compactness lemma and an interesting ODE type formula obtained below.

\begin{lemma}\label{lem:compsig}
Let $u:B_1\to \R$ be an even solution of  \eqref{signoriniB1} satisfying $\phi(0^+,u)= \lambda$ with $\lambda$ odd, 
set $u_r(x):= u(rx)$, and let $\tilde u_r := u_r/\|u_r\|_{L^2(\partial B_1)}$.
Given $\eta>0$ there exists  $\delta = \delta(n,\lambda, \eta)$  such that, if  for some $r\in (0,1/2)$ and for some $\lambda$-homogeneous even solution $q$ of  \eqref{signoriniB1} we have
\[
\| \tilde u_r -q\|_{L^\infty(B_2)} \le \delta,
\]
then
\[
\tilde u_r= 0 \qquad \mbox{on } \{x_n=0\} \cap  (B_1\setminus B_{1/2} )\cap \big\{ T[q]  \ge \eta\big\},
\]
where $T[q]$ is defined as in \eqref{defq0}. 
\end{lemma}

\begin{proof}
Fix $z= (z',0)\in (B_1\setminus B_{1/2} )$ such that $T[q](z') \ge \eta$, and given $c>0$ we define
\[
\phi_{c}(x) :=  -(n-1)|x_n|^2 + |x'|^2 + c.
\]
Let  $\varrho>0$ be sufficiently small (depending only on $n$ and  $\eta$) and take $\delta= \varrho^3$. Then,  
for $|x|= \varrho$ we have 
\begin{equation}\label{ahiohaoihaioh}
\begin{split}
u_r(z+x)  \le q(z+x) +\delta &= -|x_n| q_0(z')  + O(\varrho^3) +\delta  \le -\eta |x_n| +O(\varrho^3) +\delta
\\ & \le  -n| x_n|^2 + |x|^2 \le \phi_{c}(x)    \qquad \forall\,c \geq 0.
\end{split}
\end{equation}
Since $\phi_{c} > q(z+\,\cdot\,)$ inside $B_\varrho$  for $c$ large,
we can decrease $c$ until a contact point occur inside $\overline{B_\varrho}$.
Since $\phi_{0}(0) = 0\leq q(z)$ (since $z\in \{x_n=0\}$), we see that such a contact point must occur for some value $c_*\geq 0$.

If $c_*=0$ then we have  $u_r(z) \le \phi_{0}(0) = 0$, as wanted.
Hence, it suffices to show that $c_*>0$ is impossible.

Assume by contradiction that there exists $c_*>0$
such that
$\phi_{c_*} \ge  q(z+\,\cdot\,)$ in $\overline{B_\varrho}$, and 
 $\phi_{c_*}(x_\circ)=  q(z+x_\circ)$ for some  $x_\circ \in \overline{B_\varrho}$.
By  \eqref{ahiohaoihaioh}  we see that $\phi_{c_*}$  and $u_r(z+\,\cdot\,)$ must touch at an interior point, that is $x_\circ\not\in \partial B_\varrho$. 
Also, since $\phi_{c_*}$ is harmonic, it cannot touch $u_r(z+\,\cdot\,)$ at some point where it is harmonic too. Thus, $x_\circ$ must belong to $\{x_n=0\}\cap \{u_r(z+\cdot) =0\}$. 
This gives $0=u_r(z+x_\circ)=\phi_{c_*} (x_\circ) =  |x_\circ|^2 + c_*>0$, a contradiction.
\end{proof}

Another fundamental tool is the following ODE-type formula.

\begin{lemma}\label{lem:ODEalessio11}
Let $u:B_1\to \R$ be an even solution of  \eqref{signoriniB1} satisfying $\phi(0^+,u)= \lambda$, with $\lambda$ odd. 
Set $u_r(x):= u(rx)$, $h(r): = \|u_r\|_{L^2(\partial B_1)}$, and $\tilde u_r := u_r/h(r)$. 
Let  $q$ be an even $\lambda$-homogeneous solution of  \eqref{signoriniB1}, and define
\begin{equation}\label{psi-alessio2}
\psi(r;q): =  \int_{\partial B_1}  \tilde u_r q  - 2\int_{\partial B_{1/2}}  \tilde u_r q .
\end{equation}
Then
\[
\frac{d}{dr} \psi(r;q)  = -\theta(r)  \psi(r;q)  - \frac{1}{r} \int_{B_1\setminus B_{1/2}} \tilde u_r \Delta q,
\qquad
\text{where}\quad
\theta(r) : =  \bigg(\frac{h'(r)}{h(r)} +\frac \lambda r\bigg) = \big(  \log(h(r)/r^\lambda)\big)'.
\]
\end{lemma}

\begin{proof}
As in the proof of Lemma~\ref{lem:derprodsig}, we obtain
\[
\frac{d}{dr} \int_{\partial B_\varrho}  u_r q = \frac{\lambda}{r} \int_{\partial B_\varrho} u_r q - \frac{\varrho}{r} \int_{B_\varrho} u_r \Delta q.
\]
Thus, since $\tilde u_r = u_r/h(r)$, we deduce that
\[
\frac{d}{dr} \int_{\partial B_\varrho}  \tilde u_r q  = \bigg( - \frac {h'(r)}{h(r)} + \frac{\lambda}{r}\bigg) \int_{\partial B_\varrho} \tilde u_r q - \frac{\varrho}{r} \int_{B_\varrho} \tilde u_r \Delta q ,
\]
and the lemma follows combining the identities for $\varrho =1$ and  $\varrho =1/2$.
\end{proof}

In the sequel, for $\lambda=2m+1$, $m\in \mathbb N$, we denote 
\[
\mathcal Q_\lambda := \{\mbox{ even $\lambda$-homogeneous solutions of \eqref{signoriniB1} }\}.
\]
Also, given
$f\in L^1_{\rm loc}(\R^n) $,  we define the radial symmetrization in the first $(n-1)$ variables as 
\begin{equation}
\label{eq:symm}
\widehat f(x', x_n) : =   \ave_{{\rm SO}(n-1) } f(Mx', x_n) \,dM,\qquad x=(x',x_n) \in \R^{n-1}\times \R,
\end{equation}
where the previous average is with respect to the Haar meaure of ${\rm SO}(n-1)$.

\begin{lemma}\label{lem:QQQ}
Given $\lambda\ge 3$ odd,  there exists a unique $Q\in \mathcal Q_\lambda$ 
satisfying
\begin{equation}\label{hbaiohaiohaio}
Q=\widehat Q
\qquad \mbox{ and  }\quad  \|Q\|_{L^2(\partial B_1)}=1.
\end{equation} 
Moreover, for any other $q\in \mathcal Q_\lambda$  we have 
\[
\int_{ \partial B_1} q Q \ge c_{n,\lambda}  \|q\|_{L^2(\partial B_1)} >0
\]
where $c_{n,\lambda}$ is some positive constant depending only on $n$ and $\lambda$.
\end{lemma}
\begin{proof}
We begin by noticing that $Q=\widehat Q$ belongs to  $\mathcal Q_\lambda$ if and only if
\begin{equation}\label{ahoihaoihia}
Q(x)=\sum_{k=0}^{\frac{\lambda-1}{2}} a_k |x'|^{\lambda-1-2k} |x_n|^{1+2k},\qquad a_0 \leq 0,\qquad
\Delta \bigg( \sum_{k=0}^{\frac{\lambda-1}{2}} a_k |x'|^{\lambda-1-2k} x_n^{1+2k} \bigg) =0.
\end{equation}
Setting $r':=|x'|$ and noticing that
$\Delta Q= \partial_{r'r'}Q + \frac{n-2}{r'}\partial_{r'}Q + \partial_{x_nx_n}Q$, we can rewrite \eqref{ahoihaoihia} as
\[
\sum_{j=0}^{\frac{\lambda-3}{2}}  \Big(a_j (\lambda-1-2j) (\lambda -2-2j+n-2) + a_{j+1}  (3+2j)(2+2j) \Big) (r')^{\lambda-3-2j} x_n^{1+2j}=0.
\] 
Therefore, \eqref{ahoihaoihia} is satisfied if and only if 
\[
a_j (\lambda-1-2j) (\lambda -2-2j+n-2) + a_{j+1}  (3+2j)(2+2j)   =0 \qquad \mbox{for all } j=0,1, \dots {\textstyle \frac{\lambda-3}{2}}.
\]
This means that all the coefficients are uniquely determined (by induction over $j$) once $a_0\leq 0$ is fixed,
and $a_{0} < 0$ is uniquely determined imposing that  $\|Q\|_{L^2(\partial B_1)}=1$.  
This concludes the first part of the proof.

To prove the second part, note that if $q\in  \mathcal Q_\lambda$ then $ \widehat q\in \mathcal Q_\lambda$ (see \eqref{eq:symm}).
We now recall that, to define the trace operator $T$, we used the expansion
\[
q(x',x_n) = -|x_n| \big(q_0(x') + x_n^2 q_1(x', x_n)  \big),\qquad \text{so that }T[q]=q_0.
\]
Since $T[q]\equiv 0$ implies $q\equiv 0$, it follows by compactness that 
$$
\big\|T[q]\big\|_{L^2(\partial B_1)}\ge \tilde c_{n,\lambda}\|q\|_{L^2(\partial B_1)}\qquad \forall\,q\in\mathcal Q_\lambda,\,\text{ for some $c_{n,\lambda}>0$}.
$$
Also
\[
\widehat   q(x',x_n) = -|x_n| \big(\widehat{q_0}(x') + x_n^2 \widehat{q_1}(x', x_n)  \big),\qquad \text{that is}\quad T[\widehat q] = \widehat{T[q]}.
\]
Thus, given $q\in \mathcal Q_\lambda$, since $\widehat  q\in\mathcal Q_\lambda$ depends only on the variables $r'=|x'|$ and $x_n$,
it follows from the first part of the proposition that $\widehat q$ must by a positive multiple of $Q$, that is, $\widehat q= t Q$, where $t\ge c_{n, \lambda}\|q\|_{L^2(\partial B_1)}>0$. Hence, since $\widehat  Q = Q$ and using the invariance of the Haar measure $dM$ on ${\rm SO}(n-1)$ under the transformation $M\mapsto M^{-1}$, we get
\[
\begin{split}
\int_{ \partial B_1} q Q   &= \int_{ \partial B_1} q \widehat Q  =\int_{ \partial B_1}dx \ave_{{\rm SO}(n-1) } q(x',x_n)\,Q(Mx', x_n) \,dM\\
&=\int_{ \partial B_1}dx \ave_{{\rm SO}(n-1) } q(M^{-1}x', x_n)\,Q(x',x_n) \,dM\\
&=\int_{ \partial B_1}dx \ave_{{\rm SO}(n-1) } q(Mx', x_n)\,Q(x',x_n) \,dM
= \int_{ \partial B_1} \widehat q Q = t  \int_{ \partial B_1} Q^2   \ge c_{n,\lambda}\|q\|_{L^2(\partial B_1)}.
\end{split}
\]
\end{proof}

In the following proposition we will use the notation $X\asymp Y$ for $X\le C(n,\lambda) Y$ and $Y \le C(n,\lambda) X$.

\begin{proposition}\label{prop:nondeglambda}
Let $u:B_1\to \R$ be an even solution of  \eqref{signoriniB1} with $\phi(0^+,u)= \lambda$,  where $\lambda$ is an odd integer. Suppose that $\|u\|_{L^2(\partial B_1)} =1$, and set $u_r(x):=u(rx).$ 
Then
\[0< c r^{\lambda} \leq \|u_r\|_{L^2(\partial B_1)} \leq r^{\lambda}\qquad \forall\, r\in (0,1],\]
where $c$ depends only on $n$ and $\lambda$.
\end{proposition}

\begin{proof}
The inequality $\|u_r\|_{L^2(\partial B_1)} \le r^\lambda$ follows from the fact that $r^{-2\lambda} H(r,u)$ is monotone nondecreasing  since $\phi(r,u)\geq  \lambda$ (see \cite[Lemma 2]{ACS08}). We need to show the bound from below (the nondegeneracy). 

Define 
\begin{equation}\label{themax11}
\Psi(r): = \max\big\{   \psi(r;q) \ : \  q\in \mathcal Q_\lambda \mbox{ and } \|q \|_{L^2(\partial B_1)}=1\big\},
\end{equation}
where $\psi$ is given by \eqref{psi-alessio2},
and let $q_r^*$ be the function at which the above maximum is attained (note $ \mathcal Q_\lambda$ is a closed convex subset of a finite dimensional vector space). 
Also, let $Q$ be as in Lemma~\ref{lem:QQQ}, and define $\Phi(r) : = \psi(r, Q)$.
Then, as a consequence of Lemma~\ref{lem:ODEalessio11}, we have 
\[
\frac{d}{dr} \Psi(r)  = \theta(r)  \Psi(r)   - \frac{1}{r} \int_{B_1\setminus B_{1/2}} \tilde u_r \Delta q^*_r  \qquad \mbox{ for a.e. } r>0,
\]
and
\begin{equation}\label{reugheirhn2}
\frac{d}{dr} \Phi(r)  = \theta(r)  \Phi(r)   - \frac{1}{r} \int_{B_1\setminus B_{1/2}} \tilde u_r \Delta Q\qquad \forall\,r>0.
\end{equation}
We now claim that 
\[
\Psi(r)  \asymp  \Phi(r)  \asymp \frac{\Psi(r) }{\Phi(r)}\asymp  1 \qquad \mbox{as} \quad r\downarrow 0,
\]
Indeed, the accumulation points of $\tilde u_r$ (as $r\downarrow 0$ and in the $C^0_{\rm loc}(\R^n)$ topology) belong to the unit ball of  $\mathcal Q_\lambda$ (see \cite{ACS08}) and therefore $\tilde u_r -q_r =o(1)$ for some $q_r\in \mathcal Q$.
Hence, by definition of $\Psi$,
\[
\begin{split}
\Psi(r) \ge  \psi(r;q_r) &=  \int_{\partial B_1}\tilde u_r q_{r}  - 2\int_{\partial B_{1/2}}  \tilde u_r q_{r}
=   \int_{\partial B_{1}}  q_{r}^2  - 2\int_{\partial B_{1/2}}  q_{r}^2 +o(1) 
\\
& =   (1-2^{-n-1-2\lambda})\int_{\partial B_{1}}  q_{r}^2 + o(1)  \ge \frac 1 2 >0.
\end{split}
\]
Note that the above computation shows also that
$\psi(r,q) = (1-2^{-n-1-2\lambda})\int_{\partial B_{1}} q_r q+o(1)$,
from which it follows that
 $q_r^* = q_r+o(1)$ as $r\downarrow 0$
(recall that  $q_r^*$ is a maximizer in \eqref{themax11}).

Similarly,  using   Lemma~\ref{lem:QQQ},
\[
\begin{split}
\Phi(r) &= \int_{\partial B_{1}}  \tilde u_r Q  - 2\int_{\partial B_{1/2}}  \tilde u_r Q
=   \int_{\partial B_{1}}  q_{r}Q  - 2\int_{\partial B_{1/2}}  q_{r}Q+o(1)
\\
& \ge c_{n,\lambda}    (1-2^{-n-1-2\lambda})  +o(1) \geq \frac{c_{n,\lambda}}2>0,
\end{split}
\]
where $c_{n,\lambda}$ is the constant from Lemma~\ref{lem:QQQ}. 
Finally, it is clear that $\Psi(r)$ and $\Phi(r)$ are bounded by above, so the claim is proved.

Using the expressions for $\frac{d}{dr}\Psi$ and $\frac{d}{dr}\Phi$, we find
\[
\frac{d}{dr}  \bigg( \frac{\Psi(r)}{\Phi(r)}\bigg)  =   -\frac{1}{r}  \frac{ \Psi(r) \int_{B_1\setminus B_{1/2}} \tilde u_r \Delta q_r^*  - \Phi(r) \int_{B_1\setminus B_{1/2}} \tilde w_r \Delta Q}{\Phi(r)^2}
\]
We claim that, given $\varepsilon>0$, for $r$ sufficiently small,
\begin{equation}\label{ggooaall11}
\left|\int_{B_1\setminus B_{1/2}} \tilde u_r \Delta q_r^*\right| \leq \varepsilon \left|\int_{B_1\setminus B_{1/2}} \tilde u_r \Delta Q\right|.
\end{equation}
Indeed,  
introducing the notation $B_r':=B_r\cap\{x_n=0\}$ and
using  Lemma~\ref{lem:compsig}, given $\eta>0$ and choosing $r>0$ is sufficiently small so that $\|\tilde u_r - q_r^*\|_{L^\infty(B_2)} \le \delta (n,\eta)$ (recall that  $q_r^* = q_r+o(1)$ as $r\downarrow 0$), we have
\[\begin{split}
0\leq -\int_{B_1\setminus B_{1/2}}  u_r \Delta q_{r}^* & =  \int_{B'_1\setminus B'_{1/2}}  u(rx',0) T[q_r^*] \,dx'
=  \eta \int_{(B'_1\setminus B'_{1/2})\cap \{T[q_r^*] \leq \eta\}} u_r \,dx'\le  \eta \int_{B'_1\setminus B'_{1/2}}  u_r\,dx'
\end{split}\]
(recall that $u_r\geq 0$ on $\{x_n=0\}$),
while 
\[
-\int_{B_1\setminus B_{1/2}}  u_r \Delta Q = 2|a_0| \int_{B'_1\setminus B'_{1/2}}  u(r x',0)  |x'|^{\lambda-1} \,dx'  \ge c_{n,\lambda}' \int_{B'_1\setminus B'_{1/2}}  u_r\,dx',
\]
for some constant $c_{n,\lambda}'>0$.
Hence, dividing by $h(r)$, we obtain 
\[
0\leq -\int_{B_1\setminus B_{1/2}}  \tilde u_r \Delta q_{r} \le C_{n,\lambda}\eta \int_{B_1\setminus B_{1/2}}  \tilde u_r \Delta Q,
\]
and  \eqref{ggooaall11} follows.

Then, thanks to \eqref{ggooaall11}, we deduce that
\[
\frac{d}{dr} \bigg( \frac{\Psi(r)}{\Phi(r)}\bigg)   =   -\frac{1}{r}  \frac{ \Psi(r) \int_{B_1\setminus B_{1/2}} \tilde u_r \Delta q_r^*  - \Phi(r) \int_{B_1\setminus B_{1/2}} \tilde u_r \Delta Q}{\Phi(r)^2}  \asymp \frac{1}{r} \int_{B_1\setminus B_{1/2}} \tilde u_r \Delta Q
\]
for $r\leq r_0$ small enough.

Integrating the above ODE over $[\hat r,r_0]$, since the integral of $\frac{d}{dr} \big(\frac{\Psi(r)}{\Phi(r)}\big)$ over $[\hat r, r_0]$ is uniformly bounded as $\hat r \to 0$,
we deduce that the negative term $\frac1r \int_{B_1\setminus B_{1/2}} \tilde u_r \Delta Q$ is integrable over $[0,r_0].$
Hence, since $\Phi(r)\asymp 1$ and $\theta(r)= \frac{d}{dr} \log (h(r)/r^\lambda)$, it follows from \eqref{reugheirhn2} that
$$
\frac{d}{dr}\log \Phi(r)= \frac{d}{dr} \log (h(r)/r^\lambda)+g(r),\qquad \text{with $g \in L^1([0,r_0])$}.
$$
Integrating over $[\hat r,r_0]$ and using again that $\Phi(r)\asymp 1$, we deduce that $\log (h(\hat r)/\hat r^\lambda)$ is uniformly bounded as $\hat r \to 0$,
therefore $h(r)\asymp r^{\lambda}$, as desired.
\end{proof}

As a consequence of  Propositions~\ref{uniquelambda} and~\ref{prop:nondeglambda}, 
we get the the uniqueness and nondegeneracy of blow-ups:

\begin{theorem}\label{unique:blow up}
Let $u:B_1\to \R$ be an even solution of  \eqref{signoriniB1} with $\phi(0^+,u)= \lambda$, where $\lambda=2m+1$ is an odd integer.
Then the limit
\[
\tilde q := \lim_{r\downarrow 0} \frac{u(r\,\cdot\,)}{r^\lambda}
\]
exists, is non-zero, and it is a $\lambda$-homogeneous even solution of \eqref{signoriniB1}.
\end{theorem}
Thanks to this result,
by classical arguments (see Proposition \ref{prop:recti} and Lemma \ref{dimred})
one easily obtains the following rectifiability result, that was already proved with completely different methods in \cite{FS17}:

\begin{corollary}
Let $u:B_1\to \R$ be an even solution of  \eqref{signoriniB1}. 
Then, for any odd integer $\lambda \ge3$, the set of free boundary points of frequency $\lambda$ is $(n-2)$-rectifiable.
\end{corollary}


\begin{thebibliography}{000000}




\bibitem[ALS13]{ALS} J. Andersson, E. Lindgren, H. Shahgholian, \emph{Optimal regularity for the no-sign obstacle problem}, Comm. Pure Appl. Math. \textbf{66} (2013), 245--262.





\bibitem[ACS08]{ACS08} I. Athanasopoulos, L. Caffarelli, S. Salsa, \emph{The structure of the free boundary for lower dimensional obstacle problems}, Amer. J. Math. \textbf{130} (2008) 485--498.

\bibitem[Bai74]{Baiocchi} C. Baiocchi, \emph{Free boundary problems in the theory of fluid flow through porous media}, in Proceedings of the ICM 1974.



\bibitem[BK74]{BK74} H. Br\'ezis, D. Kinderlehrer, \emph{The smoothness of solutions to nonlinear variational inequalities}, Indiana Univ. Math. J. \textbf{23} (1973/74), 831--844.


\bibitem[BZ80]{BZ} Y. D. Burago, V. A. Zallager, \emph{Geometric Inequalities}, Grundlehren der mathematischen Wissenschaften 285, Springer-Verlag, 1980.


\bibitem[Caf77]{C-obst} L. Caffarelli, \emph{The regularity of free boundaries in higher dimensions}, Acta Math. \textbf{139} (1977), 155--184.



\bibitem[Caf98]{C-obst2} L. Caffarelli, \emph{The obstacle problem revisited}, J. Fourier Anal. Appl. \textbf{4} (1998), 383--402.

\bibitem[CC95]{CC95} L. Caffarelli, X. Cabr\'e, {\it Fully Nonlinear Elliptic Equations.} American Mathematical Society Colloquium Publications, 43. American Mathematical Society, Providence, RI, 1995.

\bibitem[CKS00]{CKS00} L. Caffarelli, L. Karp,  H. Shahgholian, \emph{Regularity of a free boundary with applications to the Pompeiu problem}, Ann. Math. \textbf{151} (2000), 269--292.


\bibitem[CR76]{CR76} L. Caffarelli, N. Rivi\`ere, \emph{Smoothness and analyticity of free boundries in variational inequalities}, Ann. Scuola Norm. Sup. Pisa Cl. Sci. \textbf{3} (1976), 289--310.

\bibitem[CR77]{CR77} L. Caffarelli, N. Rivi\`ere,  \emph{Asymptotic behavior of free boundaries at their singular points}, Ann. of Math. \textbf{106} (1977), 309--317.

\bibitem[CSS08]{CSS08}
L. Caffarelli, S. Salsa, L. Silvestre, \emph{Regularity estimates for the solution and the free boundary of the
obstacle problem for the fractional Laplacian}, Invent. Math. \textbf{171} (2008), 425--461.


\bibitem[CJK07]{CJK} S. Choi, D. Jerison, I. Kim, \emph{Regularity for the one-phase Hele-Shaw problem from a Lipschitz initial surface}, Amer. J. Math. \textbf{129} (2007), 527--582.

\bibitem[CSV18]{CSV17} M. Colombo, L. Spolaor, B. Velichkov, \emph{A logarithmic epiperimetric inequality for the obstacle problem}, Geom. Funct. Anal. \textbf{28} (2018), 1029--1061.

\bibitem[DL76]{DL} G. Duvaut, J. L. Lions, \emph{Inequalities in Mechanics and Physics}, Springer, 1976.



\bibitem[Eva10]{Eva10} L. C.
Evans, \emph{Partial differential equations. Second edition.} Graduate Studies in Mathematics, 19. American Mathematical Society, Providence, RI, 2010. xxii+749 pp.


\bibitem[Fed69]{Fed69} H. Federer, \emph{Geometric Measure Theory}, Springer-Verlag, 1969.

\bibitem[Fef09]{Fef09} C. Fefferman, \emph{Extension of $C^{m,\omega}$-smooth functions by linear operators}, Rev. Mat. Iberoam. \textbf{25} (2009), 1--48.

	
\bibitem[FS19]{AlessioJoaquim} A. Figalli, J. Serra, \emph{On the fine structure of the free boundary for the classical obstacle problem}, Invent. Math. \textbf{215} (2019), 311--366.

\bibitem[FS18]{FS17} M. Focardi, E. Spadaro, \emph{On the measure and structure of the free boundary of the lower dimensional obstacle problem}, Arch. Rat. Mech. Anal. \textbf{230} (2018), 125--184.

\bibitem[Fri82]{obst-appl4} A. Friedman, \emph{Variational Principles and Free Boundary Problems}, Wiley, New York, 1982.

\bibitem[HS1898]{HS} H. S. H. Shaw, \emph{Investigation of the Nature of Surface Resistance of Water and of Stream-line Motion Under Certain Experimental Conditions}, Inst. NA. 1898.

\bibitem[GP09]{GP09} N. Garofalo, A. Petrosyan, \emph{Some new monotonicity formulas and the singular set in the lower dimensional obstacle problem}, Invent. Math. \textbf{177} (2009), 415--461.


\bibitem[Giu84]{Giusti} E. Giusti, \emph{Minimal Surfaces and Functions of Bounded Variation}, Monographs in Mathematics 80, Birkh\"auser, 1984.

\bibitem[Kin73]{Kin73} D. Kinderlehrer, \emph{How a minimal surface leaves an obstacle}, Acta Math. \textbf{130} (1973), 221--242.

\bibitem[KN77]{KN77} D. Kinderlehrer, L. Nirenberg, \emph{Regularity in free boundary problems}, Ann. Sc. Norm. Sup. Pisa \textbf{4} (1977), 373--391.


\bibitem[LS69]{LS69} H. Lewy, G. Stampacchia, \emph{On the regularity of the solution of a variational inequality}, Comm. Pure Appl. Math. \textbf{22} (1969), 153--188.



\bibitem[LS67]{LS} J. L. Lions, G. Stampacchia, \emph{Variational inequalities}, Comm. Pure Appl. Math. \textbf{20} (1967), 493-519.






\bibitem[Mon03]{M03} R. Monneau, \emph{On the number of singularities for the obstacle problem in two dimensions}, J. Geom. Anal. \textbf{13} (2003), 359--389.

\bibitem[NV17]{NV17} A. Naber, D. Valtorta, \emph{Volume estimates on the critical sets of solutions to elliptic PDEs}, Comm. Pure Appl. Math. \textbf{70} (2017), 1835--1897.

\bibitem[PSU12]{PSU12} A. Petrosyan, H. Shahgholian,  N. Uraltseva, \emph{Regularity of Free Boundaries in Obstacle-Type Problems}, Graduate Studies in Mathematics, Vol. 136., AMS, 2012.

\bibitem[Rod87]{obst-appl1} J. F. Rodrigues, \emph{Obstacle Problems in Mathematical Physics}, North-Holland Mathematics Studies, vol. 134, North-Holland Publishing Co., Amsterdam, 1987.


\bibitem[Sak91]{Sak91} M. Sakai, \emph{Regularity of a boundary having a Schwarz function}, Acta Math. \textbf{166} (1991), 263--297.

\bibitem[Sak93]{Sak93} M. Sakai, \emph{Regularity of free boundaries in two dimensions}, Ann. Scuola Norm. Sup. Pisa Cl. Sci.  \textbf{20} (1993), 323--339.

\bibitem[Sch74]{Sch1} D. G. Schaeffer, \emph{An example of generic regularity for a nonlinear elliptic equation}, Arch. Rat. Mech. Anal. \textbf{57} (1974), 134--141.

\bibitem[Sch76]{Sch76} D. Schaeffer, \emph{Some examples of singularities in a free boundary}, Ann. Scuola Norm. Sup. Pisa \textbf{4} (1976), 131--144.

\bibitem[Ser15]{Serfaty} S. Serfaty, \emph{Coulomb Gases and Ginzburg-Landau Vortices}, Zurich Lectures in Advanced Mathematics, EMS books, 2015.


\bibitem[SU03]{SU} H. Shahgholian, N. Uraltseva, \emph{Regularity properties of a free boundary near contact points with the fix boundary}, Duke Math. J. \textbf{116} (2003), 1--34.


\bibitem[Sim83]{Sim83} L. Simon, \emph{Lectures on Geometric Measure Theory.} Proc. Centre for Mathematical Analysis, Australian National University, Canberra, 1983.


\bibitem[Sma93]{Smale} N. Smale, \emph{Generic regularity of homologically area minimizing hypersurfaces in eight dimensional manifolds}, Comm. Anal. Geom. \textbf{1} (1993), 217--228.

\bibitem[W99]{W99} G. Weiss, \emph{A homogeneity improvement approach to the obstacle problem}, Invent. Math. \textbf{138} (1999), 23--50.


\bibitem[Whi97]{Whi97} B. White, \emph{Stratification of minimal surfaces, mean curvature flows, and harmonic maps}, J. Reine Angew. Math. \textbf{488} (1997), 1--35.	


\end{thebibliography}
\end{document}